\documentclass[11pt,a4paper]{amsart}
\usepackage[utf8]{inputenc}
\usepackage[english]{babel}
\usepackage{amsmath}
\usepackage{amsfonts}
\usepackage{amssymb}
\usepackage{graphicx}
\usepackage{stackengine}
\usepackage{fullpage}
\usepackage{hyperref}

\numberwithin{equation}{section}

\newtheorem{thm}{Theorem}[section]
\newtheorem{lemma}[thm]{Lemma}

\newtheorem{prop}[thm]{Proposition}
\newtheorem{cor}[thm]{Corollary}
\newtheorem{claim}[thm]{Claim}
\newtheorem{conj}[thm]{Conjecture}
\theoremstyle{remark}
\newtheorem{rem}[thm]{Remark}

\newcommand{\ind}{\mathbf 1}
\newcommand{\C}{\mathbb C}
\newcommand{\D}{\mathbb D}
\newcommand{\E}{\mathbb E}
\newcommand{\N}{\mathbb N}
\newcommand{\R}{\mathbb R}
\newcommand{\Z}{\mathbb Z}
\newcommand{\PP}{\mathbb P}

\newcommand{\LC}{\mathcal C}

\newcommand{\LL}{\mathcal L}
\newcommand{\Ns}{\mathsf N}
\newcommand{\LK}{\mathcal K}

\newcommand{\LP}{\mathcal P}

\newcommand{\He}{\mathtt H}

\newcommand{\ED}{\operatorname{ED}}

\newcommand{\Var}{\operatorname{Var}}

\newcommand{\diam}{\operatorname{diam}}

\newcommand{\CR}{\operatorname{CR}}
\newcommand{\Supp}{\operatorname{Supp}}
\newcommand{\cst}{\mathtt{C}}
\newcommand{\clus}{\mathcal{C}}
\newcommand{\inter}{\operatorname{Int}}
\newcommand{\Gen}{\operatorname{Gen}}
\newcommand{\fld}{\mathtt{h}}

\renewcommand{\Re}{\operatorname{Re}}
\renewcommand{\Im}{\operatorname{Im }}

\title{Relation between Wick powers and excursion clusters of the 2D GFF}

\author{Titus Lupu}
\address {CNRS and LPSM, UMR 8001,
Sorbonne Université,
4 place Jussieu,
75252 Paris cedex 05,
France}
\email
{titus.lupu@sorbonne-unversite.fr}


\begin{document}

\begin{abstract}
We study the decomposition of the Wick powers of the continuum GFF in dimension $2$ via the first passage sets (FPS) and the excursion clusters (sign components) of the GFF.
These sets are non-thin for the GFF, that is to say the field has non-trivial restriction to such a set,
which is a measure, negative or positive depending on the sign.
In this work we show that all the odd Wick powers of the GFF can be restricted to the FPS and the excursion clusters,
and the restrictions are generalized functions supported on these fractal sets.
By contrast, the restriction of an even Wick power to an FPS or excursion cluster is diverging,
and to get something converging an additional compensation is required, which is provided by a smooth function living outside of the set and blowing up in a non-integrable way when approaching the set.
We further provide expressions of restricted odd Wick powers and
restricted-compensated even Wick powers as limits of functions living outside the FPS/excursion cluster.

Then, we study the $\varepsilon$-neighborhoods, in the sense of conformal radius, of first passage sets and excursion clusters.
We show that such $\varepsilon$-neighborhoods admit asymptotic expansions in $L^2$ into
half-integer powers
$\vert\log \varepsilon\vert^{-(n+1/2)}$, $n\in\N$, of
$1/\vert\log \varepsilon\vert$.
The coefficients of the expansion involve the restrictions of the odd Wick powers.
By contrast, the even Wick powers do not appear in the expansion.
Our expansion is reminiscent of Le Gall's expansion for the Wiener sausage in dimension 2,
with however some important differences.
The most important one is that the powers of $1/\vert\log \varepsilon\vert$ are different.
In the case of the Wiener sausage the powers are integer,
$\vert\log \varepsilon\vert^{-n}$, $n\in\N\setminus \{0\}$.
\end{abstract}

\maketitle

\tableofcontents

\section{Introduction}
\label{Sec intro}

In a series of works \cite{ALS1,ALS2,ALS4}, Aru, Lupu and Sep\'ulveda 
developed the notion of connected components of the level sets of the continuum Gaussian free field (GFF) in dimension 2.
Despite the continuum GFF $\Phi$ being a generalized function, not defined pointwise, one can still construct such
connected components of the level sets by iterating SLE$_{4}$ type processes.
Specifically, Aru, Lupu and Sep\'ulveda constructed the \textit{first passage sets}  (FPS)
and the \textit{excursion clusters} (or excursion sets or sign clusters) of the GFF.
Informally, an FPS of level $a$ can be thought of as the set of points in the domain that are a connected
to the boundary by a continuous path along which the GFF is above the value $a$.
Excursion sets are a closely related notion and can be thought of as connected components of constant sign for 
$\Phi$.
These informal descriptions make precise sense by taking the scaling limit of the GFF on metric graphs.
The FPS and the excursion sets have $0$ Lebesgue measure. 
More precisely, their $\varepsilon$-neighborhoods have an area of order
$\vert\log\varepsilon\vert^{-1/2}$.
However, these sets are non-thin for the GFF,
that is to say, the GFF admits a non-trivial restriction to these sets.
Depending on the sign, such a restriction is a positive or negative measure,
more precisely a Minkowski content measure in the gauge
$\vert\log r\vert^{1/2} r^{2}$.

\bigskip

In this work we investigate how the Wick powers $:\Phi^{n}:$ of the GFF relate to the FPS and the excursion clusters.
For $n\geq 2$, one cannot directly define the powers $\Phi^{n}$,
and a renormalization procedure involving Hermite polynomials is required.
In particular, because of the compensation terms,
the even powers $:\Phi^{2k}:$ are not positive, are not measures, but are more complicated generalized functions.
We start by considering an FPS $A$ of $\Phi$, and we take the conditional expectation field
\begin{displaymath}
\psi_{n,A} = \E\big[:\Phi^{n}:\vert A\big]\,.
\end{displaymath}
In Section \ref{Sec Wick FPS}
we investigate both the nature of the field $\psi_{n,A}$
and identify the difference $:\Phi^{n}:\,-\psi_{n,A}$
(Theorem \ref{Thm decomp FPS Wick}).
One remarkable point is that the odd powers $\psi_{2k+1,A}$ and the even powers $\psi_{2k,A}$
behave differently (Corollary \ref{Cor support psi}).
The odd powers $\psi_{2k+1,A}$ are generalized functions supported on $A$ itself,
and are actually restrictions of $:\Phi^{2k+1}:$ to $A$
(Corollary \ref{Cor restr A}).
The fact that such restrictions exist is already a remarkable fact, since in general one cannot restrict a generalized
function to an arbitrary closed subset. 
But the even powers $\psi_{2k,A}$ behave differently.
They live both on $A$ and on the complementary
$D\setminus A$ of $A$ in the domain $D$.
On $D\setminus A$, $\psi_{2k,A}$ coincides with a smooth function obtained through the ratio of conformal radii:
\begin{displaymath}
(-1)^{k}\dfrac{(2k)!}{2^{k} k!}
\Big(\dfrac{1}{2\pi}\log\Big(
\dfrac{\CR(z,D)}{\CR(z,D\setminus A)}
\Big)\Big)^{k}.
\end{displaymath}
However, one cannot separate $\psi_{2k,A}$ into a part on $A$ and a part on $D\setminus A$,
as the function above, blowing up near $A$, is not integrable on $D\setminus A$.
In particular, the difference
\begin{displaymath}
\psi_{2k,A}-\ind_{D\setminus A}(-1)^{k}\dfrac{(2k)!}{2^{k} k!}
\Big(\dfrac{1}{2\pi}\log\Big(
\dfrac{\CR(z,D)}{\CR(z,D\setminus A)}
\Big)\Big)^{k}
\end{displaymath}
does not make sense.
This phenomenon is reminiscent, in a more sophisticated setting,
of the finite part of $1/x^{2}$,
$\operatorname{F.P.}(x^{-2})$.
There one compensates the non-integrable function $1/x^{2}$ on $\R$
with a Dirac in $\{0\}$ with an infinite negative mass.
The end result, $\operatorname{F.P.}(x^{-2})$,
can be tested against any smooth compactly supported function on $\R$,
coincides on $\R\setminus \{0\}$ with $1/x^{2}$,
but cannot be separated into a function on $\R$ and some generalized function on $\{0\}$.
See Remark \ref{Rem non sep}.

The fields $\psi_{n,A}$, for $n$ both odd and even,
admit explicit approximations through smooth functions
living outside $A$.
In Theorem \ref{Thm psi GMC} we provide such an explicit expression,
obtained through the germs of the Gaussian multiplicative chaos at the
intermittency parameter $\gamma$ tending to $0$.
We also state it below.
Let $\He_n$ denote the $n$-th Hermite polynomial,
and let $V_A$ be the function on $D\setminus A$ given by
\begin{displaymath}
V_A(z) =
\dfrac{1}{2\pi}\log\Big(
\dfrac{\CR(z,D)}{\CR(z,D\setminus A)}
\Big).
\end{displaymath}

\begin{thm}[Theorem \ref{Thm psi GMC}]
\label{Thm psi GMC intro}
Let $n\geq 1$.
Then the field $\psi_{n,A}$ is the limit, as $\gamma\to 0^{+}$,
of the following continuous function supported of $D\setminus A$:
\begin{displaymath}
\ind_{D\setminus A}
\He_{n}(\gamma V_{A}^{1/2})V_{A}^{n/2} \exp\Big(-\dfrac{\gamma^{2}}{2} V_{A}\Big).
\end{displaymath}
The convergence holds both in
$L^{2}(d\PP,\sigma(\Phi),H^{-\eta}(\C))$ 
and almost surely in $H^{-\eta}(\C)$ (for the Sobolev norm) for every $\eta>0$.
\end{thm}

\bigskip

We also consider the $\varepsilon$-neighborhoods of $A$ defined in terms of conformal radii:
\begin{displaymath}
\Ns_{\varepsilon}(A) = \{ z\in D\setminus A \vert \CR(z,D\setminus A)<\varepsilon \CR(z,D)\}.
\end{displaymath}
We derive an asymptotic expansion, in the functional sense,
of the indicator function $\ind_{\Ns_{\varepsilon}(A)}$, as $\varepsilon\to 0$.
The leading term is of course $\vert\log\varepsilon\vert^{-1/2}$
times the Minkowski content measure of $A$ in the gauge $\vert\log r\vert^{1/2} r^{2}$.
With the notations above, the Minkowski content measure is also $\psi_{1,A}$.
We show that the higher order correction terms come into half-integer powers
$\vert\log\varepsilon\vert^{-(k+1/2)}$, $k\in\N$,
and involve the restricted odd Wick powers $\psi_{2k+1,A}$.
By contrast, the even Wick powers do not appear in the asymptotic expansions,
which is consistent with the above mentioned fact that one cannot restrict $:\Phi^{2k}:$
to $A$ in the first place.
The expansion is derived in Section \ref{Sec A E}.
It is Theorem \ref{Thm A E FPS} there.
Here we state a shorter version of it.

\begin{thm}[Theorem \ref{Thm A E FPS}]
\label{Thm A E FPS intro}
The indicator function $\ind_{\Ns_{\varepsilon}(A)}$ satisfies the following asymptotic expansion:
for every $N\geq 0$,
\begin{displaymath}
\ind_{\Ns_{\varepsilon}(A)}
=
\dfrac{1}{\sqrt{2\pi}}
\sum_{k=0}^{N} (-1)^{k}
\dfrac{1}{2^{k} k! (k+1/2)}
\dfrac{\psi_{2k+1,A}}{\big(\frac{1}{2\pi}\vert\log \varepsilon\vert\big)^{k + 1/2}}
~+~R_{N,\varepsilon},
\end{displaymath}
where the error term $R_{N,\varepsilon}$ satisfies, as $\varepsilon \to 0$,
\begin{displaymath}
\forall \eta >0,~
\E\big[\Vert R_{N,\varepsilon} \Vert_{H^{-\eta}(\C)}^{2}\big]^{1/2}
=\, o (\vert\log \varepsilon\vert^{-(N+1/2)}).
\end{displaymath}
\end{thm}

With this asymptotic expansion we derive a multi-scale description of the fields
$\psi_{2n+1,A}$, where they appear as limits of linear combinations of indicator functions
$\ind_{\Ns_{\varepsilon^{\alpha_{i}}}(A)}$,
at different scales $\varepsilon^{\alpha_{i}}$.

\begin{cor}[Corollary \ref{Cor psi Vandermonde}]
\label{Cor psi Vandermonde intro}
Let $n\geq 1$
and let $\alpha_{n}>\alpha_{n-1}>\dots >\alpha_{1}>\alpha_{0}=1$.
Let $(c_{0}, c_{1}, \dots, c_{n})$
be the unique solution to the linear system
\begin{displaymath}
\forall k\in \{0,\dots, n-1\},
\sum_{i=0}^{n} c_{i} \alpha_{i}^{-(k+1/2)} = 0,
\qquad
\sum_{i=0}^{n} c_{i} \alpha_{i}^{-(n+1/2)} = 1.
\end{displaymath}
Then
\begin{displaymath}
\psi_{2n+1, A} = \lim_{\varepsilon \to 0} (-1)^{n}\dfrac{2^{n} n! (n+1/2)}{(2\pi)^{n}}\vert\log\varepsilon\vert^{n+1/2}
\sum_{i=0}^{n} c_{i} \ind_{\Ns_{\varepsilon^{\alpha_{i}}}(A)},
\end{displaymath}
with convergence in $L^{2}(d\PP,\sigma(\Phi),H^{-\eta}(\C))$ for $\eta>0$.
\end{cor}

We see our asymptotic expansion, as an analogue, in the GFF setting,
of Le Gall's expansion of the 2D Wiener sausage
($\varepsilon$-neighborhood of a Brownian trajectory)
into powers of $1/\vert\log\varepsilon\vert$ \cite{LeGallLocTime,LeGallStFlour}.
To have a point of comparison, we recall Le Gall's result in 
Section \ref{Subsec Le Gall sausage} (Theorem \ref{Thm Le Gall}).
The indicator function of a Wiener sausage expands in an $L^2$ sense into integer powers
$\vert\log\varepsilon\vert^{-n}$, $n\in\N\setminus \{0\}$,
and the expansion involves Brownian renormalized self-intersection local times,
which are generalized functions supported on the range of the Brownian motion.
The leading coefficient attached to $\vert\log\varepsilon\vert^{-1}$
is a positive measure, the occupation measure of the Brownian motion.
The analogy between the two expansions is clear, but there are some important differences.
First of all, our $\varepsilon$-neighborhood is defined through conformal radii,
and not the Euclidean distance, as for the Wiener sausage.
But this is a minor point.
Most importantly, the powers of $1/\vert \log\varepsilon\vert$ are not the same,
although the increment is $1$ in both cases.
Moreover, the renormalization for Wick powers is done with Hermite polynomials,
whereas for self-intersection local times it is done with generalized Laguerre polynomials of order $-1$.
The analogy between the two expansion becomes even more intriguing if one reminds that
via the isomorphism theorems (Brownian motion representations of the GFF),
the FPS $A$ can be represented as a cluster of Brownian trajectories (loops and boundary excursions)
\cite{ALS2}.
The difference in exponents reflects the accumulation of small Brownian loops (ultraviolet divergence)
in the cluster.
Note however that we do not derive our Theorem \ref{Thm A E FPS intro} from Le Gall's expansion,
and do not use the Brownian motion representations at all.
Deriving one expansion from the other is actually far from being straightforward,
and there is a heuristic but in-depths discussion on this point in Section \ref{Subsec Wiener to FPS}.

To prove Theorem \ref{Thm A E FPS intro} we ultimately rely on SLE theory.
More precisely, we approximate the first passage sets by the two-valued sets (TVS), 
the latter having the advantage of being thin for the GFF.
We start by observing that the asymptotic expansion holds if one replaces the indicator function and the fields by their conditional expectations given a TVS and its labels (Proposition \ref{Prop cond series}).
This observation is key in the proof,
but there is still a lot of technical work after this point.
This is because there is an interchange of limits involved:
limit in $\varepsilon$ versus the limit in the level of the TVS.

\bigskip

In Section \ref{Sec Wick exc decomp}, we explain how the results for the FPS translate into results for the excursion clusters.
The excursion clusters are closely related to first passage sets,
since conditionally on its outer boundary $\Gamma$, an excursion cluster is an FPS in the domain
$\inter(\Gamma)$ delimited by $\Gamma$,
from a level $\pm 2\lambda$ to level $0$.
Here $2\lambda$ is the Schramm-Sheffield height gap of the GFF \cite{SchSh,SchSh2}.
In particular,
Theorems \ref{Thm psi GMC intro} and \ref{Thm A E FPS intro},
and Corollary \ref{Cor psi Vandermonde intro}
have their analogues for the excursion clusters.
These are Theorem \Ref{Thm fields on C j},
Theorem \ref{Thm A E C j} and Corollary \ref{Cor psi Vandermonde indiv exc}.
Further, we derive a full decomposition of the Wick powers
$:\Phi^{n}:$ through the full collection of excursion clusters.
To this end we first group the excursion clusters into generations,
corresponding to the number of other cluster surrounding a cluster,
then inside each generation we order the clusters according to the diameter.
The latter ordering inside a generation is not essential, but just a convenience.
See Theorem \ref{Thm decomp gen Wick},
Theorem \ref{Thm per cluster enum},
and Corollary \ref{Cor decomp gen plus enum}.

\bigskip

Here we would also like to mention an important Leitmotiv that will be recurring in this article.
It is related to the fact that the Wick powers $:\Phi^{n}:$
are not exactly local functions of the GFF $\Phi$.
Given a point $z_{0}\in D$, to know $:\Phi^{n}:$ in the neighborhood of $z_{0}$,
it is not enough to know $\Phi$ in the neighborhood of $z_{0}$.
One also needs to know the global shape of the domain $D$.
This global shape enters the definition $:\Phi^{n}:$ through the Green's function
$G_{D}$ used in the renormalization.
In particular, what matters is the constant order term in $G_{D}$
after the logarithmic singularity on the diagonal.
Should this constant order term change, then so would change $:\Phi^{n}:$\,.
There is a rule describing how $:\Phi^{n}:$ transforms under such circumstances.
It is given by the change of variance identity \eqref{Eq change var},
which is a remarkable identity for Hermite polynomials.
Typically, we need this identity when we cut our domain $D$ by removing for instance an FPS $A$.
In the remaining random domain $D\setminus A$,
the Green's function $G_{D\setminus A}$ has a constant order term after the singularity
which differs from that of $G_{D}$.
Therefore, renormalizing fields in $D\setminus A$ with $G_{D\setminus A}$
returns a different outcome from 
renormalizing with $G_{D}$.
In this situations we need the change of variance identity \eqref{Eq change var}
to relate the two renormalizations.

The change of variance identity \eqref{Eq change var} has a more abstract group-theoretic interpretation
that originates from the \textit{umbral calculus} \cite{RotaAll73Umbral,Roman84Umbral}.
In this interpretation, the families of Hermite polynomials
with different normalizations (corresponding to different variances for the Gaussian density on $\R$)
form a one-parameter subgroup of the group of polynomial sequences under the \textit{umbral composition}.
The generalized Laguerre polynomials of order $-1$, used in the renormalization of Brownian self-intersection local times, also have a similar property \eqref{Eq change norm Laguerre}.
This group-theoretic point of view is detailed in Section \ref{Subsec umbral}.
Further, the asymptotic expansions, both in the GFF setting (Theorem \ref{Thm A E FPS intro})
and in the Brownian setting (Theorem \ref{Thm Le Gall} by Le Gall) are closely related to
these umbral one-parameter subgroup identities.
Indeed, one can globally modify the setting 
(perturbation of domain in the GFF case, change of life-time in the Brownian case)
without modifying, with positive probability and in a small region, the set to asymptotically expand.
The asymptotic expansion being local, it should remain locally unchanged if the set to expand is locally unchanged.
But the Wick powers and the renormalized self-intersection local times do change under a global change of parameters,
according to the above mentioned identities \eqref{Eq change var} and \eqref{Eq change norm Laguerre}.
This discrepancy between the purely local nature of the asymptotic expansions and the more global nature of the renormalized fields puts hard algebraic \textit{a priori} constraints on the form of  the asymptotic expansions.
This point of view is developped in Section \ref{Sec algeb}.
The fact that the asymptotic expansions satisfy these constraints comes from what we call
\textit{reexpansion identities},
which are purely identities on Hermite polynomials (Proposition \ref{Prop reexp})
and on generalized Laguerre polynomials of order $-1$ (Proposition \ref{Prop reexp Laguerre}).
Moreover, one can already read the combinatorial coefficients appearing in the asymptotic expansions
from the reexpansion identities.

\bigskip

We conclude our work by a series of conjectures and open questions in Section \ref{Sec open}.
In Section \ref{Subsec FPS Euclid} we give a very precise conjecture on the asymptotic expansion of
first passage sets when the $\varepsilon$-neighborhood is defined not through the conformal radius but through the Euclidean distance.
This conjecture involves expectations w.r.t. a measure on SLE$_{4}$ loops up to change of scale,
which is a ``quotient" of the Malliavin–Kontsevich–Suhov (MKS) measure on SLE$_{4}$ loops.
Indeed, the measure on SLE$_{4}$ loops is supposed to describe, up to scaling,
the microscopic holes inside the FPS.
In Section \ref{Subsec Wiener CR} we wonder what would become Le Gall's expansion for the Brownian motion if one replaces the Euclidean distance by the conformal radius.
We provide a conjecture for this, which again involves an MKS measure,
but this time on SLE$_{8/3}$ loops.
In Section \ref{Subsec Wiener to FPS} we discuss about relating the expansion for the FPS
to Le Gall's expansion for the Wiener sausage via Brownian motion representations of the GFF.
We conclude that this might be possible,
but in a sophisticated highly non-obvious way, as one needs a mechanism that accounts for the change of exponents.
In Section \ref{Subsec other c}
we consider clusters in a Brownian loop soup of central charge 
$c\in (0,1)$.
The case $c=1$ corresponds to the GFF setting via the Brownian motion representations,
and the limit $c\to 0$,
corresponds, loosely speaking, to a single Brownian trajectory.
We ask the question of the asymptotic expansion of the $\varepsilon$-neighborhoods of clusters
for $c\in (0,1)$, and provide a conjecture.
Although our conjecture is not entirely precise and does not contain the exact coefficients,
we still conjecture that the powers of $1/\vert\log \varepsilon\vert$ that appear are
$\vert\log \varepsilon\vert^{-(n-c/2)}$, $n\in\N\setminus\{0\}$.
So the conjectured exponents form an interpolation between the GFF setting (Theorem \ref{Thm A E FPS intro})
and Le Gall's expansion for the Wiener sausage (Theorem \ref{Thm Le Gall}).
We further conjecture that the fields that appear are the hypothetical
renormalized signed \textbf{fractional} powers of the occupation field of the loop soup,
that were conjectured to exist by Jego, Lupu and Qian in \cite{JegoLupuQianFields}.
In particular, the more obvious renormalized intersection local times of the Brownian loop soup that were constructed by Le Jan \cite{LeJan2011Loops} should \textbf{not} appear in the expansion,
just like the even Wick powers of the GFF do not appear in Theorem \ref{Thm A E FPS intro}.
Finally, Section \ref{Subsec deep umbral} contains a broad question on the extent of the relation between the umbral calculus and the renormalization in dimension $2$.

\subsection*{Organization of the article}

In Section \ref{Sec prelim} we review some known prerequisites for our work.
It covers many different points.
In Section \ref{Subsec Hermite} we recall some remarkable identities satisfied by the Hermite polynomials,
in particular the change of variance identity \eqref{Eq change var},
which will be used countless times in this work.
In Section \ref{Subsec Wick},
we review the Sobolev spaces $H^{-\eta}(\C)$ and the Wick powers $:\Phi^{n}:$\,.
One take-away from this would be the following.
Although the Sobolev spaces $H^{s}(\C)$ are defined in Fourier,
for $s= -\eta <0$,
one can express the Sobolev norm-squared
$\Vert\cdot\Vert_{H^{-\eta}(\C)}^{2}$ as a double integral with respect to a kernel $\LK_{\eta}$.
The kernels $\LK_{\eta}$ have an explicit expression through Bessel functions and are known as 
\textit{Bessel potentials}.
Therefore, when dealing with random fields in $H^{-\eta}(\C)$,
one can use a moment method and never ever needs to go into Fourier or a spectral decomposition of the Laplacian.
In Section \ref{Subsec TVS FPS},
we review the notions of local set, two-valued set, first passage set and excursion cluster of the GFF,
and emphasize the tools that we need from this theory.
In Section \ref{Subsec Le Gall sausage},
we recall the notion of renormalized self-intersection local times of the planar Brownian motion
and recall Le Gall's asymptotic expansion of the Wiener sausage in dimension 2.
This will not be used as a tool in our work,
but will serve as a point of comparison with our 
Theorem \ref{Thm A E FPS intro} (Theorem \ref{Thm A E FPS} in Section \ref{Sec A E}).
In Section \ref{Subsec umbral} we present some elements of umbral calculus,
in particular the umbral composition which endows the set of polynomial sequences with a group structure.
We observe that the change of variance identity \eqref{Eq change var} for Hermite polynomials
can be naturally expressed in terms of the umbral composition,
and is related to a one-parameter subgroup of the group of polynomial sequences.
Further, the generalized Laguerre polynomials used in the renormalization of Brownian
self-intersection local times also satisfy an analogous change of normalization identity \eqref{Eq change norm Laguerre},
which in turn is also related to a (different) one-parameter subgroup of the group of polynomial sequences.

Section \ref{Sec estimates} contains some estimates on Green's function and conformal radius that would be further used in this work.
In Section \ref{Subsec estim Green} we provide an upper bound on $G_{D}(z,w)$ useful in a situation when
$z$ and $w$ are simultaneously close to each other and close to the boundary $\partial D$.
The question is who wins, the logarithmic singularity on the diagonal or the $0$ boundary condition.
Proposition \ref{Prop Green} provides in this case a sufficient condition for $G_{D}(z,w)$ to be small.
Importantly, the dependence on the domain of the upper bound is explicit, as we will further apply the bound to random domains.
Section \ref{Subsec estim CR} contains estimates on the variations of Green's functions and conformal radii,
with again an explicit dependence on the domain.

In Section \ref{Sec Wick FPS} we take $A$ an FPS of a GFF $\Phi$ and introduce the fields
$\psi_{n,A}$:
\begin{displaymath}
\psi_{n,A} = \E\big[:\Phi^{n}:\vert A\big]\,.
\end{displaymath}
We study the properties of these fields, and in particular the distinction between $n$ odd and $n$ even.
In Theorem \ref{Thm decomp FPS Wick}
we give the decomposition of the Wick powers $:\Phi^{n}:$ induced by the FPS $A$,
that involves $\psi_{n,A}$.
In the Subsection \ref{Subsec Wick TVS},
we consider a two-valued set (TVS) instead of an FPS and provide the decomposition of
the Wick powers $:\Phi^{n}:$ induced by the TVS.
It is similar but simpler, since unlike the FPS, a TVS is thin for the GFF. 

Section \ref{Sec A E} is dedicated to the asymptotic expansion of $\varepsilon$-neighborhoods of an FPS.
There we state and prove Theorem \ref{Thm A E FPS intro} (Theorem \ref{Thm A E FPS}) and 
Corollary \ref{Cor psi Vandermonde intro} (Corollary \ref{Cor psi Vandermonde}).
The Subsection \ref{Subsec cond TVS} contains the key observation on what happens if we additionally condition the expansion by the TVS and its labels (Proposition \ref{Prop cond series}).
In the rest of the section we derive a proof of Theorem \ref{Thm A E FPS intro} (Theorem \ref{Thm A E FPS})
out of the above observation.
The main effort goes into
justifying an interchange of limits,
limit in $\varepsilon$ versus limit in the upper level of the TVS.

In Section \ref{Sec algeb} we discuss how there are \textit{a priori} constraints
to be satisfied by the asymptotic expansions,
both in the case of the FPS (Section \ref{Subsec algeb})
and in the case of the Wiener sausage (Section \ref{Subsec BM algeb}),
and there is a parallel between the two settings.
The fact that the constraints are satisfied comes from reexpansion identities
for Hermite polynomials (Proposition \ref{Prop reexp}, already used to prove Theorem \ref{Thm A E FPS intro}),
respectively for generalized Laguerre polynomials of order $-1$ (Proposition \ref{Prop reexp Laguerre}).

In Section \ref{Sec psi GMC} we explain what the expansion of the Gaussian multiplicative chaos into power
series w.r.t. the intermittency parameter $\gamma$ tells us on the fields $\psi_{n,A}$.
In Section \ref{Subsec psi GMC 1} we state and prove Theorem \ref{Thm psi GMC intro}
(Theorem \ref{Thm psi GMC}).
In Section \ref{Subsec psi GMC multiscale} we explain how Theorem \ref{Thm A E FPS intro}
and Theorem \ref{Thm psi GMC intro} are actually related, although one does not need the one to prove the other.

In Section \ref{Sec Wick exc decomp} we explain how our previous results translate to the excursion clusters.
In Subsection \ref{Subsec Wick exc gen} we consider the excursion clusters grouped into generations
according to the number of other clusters surrounding a cluster.
In Theorem \ref{Thm decomp gen Wick} we provide a decomposition of Wick powers $:\Phi^{n}:$
w.r.t. this generational picture.
In Subsection \ref{Subsec fields indiv} we consider what happens at the level of an individual 
excursion cluster and give the analogues of Theorem \ref{Thm psi GMC intro}, Theorem \ref{Thm A E FPS intro}
and Corollary \ref{Cor psi Vandermonde intro}.
In Subsection \ref{Subsec Wick exc indiv} we explain how inside each generation the fields
decompose along individual excursion clusters (Theorem \ref{Thm per cluster enum}).
The Subsection \ref{Subsec remark nestes CLE 4} contains more of a side remark, outside our main focus.
It presents the decomposition of Wick powers induced by bounded-type thin local sets of the GFF \cite{ASW},
and derives out of it a description of the Wick powers $:\Phi^{n}:$
through the Miller-Sheffield coupling between the GFF and the nested CLE$_{4}$.

The final Section \ref{Sec open} is dedicated to open questions and conjectures.
Section \ref{Subsec FPS Euclid} contains a precise conjecture for the asymptotic expansion of an FPS
when the $\varepsilon$-neighborhood is defined through Euclidean distance.
Section \ref{Subsec Wiener CR} contains a precise conjecture for the asymptotic expansion of a Wiener
sausage defined via the conformal radius.
Both conjectures involve Malliavin–Kontsevich–Suhov measures on SLE loops.
Section \ref{Subsec Wiener to FPS} contains a discussion on a possible relation between the expansion for the Wiener sausage and the expansion for FPS/excursion clusters.
Section \ref{Subsec other c} contains a conjecture on the asymptotic expansion for
clusters in a Brownian loop soup of central charge
(intensity parameter) $c\in(0,1)$.
Section \ref{Subsec deep umbral} contains a very broad question on the relation between renormalization and umbral calculus.

\section{Preliminaries}
\label{Sec prelim}

\subsection{Hermite polynomials}
\label{Subsec Hermite}

Here we recall some basic facts about the Hermite polynomials.
For references, see \cite{RotaAll73Umbral,AbramowitzStegun84,Roman84Umbral}.

Let $(\He_{n})_{n\geq 0}$ be the family of monic polynomials
which are orthogonal w.r.t. the Gaussian measure on $\R$,
$e^{-x^{2}/2} dx$.
These are the probabilistic Hermite polynomials.
There are several equivalent expressions for these polynomials.
First,
\begin{equation}
\label{Eq Herm 1}
\He_{n}(x) = (-1)^{n}e^{x^{2}/2} \dfrac{d^{n}}{dx^{n}} e^{-x^{2}/2}.
\end{equation}
Further, let $Y$ be a Gaussian r.v. distributed according to
$\mathcal{N}(0,1)$. Then,
\begin{equation}
\label{Eq Herm 2}
\He_{n}(x) = \E[(x+ i Y)^{n}],
\end{equation}
where $i=\sqrt{-1}$. 
Note that the coefficients of $\He_{n}$ are real (and actually relative integers)
since the odd powers of $i$ in \eqref{Eq Herm 2} vanish.
Finally, one can express the coefficient of $\He_{n}$ explicitly:
\begin{equation}
\label{Eq Herm 3}
\He_{n}(x) = \sum_{0\leq k\leq \lfloor n/2\rfloor}
(-1)^{k} \dfrac{n!}{2^{k} k! (n-2k)!} x^{n-2k}.
\end{equation}
We have
\begin{displaymath}
\He_{0}(x) = 1,
\qquad
\He_{1}(x) = x,
\qquad
\He_{2}(x) = x^{2}-1,
\qquad
\He_{3}(x) = x^{3}-3x,
\qquad
\He_{4}(x) = x^{4}-6 x^{2} + 3,
\text{ etc.}
\end{displaymath}

Next we recall some classical identities satisfied by Hermite polynomials.
First, a binomial identity:
\begin{equation}
\label{Eq Herm bin}
\He_{n}(x + y) = \sum_{j=0}^{n} \dfrac{n!}{j! (n-j)!} \He_{j}(x) y^{n-j}.
\end{equation}
It follows for instance directly from \eqref{Eq Herm 2}.
From \eqref{Eq Herm 2} also follows an expression for the exponential generating function of Hermite polynomials:
\begin{equation}
\label{Eq Herm exp}
\sum_{n\geq 0}\dfrac{\gamma^{n}}{n!} \He_{n}(x) =
e^{\gamma x - \gamma^{2}/2}.
\end{equation}

For the purpose of the Wick renormalisation (Section \ref{Subsec Wick}),
we will further need polynomials in two variables:
\begin{equation}
\label{Eq def Q}
Q_{n}(x,u) = u^{n/2}\He_{n}(x u^{-1/2}) = 
\sum_{0\leq k\leq \lfloor n/2\rfloor}
(-1)^{k} \dfrac{n!}{2^{k} k! (n-2k)!} x^{n-2k} u^{k}.
\end{equation}
The variable $u$ represents a variance.
The identities \eqref{Eq Herm bin} and \eqref{Eq Herm exp}
immediately generalize to $Q_{n}$:
\begin{equation}
\label{Eq Q bin}
Q_{n}(x+y,u) =
\sum_{j=0}^{n} \dfrac{n!}{j! (n-j)!} Q_{j}(x,u) y^{n-j},
\end{equation}
\begin{equation}
\label{Eq Q exp}
\sum_{n\geq 0}\dfrac{\gamma^{n}}{n!} Q_{n}(x,u) =
e^{\gamma x - \gamma^{2} u/2}.
\end{equation}
We will also need the following change of variance formula:
\begin{equation}
\label{Eq change var}
Q_{n}(x,u_1+u_2) =
\sum_{0\leq k\leq \lfloor n/2\rfloor}
(-1)^{k} \dfrac{n!}{2^{k} k! (n-2k)!} Q_{n-2k}(x,u_1) u_{2}^{k}.
\end{equation}
Note that the coefficients appearing in \eqref{Eq change var}
are exactly the same as in \eqref{Eq Herm 3} and \eqref{Eq def Q}.
This is related to a group structure on the space of polynomial sequences and to the so-called \textit{umbral calculus} \cite{RotaAll73Umbral,Roman84Umbral}.
This will be further detailed in Section \ref{Subsec umbral}.

\subsection{Wick powers of the Gaussian free field}
\label{Subsec Wick}

Let $D\subset \C$ be an open, non-empty, bounded, connected and simply connected domain.
Let $\Phi_{(0)}$ be a continuum Gaussian free field (GFF) on $D$ with $0$ boundary conditions.
It's covariance kernel is given by the Green's function $G_{D}(z,w)$ of $-\Delta$ on $D$,
with $0$ boundary conditions. We chose the normalization so that
\begin{equation}
\label{Eq G CR}
G_{D}(z,w) = \dfrac{1}{2\pi}
\log \dfrac{1}{\vert w-z\vert}
+
g_{D}(z,w),
\end{equation}
where $g_{D}(z,w)$ is harmonic on $D$ w.r.t. both variables.
The function $g_{D}(z,w)$ is bounded on $D^{2}$ from above by
$\frac{1}{2\pi}\log \diam(D)$.
Note that for $z\in D$,
\begin{displaymath}
g_{D}(z,z) = \dfrac{1}{2\pi}
\log \CR (z,D),
\end{displaymath}
where $\CR (z,D)$ is the conformal radius of $D$ seen from $z$.
By distortion inequalities,
\begin{equation}
\label{Eq Koebe}
d(z,\partial D)\leq \CR (z,D) \leq 4 d(z,\partial D).
\end{equation}
See \cite{Ahlfors2010ConfInv}.

Usually, on views $\Phi_{(0)}$ as a random element of Sobolev spaces
$H^{-\eta}(D)$ for $\eta>0$ \cite{Sheffield07GFF,Dubedat09PartitionFunc,PowWernerGFF}.
However, these functional spaces do not provide enough control on the behavior of 
$\Phi_{(0)}$ near the boundary $\partial D$.
Therefore, we will use the Sobolev spaces $H^{-\eta}(\C)$.
Alternatively, we could have used $H^{-\eta}(\widetilde{D})$
for $\widetilde{D}$ an open domain containing $\overline{D}$, the topological closure of $D$.

The Sobolev spaces $H^{s}(\C)$ are defined in terms of the Fourier transform.
We will use the following convention for the Fourier transform:
\begin{displaymath}
\hat{f}(\xi) = \dfrac{1}{2\pi}\int_{\C} f(z) e^{-i \xi\cdot z} d^{2} z,
\end{displaymath}
where $\xi\cdot z$ is the inner product
\begin{displaymath}
\xi\cdot z = \Re (\xi) \Re (z) + \Im (\xi) \Im (z)
=\dfrac{1}{2} (\xi \bar{z} + \bar{\xi} z).
\end{displaymath}
Then $H^{s}(\C)$ is the completion of $\mathcal{S}(\C)$, the Schwartz space of smooth functions with rapidly decreasing derivatives,
for the Sobolev norm
\begin{equation}
\label{Eq def Sobo norm}
\Vert f\Vert_{H^{s}(\C)}^{2}
= \int_{\C}(1 + \vert\xi\vert^{2})^{s}\vert \hat{f}(\xi)\vert^{2} d^{2} \xi.
\end{equation}
The spaces $H^{s}(\C)$ are separable Hilbert spaces, 
decreasing in $s$.
We refer to \cite[Section 7.62]{AdamsFournierSobo}.
For $s= -\eta <0$, the Sobolev norm has a translation and rotation invariant kernel:
\begin{displaymath}
\forall f\in \mathcal{S}(\C),~
\Vert f\Vert_{H^{-\eta}(\C)}^{2}
= \int_{\C}(1 + \vert\xi\vert^{2})^{-\eta}\vert \hat{f}(\xi)\vert^{2} d^{2} \xi
= \int_{\C^{2}} f(z)\LK_{\eta}(\vert w-z\vert)\overline{f(w)}\,d^{2} z\,d^{2} w,
\end{displaymath}
where
\begin{equation}
\label{Eq Bessel potential}
\LK_{\eta}(r) = \lim_{R\to +\infty}\dfrac{1}{4\pi^{2}}
\int_{\vert \xi\vert < R}(1 + \vert\xi\vert^{2})^{-\eta} \cos (r\Re(\xi))\, d^{2}\xi =
\dfrac{2^{1-\eta}}{2\pi \Gamma(\eta)}\dfrac{K_{\eta - 1}(r)}{r^{1-\eta}},
\end{equation}
with $K_{\eta - 1}$ being the modified Bessel function of the second kind
of order $\eta - 1$.
The kernels $\LK_{\eta}$ are the so-called \textit{Bessel potentials}.
See \cite[Section 7.63]{AdamsFournierSobo} and
\cite{AronszajnSmith61Bessel1,AronszajnMullaSzeptycki63Bessel}.
For $\eta>0$, the kernel $\LK_{\eta}(r)$ is continuous positive on
$(0,+\infty)$.
If $\eta\in (0,1)$, $\LK_{\eta}(r)$ diverges as a power near $0$:
\begin{equation}
\label{Eq asymp kernel}
\LK_{\eta}(r) \asymp \dfrac{1}{r^{2 - 2\eta}}~~
\text{as } r\to 0.
\end{equation}
If $\eta = 1$, $\LK_{\eta}(r)$ has a logarithmic divergence at $r=0$,
and if $\eta>1$, $\LK_{\eta}$ is continuous on $[0,+\infty)$.

On can extend the GFF $\Phi_{(0)}$ to $\C$ by setting it to $0$ outside $D$.
Then it is a random element of $H^{-\eta}(\C)$ for all $\eta >0$,
and
\begin{displaymath}
\E\big[\Vert\Phi_{(0)}\Vert_{H^{-\eta}(\C)}^{2}\big]
=\int_{D^{2}} G_{D}(z,w)\LK_{\eta}(\vert w-z\vert)
\,d^{2} z\,d^{2} w
< +\infty,
\end{displaymath}
the convergence following from \eqref{Eq asymp kernel}.

\medskip

Classically, in dimension $2$ one can define the renormalized (Wick) powers of the GFF.
See \cite{Simon74EQFT} for the use of Wick powers in Quantum Field Theory and
\cite{KangMakarovGFFCFT} for the use in Conformal Field Theory.
We will denote by $:\Phi_{(0)}^{n}:$ the $n$-th Wick power of $\Phi_{(0)}$.
Formally, it is given by
\begin{displaymath}
:\Phi_{(0)}^{n}:(z)~= Q_{n}(\Phi_{(0)}(z),\Var(\Phi_{(0)}(z)))
=Q_{n}(\Phi_{(0)}(z),G_{D}(z,z)),
\end{displaymath}
where $Q_{n}$ is given by \eqref{Eq def Q}.
Rigorously, $:\Phi_{(0)}^{n}:$ can be defined through a regularization procedure.

Let $\rho$ be a $\LC^{\infty}$-smooth non-negative function on $\C$, with compact support contained in the unit disk
$\D = \{z\in\C\vert ~\vert z\vert<1\}$,
and satisfying
\begin{displaymath}
\int_{\C}\rho(z)\, d^{2} z = 1.
\end{displaymath}
For $\varepsilon>0$, define
\begin{displaymath}
\rho_{\varepsilon}(z) = \varepsilon^{-2}\rho(\varepsilon^{-1} z).
\end{displaymath}
Then, as $\varepsilon\to 0$, $\rho_{\varepsilon}$ approximates the Dirac mass at $0$.
Let $\Phi_{(0),\varepsilon}$ be the regularization by convolution
\begin{displaymath}
\Phi_{(0),\varepsilon}(z) = (\Phi_{(0)}\ast\rho_{\varepsilon})(z)
= \int_{\C}\Phi_{(0)}(w)\rho_{\varepsilon}(z-w)\,d^{2}w.
\end{displaymath}
Then $\Phi_{(0),\varepsilon}$ is $\LC^{\infty}$-smooth,
belongs to all the Sobolev spaces $H^{s}(\C)$ for $s\in\R$,
and importantly, is still Gaussian.
Denote
\begin{displaymath}
G_{D,\varepsilon,\varepsilon'}(z,w)
=\int_{\C^{2}} \rho_{\varepsilon}(z-x)G_{D}(x,y)\rho_{\varepsilon'}(w-y)
\,d^{2}x\,d^{2}y
=\E\big[\Phi_{(0),\varepsilon}(z)\Phi_{(0),\varepsilon'}(w)\big],
\end{displaymath}
where $G_{D}$ is extended by $0$ outside $D^{2}$.

\begin{lemma}
\label{Lem G eps}
For every $z,w\in\C$ and $\varepsilon, \varepsilon'\in (0,1]$,
\begin{displaymath}
G_{D,\varepsilon,\varepsilon'}(z,w)
\leq\Big(\dfrac{1}{2\pi}\log (\varepsilon^{-1}) + c\Big)^{1/2}
\Big(\dfrac{1}{2\pi}\log (\varepsilon'^{-1}) + c\Big)^{1/2},
\end{displaymath}
where
\begin{displaymath}
c = \dfrac{1}{2\pi}
\int_{\C^{2}} 
\log (\vert y-x\vert^{-1})
\rho(x)\rho(y)
\,d^{2}x\,d^{2}y
+
\dfrac{1}{2\pi}
0\vee \log(\diam(D))
< +\infty.
\end{displaymath}
Moreover, for every $\varepsilon, \varepsilon'\in (0,1]$,
and $z,w\in\C$ such that $\vert w-z\vert\geq 2\,\varepsilon\vee\varepsilon'$,
\begin{displaymath}
G_{D,\varepsilon,\varepsilon'}(z,w)
\leq
\dfrac{1}{2\pi}
\Big(
0\vee\log \dfrac{2}{\vert w-z\vert}
+
0\vee\log(\diam(D))
\Big).
\end{displaymath}
\end{lemma}

\begin{proof}
We have
$G_{D,\varepsilon,\varepsilon'}(z,w)\leq 
G_{D,\varepsilon,\varepsilon}(z,z)^{1/2}
G_{D,\varepsilon',\varepsilon'}(w,w)^{1/2}$.
Further,
\begin{eqnarray*}
G_{D,\varepsilon,\varepsilon}(z,z)
&\leq & 
\dfrac{1}{2\pi}
\int_{\C^{2}} 
0\vee\log \Big(\dfrac{\diam (D)}{\vert y-x \vert}\Big)
\rho_{\varepsilon}(x)\rho_{\varepsilon}(y)
\,d^{2}x\,d^{2}y
\\
& = &
\dfrac{1}{2\pi}
\int_{\C^{2}} 
0\vee\log \Big(\dfrac{\diam (D)}{\varepsilon\vert y-x \vert}\Big)
\rho(x)\rho(y)
\,d^{2}x\,d^{2}y.
\end{eqnarray*}
For $\varepsilon\in (0,1]$,
\begin{displaymath}
G_{D,\varepsilon,\varepsilon}(z,z)
\leq 
\dfrac{1}{2\pi}\log (\varepsilon^{-1})
+
\dfrac{1}{2\pi}
\int_{\C^{2}} 
\log (\vert y-x\vert^{-1})
\rho(x)\rho(y)
\,d^{2}x\,d^{2}y
+
\dfrac{1}{2\pi}
0\vee \log(\diam(D)).
\end{displaymath}
The second bound is obvious.
\end{proof}

Set
\begin{displaymath}
:\Phi_{(0),\varepsilon}^{n}:(z)~= 
Q_{n}(\Phi_{(0),\varepsilon}(z),\Var(\Phi_{(0),\varepsilon}(z)))
=Q_{n}(\Phi_{(0),\varepsilon}(z),G_{D,\varepsilon,\varepsilon}(z,z)).
\end{displaymath}
Then
\begin{equation}
\label{Eq cov Wick eps}
\E\big[:\Phi_{(0),\varepsilon}^{n}:(z):\Phi_{(0),\varepsilon'}^{m}:(w)\big]
=n!G_{D,\varepsilon,\varepsilon'}(z,w)^{n}
\text{ if } n=m, \text{ and } 0 \text{ otherwise}.
\end{equation}
The Wick renormalization procedure actually removes all the diagonal terms in the covariance.
We refer to \cite[Theorem 3.12]{Janson1997GaussHilbSpaces} for the
general expression of moments of Wick powers and their representation through Feynman diagrams without self-loop.
See also \cite[Theorem I.3]{Simon74EQFT}.

Fix $\eta >0$ and consider the space
$L^{2}(d\PP,\sigma(\Phi_{(0)}),H^{-\eta}(\C))$,
i.e. the space of random variables $X$ with values in the Sobolev space
$H^{-\eta}(\C)$, measurable w.r.t. $\Phi_{(0)}$,
such that $\E[\Vert X\Vert_{H^{-\eta}(\C)}^{2}]<+\infty$.
This space is complete; see \cite[Section 1.2.b]{HNVW16AnBanach1}.
Further, \eqref{Eq cov Wick eps} ensures that 
$(:\Phi_{(0),\varepsilon}^{n}:)_{\varepsilon>0}$
forms a Cauchy family in $L^{2}(d\PP,\sigma(\Phi_{(0)}),H^{-\eta}(\C))$
as $\varepsilon>0$.
Therefore, there is a limit random field $:\Phi_{(0)}^{n}:$ with values in
$H^{-\eta}(\C)$, measurable w.r.t. $\Phi_{(0)}$,
such that
\begin{displaymath}
\lim_{\varepsilon\to 0}
\E\big[\Vert :\Phi_{(0)}^{n}: - :\Phi_{(0),\varepsilon}^{n}:
\Vert_{H^{-\eta}(\C)}^{2}\big] = 0.
\end{displaymath}
It is easy to check that there is consistency for different values of $\eta$,
and that the limit does not depend on particular choice of $\rho$.
Similarly, for every $f$ continuous function on $\C$,
\begin{displaymath}
\lim_{\varepsilon\to 0}
\E\big[\vert (:\Phi_{(0)}^{n}:,f) - (:\Phi_{(0),\varepsilon}^{n}:,f)
\vert^{2}\big] = 0.
\end{displaymath}
Moreover, for every $\eta>0$,
\begin{displaymath}
\lim_{\varepsilon\to 0}
\E\Big[\sup_{\substack{f\in H^{+\eta}(\C)\\ \Vert f\Vert_{H^{\eta}(\C)}\leq 1}}
\vert (:\Phi_{(0)}^{n}:,f) - (:\Phi_{(0),\varepsilon}^{n}:,f)
\vert^{2}\Big] 
\leq \lim_{\varepsilon\to 0}
\E\big[\Vert :\Phi_{(0)}^{n}: - :\Phi_{(0),\varepsilon}^{n}:
\Vert_{H^{-\eta}(\C)}^{2}\big]= 0.
\end{displaymath}
From \eqref{Eq cov Wick eps}, by taking the limit,
follows that the family $(:\Phi_{(0)}^{n}:)_{n\geq 0}$
is orthogonal in $L^{2}(d\PP,\sigma(\Phi_{(0)}),H^{-\eta}(\C))$,
and that
\begin{displaymath}
\E\big[\Vert :\Phi_{(0)}^{n}:\Vert_{H^{-\eta}(\C)}^{2}\big] = 
n!\int_{D^{2}} G_{D}(z,w)^{n}\LK_{\eta}(\vert w-z\vert)
\,d^{2} z\,d^{2} w
< +\infty.
\end{displaymath}

\medskip

For $\gamma\in (-2\sqrt{2\pi},2\sqrt{2\pi})$, 
one can define the Gaussian multiplicatif chaos (GMC)
$: e^{\gamma \Phi_{(0)}} :$,
which is a renormalization of $e^{\gamma \Phi_{(0)}}$.
Indeed,
\begin{equation}
\label{Eq def GMS}
: e^{\gamma \Phi_{(0)}} :~=
\lim_{\varepsilon\to 0}
\exp\Big(\gamma \Phi_{(0),\varepsilon}(z)
-\dfrac{\gamma^{2}}{2}\Var(\Phi_{(0),\varepsilon}(z))\Big)~
d^{2}z
=
\lim_{\varepsilon\to 0}
\exp\Big(\gamma \Phi_{(0),\varepsilon}(z)
-\dfrac{\gamma^{2}}{2}G_{D,\varepsilon,\varepsilon}(z,z)\Big)~
d^{2}z
.
\end{equation}
The field $: e^{\gamma \Phi_{(0)}} :$ is a random positive finite measure on $D$,
and the convergence in \eqref{Eq def GMS} holds in probability for the weak convergence of measures.
For the references, see \cite{HoeghKrohn71,Kahane85GMC,DuplantierSheffield,RhodesVargas14GMCReview,
BerestyckiGMC,Aru20GMCreview}.
Note that our normalization for the GFF differs from the Liouville field theory literature, and so does our range for the intermittency parameter $\gamma$.
The two-point correlation function of $: e^{\gamma \Phi_{(0)}} :$ equals
\begin{equation}
\label{Eq two point GMC}
e^{\gamma^{2} G_{D}(z,w)}
=
\vert w-z\vert^{-\frac{\gamma^{2}}{2\pi}}
\,e^{\gamma^{2} g_{D}(z,w)}.
\end{equation}
For $\vert\gamma\vert < 2\sqrt{\pi}$,
\eqref{Eq two point GMC} is integrable on $D^{2}$.
Moreover, for $\vert\gamma\vert < 2\sqrt{\pi}$
and $\eta > \gamma^{2}/(4\pi)$,
$: e^{\gamma \Phi_{(0)}} :$ belongs to 
$L^{2}(d\PP,\sigma(\Phi_{(0)}),H^{-\eta}(\C))$
and the convergence \eqref{Eq def GMS} holds in
$L^{2}(d\PP,\sigma(\Phi_{(0)}),H^{-\eta}(\C))$.

The identity \eqref{Eq Q exp} translates into an expansion
of the GMC $: e^{\gamma \Phi_{(0)}} :$ into the Wick powers $:\Phi_{(0)}^{n}:$.
Thus, for $\vert\gamma\vert < 2\sqrt{\pi}$ (the $L^2$ regime),
\begin{equation}
\label{Eq GMC Wick}
\ind_{D}: e^{\gamma \Phi_{(0)}} :
~=\ind_{D}\sum_{n\geq 0}\dfrac{\gamma^{n}}{n!}:\Phi_{(0)}^{n}:
~=\ind_{D} + \sum_{n\geq 1}\dfrac{\gamma^{n}}{n!}:\Phi_{(0)}^{n}:\,.
\end{equation}
The convergence holds in 
$L^{2}(d\PP,\sigma(\Phi_{(0)}),H^{-\eta}(\C))$
for $\eta > \gamma^{2}/(4\pi)$.
In other words, for $\eta\geq 1$,
the radius of convergence of the power series \eqref{Eq GMC Wick}
in $L^{2}(d\PP,\sigma(\Phi_{(0)}),H^{-\eta}(\C))$
equals $2\sqrt{\pi}$,
and for $\eta\in (0,1)$,
this radius of convergence equals $2\sqrt{\eta\pi}$.
Note that for the aforementioned ranges of $\gamma$ and $\eta$,
the convergence \eqref{Eq GMC Wick} also holds almost surely in
$H^{-\eta}(\C)$, not just in the $L^2$ sense,
and as a stochastic process depending on $\gamma$ (i.e. simultaneously for all admissible $\gamma$-s).
This is because then we have
\begin{displaymath}
\sum_{n\geq 1}\dfrac{\gamma^{n}}{n!}
\Vert :\Phi_{(0)}^{n}:\Vert_{H^{-\eta}(\C)}
< + \infty~\text{ a.s.}
\end{displaymath}

\medskip

In this paper we will also need to consider the case of constant non-zero boundary condition for the GFF. 
One can also handle non-constant sufficiently regular boundary conditions, but we will not pursue this here.
So let be a constant $v\geq 0$ and let $\Phi$ be the GFF on $D$ with boundary
condition $v$ on $\partial D$.
As previously, we will extend $\Phi$ by $0$ outside $D$.
Then $\Phi$ has the same distribution as
$\Phi_{(0)} + v \ind_{D}$.
Again, it is a random element of $H^{-\eta}(\C)$ for $\eta>0$.
As previously, one can define its Wick powers and multiplicative chaos:
\begin{displaymath}
:\Phi^{n}: =
\lim_{\varepsilon\to 0}
Q_{n}(\Phi_{\varepsilon}(z),G_{D,\varepsilon,\varepsilon}(z,z))\,d^{2}z,
\qquad
: e^{\gamma \Phi} :~
=
\lim_{\varepsilon\to 0}
\exp\Big(\gamma \Phi_{\varepsilon}(z)
-\dfrac{\gamma^{2}}{2}G_{D,\varepsilon,\varepsilon}(z,z)\Big)~
d^{2}z
,
\end{displaymath}
with the same type and range of convergences as previously.
Through the binomial formula \eqref{Eq Q bin},
$:\Phi^{n}:$ identifies, for $n\geq 1$, to
\begin{displaymath}
\ind_{D}\sum_{j=0}^{n} \dfrac{n!}{j! (n-j)!} v^{n-j} :\Phi_{(0)}^{j}:,
\end{displaymath}
and $: e^{\gamma \Phi} :$ identifies to
$e^{\gamma v \ind_{D}} : e^{\gamma \Phi_{(0)}} :\,$.
The analogue of the expansion \eqref{Eq GMC Wick} holds,
\begin{equation}
\label{Eq GMC Wick gen}
\ind_{D}: e^{\gamma \Phi} :
~=\ind_{D}\sum_{n\geq 0}\dfrac{\gamma^{n}}{n!}:\Phi^{n}:
~=\ind_{D}+\sum_{n\geq 1}\dfrac{\gamma^{n}}{n!}:\Phi^{n}:\,,
\end{equation}
for the same radii of convergence.

\subsection{Two-valued sets, first passage sets, and excursion decomposition of the GFF}
\label{Subsec TVS FPS}

\subsubsection{Two-valued sets}
\label{Subsubsec TVS}

We will use the local sets and the local set decompositions of the GFF.
For references on this theory, see \cite{SchSh2,MS1,Sepulveda19ThinLocalSet,ASW,ALS1,ALS3}.
We start we the two-valued sets, introduced in \cite{ASW},
and further developed in \cite{AruSepulveda18TVS,SchougSepulvedaViklund22TVS,ALS3}.

First, recall the notions of height gap and level lines of the 2D continuum GFF,
introduced by Schramm and Sheffield in \cite{SchSh,SchSh2}.
See also \cite{WaWu17LLGFF} for an overview.
One can think of the 2D continuum GFF, which is only a generalized function,
as spanned by cliffs of constant height.
The height of these cliffs is a deterministic universal constant,
called \textit{height gap}, and following \cite{SchSh,SchSh2} denoted $2\lambda$.
With our normalization of the GFF, $2\lambda=\sqrt{\pi/2}$.
The cliffs themselves are called \textit{level lines}
and are locally SLE$_4$-type curves (Schramm-Loewner Evolution with $\kappa=4$).
On each side of the level line, the GFF has a well defined constant value 
(and this despite not being a point-wise defied function), the two values from both sides differing by $\pm 2\lambda$.

As previously, let $D\subset \C$ be an open, non-empty, bounded, connected and simply connected domain.
Let be a constant $v\geq 0$ and $\Phi$ a GFF on $D$ with constant
boundary condition $v$ on $\partial D$.
Let $b>v$ and $a<v$. 
Informally, a two-valued set (TVS) of the GFF $\Phi$ with levels $a$ and $b$ can be thought of as the subset of points of 
$\overline{D}$ that can be connected to the boundary $\partial D$
by a path along which the values of $\Phi$ are contained in the interval $[a,b]$.
This however, does not make immediate sense since $\Phi$ does not have point values.
Yet, the two-valued sets can be constructed as soon as $b-a\geq 2\lambda$.
The construction proceeds by iterating level lines of the GFF;
see \cite{ASW}.
One obtains a random compact connected subset $A_{a,b}$ of $\overline{D}$
such that $\partial D\subset A_{a,b}$.
Further, the GFF $\Phi$ admits a decomposition
\begin{equation}
\label{Eq decomps TVS}
\Phi = \ell_{a,b} + \Phi_{D\setminus A_{a,b}},
\end{equation}
where $\ell_{a,b}$ is a random function from $D\setminus A_{a,b}$ to $\{a,b\}$
that is constant on each connected component of $D\setminus A_{a,b}$,
and where conditionally on $(A_{a,b}, \ell_{a,b})$,
the field $\Phi_{D\setminus A_{a,b}}$ is distributed as a GFF on $D\setminus A_{a,b}$.
Note that $D\setminus A_{a,b}$ has infinitely many connected components,
that is to say $A_{a,b}$ has infinitely many holes.
The function $\ell_{a,b}$ is referred to as the \textit{label function}.
Outside $D\setminus A_{a,b}$, we extend $\ell_{a,b}$ by $0$.
The couple TVS-label function $(A_{a,b}, \ell_{a,b})$ is measurable w.r.t. the GFF $\Phi$ \cite{ASW}.
The Hausdorff dimension of $A_{a,b}$ is a.s.
$2-2\lambda^2/(b-a)^{2}\in [3/2,2)$ \cite{SchougSepulvedaViklund22TVS}.
The TVS $A_{a,b}$ is a \textit{thin} local set (see \cite{Sepulveda19ThinLocalSet, ASW} for this notion): 
the GFF $\Phi$ does not charge $A_{a,b}$ itself.
If one looks at an $\varepsilon$-neighborhood of $A_{a,b}$,
then the decomposition \eqref{Eq decomps TVS} implies that
\begin{displaymath}
\ind_{\{z\in\C\vert d(z,A_{a,b})<\varepsilon\}}\Phi
\end{displaymath}
converges to $0$, as $\varepsilon\to 0$,
in $L^{2}(d\PP,\sigma(\Phi),H^{-\eta}(\C))$ for $\eta >0$.
Also note that there is a monotonicity of the TVS:
if $a'\leq a$ and $b'\geq b$, then
$A_{a,b}\subset A_{a',b'}$ a.s.

Fix $z\in D$. Note that a.s., $z\not\in A_{a,b}$.
The joint law of the ratio of conformal radii
$\CR(z,D)/\CR(z,D\setminus A_{a,b})$ and the label $\ell_{a,b}(z)$ can be expressed through a
one-dimensional Brownian motion.
Note that by conformal invariance,
this joint law depends neither on $z$ nor $D$.
We will use this law in the proof of Lemma \ref{Lem area 0}.
So let $(W_{t})_{t\geq 0}$ be a standard one-dimensional Brownian motion
starting from $W_{0} = v$.
Let
\begin{displaymath}
T_{a,b} = \min \{t\geq 0\vert W_t\in\{a,b\}\}.
\end{displaymath}

\begin{thm}[Aru-Sep\'ulveda-Werner \cite{ASW}]
\label{Thm ASW CR ell a b}
Let $b>v$ and $a<v$, such that $b-a\geq 2\lambda$.
Fix $z\in D$.
Then
\begin{displaymath}
\Big(\dfrac{1}{2\pi}\log\Big(\dfrac{\CR(z,D)}{\CR(z,D\setminus A_{a,b})}\Big), \ell_{a,b}(z)\Big)
\end{displaymath}
has the same law as $(T_{a,b},W_{T_{a,b}})$ under $W_{0} = v$.
\end{thm}

For $z\in D\setminus A_{a,b}$,
we will denote by $\Gamma_{a,b}(z)$
the loop of $A_{a,b}$ that surrounds $z$,
that is to say the boundary of the connected component of $z$ in $D\setminus A_{a,b}$.
It is a continuous simple closed curve (Jordan curve)
that locally looks like SLE$_4$.
We will consider the extremal distance (or extremal length)
between $\Gamma_{a,b}(z)$ and the boundary $\partial D$,
which we will denote by $\ED(\partial D,\Gamma_{a,b}(z))$.
This is the same as the electrical resistance,
and is also a conformal invariant.
Actually, for annular domains (one hole),
the extremal distance between the two boundary components parametrizes the moduli space
for the conformal equivalence.
We refer to \cite{Ahlfors2010ConfInv}.
Recall $(W_{t})_{t\geq 0}$ the above one-dimensional Brownian motion starting from
$W_0 = v$.
To express the law of $\ED(\partial D,\Gamma_{a,b}(z))$, we need to introduce the following time.
On the event $\{W_{T_{a,b}}=a\}$, set $\tau_{a,b}$ be equal to
\begin{displaymath}
0\vee \sup\{t\in[0, T_{a,b}]\vert W_t=a+2\lambda\},
\end{displaymath}
and on the event $\{W_{T_{a,b}}=b\}$, set $\tau_{a,b}$ be equal to
\begin{displaymath}
0\vee \sup\{t\in[0, T_{a,b}]\vert W_t=b-2\lambda\}.
\end{displaymath}
Note that $\tau_{a,b}$ is not a stopping time.

\begin{thm}[Aru-Lupu-Sep\'ulveda \cite{ALS3}]
\label{Thm ED ALS3}
Let $b>v$ and $a<v$, such that $b-a\geq 2\lambda$.
Also assume that $b-a\in 2\lambda \N$ (integer multiple of $2\lambda$).
Fix $z\in D$.
Then
\begin{displaymath}
\Big(\dfrac{1}{2\pi}\log\Big(\dfrac{\CR(z,D)}{\CR(z,D\setminus A_{a,b})}\Big), \ell_{a,b}(z), 
\ED(\partial D,\Gamma_{a,b}(z))\Big)
\end{displaymath}
has the same law as $(T_{a,b},W_{T_{a,b}},\tau_{a,b})$ under $W_{0} = v$.
\end{thm}

It is believed that the above identity in law is true without the condition $b-a\in 2\lambda \N$,
and that this is just a technical limitation of the proof provided in \cite{ALS3}.
Theorem \ref{Thm ED ALS3} will play a key role in the proof of Lemma \ref{Lem ED Gauss},
and by extension in that of our main Theorem \ref{Thm A E FPS}, at the technical level.

\subsubsection{First passage sets}
\label{Subsubsec FPS}

The theory of first passage sets (FPS) of the GFF has been introduced in \cite{ALS1,ALS2}.
As previously, let be a constant $v\geq 0$ and $\Phi$ a GFF on $D$ with constant
boundary condition $v$ on $\partial D$.
Let $a<v$.
The first passage set of $\Phi$ of level $a$, denoted $A_{a}$, is
\begin{displaymath}
A_{a} = \overline{\bigcup_{b\geq a+ 2\lambda}A_{a,b}}.
\end{displaymath}
One can think of $A_{a}$ as the set of points in $\overline{D}$
that can be connected to the boundary $\partial D$ by a path along which $\Phi$ takes values larger or equal to $a$.
Again, this picture is informal, since $\Phi$ is not defined pointwise,
but can be made precise through metric graph approximation \cite{ALS2}.
The FPS $A_{a}$ is a random compact connected subset of $\overline{D}$,
containing $\partial{D}$.
It has infinitely many holes.
It also has a.s. Hausdorff dimension $2$, but $0$ Lebesgue measure.
As a random variable, it is measurable w.r.t. the field $\Phi$.
The GFF $\Phi$ admits the following decomposition via $A_{a}$:
\begin{displaymath}
\Phi = \nu_{A_a} + a\ind_{D} + \Phi_{D\setminus A_{a}},
\end{displaymath}
where $\nu_{A_a}$ is a finite positive measure supported on $A_{a}$,
and, as a random variable, measurable w.r.t. $A_{a}$,
and $\Phi_{D\setminus A_{a}}$ is a random field,
distributed conditionally on $A_{a}$,
as a GFF on $D\setminus A_{a}$ with $0$ boundary conditions.
The sum $a\ind_{D} + \Phi_{D\setminus A_{a}}$ is distributed, conditionally on $A_{a}$,
as a GFF on $D\setminus A_{a}$ with boundary condition $a$.
The FPS $A_{a}$ is a local set of the GFF $\Phi$,
but unlike the TVS, it is \textit{non-thin}.
This means that the restriction of $\Phi$ to $A_{a}$ itself is non-trivial,
and in this case it equals the measure $\nu_{A_a}$.
This measure has been identified in \cite{ALS1} as a Minkowski content.

\begin{thm}[Aru-Lupu-Sep\'ulveda \cite{ALS1}]
\label{Thm Mink ALS1}
Fix $a<v$.
The measure $\nu_{A_a}$ is a.s. finite, positive and supported on $A_{a}$.
Moreover, its compact support equals a.s. the whole $A_{a}$,
that is to say a.s., for every $U$ open subset of $\C$
such that $A_a\cap U\neq \emptyset$,
$\nu_{A_a}(U)>0$.
The measure $\nu_{A_a}$ can be described as follows:
\begin{eqnarray*}
\nu_{A_a} &=& \lim_{\varepsilon \to 0}
\dfrac{1}{2} \vert \log\varepsilon\vert^{1/2} \ind_{z\in D\setminus A_{a}, \CR(z,D\setminus A_{a})<\varepsilon \CR(z,D)} d^{2} z
\\
&=& \lim_{\varepsilon \to 0}
\dfrac{1}{2} \vert \log\varepsilon\vert^{1/2} \ind_{z\in D\setminus A_{a},d(z,A_{a})<\varepsilon d(z,\partial D)} d^{2} z
,
\end{eqnarray*}
where the convergence is a.s., for the weak topology on measures.
Fix a deterministic continuous cutoff function $f_{0}:D\rightarrow [0,1]$, compactly supported in $D$.
Then
\begin{eqnarray*}
f_0\,\nu_{A_a} &=& \lim_{\varepsilon \to 0}
\dfrac{1}{2} \vert \log\varepsilon\vert^{1/2} \ind_{z\in D\setminus A_{a}, \CR(z,D\setminus A_{a})<\varepsilon} f_0(z)d^{2} z
\\
&=& \lim_{\varepsilon \to 0}
\dfrac{1}{2} \vert \log\varepsilon\vert^{1/2} \ind_{z\in D\setminus A_{a},d(z,A_{a})<\varepsilon} 
f_0(z) d^{2} z
.
\end{eqnarray*}
Moreover, if 
\begin{equation}
\label{Eq cond Leb boundary ALS1}
\operatorname{Leb}(\{ z\in D\vert d(z,\partial D)<\varepsilon\})
=
o(\vert \log\varepsilon\vert^{-1/2}),
\end{equation}
then
\begin{eqnarray}
\label{Eq nu a Mink non norm}
\nu_{A_a} &=& \lim_{\varepsilon \to 0}
\dfrac{1}{2} \vert \log\varepsilon\vert^{1/2} \ind_{z\in D\setminus A_{a}, \CR(z,D\setminus A_{a})<\varepsilon} d^{2} z
\\
\nonumber
&=& \lim_{\varepsilon \to 0}
\dfrac{1}{2} \vert \log\varepsilon\vert^{1/2} \ind_{z\in D\setminus A_{a},d(z,A_{a})<\varepsilon} d^{2} z
.
\end{eqnarray}
\end{thm}

\begin{rem}
The \eqref{Eq cond Leb boundary ALS1}
ensures that in \eqref{Eq nu a Mink non norm}
there are no deterministic terms coming from the boundary $\partial D$, on top of $\nu_{A_a}$.
\end{rem}

Fix $z\in D$. Then a.s., $z\not\in A_{a}$.
Next we will describe the law of the ratio of conformal radii
$\CR(z,D)/\CR(z,D\setminus A_{a})$.
As previously, let let $(W_{t})_{t\geq 0}$ be a standard one-dimensional Brownian motion
starting from $W_{0} = v$.
Let $T_a$ be the first hitting time
\begin{displaymath}
T_a = \min\{t\geq 0\vert W_t = a\}.
\end{displaymath}

\begin{thm}[Aru-Lupu-Sep\'ulveda \cite{ALS1}]
\label{Thm law CR FPS}
Let $a<v$. Fix $z\in D$.
Then the r.v. 
\begin{displaymath}
\dfrac{1}{2\pi}\log\Big(\dfrac{\CR(z,D)}{\CR(z,D\setminus A_{a})}\Big)
\end{displaymath}
has the same distribution as $T_a$ under $W_0 = v$.
\end{thm}

The above identity in law plays a key role for figuring out the precise form of the asymptotic expansion for first passage sets (Theorem \ref{Thm A E FPS}).
See the explanation at the end of Section \ref{Subsec Pres A E} and in Section \ref{Subsec cond TVS}.

\medskip

In our proof of Theorem \ref{Thm A E FPS},
we will require further properties relating the TVS and the FPS.
Let $b\geq a + 2\lambda$, with $b>v$.
\begin{itemize}
\item A.s., $A_{a,b}\subset A_a$.
\item The couple TVS-label function $(A_{a,b},\ell_{a,b})$ is measurable, as a random variable,
w.r.t. $A_a$.
\item A.s., $A_a \subset A_{a,b}\cup \{z\in D\setminus A_{a,b}\vert \ell_{a,b}(z)=b\}$.
\item Conditionally on $(A_{a,b},\ell_{a,b})$,
the FPS $A_a$ is obtained by sampling inside each connected component of
$\{z\in D\setminus A_{a,b}\vert \ell_{a,b}(z)=b\}$
an independent first passage set from boundary value $b$ to level $a$,
and adding them all to $A_{a,b}$.
\end{itemize}

\medskip

Finally, let us mention that the first passage set $A_{a}$ is distributed as a topological closure of
the cluster connected to $\partial D$ in a Poisson point process of Brownian loops inside $D$
(Brownian loop soup) and boundary-to-boundary Brownian excursions,
the intensity of boundary excursions being proportional to $(v-a)^{2}$.
The Brownian loop soup is the one appearing in Lawler-Werner \cite{LawlerWerner2004ConformalLoopSoup}
and in Sheffield-Werner \cite{SheffieldWerner2012CLE}.
This \textit{a priori} completely different description of first passage sets
originates from random walk (Brownian motion in continuum) representations of the GFF,
and is obtained by metric graph approximation, where an analogous picture holds.
We refer to \cite{ALS2}.
While we will not use this fact in the sequel, 
it was the original hint that Le Gall's asymptotic expansion for the 2D Wiener sausage
(Theorem \ref{Thm Le Gall})
should have an analogue for the first passage sets
(Theorem \ref{Thm A E FPS}).
Note however that our Theorem \ref{Thm A E FPS} is obtained not as consequence of 
Le Gall's expansion, but in a completely different way, from Gaussian considerations.
The main obstacle for using Le Gall's expansion is the ultraviolet divergence
induced by the proliferation of tiny loops in the Brownian loop soup.
Though there should be a way to relate the two expansions beyond a structural analogy,
this is not at all straightforward.
See also Remark \ref{Rem an Le Gall} and Section \ref{Subsec Wiener to FPS}.

\subsubsection{Excursion clusters}
\label{Subsubsec exc}

Despite the fact that the 2D continuum GFF is not defined point-wise, it has well-defined
\textit{sign components}, which we will refer to as \textit{excursion clusters},
but could also be called \textit{excursion sets} or \textit{sign clusters}.
The construction of these is given in \cite{ALS4},
and is closely related to the theory of first passage sets
\cite{ALS1,ALS2} (see Section \ref{Subsubsec FPS}) above.
Here we will recall the main points.

As previously, let $D\subset \C$ be an open, non-empty, bounded, connected and simply connected domain.
Let $\Phi$ a GFF on $D$ with $0$ boundary condition on $\partial D$.
As in \cite{ALS4}, we will order the excursion clusters in the decreasing order of the diameter.
So we have a random countably infinite collection $(\clus_{j})_{j\geq 0}$
of compact, connected, two-by-two disjoint subsets of $\D$.
These are the excursion clusters.
Each comes with a random sign $\sigma_{j}\in\{-1,1\}$.
Moreover, the whole collection of random variables $(\clus_{j},\sigma_{j})_{j\geq 0}$
is measurable w.r.t. the GFF $\Phi$.
Conditionally on $(\clus_{j})_{j\geq 0}$,
the sequence of signs $(\sigma_{j})_{j\geq 0}$ is i.i.d. and uniform (plus or minus with probability $1/2$ each).
Each cluster $\clus_{j}$ caries a positive and finite measure $\nu_{j}$,
given by
\begin{displaymath}
\nu_{j} = 
\lim_{\varepsilon \to 0}
\dfrac{1}{2} \vert \log\varepsilon\vert^{1/2} \ind_{d(z,\clus_{j})<\varepsilon} d^{2} z,
\end{displaymath}
with a.s. convergence for the weak topology of measures.
The measure $\nu_{j}$ is a Minkowski content measure,
with the same gauge as for the first passage sets
(see Theorem \ref{Thm Mink ALS1}).
The whole GFF $\Phi$ can be decomposed
\begin{equation}
\label{Eq exc GFF}
\Phi = \sum_{j\geq 0} \sigma_{j} \nu_{j},
\end{equation}
where the convergence of the infinite sum holds in $L^{2}(d\PP,\sigma(\Phi),H^{-\eta}(\C))$
for every $\eta >0$.
Note that this does not at all mean that the GFF $\Phi$ is a signed measure.
The signs $\sigma_{j}$ are important for the convergence in \eqref{Eq exc GFF},
and the non-signed sum of measures
\begin{displaymath}
\sum_{j\geq 0} \nu_{j}
\end{displaymath}
diverges in every open subset of $D$.
For clarity, let us also point out that the clusters $\clus_{j}$ present a nested structure:
each cluster has countably infinitely many holes,
and inside each hole there are infinitely many other clusters.
This nested structure will be used in the upcoming Section \ref{Subsubsec gen exc}.

Next we explain the relation between excursion clusters and first passage sets.
Denote by $\Gamma_{j}$ the outer boundary of $\clus_{j}$.
It is a CLE$_4$ type loop inside $D$.
Denote by $\inter(\Gamma_{j})$ the interior surrounded by the Jordan loop $\Gamma_{j}$.
It is an open simply-connected subdomain of $D$.
Let $\mathcal{F}_{{\rm ext}, j}$ be the sigma-algebra containing the information of
$(\clus_{i},\sigma_{i})_{0\leq i\leq j-1}$, as well as that of
$\Gamma_{j}$ and $\sigma_{j}$.
Then, conditionally on $\mathcal{F}_{{\rm ext}, j}$, the field
\begin{displaymath}
\ind_{\inter(\Gamma_{j})} \sigma_{j} \Phi
\end{displaymath}
is distributed as a GFF in the domain $\inter(\Gamma_{j})$,
with boundary condition $2\lambda$ on $\Gamma_{j}$,
$2\lambda$ being the height gap.
Moreover, $\clus_{j}$ is the first passage set of this field from level $2\lambda$ on $\Gamma_{j}$
to level $0$.
Note that the fact that we ordered the clusters $\clus_{j}$
in the inverse order of the diameter is not an issue for the conditional laws given
$\mathcal{F}_{{\rm ext}, j}$.
Indeed, $\diam ( \clus_{j}) = \diam (\Gamma_{j})$
and all the $(\clus_{i})_{0\leq i\leq j-1}$ are necessarily outside $\Gamma_{j}$.

The clusters $\clus_{j}$ can be described through the nested CLE$_{4}$ and the Miller-Sheffield coupling.
The labels of the nested CLE$_{4}$ form a branching random walk with step $2\lambda$,
and countable infinite offspring each time.
A first generation CLE$_{4}$ loop is an outer boundary $\Gamma_{j}$ of an excursion cluster $\clus_{j}$
that is not surrounded by any other.
The corresponding value of the branching random walk is $2\lambda \sigma_{j}$.
To get the inner boundary components of $\clus_{j}$,
one has to follow the branching random walk, and along each branch stop when the walk hits $0$.
So the inner boundary components of $\clus_{j}$ are CLE$_{4}$ loops of random generation,
and their corresponding labels are all $0$.
Of course, there are exceptional branches along which the branching random walk never returns to $0$.
These branches are important because they give rise to the measure $\sigma_{j} \nu_{j}$.
Since on each inner boundary component of $\clus_{j}$ the value of the GFF is $0$,
one iterates the above procedure inside each hole of $\clus_{j}$,
with independence between holes, and so on.
This is explained in \cite[Remark 17]{ALS4}.

A second description of the excursion clusters $\clus_{j}$ is that of
topological closures of clusters in a Brownian loop soup. 
This comes from Brownian motion representations of the GFF.
This precise identity is obtained by  combining the description of 
first passage sets as clusters of Brownian loops and boundary excursions
(\cite{ALS2}, see Section \ref{Subsubsec FPS})
with the decomposition of boundary-touching loops into a Poisson point process of boundary excursions obtained by Qian and Werner \cite{QianWerner19Clusters}.
We refer to \cite[Corollary 5.4]{ALS2}.
We will not use this description in the present work,
but this is something to keep in mind in view of Le Gall's asymptotic expansion
(Theorem \ref{Thm Le Gall}).

\subsubsection{Grouping the excursion clusters into generations}
\label{Subsubsec gen exc}

This is a continuation of the previous Section \ref{Subsubsec exc},
with the same setting and notations.
There is a disadvantage in having the excursion clusters simply enumerated 
$(\clus_j)_{j\geq 0}$ as previously.
First, for any $j\geq 0$, the cut domain
\begin{equation}
\label{Eq D minus clus}
D\setminus \bigcup_{i=0}^{j}\clus_i
\end{equation}
does not have all its connected components simply connected.
On top of infinitely many simply connected connected components, 
it has finitely many connected components with holes.
This is just a minor inconvenience,
since the theory we are developing can be extended to
multiply connected domains.
There is however substantial issue.
Conditionally on $(\clus_{i},\sigma_{i})_{0\leq i\leq j}$,
the restriction of the GFF $\Phi$
to the cut domain \eqref{Eq D minus clus}
is not a GFF on this domain,
but a conditioned GFF,
containing the conditioning on not having excursion clusters of larger diameter
than any of the $(\clus_{i})_{0\leq i\leq j}$.
This field is only mildly different from a GFF and is absolutely continuous w.r.t. the latter,
since the conditioning is done on an event of positive probability.
Still, it is not Gaussian, or at least not supposed to be so.
Also, there is no other way to enumerate the excursion clusters
which would avoid dealing with conditioned fields in the cut domains.

To avoid the complications explained above, 
we will regroup the excursion clusters into generations,
which is a standard procedure in the literature.
Let $m\geq 0$. The $m$-th generation is
\begin{displaymath}
\Gen(m) = \{j\geq 0 \vert \clus_{j} \text{ is surrounded by exactly } m \text{ other excursion clusters }\}.
\end{displaymath}
Then $(\Gen(m))_{m\geq 0}$ forms a partition of the set of indices $\N$.
A.s., each $\Gen(m)$ contains infinitely many indices.
We will denote by
\begin{displaymath}
A^{\rm gen}_{m} = \overline{D}\setminus\bigcup_{j\in \Gen(m)} \inter(\Gamma_{j}),
\qquad
A^{\rm gen, +}_{m} = 
A^{\rm gen}_{m}\cup
\bigcup_{j\in \Gen(m)} \clus_{j}.
\end{displaymath}
The $A^{\rm gen}_{m}$ and $A^{\rm gen, +}_{m}$
are compact connected subsets of $\overline{D}$
containing $\partial D$.
Moreover, all the connected components of $D\setminus A^{\rm gen}_{m}$ and
$D\setminus A^{\rm gen, +}_{m}$ are simply connected.
We have the inclusions $A^{\rm gen}_{m}\subset A^{\rm gen, +}_{m}\subset A^{\rm gen}_{m+1}$.

Denote by $\mathcal{F}_{m}^{\rm gen}$
the sigma-algebra containing the information of the
clusters $\clus_j$
for $j\in \cup_{0\leq k\leq m-1} \Gen(k)$
(so $\Gen(m)$ not included),
and of the signs $\sigma_{j}$ for $j\in \cup_{0\leq k\leq m} \Gen(k)$
($\Gen(m)$ included).
Let $\mathcal{F}_{m}^{\rm gen, +}$
denote the sigma-algebra containing the information of the
clusters and signs $(\clus_j,\sigma_{j})$
for $j\in \cup_{0\leq k\leq m} \Gen(k)$
($\Gen(m)$ included both for clusters and signs).
By construction, we have the inclusions
$\mathcal{F}_{m}^{\rm gen}\subset \mathcal{F}_{m}^{\rm gen, +}\subset \mathcal{F}_{m+1}^{\rm gen}$.
Note that, importantly,
neither $\mathcal{F}_{m}^{\rm gen}$ nor $\mathcal{F}_{m}^{\rm gen, +}$
contains the knowledge of the set of indices $\Gen(k)$ for $k\leq m$:
clusters, signs, but not the associated indices.
This point is perhaps not clear from the notations,
and is an artifact of using the enumeration indices.

We will denote by
$\Phi_{D\setminus A^{\rm gen}_{m}}$ the field
\begin{displaymath}
\ind_{D\setminus A^{\rm gen}_{m}} \Phi,
\end{displaymath}
and by $\Phi_{D\setminus A^{\rm gen, +}_{m}}$ the field
\begin{displaymath}
\ind_{D\setminus A^{\rm gen, +}_{m}} \Phi .
\end{displaymath}
Then, conditionally on $\mathcal{F}_{m}^{\rm gen, +}$,
the field $\Phi_{D\setminus A^{\rm gen, +}_{m}}$
is distributed as a GFF on $D\setminus A^{\rm gen, +}_{m}$
with $0$ boundary conditions.
Conditionally on $\mathcal{F}_{m}^{\rm gen}$,
the field $\Phi_{D\setminus A^{\rm gen}_{m}}$
is distributed as a GFF on $D\setminus A^{\rm gen}_{m}$
with boundary conditions $2\lambda \sigma_{j}$ on $\Gamma_{j}$
for $j\in \Gen(m)$.
Conditionally on $\mathcal{F}_{m}^{\rm gen, +}$,
the set $\overline{A^{\rm gen}_{m+1}\setminus A^{\rm gen, +}_{m}}$
(actually same as $A^{\rm gen}_{m+1}$ by density properties)
is the TVS for levels $-2\lambda, 2\lambda$
of the conditional GFF $\Phi_{D\setminus A^{\rm gen, +}_{m}}$.
In particular,
$A^{\rm gen}_{0}$ is the TVS $A_{-2\lambda, 2\lambda}$
of the full GFF $\Phi$,
that is to say the CLE$_{4}$ gasket.
Conditionally on $\mathcal{F}_{m}^{\rm gen}$,
the set $\overline{A^{\rm gen, +}_{m}\setminus A^{\rm gen}_{m}}$
(actually same as $A^{\rm gen, +}_{m}$ by density properties)
is distributed as an FPS from level $2\lambda$ to level $0$
of a GFF on $D\setminus A^{\rm gen}_{m}$ with boundary condition $2\lambda$.
This is because of the symmetry in law between boundary conditions $2\lambda$
and $-2\lambda$.
So the whole sequence $(A^{\rm gen}_{m}, A^{\rm gen, +}_{m})_{m\geq 0}$
is obtained by alternating TVS and FPS,
and always having in the remaining unexplored domain a field conditionally distributed as a GFF.

Now, let us denote
\begin{displaymath}
\vert \nu^{\rm gen}_{m}\vert
=
\sum_{j\in \Gen (m)} \nu_j.
\end{displaymath}
Then $\vert \nu^{\rm gen}_{m}\vert$ is a positive and a.s. finite measure,
despite the sum having infinitely many terms.
Moreover,
\begin{displaymath}
\E[\vert \nu^{\rm gen}_{m}\vert]
= 
2\lambda\E[\ind_{D\setminus A^{\rm gen}_{m}}]
=
2\lambda \ind_{D} <+\infty,
\end{displaymath}
where the first equality is due to $\overline{A^{\rm gen, +}_{m}\setminus A^{\rm gen}_{m}}$
being a conditional FPS in $D\setminus A^{\rm gen}_{m}$,
and the second equality is simply because $A^{\rm gen}_{m}$ has a.s.
$0$ Lebesgue measure.
Let be
\begin{displaymath}
\nu^{\rm gen}_{m}
=
\sum_{j\in \Gen (m)} \sigma_{j}\nu_j.
\end{displaymath}
Then $\nu^{\rm gen}_{m}$ is a \textit{bona fide}
signed measure with finite total variation measure given by $\vert \nu^{\rm gen}_{m}\vert$.
Further, we have the decomposition
\begin{displaymath}
\Phi_{D\setminus A^{\rm gen}_{m}}
=
\nu^{\rm gen}_{m}
+
\Phi_{D\setminus A^{\rm gen, +}_{m}},
\end{displaymath}
and in particular,
\begin{displaymath}
\Phi = \sum_{k=0}^{m}\nu^{\rm gen}_{k}
\, + \Phi_{D\setminus A^{\rm gen, +}_{m}}.
\end{displaymath}
Note that if instead one considers the total variation $\vert \nu^{\rm gen}_{m}\vert$,
then conditionally on $\mathcal{F}_{m}^{\rm gen}$,
the field
\begin{displaymath}
\vert \nu^{\rm gen}_{m}\vert
+
\Phi_{D\setminus A^{\rm gen, +}_{m}},
\end{displaymath}
is distributed as a GFF on $D\setminus A^{\rm gen}_{m}$
with boundary conditions $2\lambda$.
We will rely on this fact in Lemma \ref{Lem moment nu vert eps}.

\subsection{Le Gall's asymptotic expansion for the Wiener sausage in dimension 2}
\label{Subsec Le Gall sausage}

In this section we will recall Le Gall's asymptotic expansion for the Wiener sausage in dimension 2
\cite{LeGallLocTime,LeGallStFlour}.
It was a major inspiration for our expansion for first passage sets (Theorem \ref{Thm A E FPS}).
Note however, that on a technical level, the proofs are very different.

Let us first introduce the setting. Let $M>0$ be a parameter corresponding to a square-mass.
We consider the massive Laplacian $\frac{1}{2}\,\Delta - M$ on $\C$,
and let $G_{M}(z,w)$ be its Green's function:
\begin{displaymath}
G_{M} = \Big(-\dfrac{1}{2}\Delta+M\Big)^{-1}.
\end{displaymath}
This massive Green's function is symmetric, and translation an rotation invariant:
\begin{displaymath}
G_{M}(z,w) = G_{M}(0,\vert w - z\vert).
\end{displaymath}
By writing the Laplacian in polar coordinates,
we see that the function $r\mapsto G_{M}(0,r)$ is solution to a Bessel differential equation.
In particular, $G_{M}$ can be expressed via Bessel functions:
\begin{displaymath}
G_{M}(z,w) = \dfrac{1}{\pi} K_{0}\big(\sqrt{2M} \vert w - z\vert\big),
\end{displaymath}
where $K_0$ is the modified Bessel function of the second kind of order $0$.
See \cite[Equation (2.a)]{LeGallLocTime} and \cite[Section 7.2]{ItoMcKean74Diffusions}.
From this we get, as $\vert w - z\vert\to 0$ the asymptotic expansion
\begin{equation}
\label{Eq G M sing}
G_{M}(z,w) = \dfrac{1}{\pi}\log \dfrac{1}{\vert w - z\vert} + \cst (M) + 
O(\vert w - z\vert^{2} \vert (\log\vert w - z\vert)\vert),
\end{equation}
where
\begin{displaymath}
\cst (M) = \dfrac{1}{\pi}\Big(\dfrac{\log(2) - \log(M)}{2} - \gamma_{\rm EM}\Big),
\end{displaymath}
and $\gamma_{\rm EM}$ is the Euler–Mascheroni constant,
\begin{displaymath}
\gamma_{\rm EM} = \int_{0}^{1}\dfrac{1 - e^{-u}}{u} \, du \, 
-\int_{1}^{+\infty}\dfrac{e^{-u}}{u}\, du.
\end{displaymath}
See \cite[Equation (2.b)]{LeGallLocTime}.

Let $(B_{t})_{t\geq 0}$ be a standard Brownian motion on $\C$ with $B_{0} = 0$.
The $\zeta_M$ be an independent exponential time with mean $M^{-1}$.
Then $(B_{t})_{0\leq t< \zeta_M}$ is the killed Brownian motion,
with infinitesimal generator $1/2\,\Delta - M$.
Let $\Theta_{\zeta_M}$ denote the occupation measure of $(B_{t})_{0\leq t< \zeta_M}$:
\begin{equation}
\label{Eq occup meas}
(\Theta_{\zeta_M},f) = \int_{0}^{\zeta_M} f(B_t)\, dt.
\end{equation}
It is a measure supported on the range $B([0, \zeta_M])$.
For $\varepsilon>0$, let $S_\varepsilon(\zeta_M)$
denote the $\varepsilon$-neighborhood of the trajectory
$B([0, \zeta_M])$:
\begin{displaymath}
S_\varepsilon(\zeta_M) = \{ z\in\C\vert \, d(z, B([0, \zeta_M]))<\varepsilon\}.
\end{displaymath}
This is the so-called \textit{Wiener sausage}.
It standard and easy to show that the renormalized indicator function
\begin{displaymath}
\dfrac{1}{\pi} \vert \log\varepsilon \vert \, \ind_{S_\varepsilon(\zeta_M)}
\end{displaymath}
converges, as $\varepsilon\to 0$, 
to the occupation measure $\Theta_{\zeta_M}$,
in $L^{2}(d\PP,\sigma((B_{t})_{0\leq t< \zeta_M}),H^{-\eta}(\C))$
for $\eta >0$.
So, in a sense, as $\varepsilon\to 0$,
\begin{displaymath}
\ind_{S_\varepsilon(\zeta_M)}
\sim \dfrac{\pi}{\vert \log\varepsilon \vert}\,\Theta_{\zeta_M}.
\end{displaymath}
Le Gall in \cite{LeGallLocTime} pushed beyond this main order equivalent,
and gave a full asymptotic expansion of $\ind_{S_\varepsilon(\zeta_M)}$
into integer powers of $1/\vert \log\varepsilon \vert$.

The next terms in Le Gall's expansion involve the \textit{renormalized self-intersection local times}.
For $n\geq 2$, the powers $\Theta_{\zeta_M}^{n}$ are not defined,
since a.s., the measure $\Theta_{\zeta_M}$ is singular w.r.t. Lebesgue measure.
However, one can define a renormalized version $:\Theta_{\zeta_M}^{n}:$
via a polynomial renormalization procedure. 
These are the Brownian analogues of Wick powers of the GFF; see Section \ref{Subsec Wick}.
However, the combinatorics of the counterterms differs from the Gaussian case. 
Next we introduce it.

We will denote by $(L_{n}^{(-1)})_{n\geq 0}$
the family of generalized Laguerre polynomials of order $-1$.
We refer to \cite{RotaAll73Umbral,AbramowitzStegun84,Roman84Umbral} for references.
The coefficients of $L_{n}^{(-1)}$ are explicit:
$L_{0}^{(-1)}(x)=1$ and for $n\geq 1$,
\begin{displaymath}
L_{n}^{(-1)}(x) = \sum_{k=1}^{n}(-1)^{n-k} \dfrac{n! (n-1)!}{(n-k)! k! (k-1)! }x^{k} .
\end{displaymath}
In this way, $L_{1}^{(-1)}(x) = x$ and
\begin{displaymath}
L_{2}^{(-1)}(x) = x^{2} - 2 x,
\qquad
L_{3}^{(-1)}(x) = x^{3} - 6 x^{2} + 6 x,
\qquad
L_{4}^{(-1)}(x) = x^{4} - 12 x^{3} + 36 x^{2} - 24 x,
\text{ etc.}
\end{displaymath}
We normalized our polynomials so that the leading term of $L_{n}^{(-1)}(x)$ is $x^{n}$.
In literature one often finds different conventions.
The polynomials $(L_{n}^{(-1)})_{n\geq 1}$ are orthogonal for the infinite measure
$\ind_{x>0}\, x^{-1} e^{-x}\, dx$.
Note that for $n\geq 1$, $L_{n}^{(-1)}(0) = 0$. 
In the context of renormalization, we will need the following two-variable homogeneous version of generalized Laguerre polynomials:
\begin{equation}
\label{Eq Lambda}
\Lambda_{n}(x,u) = u^{n} L_{n}^{(-1)}(x u^{-1}) = 
\sum_{k=1}^{n}(-1)^{n-k} \dfrac{n! (n-1)!}{(n-k)! k! (k-1)! }x^{k} u^{n-k}.
\end{equation} 

Further, one regularizes the occupation field $\Theta_{\zeta_M}$ by circle averages.
For $r>0$ and $z\in\C$, let be
\begin{displaymath}
\Theta_{\zeta_M, r}(z) = \lim_{\delta\to 0} \dfrac{1}{4\pi r \delta}\Theta_{\zeta_M}(\{w\in\C\vert\,\vert w-z\vert\in (r-\delta,r+\delta)\}).
\end{displaymath}
Then $z\mapsto \Theta_{\zeta_M, r}(z)$ is a continuous function on $\C$.
For the existence and the continuity of $\Theta_{\zeta_M, r}$,
we refer to \cite[Proposition 1.1]{jegoBMC}.
Further, as $r\to 0$, $\Theta_{\zeta_M, r}$ converges to $\Theta_{\zeta_M}$
in $L^{2}(d\PP,\sigma((B_{t})_{0\leq t< \zeta_M}),H^{-\eta}(\C))$ for $\eta >0$.
Denote by $h(r)$ the (deterministic) quantity
\begin{displaymath}
h_M(r) = \E[\Theta_{\zeta_M, r}(0)] = G_{M}(0,r).
\end{displaymath}
According to \eqref{Eq G M sing}, as $r\to 0$,
\begin{equation}
\label{Eq h M exp}
h_M(r) = \dfrac{1}{\pi} \vert \log r\vert +
\cst (M) + 
O(r^{2} \vert \log r\vert).
\end{equation}
The renormalized self-intersection local time field of multiplicity $n$,
denoted $:\Theta_{\zeta_M}^{n}:$, is
\begin{equation}
\label{Eq renorm mult loc time}
:\Theta_{\zeta_M}^{n}: \, = \lim_{r\to 0} 
\Lambda_{n}(\Theta_{\zeta_M, r},h_M(r)),
\end{equation}
with convergence in $L^{2}(d\PP,\sigma((B_{t})_{0\leq t< \zeta_M}),H^{-\eta}(\C))$ for $\eta >0$.
We refer to Le Gall \cite{LeGallLocTime,LeGallStFlour} for details.
We have $:\Theta_{\zeta_M}^{1}: \, = \Theta_{\zeta_M}$,
and for $n\geq 2$, the $:\Theta_{\zeta_M}^{n}:$ are more complicated generalized functions,
with compact support equal to $B([0, \zeta_M])$.

Note that unlike the case of Gaussian Wick powers,
regularizing by circle averages is important,
and one cannot freely replace these in \eqref{Eq renorm mult loc time}
by other types of regularization.
The fact that Gaussianity is preserved by most reasonable regularizations
(convolution, spectral cutoff, white noise cutoff, etc.) 
allows for much more flexibility in constructing Gaussian Wick powers.
By contrast, in the case of $:\Theta_{\zeta_M}^{n}:$,
the combinatorics of the counterterms depend on the way $\Theta_{\zeta_M}$
is regularized.
In general, if one replaces in \eqref{Eq renorm mult loc time}
the circle averages by other regularizations,
the leading diverging term will still cancel out,
but not necessarily all the lower diverging terms.
To get rid of all the diverging terms,
one needs in general to replace the generalized Laguerre polynomials
by scale-dependent perturbations of these polynomials.
A more general theory of the renormalization of self-intersection local times is presented in
Dynkin \cite{Dynkin1984Polynomials}.

Now we are ready to state the Le Gall's result.

\begin{thm}[Le Gall \cite{LeGallLocTime,LeGallStFlour}]
\label{Thm Le Gall}
Let $N\geq 1$. 
Fix $f$ a smooth function with compact support in $\C$.
Then
\begin{equation}
\label{Eq Le Gal exp f}
\int_{S_\varepsilon(\zeta_M)} f(z) d^2 z\,
=
\sum_{n=1}^{N} (-1)^{n-1} \dfrac{1}{n!}\dfrac{(:\Theta_{\zeta_M}^{n}:,f)}{h_M(\varepsilon)^{n}}
\, + R_{N,\varepsilon}(f),
\end{equation}
where the error term $R_{N,\varepsilon}(f)$ satisfies, as $\varepsilon\to 0$,
\begin{displaymath}
\E[R_{N,\varepsilon}(f)^{2}]^{1/2} = o(\vert \log\varepsilon\vert^{-N}).
\end{displaymath}
\end{thm}

Next are some remarks concerning the Le Gall's expansion.
First, this expansion is obtained by a second moment method.
The two-point mixed correlations
$\langle :\Theta_{\zeta_M}^{m}:(z) :\Theta_{\zeta_M}^{n}:(w)\rangle$
are explicit.
Then one proceeds with an expansion of the two-point correlations of the Wiener sausage,
and more technically,
with the expansion of mixed correlations
$\langle \ind_{S_\varepsilon(\zeta_M)}(z) :\Theta_{\zeta_M}^{n}:(w)\rangle$.
By contrast, we do not know explicit expressions for the two-point correlations of the
first passage sets of the GFF.
Therefore, the proof of our Theorem \ref{Thm A E FPS}
proceeds in a very different way, 
relying on explicit expressions only for the conditional expectations (first moments),
and at the level of second moments, only using upper bounds but no explicit expressions or expansions.

Then, although this is not in \cite{LeGallLocTime,LeGallStFlour},
one can show with minimal additional effort a slightly stronger version
of the expansion \ref{Eq Le Gal exp f}:
\begin{equation}
\label{Eq Le Gall exp norm}
\ind_{S_\varepsilon(\zeta_M)}
=
\sum_{n=1}^{N} (-1)^{n-1} \dfrac{1}{n!}\dfrac{:\Theta_{\zeta_M}^{n}:}{h_M(\varepsilon)^{n}}
\, + R_{N,\varepsilon},
\end{equation}
where the error term $R_{N,\varepsilon}$ satisfies, as $\varepsilon\to 0$,
\begin{displaymath}
\E[\Vert R_{N,\varepsilon}\Vert_{H^{-\eta}(\C)}^{2}]^{1/2} = o(\vert \log\varepsilon\vert^{-N}),
\end{displaymath}
for every $\eta>0$.
This is doable by following Le Gall's proof step by step,
and plunging in everywhere the Bessel potentials $\LK_{\eta}$ \eqref{Eq Bessel potential}.

Further, according to \eqref{Eq h M exp}, the difference
\begin{displaymath}
h_M(\varepsilon) - \dfrac{1}{\pi} \vert \log \varepsilon\vert -
\cst (M)
\end{displaymath}
is negligible compared to any power of $1/ \vert \log \varepsilon\vert$.
Therefore, one can also rewrite the expansion as
\begin{displaymath}
\ind_{S_\varepsilon(\zeta_M)}
=
\sum_{n=1}^{N} (-1)^{n-1} \dfrac{1}{n!}\dfrac{:\Theta_{\zeta_M}^{n}:}
{\big(\frac{1}{\pi}\vert \log \varepsilon\vert + \cst (M)\big)^{n}}
\, + o(\vert \log\varepsilon\vert^{-N}),
\end{displaymath}
where $o(\vert \log\varepsilon\vert^{-N})$ is in the sense of $L^{2}(d\PP,\sigma((B_{t})_{0\leq t< \zeta_M}),H^{-\eta}(\C))$ for $\eta >0$.
One can also reexpand $\big(\frac{1}{\pi}\vert \log \varepsilon\vert + \cst (M)\big)^{-n}$
into powers of $1/ \vert \log \varepsilon\vert$ and get
\begin{displaymath}
\ind_{S_\varepsilon(\zeta_M)}
=
\sum_{n=1}^{N} (-1)^{n-1} 
\Big(\sum_{k=1}^{n} 
\dfrac{(n-1)!}{(n-k)! k! (k-1)!} \cst(M)^{n-k}\,
:\Theta_{\zeta_M}^{k}:\Big)\dfrac{1}{\big(\frac{1}{\pi} \vert \log \varepsilon\vert\big)^{n}}
\, + o(\vert \log\varepsilon\vert^{-N}).
\end{displaymath}
For the details of the computations, we refer to
Proposition \ref{Prop reexp Laguerre}.

Finally, a very important remark.
The expansions \eqref{Eq Le Gal exp f}
or \eqref{Eq Le Gall exp norm} do \textbf{not} at all imply, in any way, an equality,
for fixed $\varepsilon$, between $\ind_{S_\varepsilon(\zeta_M)}$
and
\begin{equation}
\label{Eq sausage inf sum}
\sum_{n=1}^{+\infty} (-1)^{n-1} \dfrac{1}{n!}\dfrac{:\Theta_{\zeta_M}^{n}:}{h_M(\varepsilon)^{n}}.
\end{equation}
Indeed, the fields $:\Theta_{\zeta_M}^{n}:$ do not charge the immediate neighborhood
of the Brownian trajectory, while $\ind_{S_\varepsilon(\zeta_M)}$ does.
Most likely, the infinite sum \eqref{Eq sausage inf sum},
which is a power series in $h_M(\varepsilon)^{-1}$,
does not converge in any meaningful sense.
Actually, one can show the following.

\begin{prop}
\label{Prop 0 rc BM}
Let $f$ be a smooth non-negative function with compact support on $\C$,
and not identically zero.
Then the power series
\begin{displaymath}
\sum_{n=1}^{+\infty} \dfrac{\alpha^{n}}{n!} \E[(:\Theta_{\zeta_M}^{n}:\,,f)^{2}]^{1/2}
\end{displaymath}
has zero radius of convergence.
Similarly, for ever $\eta>0$, the power series
\begin{displaymath}
\sum_{n=1}^{+\infty} \dfrac{\alpha^{n}}{n!} \E[\Vert:\Theta_{\zeta_M}^{n}:\Vert_{H^{-\eta}(\C)}^{2}]^{1/2}
\end{displaymath}
has zero radius of convergence.
\end{prop}

\begin{proof}
We will consider only the first part, as the second is similar.
The two-point correlation function of $:\Theta_{\zeta_M}^{n}:$ is
\begin{displaymath}
\langle :\Theta_{\zeta_M}^{n}:(z) :\Theta_{\zeta_M}^{n}:(w)\rangle
= (n!)^{2}
(G_M(0,z) + G_M(0,w))G_M(z,w)^{2n-1} .
\end{displaymath}
Taking into account the logarithmic singularity of $G_M(z,w)$ 
on the diagonal, we get that
\begin{displaymath}
 \E[(:\Theta_{\zeta_M}^{n}:\,,f)^{2}]
 \geq 
 c (n!)^{2} \theta^{n} (2n-1)!\, .
\end{displaymath}
for constants $c,\theta>0$ not depending on $n$.
Then indeed, the power series in the variable $\alpha$,
\begin{displaymath}
\sum_{n=1}^{+\infty} \dfrac{\alpha^{n}}{n!} [\theta^{n} (n!)^{2} (2n-1)!]^{1/2}\, ,
\end{displaymath}
has zero radius of convergence.
\end{proof}

\subsection{Umbral composition of polynomial sequences}
\label{Subsec umbral}

The change of variance identity for Hermite polynomials \eqref{Eq change var}
plays a key role in the context of Wick renormalization for Gaussians.
Further in this paper, it will be ubiquitous.
It appears behind Lemma \ref{Lem decomp Phi eps},
Corollary \ref{Cor cond exp eps},
Theorem \ref{Thm decomp FPS Wick} and
Proposition \ref{Prop decomp Wick TVS}
in Section \ref{Sec Wick FPS},
Propositions \ref{Prop reexp} and \ref{Prop cond series} in Section \ref{Sec A E},
Lemma \ref{Lem psi hat psi} and Section \ref{Subsec algeb} in general,
and Lemma \ref{Lem decomp Phi gen 1}, Proposition \ref{Prop decomp gen 0 Wick},
Proposition \ref{Prop decomp gen Wick iter},
Theorem \ref{Thm decomp gen Wick},
Proposition \ref{Prop gen to indiv}
and Proposition \ref{Prop decomp Wick BTLS} in Section \ref{Sec Wick exc decomp}.
Given this importance,
we would like to present an algebraic viewpoint on this identity,
related to the umbral calculus and the umbral composition of polynomial sequences.
Moreover, this point of view also encompasses the generalized Laguerre polynomials
used for renormalizing the Brownian self-intersection local times (Section \ref{Subsec Le Gall sausage}).
What we will present here is completely standard in the field of umbral calculus \cite{RotaAll73Umbral,Roman84Umbral}
(algebra and combinatorics),
but we have not encountered it the literature dealing with renormalization.
We hope that our paper will contribute to strengthening the link between the two domains.

In the sequel, we will consider polynomials with coefficients in $\R$.
However, what we present holds if one replaces $\R$ by any other number field
($\C$, $\mathbb{Q}$, $\mathbb{Q}[\sqrt{2}]$, $\Z/ p\Z$, etc.).
For us, a \textit{polynomial sequence} will be a family of
polynomials $(P_{n}(x))_{n\geq 0}$
such that for every $n\geq 0$, $\deg P_{n} = n$.
In particular, $(P_{n}(x))_{n\geq 0}$ forms a basis of $\R[x]$.
We will denote by $\mathtt{S}_{\R}$ the set of all polynomial sequences 
with coefficients in $\R$. 
A polynomial sequence can be encoded by its matrix of coefficients:
\begin{displaymath}
P_{n}(x) = \sum_{k=0}^{n} a_{k,n} x^{k},
\end{displaymath}
with $a_{n,n}\neq 0$.
In this way, an element of $\mathtt{S}_{\R}$
can be seen as an upper-triangular matrix of infinite size, 
with non-zero coefficients on the diagonal.

The set $\mathtt{S}_{\R}$ can be endowed with a composition law,
called \textit{umbral composition},
and which we will denote here by $\odot$.
Given two polynomial sequences $(P_{n}(x))_{n\geq 0}$ and
$(\widetilde{P}_{n}(x))_{n\geq 0}$ with coefficients
\begin{displaymath}
P_{n}(x) = \sum_{k=0}^{n} a_{k,n} x^{k},
\qquad
\widetilde{P}_{n}(x) = \sum_{k=0}^{n} \tilde{a}_{k,n} x^{k},
\end{displaymath}
the umbral composition $P\odot\widetilde{P}$
is the polynomials sequence
\begin{equation}
\label{Eq def umbral}
(P\odot\widetilde{P})_{n}(x) = 
\sum_{k=0}^{n} \tilde{a}_{k,n}P_{k}(x)
=
\sum_{j=0}^{n}\Big(\sum_{k=j}^{n} a_{j,k}\tilde{a}_{k,n}\Big) x^{j}.
\end{equation}

\begin{prop}
\label{Prop group}
The composition law $\odot$ is associative,
but it is not commutative.
The monomial sequence $(x^{n})_{n\geq 0}$
is the unique left-neutral and the unique right-neutral element for $\odot$.
Endowed with this composition law, $(\mathtt{S}_{\R},\odot)$ forms a group,
and in particular every element of $\mathtt{S}_{\R}$ has an inverse for $\odot$.
\end{prop}

\begin{proof}
From the second equality in \eqref{Eq def umbral}
we see that the composition law $\odot$ is the same as matrix multiplication of
two upper-triangular matrices of infinite size.
Since both matrices are upper-triangular, 
this multiplication makes sense despite the infinite size.
One can deduce from this all the above elementary properties.
The monomial sequence $(x^{n})_{n\geq 0}$
corresponds to an infinite-size matrix with $1$ on the diagonal and $0$ elsewhere.
The invertibility for $\odot$ comes from the fact that the diagonal coefficients are non-zero,
and one can for instance use the Gaussian-Jordan elimination, as the matrices are already in row-echelon form.
\end{proof}

Now recall $Q_{n}(x,u)$ \eqref{Eq def Q} the Hermite polynomials with two variables.
For fixed $u$, we will see $(Q_{n}(x,u))_{n\geq 0}$ as a polynomial sequence
with respect to the polynomial variable $x$,
and we will use the short notation $Q(\cdot,u)$ for it.
By using the composition law $\odot$,
the change of variance identity \eqref{Eq change var} can be rewritten as
\begin{equation}
\label{Eq change var umbral}
Q(\cdot,u_{1} + u_{2}) = Q(\cdot,u_1)\odot Q(\cdot,u_2).
\end{equation}
Note that at an algebraic level, $u_{1}$ and $u_{2}$ can assume arbitrary values in $\R$,
although in a renormalization context, the variances are positive and blow up to infinity.
One can further restate \eqref{Eq change var umbral} as follows.

\begin{cor}
\label{Cor Hermite subgroup}
The map $u\mapsto Q(\cdot,u)$ is an injective group homomorphism from
$(\R,+)$ to $(\mathtt{S}_{\R},\odot)$.
In particular, $(Q(\cdot,u))_{u\in\R}$ forms a one-parameter subgroup of $(\mathtt{S}_{\R},\odot)$.
\end{cor}

A similar property holds for the  two-variable generalized Laguerre polynomials 
$\Lambda_{n}(x,u)$ \eqref{Eq Lambda} that are used for renormalizing the Brownian self-intersection local times 
(Section \ref{Subsec Le Gall sausage}).
One has a change of normalization identity:
\begin{equation}
\label{Eq change norm Laguerre}
\Lambda_{n}(x,u_{1}+u_{2}) = 
\sum_{k=1}^{n}(-1)^{n-k} \dfrac{n! (n-1)!}{(n-k)! k! (k-1)! }\Lambda_{k}(x,u_{1}) u_{2}^{n-k},
\end{equation}
where the coefficients appearing are the same as in the definition of $\Lambda_{n}(x,u)$ \eqref{Eq Lambda}.
For the importance of this identity in a renormalization context, we refer to 
Proposition \ref{Prop powers loc time change norm} and Section \ref{Subsec BM algeb} in general,
as well as to Section \ref{Subsec Wiener to FPS}.
This identity can be rewritten as 
\begin{displaymath}
\Lambda(\cdot,u_{1}+u_{2}) = \Lambda(\cdot,u_{1})\odot \Lambda(\cdot,u_{2}).
\end{displaymath}
One can further restate this as follows.

\begin{cor}
\label{Cor Laguerre subgroup}
The map $u\mapsto \Lambda(\cdot,u)$ is an injective group homomorphism from
$(\R,+)$ to $(\mathtt{S}_{\R},\odot)$.
In particular, $(\Lambda(\cdot,u))_{u\in\R}$ forms a one-parameter subgroup of $(\mathtt{S}_{\R},\odot)$.
\end{cor}

Finally, we would like to point out that both the Hermite polynomials and the generalized Laguerre polynomials
belong to a more specific subgroup of $(\mathtt{S}_{\R},\odot)$,
with more structure. 
This is the subgroup of Sheffer sequences.
We refer to \cite{RotaAll73Umbral,Roman84Umbral}
and in particular to \cite[Theorem 7]{RotaAll73Umbral}.
We do not know whether this fact is relevant in the context of renormalization,
but possibly it is.

\section{Estimates on Green's function and conformal radius}
\label{Sec estimates}

\subsection{An estimate for the Green's function both near the boundary and diagonal}
\label{Subsec estim Green}

The Green's function $G_{D}(z,w)$ has on one hand a logarithmic divergence when $z$ and $w$ are close,
and on the other hand converges to $0$ when one of the two points is close to the boundary $\partial D$.
The following estimate will be useful in the case when $z$ and $w$ are both close
to each other and close to the boundary,
and will give a sufficient condition for $G_{D}(z,w)$ to be small nevertheless.
It is used in the proof of Lemmas \ref{Lem key estimate}, \ref{Lem f q}, \ref{Lem a s bound},
\ref{Lem a s bound q} and \ref{Lem conv 0 remainder Sob}.

\begin{prop}
\label{Prop Green}
There is a universal constant $C>0$,
such that for every $D\subset \C$ open, non-empty, bounded, connected and simply connected domain,
and for every $z\neq w\in D$,
\begin{equation}
\label{Eq est Green}
G_{D}(z,w) \leq C\Big( 1\wedge \Big(
\dfrac{d(z,\partial D)\wedge d(w,\partial D)}{\vert w-z\vert}\Big)^{1/2}
\Big)
\log\Big(\dfrac{2 \diam (D)}{\vert w-z\vert}\Big).
\end{equation}
\end{prop}

\begin{proof}
Fix a domain $D$ as above and two points $z\neq w\in D$.
Let $R=2 \diam (D)$.
By symmetry, we may assume that
$d(z,\partial D)\leq d(w,\partial D)$.
Let $\widetilde{D} = \D(w,R)$ be
the open disk of center $w$ and radius $R$.
The for every $x\in \widetilde{D}$,
\begin{displaymath}
G_{\widetilde{D}}(w,x) = \dfrac{1}{2\pi}\log(R/\vert w-x\vert).
\end{displaymath}
By construction, 
$D\subset \widetilde{D}$,
and in particular, $G_{D}(z,w)\leq G_{\widetilde{D}}(z,w)$.
Let $(\wp_{\widetilde{D}}^{z,w}(t))_{0\leq t\leq t_{\rm max}}$
be the Brownian excursions in $\widetilde{D}$ from $z$ to $w$.
Then
\begin{displaymath}
\dfrac{G_{D}(z,w)}{G_{\widetilde{D}}(z,w)}
=\PP(\wp_{\widetilde{D}}^{z,w} \text{ stays in } D).
\end{displaymath}
Consider $\D(z,\vert w-z\vert/2)$ the open disk centered in $z$ of radius
$\vert w-z\vert/2$,
which is again contained in $\widetilde{D}$.
Let $\widetilde{T}\in (0, t_{\rm max})$
be the first time $\wp_{\widetilde{D}}^{z,w}(t)$
exits $\D(z,\vert w-z\vert/2)$.
Then
\begin{displaymath}
\PP(\wp_{\widetilde{D}}^{z,w} \text{ stays in } D)
\leq\PP(\wp_{\widetilde{D}}^{z,w} \text{ stays in } D \text{ on time interval }
[0, \widetilde{T}]).
\end{displaymath}
Let $(B_{t})_{t\geq 0}$ be a Brownian motion starting from $z$,
and let $T$ be its first exit time from $\D(z,\vert w-z\vert/2)$.
Then the law of 
$(\widetilde{T},(\wp_{\widetilde{D}}^{z,w}(t))_{0\leq t\leq \widetilde{T}})$
is absolutely continuous with respect to that of
$(T,(B_{t})_{0\leq t\leq T})$,
and the Radon-Nikodym derivative is given by
\begin{displaymath}
\dfrac{G_{\widetilde{D}}(B_{T},w)}{G_{\widetilde{D}}(z,w)}.
\end{displaymath}
This Radon-Nikodym derivative is bounded from above by
\begin{displaymath}
\dfrac{G_{\widetilde{D}}((z+w)/2,w)}{G_{\widetilde{D}}(z,w)}
=\dfrac{\log(2 R/ \vert w-z\vert)}{\log(R/ \vert w-z\vert)}
= 1 + \dfrac{\log 2}{\log(R/ \vert w-z\vert)}
\leq 2.
\end{displaymath}
Thus,
\begin{displaymath}
\PP(\wp_{\widetilde{D}}^{z,w} \text{ stays in } D \text{ on }
[0, \widetilde{T}])
\leq
2 \mathbb{P}(B_{t}\text{ stays in } D \text{ on } [0,T]).
\end{displaymath}
Assume that $d(z,\partial D)<\vert w-z\vert/2$.
Then, by the Beurling's estimate,
\begin{displaymath}
\mathbb{P}(B_{t}\text{ stays in } D \text{ on } [0,T])
\leq C_{\rm Beurling}
\Big(\dfrac{2 d(z,\partial D)}{\vert w-z\vert}\Big)^{1/2},
\end{displaymath}
where $C_{\rm Beurling}\geq 1$ is a universal constant not depending
$D$, $z$ or $w$.
We refer to \cite[Section 3.8]{LawlerConformallyInvariantProcesses}
and \cite{Oksendal83Beurling}.
In case $d(z,\partial D)\geq \vert w-z\vert/2$,
then trivially $\PP(\wp_{\widetilde{D}}^{z,w} \text{ stays in } D)$
can be always bounded by $1$.
Thus, \eqref{Eq est Green} holds by taking
\begin{displaymath}
C = \dfrac{1}{2\pi} 2\sqrt{2} C_{\rm Beurling}.
\qedhere
\end{displaymath}

\end{proof}

\subsection{An estimate for the variations of the Green's function and conformal radius}
\label{Subsec estim CR}

Here we derive estimates on the variation of Green's function and conformal radius,
that will be used in the proof of Proposition \ref{Prop V G eps}.

Recall that by \eqref{Eq G CR}, for $D\subsetneq \C$ and open simply connected domain,
\begin{displaymath}
g_{D}(z,w) = G_{D}(z,w) - 
\dfrac{1}{2\pi}\log \dfrac{1}{\vert w-z\vert}.
\end{displaymath}

\begin{prop}
\label{Prop var g D}
For every $D\subsetneq \C$ open simply connected domain,
for every $z,z',w,w'\in D$ such that
$\vert z' - z\vert \leq \frac{1}{2}d(z,\partial D)$ and
$\vert w' - w\vert \leq \frac{1}{2}d(w,\partial D)$,
we have
\begin{displaymath}
\vert g_{D}(z',w') - g_{D}(z,w)\vert
\leq 16\dfrac{\vert z' - z\vert}{d(z,\partial D)}
\Big(g_{D}(z,w)-\dfrac{1}{2\pi}\log d(w,\partial D)\Big)
+ \dfrac{\log(2)}{\pi}\dfrac{\vert w' - w\vert}{d(w,\partial D)}.
\end{displaymath}
\end{prop}

\begin{proof}
Let $(B_{t})_{t\geq 0}$ be a Brownian motion on $\C$.
Let $T_{\C\setminus D}$ denote its first exit time from $D$.
By $\E_{x}[\cdot]$ we will denote the expectation under the initial condition $B_{0}=x$.
Then
\begin{displaymath}
g_{D}(z,w) = \dfrac{1}{2\pi}\E_{z}\Big[\log\vert B_{T_{\C\setminus D}} - w\vert\Big],
\qquad
g_{D}(z',w') = \dfrac{1}{2\pi}\E_{z'}\Big[\log\vert B_{T_{\C\setminus D}} - w'\vert\Big].
\end{displaymath}
Let us denote by $P_{\D}$ the Poisson kernel of the unit disk $\D$.
We have
\begin{displaymath}
P_{\D}(\xi,e^{i\theta}) = \dfrac{1}{2\pi}\dfrac{1-\vert \xi\vert^{2}}{\vert e^{i\theta} - \xi\vert^{2}}.
\end{displaymath}
Let $R = d(z,\partial D)$.
By first stopping the Brownian motion when it exists $\D(z,R)$,
the disk of radius R and center $z$, we get
\begin{eqnarray*}
g_{D}(z,w) &=&
\dfrac{1}{2\pi}\int_{0}^{2\pi}
d\theta
\,P_{\D}(0,e^{i\theta})
\E_{z+R e^{i\theta}}\Big[\log\vert B_{T_{\C\setminus D}} - w\vert\Big],
\\
g_{D}(z',w') &=&
\dfrac{1}{2\pi}\int_{0}^{2\pi}
d\theta
\,P_{\D}((z' - z) R^{-1},e^{i\theta})
\E_{z+R e^{i\theta}}\Big[\log\vert B_{T_{\C\setminus D}} - w'\vert\Big].
\end{eqnarray*}
Let $\widehat{R} = d(w,\partial D)$.
By rescaling all the distances by $\widehat{R}$, we get 
\begin{eqnarray*}
g_{D}(z,w) &=&
\dfrac{1}{2\pi}\int_{0}^{2\pi}
d\theta
\,P_{\D}(0,e^{i\theta})
\E_{z+R e^{i\theta}}\Big[\log (\widehat{R}^{-1}\vert B_{T_{\C\setminus D}} - w\vert)\Big],
\\
g_{D}(z',w') &=&
\dfrac{1}{2\pi}\int_{0}^{2\pi}
d\theta
\,P_{\D}((z' - z) R^{-1},e^{i\theta})
\E_{z+R e^{i\theta}}\Big[\log(\widehat{R}^{-1}\vert B_{T_{\C\setminus D}} - w'\vert)\Big].
\end{eqnarray*}
By doing that, we get now that
$\log (\widehat{R}^{-1}\vert B_{T_{\C\setminus D}} - w\vert)\geq 0$.
In this way,
\begin{multline*}
\vert g_{D}(z',w') - g_{D}(z,w)\vert 
\leq 
\\
\dfrac{1}{2\pi}\int_{0}^{2\pi}
d\theta
\,
\vert P_{\D}(0,e^{i\theta}) - P_{\D}((z' - z) R^{-1},e^{i\theta})\vert
\E_{z+R e^{i\theta}}\Big[\log (\widehat{R}^{-1}\vert B_{T_{\C\setminus D}} - w\vert)\Big]
\\
+
\dfrac{1}{2\pi}\int_{0}^{2\pi}
d\theta
\,
P_{\D}((z' - z) R^{-1},e^{i\theta})
\E_{z+R e^{i\theta}}
\Big[
\Big\vert
\log
\dfrac{\vert B_{T_{\C\setminus D}} - w'\vert}{\vert B_{T_{\C\setminus D}} - w\vert}
\Big\vert
\Big]
.
\end{multline*}
Further,
\begin{displaymath}
1 -
\dfrac{\vert w' -w\vert}{\vert B_{T_{\C\setminus D}} - w\vert} 
\leq
\dfrac{\vert B_{T_{\C\setminus D}} - w'\vert}{\vert B_{T_{\C\setminus D}} - w\vert}
\leq 1 + 
\dfrac{\vert w' -w\vert}{\vert B_{T_{\C\setminus D}} - w\vert},
\end{displaymath}
and
\begin{displaymath}
-
2\log(2)
\dfrac{\vert w' -w\vert}{\vert B_{T_{\C\setminus D}} - w\vert}
\leq
\log
\dfrac{\vert B_{T_{\C\setminus D}} - w'\vert}{\vert B_{T_{\C\setminus D}} - w\vert}
\leq
\dfrac{\vert w' -w\vert}{\vert B_{T_{\C\setminus D}} - w\vert}.
\end{displaymath}
For the lower bound we used the concavity of the $\log$ function and
the fact that 
\begin{displaymath}
\dfrac{\vert w' -w\vert}{\vert B_{T_{\C\setminus D}} - w\vert}
\leq
\dfrac{\vert w' -w\vert}{d(w,\partial D)}
\leq 1/2.
\end{displaymath}
Thus,
\begin{displaymath}
\int_{0}^{2\pi}
d\theta
\,
P_{\D}((z' - z) R^{-1},e^{i\theta})
\E_{z+R e^{i\theta}}
\Big[
\Big\vert
\log
\dfrac{\vert B_{T_{\C\setminus D}} - w'\vert}{\vert B_{T_{\C\setminus D}} - w\vert}
\Big\vert
\Big]
\leq 
2 \log(2) \dfrac{\vert w' -w\vert}{d(w,\partial D)}.
\end{displaymath}
Moreover, for every $\theta\in[0,2\pi]$,
\begin{displaymath}
\vert P_{\D}(0,e^{i\theta}) - P_{\D}((z' - z) R^{-1},e^{i\theta})\vert
\leq \dfrac{2r(1+r)}{(1-r)^{2}} P_{\D}(0,e^{i\theta})
\leq 16 r P_{\D}(0,e^{i\theta}),
\end{displaymath}
where $r= \vert z'-z\vert /R$.
This concludes.
\end{proof}

\begin{cor}
\label{Cor estim CR}
For every $D\subsetneq \C$ open simply connected domain and
for every $z,z'\in D$ such that
$\vert z' - z\vert \leq \frac{1}{2}d(z,\partial D)$,
we have
\begin{displaymath}
\dfrac{1}{2\pi}
\Big\vert
\log
\dfrac{\CR(z',D)}{\CR(z,D)}
\Big\vert
\leq
\dfrac{17\log(2)}{\pi}\dfrac{\vert z' - z\vert}{d(z,\partial D)}.
\end{displaymath}
\end{cor}

\begin{proof}
We apply Proposition \ref{Prop var g D} with $w=z$ and $w'= z'$, to get
\begin{displaymath}
\dfrac{1}{2\pi}
\Big\vert
\log
\dfrac{\CR(z',D)}{\CR(z,D)}
\Big\vert
\leq
\Big(\dfrac{8}{\pi}
\log
\dfrac{\CR(z,D)}{d(z,\partial D)} + \dfrac{\log(2)}{\pi}\Big)\dfrac{\vert z' - z\vert}{d(z,\partial D)}.
\end{displaymath}
Further, we use the distortion inequality \eqref{Eq Koebe}.
\end{proof}

\begin{cor}
\label{Cor var G D}
For every $D\subsetneq \C$ open simply connected domain,
for every $z,z',w,w'\in D$ such that
$\vert z' - z\vert \leq \frac{1}{2}d(z,\partial D)$,
$\vert w' - w\vert \leq \frac{1}{2}d(w,\partial D)$,
$\vert z' - z\vert\leq \frac{1}{4}\vert z-w\vert$,
and
$\vert w' - w\vert\leq \frac{1}{4}\vert z-w\vert$,
we have
\begin{eqnarray*}
\vert G_{D}(z',w') - G_{D}(z,w)\vert &\leq &
\dfrac{\log(2)}{\pi}\dfrac{\vert z' - z\vert +\vert w' - w\vert}{\vert z-w\vert}
\\  && +
16\dfrac{\vert z' - z\vert}{d(z,\partial D)}
\Big(G_{D}(z,w)+\dfrac{1}{2\pi}\log \dfrac{\vert z-w\vert}{d(w,\partial D)}\Big)
+ \dfrac{\log(2)}{\pi}\dfrac{\vert w' - w\vert}{d(w,\partial D)}.
\end{eqnarray*}
\end{cor}

\begin{proof}
By the decomposition \eqref{Eq G CR} and Proposition \ref{Prop var g D}, we get
that $\vert G_{D}(z',w') - G_{D}(z,w)\vert$ is bounded by
\begin{multline*}
\dfrac{1}{2\pi}\Big\vert\log\dfrac{\vert z'-w'\vert}{\vert z-w\vert}\Big\vert
+
16\dfrac{\vert z' - z\vert}{d(z,\partial D)}
\Big(g_{D}(z,w)-\dfrac{1}{2\pi}\log d(w,\partial D)\Big)
+ \dfrac{\log(2)}{\pi}\dfrac{\vert w' - w\vert}{d(w,\partial D)}
\\  =
\dfrac{1}{2\pi}\Big\vert\log\dfrac{\vert z'-w'\vert}{\vert z-w\vert}\Big\vert +
16\dfrac{\vert z' - z\vert}{d(z,\partial D)}
\Big(G_{D}(z,w)+\dfrac{1}{2\pi}\log \dfrac{\vert z-w\vert}{d(w,\partial D)}\Big)
+ \dfrac{\log(2)}{\pi}\dfrac{\vert w' - w\vert}{d(w,\partial D)}.
\end{multline*}
Further,
\begin{displaymath}
\Big\vert\log\dfrac{\vert z'-w'\vert}{\vert z-w\vert}\Big\vert
\leq 2\log (2) \dfrac{\vert z' - z\vert +\vert w' - w\vert}{\vert z-w\vert},
\end{displaymath}
where we also used the concavity of $\log$.
\end{proof}

\section{Decomposition of Wick powers induced by first passage sets}
\label{Sec Wick FPS}

\subsection{Conditional expectation of Wick powers given an FPS}
\label{Subsec cond Wick FPS}

Let $D\subset \C$ be an open, non-empty, bounded, connected and simply connected domain.
Let be a constant $v>0$ and $\Phi$ a GFF on $D$ with constant
boundary condition $v$ on $\partial D$.
Let $A$ be the first passage set (FPS) of $\Phi$
from level $v$ to level $0$;
see Section \ref{Subsubsec FPS}.
Note that the level $0$ is not essential, but it simplifies the upcoming computations and expressions.
The GFF $\Phi$ can be decomposed
\begin{equation}
\label{Eq decomp Phi}
\Phi = \nu_{A} + \Phi_{D\setminus A},
\end{equation}
where $\nu_{A}$ is a positive measure supported on A,
measurable w.r.t. $A$
(Minkowski content, Theorem \ref{Thm Mink ALS1})
and conditionally on $A$,
$\Phi_{D\setminus A}$ is distributed as a GFF on $D\setminus A$
with $0$ boundary condition.

Denote $\psi_{n,A}$ the conditional expectation
\begin{displaymath}
\psi_{n,A} = \E\big[:\Phi^{n}:\vert A\big],
\end{displaymath}
where $:\Phi^{n}:$ is the $n$-th Wick power of $A$.
Since $\E\big[\Vert :\Phi^{n}:\Vert_{H^{-\eta}(\C)}^{2}\big] < +\infty$,
$\psi_{n,A}$ is a well defined random element of $H^{-\eta}(\C)$, and
\begin{displaymath}
\Vert\psi_{n,A}\Vert_{H^{-\eta}(\C)}^{2}
\leq
\E\big[\Vert :\Phi^{n}:\Vert_{H^{-\eta}(\C)}^{2}\vert A\big]
~~\text{a.s.}
\end{displaymath}
We refer to \cite[Section 2.6]{HNVW16AnBanach1}.
Note that $\psi_{1,A} = \nu_{A}$.

The random fields $\psi_{n,A}$ are well defined from abstract considerations,
and measurable w.r.t. $A$.
However, one of our goals here is to give a more explicit description of these fields.

Let $(\rho_{\varepsilon})_{\varepsilon>0}$ be a family of mollificators as in
Section \ref{Subsec Wick}. Then
\begin{displaymath}
\Phi_{\varepsilon} = \nu_{A,\varepsilon} + \Phi_{D\setminus A,\varepsilon},
\end{displaymath}
where $\Phi_{D\setminus A,\varepsilon} = \Phi_{D\setminus A}\ast\rho_{\varepsilon}$,
and
\begin{displaymath}
\nu_{A,\varepsilon}(z) = \int_{\C} \rho_{\varepsilon}(z-w)\,d\nu_{A}(w).
\end{displaymath}

\begin{lemma}
\label{Lem cond Psi eps}
For every $\eta>0$,
\begin{displaymath}
\lim_{\varepsilon\to 0}
\E\Big[
\big\Vert
\psi_{n,A}
-
\E\big[:\Phi^{n}_{\varepsilon}:\vert A\big]
\big\Vert_{H^{-\eta}(\C)}^{2}
\Big] = 0.
\end{displaymath}
\end{lemma}

\begin{proof}
Indeed,
For every $\eta>0$,
\begin{displaymath}
\E\Big[
\big\Vert
\psi_{n,A}
-
\E\big[:\Phi^{n}_{\varepsilon}:\vert A\big]
\big\Vert_{H^{-\eta}(\C)}^{2}
\Big]
\leq
\E\Big[
\big\Vert
:\Phi^{n}:
-
:\Phi^{n}_{\varepsilon}:
\big\Vert_{H^{-\eta}(\C)}^{2}
\Big] = 0.
\qedhere
\end{displaymath}
\end{proof}

Denote
\begin{displaymath}
V_{A}(z) = g_{D}(z,z) - g_{D\setminus A}(z,z),
\end{displaymath}
\begin{displaymath}
V_{A,\varepsilon}(z) = 
G_{D,\varepsilon,\varepsilon}(z,z) - 
G_{D\setminus A,\varepsilon,\varepsilon}(z,z)
=
g_{D,\varepsilon,\varepsilon}(z,z) - 
g_{D\setminus A,\varepsilon,\varepsilon}(z,z).
\end{displaymath}
For $z\in\C\setminus D$, $V_{A}(z)=0$,
and for $z\in D\setminus A$,
\begin{equation}
\label{Eq V A CR}
V_{A}(z) = \dfrac{1}{2\pi}\log\Big(
\dfrac{\CR(z,D)}{\CR(z,D\setminus A)}
\Big),
\end{equation}
where $\CR(z,D\setminus A)$ denotes the conformal radius of the connected
component of $z$ in $D\setminus A$.
On $A\setminus\partial D$, $V_{A}$ is not defined and morally values
$+\infty$, but we will not need that.
By contrast, the function $V_{A,\varepsilon}$ is smooth, with compact support on
$\C$.

\begin{lemma}
\label{Lem decomp Phi eps}
For $n\geq 1$ and $z\in\C$,
\begin{eqnarray*}
:\Phi^{n}_{\varepsilon}:(z)&=&
Q_{n}(\nu_{A,\varepsilon}(z),V_{A,\varepsilon}(z)) +
Q_{n}(\Phi_{D\setminus A,\varepsilon}(z),G_{D,\varepsilon,\varepsilon}(z,z))
\\
&&-\ind_{n \text{ even}}(-1)^{n/2}
\dfrac{n!}{2^{n/2} (n/2)!}
V_{A,\varepsilon}(z)^{n/2}
\\&& +
\sum_{\substack{0\leq j\leq n-1\\0\leq k <\lfloor j/2\rfloor}}
(-1)^{k}\dfrac{n!}{2^{k} (n-j)! k! (j-2k)!}
V_{A,\varepsilon}(z)^{k}
\nu_{A,\varepsilon}(z)^{n-j}
\,
:\Phi_{D\setminus A,\varepsilon}^{j-2k}:(z).
\end{eqnarray*}
\end{lemma}

\begin{proof}
By applying the binomial identity \eqref{Eq Q bin},
we get that $:\Phi^{n}_{\varepsilon}:(z)$ equals
\begin{multline*}
\sum_{j=0}^{n} \dfrac{n!}{j! (n-j)!} 
Q_{j}(\Phi_{D\setminus A,\varepsilon}(z),G_{D,\varepsilon,\varepsilon}(z,z))\, 
\nu_{A,\varepsilon}(z)^{n-j}
=
\\
Q_{n}(\Phi_{D\setminus A,\varepsilon}(z),G_{D,\varepsilon,\varepsilon}(z,z))
+
\sum_{j=0}^{n-1} \dfrac{n!}{j! (n-j)!} 
Q_{j}(\Phi_{D\setminus A,\varepsilon}(z),G_{D,\varepsilon,\varepsilon}(z,z))\, 
\nu_{A,\varepsilon}(z)^{n-j}
.
\end{multline*}
Further, by applying the change of variance formula \eqref{Eq change var},
this in turn equals
\begin{multline*}
Q_{n}(\Phi_{D\setminus A,\varepsilon}(z),G_{D,\varepsilon,\varepsilon}(z,z))
+
\\
\sum_{j=0}^{n-1} \dfrac{n!}{j! (n-j)!} 
\nu_{A,\varepsilon}(z)^{n-j}
\sum_{0\leq k\leq\lfloor j/2\rfloor}
(-1)^{k}\dfrac{j!}{2^{k} k! (j-2k)!}
V_{A,\varepsilon}(z)^{k}\,
:\Phi_{D\setminus A,\varepsilon}^{j-2k}:(z)
.
\end{multline*}
The term corresponding to $j$ even and $k=j/2$ is
\begin{multline*}
\sum_{0\leq k <\lfloor n/2\rfloor} 
(-1)^{k}\dfrac{n!}{2^{k} k! (n-2k)!}
\nu_{A,\varepsilon}(z)^{n-2k}
V_{A,\varepsilon}(z)^{k}
\\=
Q_{n}(\nu_{A,\varepsilon},V_{A,\varepsilon})
-\ind_{n \text{ even}}(-1)^{n/2}
\dfrac{n!}{2^{n/2} (n/2)!}
V_{A,\varepsilon}(z)^{n/2}
\end{multline*}
So one gets the desired formula.
\end{proof}

\begin{cor}
\label{Cor cond exp eps}
For $\varepsilon>0$,
\begin{displaymath}
\E\big[:\Phi^{n}_{\varepsilon}:\vert A\big] =
Q_{n}(\nu_{A,\varepsilon},V_{A,\varepsilon}).
\end{displaymath}
\end{cor}

\begin{proof}
We use the decomposition of Lemma \ref{Lem decomp Phi eps}.
For $j>2k$,
$\E\big[:\Phi_{D\setminus A,\varepsilon}^{j-2k}:\vert A\big] = 0$.
With the change of variance formula \eqref{Eq change var},
we also have
\begin{displaymath}
Q_{n}(\Phi_{D\setminus A,\varepsilon}(z),G_{D,\varepsilon,\varepsilon}(z,z))
=
\sum_{0\leq k\leq\lfloor n/2\rfloor}
(-1)^{k}\dfrac{n!}{2^{k} k! (n-2k)!}
V_{A,\varepsilon}(z)^{k}\,
:\Phi_{D\setminus A,\varepsilon}^{n-2k}:(z)
\end{displaymath}
From this we get
\begin{displaymath}
\E\big[
Q_{n}(\Phi_{D\setminus A,\varepsilon}(z),G_{D,\varepsilon,\varepsilon}(z,z))
\big\vert A
\big] =
\ind_{n \text{ even}}(-1)^{n/2}
\dfrac{n!}{2^{n/2} (n/2)!}
V_{A,\varepsilon}(z)^{n/2},
\end{displaymath}
which gives the desired result.
\end{proof}

By combining Lemma \ref{Lem cond Psi eps} and
Corollary \ref{Cor cond exp eps}, we get a first description for $\psi_{n,A}$.

\begin{prop}
\label{Prop psi eps}
For every $\eta>0$ and $n\geq 1$,
\begin{displaymath}
\lim_{\varepsilon\to 0}
\E\Big[
\big\Vert
\psi_{n,A}
-
Q_{n}(\nu_{A,\varepsilon},V_{A,\varepsilon})
\big\Vert_{H^{-\eta}(\C)}^{2}
\Big] = 0.
\end{displaymath}
\end{prop}

So formally,
\begin{displaymath}
\psi_{n,A} = Q_{n}(\nu_{A},V_{A}),
\end{displaymath}
except that the powers $\nu_{A}^{j}$ for $j\geq 2$ are not defined
and $V_{A}$ is infinite on $A\setminus\partial D$.

\begin{cor}
\label{Cor support psi}
For $n = 2k+1$ odd, a.s. $\psi_{n,A}$ is a generalized function with
compact support contained in $A$.
For $n=2k\geq 2$ even, $\psi_{n,A}$ on $D\setminus A$ (more precisely when tested against functions compactly supported on
$D\setminus A$) coincides with
\begin{equation}
\label{Eq rest D A}
(-1)^{k}\dfrac{(2k)!}{2^{k} k!}
\Big(\dfrac{1}{2\pi}\log\Big(
\dfrac{\CR(z,D)}{\CR(z,D\setminus A)}
\Big)\Big)^{k}.
\end{equation}
\end{cor}

\begin{proof}
For $n$ odd, $Q_{n}(\nu_{A,\varepsilon},V_{A,\varepsilon})$
does not contain terms of degree $0$ in $\nu_{A,\varepsilon}$.
Therefore, $Q_{n}(\nu_{A,\varepsilon},V_{A,\varepsilon})$
is supported on the $\varepsilon$-neighborhood of $A$.

For $n=2k$ even, the term in $Q_{n}(\nu_{A,\varepsilon},V_{A,\varepsilon})$
of degree $0$ in $\nu_{A,\varepsilon}$ is
\begin{displaymath}
(-1)^{k}\dfrac{(2k)!}{2^{k} k!}V_{A,\varepsilon}^{k}.
\qedhere
\end{displaymath}
\end{proof}

\begin{rem}
\label{Rem non sep}
Note that for $n=2k$ even, one cannot separate $\psi_{n,A}$ into a part on $A$
and a part on $D\setminus A$.
Indeed, \eqref{Eq rest D A} blows up near $A$ and actually is not
integrable on $D\setminus A$.
We have
\begin{multline}
\label{Eq non int}
\int_{D\setminus A}
\log\Big(
\dfrac{\CR(z,D)}{\CR(z,D\setminus A)}
\Big)^{k}\, d^{2}z
\\=
k
\int_{0}^{1}
\operatorname{Leb}(\{ z\in D\setminus A \vert 
\CR(z,D\setminus A)<u\CR(z,D)\})
\vert\log(u)\vert^{k-1} \dfrac{du}{u},
\end{multline}
Since the Lebesgue measure 
$\operatorname{Leb}(\{ z\in D\setminus A \vert 
\CR(z,D\setminus A)<u\CR(z,D)\})$
is of order $\vert\log(u)\vert^{-1/2}$ (Theorem \ref{Thm Mink ALS1}),
we get that the integral \eqref{Eq non int} equals $+\infty$.
To give an elementary analogue,
this phenomenon reminds the finite part of $1/x^{2}$.
This finite part is defined as the generalized function on $\R$
obtained as the limit
\begin{displaymath}
\operatorname{F.P.}\Big(\dfrac{1}{x^{2}}\Big)
=
\lim_{\varepsilon\to 0}
\Big(
\ind_{\vert x\vert>\varepsilon}\dfrac{1}{x^{2}}
-\dfrac{2}{\varepsilon} \delta_{0}\Big).
\end{displaymath}
Then $\operatorname{F.P.}(1/x^{2})$ cannot be separated into a function on
$\R\setminus\{0\}$ and a generalized function supported on $\{0\}$.
\end{rem}
 
\subsection{Convergence to 0 of mixed terms}
\label{Subsec prod 0}

Here we will deal with the mixed terms in Lemma \ref{Lem decomp Phi eps},
of form $V_{A,\varepsilon}^{k}\,
\nu_{A,\varepsilon}^{n-j}
\,
:\Phi_{D\setminus A,\varepsilon}^{j-2k}:$,
and show that they converge to $0$.

\begin{prop}
\label{Prop conv 0}
Fix $k\geq 0$ and $l,m\geq 1$.
Then for every $\eta>0$,
\begin{displaymath}
\lim_{\varepsilon\to 0}\E\Big[
\big\Vert
V_{A,\varepsilon}^{k}
\,\nu_{A,\varepsilon}^{l}
\, :\Phi_{D\setminus A,\varepsilon}^{m}:
\big\Vert_{H^{-\eta}(\C)}^{2}
\Big] = 0.
\end{displaymath}
\end{prop}

We start with a couple of lemmas.

\begin{lemma}
\label{Lem moment nu eps}
Let $l\geq 1$.
There are constants $c_{l},c'_{l}>0$,
depending only on $l$,
such that for every $\varepsilon>0$ and $z\in\C$,
\begin{displaymath}
\E[\nu_{A,\varepsilon}(z)^{2 l}]
= c_{l} v^{2l} + c'_{l} G_{D,\varepsilon,\varepsilon}(z,z)^{l}.
\end{displaymath}
In particular, there is a constant $C_{l}>0$,
depending on $l$ and $v$,
such that for every
$\varepsilon\in (0,1/2]$ and every $z\in\C$,
\begin{equation}
\label{Eq bound nu A eps}
\E[\nu_{A,\varepsilon}(z)^{2 l}]
\leq C_{l} \vert\log \varepsilon\vert^{l}.
\end{equation}
\end{lemma}

\begin{proof}
We have that
\begin{displaymath}
\E[\nu_{A,\varepsilon}(z)^{2 l}]
=
\E\Big[\big(\Phi_{\varepsilon}(z) - 
\Phi_{D\setminus A,\varepsilon}(z)\big)^{2 l}\Big]
.
\end{displaymath}
By the Brunn-Minkowski inequality,
\begin{eqnarray*}
\E[\nu_{A,\varepsilon}(z)^{2 l}]
&\leq &\Big(\E\big[\Phi_{\varepsilon}(z)^{2 l}\big]^{1/(2 l)}
+\E\big[\Phi_{D\setminus A,\varepsilon}(z)^{2 l}\big]^{1/(2 l)}
\Big)^{2l}
\\
&\leq &
\Big(v + 2\E\big[\Phi_{(0),\varepsilon}(z)^{2 l}\big]^{1/(2 l)}
\Big)^{2l}
\\
&\leq &
2^{2 l - 1}\Big(v^{2l} + 2^{2l}\E\big[\Phi_{(0),\varepsilon}(z)^{2 l}\big]\Big).
\end{eqnarray*}
We conclude by the Wick's formula applied to $\Phi_{(0),\varepsilon}(z)$.
\end{proof}

\begin{lemma}
\label{Lem second moment}
Fix $k\geq 0$ and $l,m\geq 1$.
Let $\eta>0$ and $\varepsilon>0$.
Then
\begin{multline*}
\E\Big[
\big\Vert
V_{A,\varepsilon}^{k}
\,\nu_{A,\varepsilon}^{l}
\, :\Phi_{D\setminus A,\varepsilon}^{m}:
\big\Vert_{H^{-\eta}(\C)}^{2}
\Big]
=
\\
\int_{\substack{d(z,D)<\varepsilon\\ d(w,D)<\varepsilon}}
\LK_{\eta}(\vert w-z\vert)
\E\Big[V_{A,\varepsilon}(z)^{k}
V_{A,\varepsilon}(w)^{k}\nu_{A,\varepsilon}(z)^{l}\nu_{A,\varepsilon}(w)^{l}
:\Phi_{D\setminus A,\varepsilon}^{m}:(z)
:\Phi_{D\setminus A,\varepsilon}^{m}:(w)
\Big]
\,d^{2}z\,d^{2}w.
\end{multline*}
\end{lemma}

\begin{proof}
Since the function $V_{A,\varepsilon}^{k}
\,\nu_{A,\varepsilon}^{l}
\, :\Phi_{D\setminus A,\varepsilon}^{m}:$
is smooth with compact support,
we have that
\begin{multline*}
\big\Vert
V_{A,\varepsilon}^{k}
\,\nu_{A,\varepsilon}^{l}
\, :\Phi_{D\setminus A,\varepsilon}^{m}:
\big\Vert_{H^{-\eta}(\C)}^{2}
=
\\
\int_{\substack{d(z,D)<\varepsilon\\ d(w,D)<\varepsilon}}
V_{A,\varepsilon}(z)^{k}
V_{A,\varepsilon}(w)^{k}
\LK_{\eta}(\vert w-z\vert)
\nu_{A,\varepsilon}(z)^{l}\nu_{A,\varepsilon}(w)^{l}
:\Phi_{D\setminus A,\varepsilon}^{m}:(z)
:\Phi_{D\setminus A,\varepsilon}^{m}:(w)
\,d^{2}z\,d^{2}w.
\end{multline*}
Further, $V_{A,\varepsilon}(z)\leq G_{D,\varepsilon,\varepsilon}(z,z)$.
Thus, for fixed $\varepsilon$, the function $V_{A,\varepsilon}$
is bounded on $\C$.
It suffices then to check that
\begin{equation}
\label{Eq abs val}
\E\Big[\nu_{A,\varepsilon}(z)^{l}\nu_{A,\varepsilon}(w)^{l}
\big\vert:\Phi_{D\setminus A,\varepsilon}^{m}:(z)
:\Phi_{D\setminus A,\varepsilon}^{m}:(w)
\big\vert
\Big]
\end{equation}
is uniformly bounded.
By Cauchy-Schwarz, \eqref{Eq abs val} is bounded by
\begin{displaymath}
\E[\nu_{A,\varepsilon}(z)^{4l}]^{1/4}
\E[\nu_{A,\varepsilon}(w)^{4l}]^{1/4}
\E[:\Phi_{D\setminus A,\varepsilon}^{m}:(z)^{4}]^{1/4}
\E[:\Phi_{D\setminus A,\varepsilon}^{m}:(w)^{4}]^{1/4}.
\end{displaymath}
To bound the factors in $\nu_{A,\varepsilon}$,
we use Lemma \ref{Lem moment nu eps}.
Further,
\begin{displaymath}
\E[:\Phi_{D\setminus A,\varepsilon}^{m}:(z)^{4}]
= N_{m}\E[G_{D\setminus A,\varepsilon,\varepsilon}(z,z)^{2 m}]
\leq 
N_{m}G_{D,\varepsilon,\varepsilon}(z,z)^{2 m}
\end{displaymath}
for some integer $N_{m}\geq 1$.
\end{proof}

\begin{lemma}
\label{Lem cond exp phi D A}
We have
\begin{multline*}
\E\Big[
V_{A,\varepsilon}(z)^{k}
V_{A,\varepsilon}(w)^{k}
\nu_{A,\varepsilon}(z)^{l}\nu_{A,\varepsilon}(w)^{l}
:\Phi_{D\setminus A,\varepsilon}^{m}:(z)
:\Phi_{D\setminus A,\varepsilon}^{m}:(w)
\Big]
=\\
m!\,\E\Big[
V_{A,\varepsilon}(z)^{k}
V_{A,\varepsilon}(w)^{k}
\nu_{A,\varepsilon}(z)^{l}\nu_{A,\varepsilon}(w)^{l}
G_{D\setminus A,\varepsilon,\varepsilon}(z,w)^{m}
\Big].
\end{multline*}
\end{lemma}

\begin{proof}
This follows by taking the conditional expectation w.r.t. $A$.
\end{proof}

\begin{lemma}
\label{Lem key estimate}
Fix $m\geq 1$.
Then there is a constant $C>0$
such that
a.s., for every $z,w\in\C$
with $\nu_{A,\varepsilon}(z)\nu_{A,\varepsilon}(w)\neq 0$,
we have
\begin{displaymath}
G_{D\setminus A,\varepsilon,\varepsilon}(z,w)^{m}
\leq
C
\int_{D^{2}}
\Big( 1\wedge
\Big(
\dfrac{\varepsilon}{\vert y-x\vert}
\Big)^{m/2}
\Big)
\log\Big(\dfrac{2 \diam (D)}{\vert y-x\vert}\Big)^{m}
\rho_{\varepsilon}(z-x)\rho_{\varepsilon}(w-y)
d^{2}x\,d^{2}y.
\end{displaymath}
\end{lemma}

\begin{proof}
By definition,
\begin{displaymath}
G_{D\setminus A,\varepsilon,\varepsilon}(z,w)
=\int_{(D\setminus A)^{2}}
G_{D\setminus A}(x,y)
\rho_{\varepsilon}(z-x)\rho_{\varepsilon}(w-y)
d^{2}x\,d^{2}y.
\end{displaymath}
By Jensen's inequality,
\begin{displaymath}
G_{D\setminus A,\varepsilon,\varepsilon}(z,w)^{m}
\leq \int_{(D\setminus A)^{2}}
G_{D\setminus A}(x,y)^{m}
\rho_{\varepsilon}(z-x)\rho_{\varepsilon}(w-y)
d^{2}x\,d^{2}y.
\end{displaymath}
Now we use the estimate of Proposition \ref{Prop Green} (Section \ref{Subsec estim Green}):
\begin{displaymath}
G_{D\setminus A}(x,y)
\leq C\Big( 1\wedge \Big(
\dfrac{d(x,A)\wedge d(y,A)}{\vert y-x\vert}\Big)^{1/2}
\Big)
\log\Big(\dfrac{2 \diam (D)}{\vert y-x\vert}\Big)
\end{displaymath}
for some universal constant $C>0$.
Further, if $\nu_{A,\varepsilon}(z)\nu_{A,\varepsilon}(w)\neq 0$,
then $z$ and $w$ are $\varepsilon$-close to $A$.
Moreover, if $\rho_{\varepsilon}(z-x)\rho_{\varepsilon}(w-y)\neq 0$,
then $\vert z-x\vert <\varepsilon$,
$\vert w-y\vert <\varepsilon$.
So we only the case $d(x,A),d(y,A) < 2\varepsilon$.
Also, the integral on $(D\setminus A)^{2}$
is smaller or equal to that on $D^{2}$.
This concludes.
\end{proof}

\begin{proof}[Proof of Proposition \ref{Prop conv 0}]
We collect the results of Lemmas \ref{Lem second moment},
\ref{Lem cond exp phi D A} and \ref{Lem key estimate}.
We want to show that
\begin{multline*}
\lim_{\varepsilon\to 0}
\int_{\substack{d(z,D)<\varepsilon\\ d(w,D)<\varepsilon}}
\LK_{\eta}(\vert w-z\vert)
\E\big[
V_{A,\varepsilon}(z)^{k}
V_{A,\varepsilon}(w)^{k}
\nu_{A,\varepsilon}(z)^{l}\nu_{A,\varepsilon}(w)^{l}
\big]
\\
\int_{D^{2}}
\Big( 1\wedge
\Big(
\dfrac{\varepsilon}{\vert y-x\vert}
\Big)^{m/2}
\Big)
\log\Big(\dfrac{2 \diam (D)}{\vert y-x\vert}\Big)^{m}
\rho_{\varepsilon}(z-x)\rho_{\varepsilon}(w-y)
d^{2}x\,d^{2}y\,
\,d^{2}z\,d^{2}w
=0.
\end{multline*}
First,
\begin{displaymath}
V_{A,\varepsilon}(z)\leq G_{D,\varepsilon,\varepsilon}(z,z)
= O(\vert\log\varepsilon\vert).
\end{displaymath}
See Lemma \ref{Lem G eps}.
Moreover, by Cauchy-Schwarz and \eqref{Eq bound nu A eps},
\begin{displaymath}
\E\big[\nu_{A,\varepsilon}(z)^{l}\nu_{A,\varepsilon}(w)^{l}
\big]
\leq
\E\big[\nu_{A,\varepsilon}(z)^{2l}\big]^{1/2}
\E\big[\nu_{A,\varepsilon}(w)^{2l}\big]^{1/2}
= 
O(\vert\log\varepsilon\vert^{l}).
\end{displaymath}
So it is enough to show that for every $\beta >0$,
the integral
\begin{equation}
\label{Eq quadrup int}
\int_{\substack{d(z,D)<\varepsilon\\ d(w,D)<\varepsilon \\ x,y\in D}}
\LK_{\eta}(\vert w-z\vert)
\Big( 1\wedge
\Big(
\dfrac{\varepsilon}{\vert y-x\vert}
\Big)^{m/2}
\Big)
\log\Big(\dfrac{2 \diam (D)}{\vert y-x\vert}\Big)^{m}
\rho_{\varepsilon}(z-x)\rho_{\varepsilon}(w-y)
d^{2}x\,d^{2}y
\,d^{2}z\,d^{2}w
\end{equation}
is $o(\vert \log\varepsilon\vert^{-\beta})$.

Similarly to Lemma \ref{Lem G eps}, 
we will use the crude bound
\begin{displaymath}
\int_{D^{2}}
\log\Big(\dfrac{2 \diam (D)}{\vert y-x\vert}\Big)^{m}
\rho_{\varepsilon}(z-x)\rho_{\varepsilon}(w-y)
d^{2}x\,d^{2}y
=
O(\vert \log\varepsilon\vert^{m}).
\end{displaymath}

We will separate the integral \eqref{Eq quadrup int} into two parts.
The first part will correspond to $\vert z-w\vert>\varepsilon^{1/2}$.
The second part will correspond to 
$\vert z-w\vert\leq\varepsilon^{1/2}$.
In the first case, we will bound
\begin{displaymath}
1\wedge
\Big(
\dfrac{\varepsilon}{\vert y-x\vert}
\Big)^{m/2}
\end{displaymath}
by
\begin{displaymath}
\Big(
\dfrac{\varepsilon}{\varepsilon^{1/2}-2\varepsilon}
\Big)^{m/2}=
O(\varepsilon^{m/4}).
\end{displaymath}
In the second case, we will simply bound by $1$.

So the first part is bounded by
\begin{displaymath}
C' \varepsilon^{m/4}\vert \log\varepsilon\vert^{m}
\int_{\substack{d(z,D)<\varepsilon\\ d(w,D)<\varepsilon \\ 
\vert z-w\vert>\varepsilon^{1/2}}}
\LK_{\eta}(\vert w-z\vert)
\,d^{2}z\,d^{2}w,
\end{displaymath}
which is $o(\vert \log\varepsilon\vert^{-\beta})$.

The second part is bounded by
\begin{equation}
\label{Eq second case}
C'' \vert \log\varepsilon\vert^{m}
\int_{\substack{d(z,D)<\varepsilon\\ d(w,D)<\varepsilon \\ 
\vert z-w\vert \leq\varepsilon^{1/2}}}
\LK_{\eta}(\vert w-z\vert)
\,d^{2}z\,d^{2}w.
\end{equation}
We will consider the more difficult case of $\eta\in (0,1)$.
With the asymptotic \eqref{Eq asymp kernel},
\eqref{Eq second case} is bounded by
\begin{displaymath}
C'''\vert \log\varepsilon\vert^{m}
\int_{0}^{\varepsilon^{1/2}}\dfrac{1}{r^{2-2\eta}} r\,dr
=
O(\varepsilon^{\eta}\vert \log\varepsilon\vert^{m}),
\end{displaymath}
which is again $o(\vert \log\varepsilon\vert^{-\beta})$.
\end{proof}

\subsection{Wick powers outside the FPS}
\label{Subsec Wick outside}

Here we will describe $:\Phi^{n}: - \psi_{n,A}$.
From the decomposition of Lemma \ref{Lem decomp Phi eps},
this is the limit of
\begin{multline*}
Q_{n}(\Phi_{D\setminus A,\varepsilon}(z),G_{D,\varepsilon,\varepsilon}(z,z))
-\ind_{n \text{ even}}(-1)^{n/2}
\dfrac{n!}{2^{n/2} (n/2)!}
V_{A,\varepsilon}(z)^{n/2}
\\=
\sum_{0\leq k <\lfloor n/2\rfloor}
(-1)^{k}\dfrac{n!}{2^{k} k! (n-2k)!}
V_{A,\varepsilon}(z)^{k}
\,
:\Phi_{D\setminus A,\varepsilon}^{n-2k}:(z)
\end{multline*}
So naturally, one expects the limit to be
\begin{displaymath}
\sum_{0\leq k <\lfloor n/2\rfloor}
(-1)^{k}\dfrac{n!}{2^{k} k! (n-2k)!}
V_{A}^{k}
\,
:\Phi_{D\setminus A}^{n-2k}:\, .
\end{displaymath}
However, $V_{A}$ blows up near $A$.
Because of this, a little bit of care is needed.

Let $\chi: [0,+\infty)\rightarrow [0,1]$
be a smooth non-decreasing cut-off function,
which values $0$ on $[0,1/2]$ and $1$ on $[1,+\infty)$.
For $q\in\N$, denote $f_{q}$ the function on $D\setminus A$
\begin{equation}
\label{Eq cutoff f q}
f_{q}(z) = \chi(2^{q}\CR(z,D\setminus A)).
\end{equation}
Outside $D\setminus A$, we set $f_{q}$ to $0$.
Clearly, for all $n\geq 1$, $k\geq 1$ and $q\in\N$,
the field $f_{q} V_{A}^{k} :\Phi_{D\setminus A}^{n}:$
is well defined.

\begin{lemma}
\label{Lem f q}
Fix $n\geq 1$ and $k\geq 1$.
The sequence $(f_{q} V_{A}^{k} :\Phi_{D\setminus A}^{n}:)_{q\geq 0}$
is Cauchy in $L^{2}(d\PP,\sigma(\Phi),H^{-\eta}(\C))$
for every $\eta>0$.
\end{lemma}

\begin{proof}
Let $q'>q\geq 0$. Note that the function $f_{q'}-f_{q}$
is non-negative.
We have that
\begin{multline*}
\E\Big[
\big\Vert
(f_{q'}-f_{q})\,
V_{A}^{k}
\, :\Phi_{D\setminus A}^{n}:
\big\Vert_{H^{-\eta}(\C)}^{2}
\Big]
=
\\
\int_{\C^{2}}
\LK_{\eta}(\vert w-z\vert)
\E\Big[
(f_{q'}-f_{q})(z)
(f_{q'}-f_{q})(w)
V_{A}(z)^{k}
V_{A}(w)^{k}
:\Phi_{D\setminus A}^{n}:(z)
:\Phi_{D\setminus A}^{n}:(w)
\Big]
\,d^{2}z\,d^{2}w
\\
=
n!\,
\int_{\C^{2}}
\LK_{\eta}(\vert w-z\vert)
\E\Big[
(f_{q'}-f_{q})(z)
(f_{q'}-f_{q})(w)
V_{A}(z)^{k}
V_{A}(w)^{k}
G_{D\setminus A}(z,w)^{n}
\Big]
\,d^{2}z\,d^{2}w.
\end{multline*}
Further, $f_{q'}-f_{q}\leq \ind_{D\setminus A}-f_{q}$.
Thus,
\begin{multline*}
\E\Big[
\big\Vert
(f_{q'}-f_{q})\,
V_{A}^{k}
\, :\Phi_{D\setminus A}^{n}:
\big\Vert_{H^{-\eta}(\C)}^{2}
\Big]
\\ \leq
n!\,
\int_{\C^{2}}
\LK_{\eta}(\vert w-z\vert)
\E\Big[
(\ind_{D\setminus A}-f_{q})(z)
(\ind_{D\setminus A}-f_{q})(w)
V_{A}(z)^{k}
V_{A}(w)^{k}
G_{D\setminus A}(z,w)^{n}
\Big]
\,d^{2}z\,d^{2}w.
\end{multline*}
The function $\ind_{D\setminus A}-f_{q}$ converges to $0$ as $q\to +\infty$
almost sure almost everywhere.
So to conclude by dominated convergence, we need only to check that
\begin{equation}
\label{Eq to bound}
\int_{\C^{2}}
\LK_{\eta}(\vert w-z\vert)
\E\Big[
V_{A}(z)^{k}
V_{A}(w)^{k}
G_{D\setminus A}(z,w)^{n}
\Big]
\,d^{2}z\,d^{2}w
< + \infty.
\end{equation}
For $G_{D\setminus A}(z,w)$,
we apply the estimate \eqref{Eq est Green} of Proposition \ref{Prop Green}.
For $V_{A}$ we use the distortion inequalities \eqref{Eq Koebe}:
\begin{displaymath}
V_{A}(z)\leq \dfrac{1}{2\pi}\log\Big(\dfrac{4 d(z,\partial D)}{d(z,A)}\Big)
\leq 
\dfrac{1}{2\pi}\log\Big(\dfrac{4 \diam (D)}{d(z,A)}\Big).
\end{displaymath}
Thus, \eqref{Eq to bound} is upper bounded by
\begin{multline*}
C\int_{\C^{2}}
\LK_{\eta}(\vert w-z\vert)
\log\Big(\dfrac{2 \diam (D)}{\vert z-w\vert}\Big)^{n}
\E\Big[
\ind_{z,w\in D\setminus A}
\log\Big(\dfrac{4 \diam (D)}{d(z,A)}\Big)^{k}
\log\Big(\dfrac{4 \diam (D)}{d(z,A)}\Big)^{k}
\\
\Big( 1\wedge \Big(
\dfrac{d(z,A)\wedge d(w,A)}{\vert z-w\vert}\Big)^{n/2}
\Big)
\Big]
\,d^{2}z\,d^{2}w,
\end{multline*}
for some constant $C>0$. 
Now let us upper bound the quantity
\begin{equation}
\label{Eq to bound 2}
\log\Big(\dfrac{4 \diam (D)}{d(z,A)}\Big)^{k}
\log\Big(\dfrac{4 \diam (D)}{d(z,A)}\Big)^{k}
\Big( 1\wedge \Big(
\dfrac{d(z,A)\wedge d(w,A)}{\vert z-w\vert}\Big)^{n/2}
\Big).
\end{equation}
In case $d(z,A)\wedge d(w,A)\leq \vert z-w\vert^{2}$,
it is bounded by
\begin{displaymath}
\log\Big(\dfrac{4 \diam (D)}{d(z,A)}\Big)^{k}
\log\Big(\dfrac{4 \diam (D)}{d(z,A)}\Big)^{k}
(1\wedge(d(z,A)\wedge d(w,A))^{n/4}),
\end{displaymath}
which in turn is bounded by a deterministic constant.
In case $d(z,A)\wedge d(w,A) > \vert z-w\vert^{2}$,
\eqref{Eq to bound 2} is bounded by
\begin{displaymath}
\log\Big(\dfrac{4 \diam (D)}{\vert z-w\vert^{2}}\Big)^{2k}
\end{displaymath}
So in any case we get the desired convergence of the integral.
\end{proof}

We will denote the limit of $f_{q} V_{A}^{k} :\Phi_{D\setminus A}^{n}:$
simply $V_{A}^{k} :\Phi_{D\setminus A}^{n}:$.
It is easy to check that it does not depend on the particular choice of $\chi$.

\begin{prop}
\label{Prop V G eps}
Fix $n\geq 1$ and $k\geq 1$.
As $\varepsilon\to 0$, $V_{A,\varepsilon}^{k} :\Phi_{D\setminus A,\varepsilon}^{n}:$
converges to $V_{A}^{k} :\Phi_{D\setminus A}^{n}:$ in
$L^{2}(d\PP,\sigma(\Phi),H^{-\eta}(\C))$
for every $\eta>0$.
\end{prop}

\begin{proof}
We need to show that
\begin{equation}
\label{Eq eps eps}
\int_{\C^{2}}
\LK_{\eta}(\vert w-z\vert)
\E\Big[
V_{A,\varepsilon}(z)^{k}
V_{A,\varepsilon}(w)^{k}
G_{D\setminus A,\varepsilon,\varepsilon}(z,w)^{n}
\Big]
\,d^{2}z\,d^{2}w,
\end{equation}
as well as 
\begin{equation}
\label{Eq 0 eps}
\int_{\C^{2}}
\LK_{\eta}(\vert w-z\vert)
\E\Big[
V_{A}(z)^{k}
V_{A,\varepsilon}(w)^{k}
G_{D\setminus A,0,\varepsilon}(z,w)^{n}
\Big]
\,d^{2}z\,d^{2}w,
\end{equation}
converge to
\begin{equation}
\label{Eq 0 0}
\int_{\C^{2}}
\LK_{\eta}(\vert w-z\vert)
\E\Big[
V_{A}(z)^{k}
V_{A}(w)^{k}
G_{D\setminus A}(z,w)^{n}
\Big]
\,d^{2}z\,d^{2}w.
\end{equation}
By $G_{D\setminus A,0,\varepsilon}(z,w)$ in \eqref{Eq 0 eps} we denoted
\begin{displaymath}
G_{D\setminus A,0,\varepsilon}(z,w) = 
\int_{(D\setminus A)}
G_{D\setminus A}(z,y)\rho_{\varepsilon}(w-y)
\,d^{2}y.
\end{displaymath}
We will focus on the convergence of \eqref{Eq eps eps}.
The convergence of \eqref{Eq 0 eps} is similar and simpler.

Let $\varepsilon\in (0,1/16]$.
Denote
\begin{displaymath}
S_{\varepsilon} = \{(z,w)\in (D\setminus A)^{2}\vert d(z,A)>\varepsilon^{1/2},d(w,A)>\varepsilon^{1/2},
\vert w-z\vert>\varepsilon^{1/2}\}.
\end{displaymath}
By Corollaries \ref{Cor estim CR} and \ref{Cor var G D},
there is a deterministic constant $C>0$
such that for every $(z,w)\in S_{\varepsilon}$, and every $(x,y)\in \D(z,\varepsilon)\times \D(w,\varepsilon)$,
all of the following holds:
\begin{displaymath}
\vert V_{A}(x) - V_{A}(z) \vert \leq C \varepsilon^{1/2},
\qquad
\vert V_{A}(y) - V_{A}(w) \vert \leq C \varepsilon^{1/2},
\end{displaymath}
\begin{displaymath}
\vert G_{D\setminus A}(x,y) - G_{D\setminus A}(z,w) \vert 
\leq 
C \varepsilon^{1/2} (G_{D\setminus A}(z,w) + \vert \log \varepsilon\vert).
\end{displaymath}
Then for every $(z,w)\in S_{\varepsilon}$,
\begin{displaymath}
\vert V_{A,\varepsilon}(z) - V_{A}(z) \vert \leq C \varepsilon^{1/2},
\qquad
\vert V_{A,\varepsilon}(w) - V_{A}(w) \vert \leq C \varepsilon^{1/2},
\end{displaymath}
\begin{displaymath}
\vert G_{D\setminus A,\varepsilon,\varepsilon}(z,w) - G_{D\setminus A}(z,w) \vert 
\leq 
C \varepsilon^{1/2} (G_{D\setminus A}(z,w) + \vert \log \varepsilon\vert).
\end{displaymath}
Therefore,
\begin{displaymath}
\int_{\C^{2}}
\LK_{\eta}(\vert w-z\vert)
\E\Big[
\ind_{(z,w)\in S_{\varepsilon}}
V_{A,\varepsilon}(z)^{k}
V_{A,\varepsilon}(w)^{k}
G_{D\setminus A,\varepsilon,\varepsilon}(z,w)^{n}
\Big]
\,d^{2}z\,d^{2}w
\end{displaymath}
converges, as $\varepsilon\to 0$, to \eqref{Eq 0 0}.
We need now to estimate
\begin{displaymath}
\int_{\C^{2}}
\LK_{\eta}(\vert w-z\vert)
\E\Big[
\ind_{(z,w)\not\in S_{\varepsilon}}
V_{A,\varepsilon}(z)^{k}
V_{A,\varepsilon}(w)^{k}
G_{D\setminus A,\varepsilon,\varepsilon}(z,w)^{n}
\Big]
\,d^{2}z\,d^{2}w,
\end{displaymath}
and show that it converges to $0$.
Since for every $z\in\C$,
\begin{displaymath}
V_{A,\varepsilon}(z)\leq G_{D,\varepsilon,\varepsilon}(z,z)
= O(\vert\log\varepsilon\vert),
\end{displaymath}
it is enough to show that
\begin{displaymath}
\int_{\C^{2}}
\LK_{\eta}(\vert w-z\vert)
\E\Big[
\ind_{(z,w)\not\in S_{\varepsilon}}
G_{D\setminus A,\varepsilon,\varepsilon}(z,w)^{n}
\Big]
\,d^{2}z\,d^{2}w =
o(\vert\log\varepsilon\vert^{-2k}).
\end{displaymath}
Denote
\begin{displaymath}
S'_{\varepsilon} =
\{(z,w)\in (D\setminus A)^{2}\vert d(z,A)\wedge d(w,A)\leq \varepsilon^{1/2},
\vert w-z\vert>\varepsilon^{1/2}\}.
\end{displaymath}
Then
\begin{multline*}
\int_{\C^{2}}
\LK_{\eta}(\vert w-z\vert)
\E\Big[
\ind_{(z,w)\not\in S_{\varepsilon}}
G_{D\setminus A,\varepsilon,\varepsilon}(z,w)^{n}
\Big]
\,d^{2}z\,d^{2}w
\\\leq 
\int_{\vert w-z\vert\leq \varepsilon^{1/2}}
\LK_{\eta}(\vert w-z\vert)
G_{D,\varepsilon,\varepsilon}(z,w)^{n}
\,d^{2}z\,d^{2}w
\\+
\int_{\C^{2}}
\LK_{\eta}(\vert w-z\vert)
\E\Big[
\ind_{(z,w)\in S'_{\varepsilon}}
G_{D\setminus A,\varepsilon,\varepsilon}(z,w)^{n}
\Big]
\,d^{2}z\,d^{2}w.
\end{multline*}
In the more difficult case of $\eta\in (0,1)$,
\begin{multline*}
\int_{\vert w-z\vert\leq \varepsilon^{1/2}}
\LK_{\eta}(\vert w-z\vert)
G_{D,\varepsilon,\varepsilon}(z,w)^{n}
\,d^{2}z\,d^{2}w
\\\leq 
C\vert \log\varepsilon\vert^{n}
\int_{0}^{\varepsilon^{1/2}}\dfrac{1}{r^{2-2\eta}} r\,dr
=
O(\varepsilon^{\eta}\vert \log\varepsilon\vert^{n})
=
o(\vert\log\varepsilon\vert^{-2k}).
\end{multline*}
For the term with $(z,w)\in S'_{\varepsilon}$,
we use the estimate \eqref{Eq est Green} to get
\begin{displaymath}
\int_{\C^{2}}
\LK_{\eta}(\vert w-z\vert)
\E\Big[
\ind_{(z,w)\in S'_{\varepsilon}}
G_{D\setminus A,\varepsilon,\varepsilon}(z,w)^{n}
\Big]
\,d^{2}z\,d^{2}w
=
O(\varepsilon^{n/4}\vert \log\varepsilon\vert^{n})
=
o(\vert\log\varepsilon\vert^{-2k}).
\qedhere
\end{displaymath}
\end{proof}

By collecting the results of Lemma \ref{Lem decomp Phi eps},
Proposition \ref{Prop psi eps},
Proposition \ref{Prop conv 0}
and Proposition \ref{Prop V G eps},
we obtain the following decomposition of the Wick powers
$:\Phi^{n}:$.

\begin{thm}
\label{Thm decomp FPS Wick}
Let $n\geq 1$. A.s., the Wick power $:\Phi^{n}:$ can be decomposed
\begin{displaymath}
:\Phi^{n}:~=\psi_{n,A}
+
\sum_{0\leq k <\lfloor n/2\rfloor}
(-1)^{k}\dfrac{n!}{2^{k} k! (n-2k)!}
V_{A}^{k}
\,
:\Phi_{D\setminus A}^{n-2k}:\,.
\end{displaymath}
\end{thm}

The above decomposition also implies a decomposition of
Sobolev norms.

\begin{cor}
\label{Cor Sob norm decomp}
Fix $\eta>0$. Let $n\geq 1$.
Then a.s.,
\begin{displaymath}
\E\big[\Vert :\Phi^{n}:\Vert^{2}_{H^{-\eta}(\C)}\big\vert A\big]
=
\Vert \psi_{n,A}\Vert^{2}_{H^{-\eta}(\C)}
+
\E\Big[
\Big\Vert
\sum_{0\leq k <\lfloor n/2\rfloor}
(-1)^{k}\dfrac{n!}{2^{k} k! (n-2k)!}
V_{A}^{k}:\Phi_{D\setminus A}^{n-2k}:
\Big\Vert^{2}_{H^{-\eta}(\C)}
\Big\vert A
\Big] .
\end{displaymath}
In particular,
\begin{displaymath}
\E\big[\Vert :\Phi^{n}:\Vert^{2}_{H^{-\eta}(\C)}\big]
=
\E\big[\Vert \psi_{n,A}\Vert^{2}_{H^{-\eta}(\C)}\big]
+
\E\Big[
\Big\Vert
\sum_{0\leq k <\lfloor n/2\rfloor}
(-1)^{k}\dfrac{n!}{2^{k} k! (n-2k)!}
V_{A}^{k}:\Phi_{D\setminus A}^{n-2k}:
\Big\Vert^{2}_{H^{-\eta}(\C)}
\Big].
\end{displaymath}
\end{cor}

\begin{proof}
One can expand $\Vert :\Phi^{n}:\Vert^{2}_{H^{-\eta}(\C)}$ as
\begin{multline*}
\Vert \psi_{n,A}\Vert^{2}_{H^{-\eta}(\C)}
+
\Big\Vert
\sum_{0\leq k <\lfloor n/2\rfloor}
(-1)^{k}\dfrac{n!}{2^{k} k! (n-2k)!}
V_{A}^{k}:\Phi_{D\setminus A}^{n-2k}:
\Big\Vert^{2}_{H^{-\eta}(\C)}
\\
+2
\Big\langle :\Phi^{n}:\, , 
\sum_{0\leq k <\lfloor n/2\rfloor}
(-1)^{k}\dfrac{n!}{2^{k} k! (n-2k)!}
V_{A}^{k}:\Phi_{D\setminus A}^{n-2k}:
\Big\rangle_{H^{-\eta}(\C)},
\end{multline*}
where $\langle \cdot,\cdot \rangle_{H^{-\eta}(\C)}$ is the
$H^{-\eta}(\C)$ inner product. 
Further,
\begin{multline*}
\E\Big[
\Big\langle :\Phi^{n}:\, , 
\sum_{0\leq k <\lfloor n/2\rfloor}
(-1)^{k}\dfrac{n!}{2^{k} k! (n-2k)!}
V_{A}^{k}:\Phi_{D\setminus A}^{n-2k}:
\Big\rangle
\Big\vert A
\Big] =
\\
\Big\langle :\Phi^{n}:\, ,
\E\Big[
\sum_{0\leq k <\lfloor n/2\rfloor}
(-1)^{k}\dfrac{n!}{2^{k} k! (n-2k)!}
V_{A}^{k}:\Phi_{D\setminus A}^{n-2k}:
\Big\vert A \Big]
\Big\rangle  = 0.
\end{multline*}
\end{proof}

Next we see that for $n$ odd, $\psi_{n,A}$ can be seen as the restriction of $:\Phi^{n}:$
to $A$. This is no longer the case if $n$ is even.

\begin{cor}
\label{Cor restr A}
Let $f_{q}$ be the cut-off functions given by \eqref{Eq cutoff f q}.
\begin{enumerate}
\item If $n$ is odd, then for every $\eta>0$.
\begin{displaymath}
\lim_{q\to 0}\E\Big[\big\Vert (1-f_{q}):\Phi^{n}: - \psi_{n,A}\big\Vert_{H^{-\eta}(\C)}^{2}\Big] = 0.
\end{displaymath}
\item If $n\geq 2$ is even, 
with $n = 2 \operatorname{mod} 4$, then
$((1-f_{q}):\Phi^{n}:\, ,1)$
converges to $+\infty$ in probability.
\item If $n\geq 2$ is even, 
with $n = 0 \operatorname{mod} 4$, then
$((1-f_{q}):\Phi^{n}:\, ,1)$
converges to $-\infty$ in probability.
\end{enumerate}
\end{cor}

\begin{proof}
By the decomposition of Theorem \ref{Thm decomp FPS Wick}, we get
\begin{displaymath}
(1-f_{q})\, :\Phi^{n}:~=(1-f_{q})\psi_{n,A}
+
(1-f_{q})\sum_{0\leq k <\lfloor n/2\rfloor}
(-1)^{k}\dfrac{n!}{2^{k} k! (n-2k)!}
V_{A}^{k}
\,
:\Phi_{D\setminus A}^{n-2k}:\,.
\end{displaymath}
By Lemma \ref{Lem f q}, the term
\begin{displaymath}
(1-f_{q})\sum_{0\leq k <\lfloor n/2\rfloor}
(-1)^{k}\dfrac{n!}{2^{k} k! (n-2k)!}
V_{A}^{k}
\,
:\Phi_{D\setminus A}^{n-2k}:
\end{displaymath}
converges to $0$ in
$L^{2}(d\PP,\sigma(\Phi),H^{-\eta}(\C))$
as $q\to + \infty$.
If $n$ is odd, then by Corollary \eqref{Cor support psi},
$\psi_{n,A}$ is supported on $A$, and thus,
$(1-f_{q})\psi_{n,A} = \psi_{n,A}$.
So we get the point (1).

Assume now that $n\geq 2$ is even.
Then by Corollary \eqref{Cor support psi},
\begin{displaymath}
(1-f_{q})\, :\Phi^{n}: - (1-f_{0})\, :\Phi^{n}:
~= -(f_{q}-f_{0})
(-1)^{n/2}\dfrac{n!}{2^{n/2} (n/2)!}V_{A}^{n/2}.
\end{displaymath}
But a.s.,
\begin{displaymath}
\lim_{q\to +\infty}\int_{D\setminus A}(f_{q}-f_{0})V_{A}^{n/2} = +\infty.
\end{displaymath}
Thus, the points (2) and (3) follow.
\end{proof}

\subsection{Decomposition of Wick powers induced by two-valued sets}
\label{Subsec Wick TVS}

As previously, let $\Phi$ be a GFF on $D$ with constant
boundary condition $v>0$ on $\partial D$.
Let $b>v$ and $a<v$ such that $b-a\geq 2\lambda$,
where $2\lambda = \sqrt{\pi/2}$ is the height gap of the GFF;
see Section \ref{Subsubsec TVS}.
Let $A_{a,b}$ be the two-valued set (TVS) of $\Phi$ for levels $a$ and $b$.
Let $\ell_{a,b}$ be the label function of $(\Phi, A_{a,b})$ : on each connected component of
$D\setminus A_{a,b}$, $\ell_{a,b}$ takes a constant value $a$ or $b$.
Outside $D\setminus A_{a,b}$, we set $\ell_{a,b}$ to $0$.
The $\Phi$ can be decomposed as
\begin{displaymath}
\Phi = \ell_{a,b} + \Phi_{D\setminus A_{a,b}},
\end{displaymath}
where conditionally on $(A_{a,b}, \ell_{a,b})$,
the field $\Phi_{D\setminus A_{a,b}}$ is distributed as a GFF on $D\setminus A_{a,b}$
with $0$ boundary condition.
Next we present the decomposition of the Wick powers $:\Phi^{n}:$ induced by the local set $A_{a,b}$.
The picture is much simpler than for the first passage sets because $A_{a,b}$ is \textbf{thin},
that is to say $\Phi$ does not ``charge'' the set $A_{a,b}$ itself,
and only the harmonic extension $\ell_{a,b}$ matters.
See \cite{Sepulveda19ThinLocalSet,ASW} for the terminology.
We will omit the details of proofs, since they are either straightforward or similar to the case of first passage sets.
We will need the function
\begin{displaymath}
V_{A_{a,b}}(z) = g_{D}(z,z) - g_{D\setminus A_{a,b}}(z,z)
=\dfrac{1}{2\pi}\log\Big(\dfrac{\CR(z,D)}{\CR(z,D\setminus A_{a,b})}\Big),
\end{displaymath}
defined on $D\setminus A_{a,b}$.
Note that for all $k\geq 1$,
\begin{displaymath}
\E\Big[\int_{D\setminus A_{a,b}}
V_{A_{a,b}}(z)^{k}\, d^{2}z \Big]
=
\int_{D}
\E\big[V_{A_{a,b}}(z)^{k}
\big] \, d^{2}z < +\infty,
\end{displaymath}
since the r.v.-s $V_{A_{a,b}}(z)$ have all the same distribution (by conformal invariance)
and have exponential tails; see Theorem \ref{Thm ASW CR ell a b}.
So this is different from the case of first passage sets; see Remark \ref{Rem non sep}.

\begin{lemma}
\label{Lem Q n x y u v}
For every $n\geq 1$, the following polynomial identity holds:
\begin{equation}
\label{Eq Q n x y u v}
Q_{n}(x_{1}+x_{2}, u_{1}+u_{2}) = 
\sum_{j=0}^{n} \dfrac{n!}{j! (n-j)!} Q_{n-j}(x_{1},u_{1})Q_{j}(x_{2},u_{2}).
\end{equation}
\end{lemma}
\begin{proof}
The binomial formula \eqref{Eq Q bin} gives
\begin{displaymath}
Q_{n}(x_{1}+x_{2}, u_{1}+u_{2}) = 
\sum_{k=0}^{n} \dfrac{n!}{k! (n-k)!} x_{1}^{n-k}Q_{k}(x_{2},u_{1}+u_{2}).
\end{displaymath}
The change of variance formula \eqref{Eq change var} gives
\begin{multline*}
Q_{n}(x_{1}+x_{2}, u_{1}+u_{2}) = 
\sum_{k=0}^{n} \dfrac{n!}{k! (n-k)!} x_{1}^{n-k}
\sum_{0\leq m\leq \lfloor k/2\rfloor}(-1)^{m}\dfrac{k!}{2^{m} m! (k-2m)!} u_{1}^{m} Q_{k-2m}(x_{2},u_{2})
\\ = 
\sum_{k=0}^{n}\sum_{0\leq m\leq \lfloor k/2\rfloor}
\dfrac{n!}{(k-2m)!(n-k + 2m)!}
(-1)^{m}
\dfrac{(n-k + 2m)!}{2^{m} m!(n-k)!}
x_{1}^{n-k} u_{1}^{m} Q_{k-2m}(x_{2},u_{2}).
\end{multline*}
By setting $j=k-2m$ and summing over $k$, we get \eqref{Eq Q n x y u v}.
\end{proof}

\begin{prop}
\label{Prop decomp Wick TVS}
With the notations above,
for every $n\geq 1$ the following decomposition holds:
\begin{equation}
\label{Eq decomp Wick TVS}
:\Phi^{n}: ~=~
\sum_{j=0}^{n} \dfrac{n!}{j! (n-j)!} Q_{n-j}(\ell_{a,b},V_{A_{a,b}})\,:\Phi_{D\setminus A_{a,b}}^{j} : \, ,
\end{equation}
where all the terms in the sum belong, as random variables,
to $L^{2}(d\PP,\sigma(\Phi),H^{-\eta}(\C))$
for every $\eta>0$.
In particular,
\begin{equation}
\label{Eq cond exp TVS}
\E\big[ :\Phi^{n}: \big\vert A_{a,b}, \ell_{a,b}\big]
= Q_{n}(\ell_{a,b},V_{A_{a,b}}).
\end{equation}
\end{prop}
\begin{proof}
At the level of formal computations,
the decomposition \eqref{Eq decomp Wick TVS} from the identity \eqref{Eq Q n x y u v}.
Indeed, $\ell_{a,b}$ corresponds to $x_{1}$,
$\Phi_{D\setminus A_{a,b}}$ corresponds to $x_{2}$, and
$V_{A_{a,b}}$ corresponds to $u_{1}$.
As for $u_{2}$, it corresponds to $G_{D\setminus A_{a,b}}(z,z)$, which is infinite.
For the convergence, this is similar and actually simpler to the case of first passage sets,
so we skip the details here.

As for \eqref{Eq cond exp TVS}, we use that for every $j\geq 1$,
\begin{displaymath}
\E\big[ :\Phi_{D\setminus A_{a,b}}^{j} : \big\vert A_{a,b}, \ell_{a,b}\big]
= 0.
\qedhere
\end{displaymath}
\end{proof}

For a generalization of Proposition \ref{Prop decomp Wick TVS} to
the case of bounded-type \textbf{thin} local sets,
see Proposition \ref{Prop decomp Wick BTLS} in Section \ref{Subsec remark nestes CLE 4}.

\section{Asymptotic expansion of the neighborhoods of first passage sets}
\label{Sec A E}

\subsection{Presentation of the expansion}
\label{Subsec Pres A E}

As in Section \ref{Sec Wick FPS},
let $D\subset \C$ be an open, non-empty, bounded, connected and simply connected domain.
Let be a constant $v>0$ and $\Phi$ a GFF on $D$ with constant
boundary condition $v$ on $\partial D$.
Let $A$ be the first passage set (FPS) of $\Phi$
from level $v$ to level $0$.
We will introduce two open neighborhoods of $A$ defined through the conformal radius.
For $\varepsilon>0$, set
\begin{equation}
\label{Eq not N eps A}
\Ns_{\varepsilon}(A) = \{ z\in D\setminus A \vert \CR(z,D\setminus A)<\varepsilon \CR(z,D)\},
\qquad
\widetilde{\Ns}_{\varepsilon}(A)
= \{ z\in D\setminus A \vert \CR(z,D\setminus A)<\varepsilon\}.
\end{equation}
Note that we actually did not include $A$ in $N_{\varepsilon}(A)$
or $\widetilde{N}_{\varepsilon}(A)$.
This actually does not matter for what follows, because a.s.
$A$ has a $0$ Lebesgue measure.
By the distortion inequalities \eqref{Eq Koebe},
the sets $\Ns_{\varepsilon}(A)$ and $\widetilde{\Ns}_{\varepsilon}(A)$
can be compared to $\varepsilon$-neighborhoods of $A$ defined in terms of the Euclidean
distance. 
Indeed, for every $z\in D\setminus A$,
\begin{displaymath}
d(z,A) \leq \CR(z,D\setminus A) \leq 4 d(z,A),
\qquad
\dfrac{1}{4}\dfrac{d(z,A)}{d(z,\partial D)}
\leq\dfrac{\CR(z,D\setminus A)}{\CR(z,D)}\leq 
4 \dfrac{d(z,A)}{d(z,\partial D)}.
\end{displaymath}
Note also that $\Ns_{\varepsilon}(A)$ is conformally invariant in law,
but $\widetilde{\Ns}_{\varepsilon}(A)$ is not.

Recall the fields $\psi_{n,A}$ introduced in Section \ref{Subsec cond Wick FPS}.
Our goal is to prove the following asymptotic expansion.

\begin{thm}
\label{Thm A E FPS}
Consider $\ind_{\Ns_{\varepsilon}(A)}$ the Lebesgue measure on 
$\Ns_{\varepsilon}(A)$ (indicator function).
It satisfies the following asymptotic expansion:
for every $N\geq 0$,
\begin{equation}
\label{Eq A E N}
\ind_{\Ns_{\varepsilon}(A)}
=
\dfrac{1}{\sqrt{2\pi}}
\sum_{k=0}^{N} (-1)^{k}
\dfrac{1}{2^{k} k! (k+1/2)}
\dfrac{\psi_{2k+1,A}}{\big(\frac{1}{2\pi}\vert\log \varepsilon\vert\big)^{k + 1/2}}
~+~R_{N,\varepsilon},
\end{equation}
where the error term $R_{N,\varepsilon}$ satisfies, as $\varepsilon \to 0$,
\begin{displaymath}
\forall \eta >0,~
\E\big[\Vert R_{N,\varepsilon} \Vert_{H^{-\eta}(\C)}^{2}\big]^{1/2}
=\, o (\vert\log \varepsilon\vert^{-(N+1/2)}).
\end{displaymath}
As for the Lebesgue measure on $\widetilde{\Ns}_{\varepsilon}(A)$,
fix a deterministic smooth cutoff function $f_{0}:D\rightarrow [0,1]$, compactly supported in $D$.
Then  $\ind_{\widetilde{\Ns}_{\varepsilon}(A)}f_{0}$ satisfies the following asymptotic expansion:
for every $N\geq 0$,
\begin{equation}
\label{Eq A E tilde N}
\ind_{\widetilde{\Ns}_{\varepsilon}(A)}f_{0}
=
\dfrac{1}{\sqrt{2\pi}}
\sum_{k=0}^{N} (-1)^{k}
\dfrac{1}{2^{k} k! (k+1/2)}
\dfrac{f_{0}\,\tilde{\psi}_{2k+1,A}}{\big(\frac{1}{2\pi}
\vert\log \varepsilon\vert\big)^{k + 1/2}}
~+~\widetilde{R}_{N,\varepsilon, f_{0}},
\end{equation}
where
\begin{equation}
\label{Eq def tilde psi}
\tilde{\psi}_{2k+1,A} = 
\sum_{j=0}^{k} \dfrac{(2k+1)!}{2^{j} j! (2(k-j)+1)!}
\,\Big(\dfrac{1}{2\pi} \log \CR(z,D)\Big)^{j}\,\psi_{2(k-j)+1,A},
\end{equation}
and the error term $\widetilde{R}_{N,\varepsilon, f_{0}}$ satisfies, as $\varepsilon \to 0$,
\begin{displaymath}
\forall \eta >0,~
\E\big[\Vert \widetilde{R}_{N,\varepsilon, f_{0}} \Vert_{H^{-\eta}(\C)}^{2}\big]^{1/2}
=\, o (\vert\log \varepsilon\vert^{-(N+1/2)}).
\end{displaymath}
Further assume that
\begin{eqnarray}
\label{Eq cond small boundary}
\forall \beta >0,~
\operatorname{Leb}(\{ z\in D \vert \,d(z,\partial D)<\varepsilon\}) 
= o(\vert \log\varepsilon\vert^{-\beta})
\text{ as } \varepsilon\to 0.
\end{eqnarray}
Then for every $N\geq 0$,
\begin{equation}
\label{Eq A E tilde N bis}
\ind_{\widetilde{\Ns}_{\varepsilon}(A)}
=
\dfrac{1}{\sqrt{2\pi}}
\sum_{k=0}^{N} (-1)^{k}
\dfrac{1}{2^{k} k! (k+1/2)}
\dfrac{\tilde{\psi}_{2k+1,A}}{\big(\frac{1}{2\pi}
\vert\log \varepsilon\vert\big)^{k + 1/2}}
~+~\widetilde{R}_{N,\varepsilon},
\end{equation}
where the error term $\widetilde{R}_{N,\varepsilon}$ satisfies, as $\varepsilon \to 0$,
\begin{displaymath}
\forall \eta >0,~
\E\big[\Vert \widetilde{R}_{N,\varepsilon} \Vert_{H^{-\eta}(\C)}^{2}\big]^{1/2}
=\, o (\vert\log \varepsilon\vert^{-(N+1/2)}).
\end{displaymath}
\end{thm}

\medskip

\begin{rem}
In the case of $\widetilde{\Ns}_{\varepsilon}(A)$, the cutoff function $f_{0}$
is to avoid the issues related to the boundary $\partial D$ of the domain.
These could typically arise if
\begin{displaymath}
\operatorname{Leb}(\{ z\in D \vert \,d(z,\partial D)<\varepsilon\}) 
\gg \vert \log\varepsilon\vert^{-\beta},
\end{displaymath}
for some $\beta>0$, 
as then deterministic contributions originating from $\partial D$ should also appear in the asymptotic expansion.
Such a cutoff is not required under the condition \eqref{Eq cond small boundary}.
Also note that this condition \eqref{Eq cond small boundary}
is automatically satisfied if the Minkowski dimension of $\partial D$ is
strictly smaller than $2$.
In the case of $\Ns_{\varepsilon}(A)$,
the cutoff is not required regardless whether \eqref{Eq cond small boundary} holds or not.
\end{rem}

\medskip

\begin{rem}
Both in \eqref{Eq A E N} and \eqref{Eq A E tilde N bis}, 
the first order term is the same, equal to
\begin{displaymath}
2\,\dfrac{\psi_{1}}{\vert \log\varepsilon\vert^{1/2}}
=
2\,\dfrac{\nu_{A}}{\vert \log\varepsilon\vert^{1/2}},
\end{displaymath}
where $\nu_{A}$ is the Minkowski content measure on $A$; see Theorem \ref{Thm Mink ALS1}.
Higher order terms however differ.
\end{rem}

\medskip

\begin{rem}
Note that both sides of \eqref{Eq A E N} satisfy the conformal invariance in law,
whereas in \eqref{Eq A E tilde N} and \eqref{Eq A E tilde N bis}, 
neither of the sides of the equality is conformally  invariant.
\end{rem}

\medskip

\begin{rem}
Theorem \ref{Thm A E FPS} does \textbf{not} imply that for fixed $\varepsilon>0$,
$\ind_{\Ns_{\varepsilon}(A)}$ equals
\begin{displaymath}
\dfrac{1}{\sqrt{2\pi}}
\sum_{k=0}^{+\infty} (-1)^{k}
\dfrac{1}{2^{k} k! (k+1/2)}
\dfrac{\psi_{2k+1,A}}{\big(\frac{1}{2\pi}
\vert\log \varepsilon\vert\big)^{k + 1/2}},
\end{displaymath} 
and similarly for $\ind_{\widetilde{\Ns}_{\varepsilon}(A)}f_{0}$ and $\ind_{\widetilde{\Ns}_{\varepsilon}(A)}$.
This cannot be the case, since $\ind_{\Ns_{\varepsilon}(A)}$
is uniformly supported on $\Ns_{\varepsilon}(A)$,
while the terms of \eqref{Eq A E N} are all supported on $A$ and do not ``charge''
$\Ns_{\varepsilon}(A)$.
Most likely, the radius of convergence of the above power series in 
$\vert\log \varepsilon\vert$ is $0$.
If we replace $\psi_{2k+1,A}$ by $:\Phi^{2k+1}:$, then
\begin{displaymath}
\E\big[\big(:\Phi^{2k+1}:\, , 1\big)^{2}\big]^{1/2}\geq
c\dfrac{(2k+1)!}{(2\pi)^{k} 2^{k}},
\end{displaymath}
for some constant $c>0$ not depending on $k$.
So,
\begin{displaymath}
\sum_{k=0}^{+\infty}
\dfrac{1}{2^{k} k! (k+1/2)}
\dfrac{\E\big[\big(:\Phi^{2k+1}:\, , 1\big)^{2}\big]^{1/2}}{\big(\frac{1}{2\pi}
\vert\log \varepsilon\vert\big)^{k + 1/2}}
= +\infty
\end{displaymath}
because the factor $(2k+1)!/k!$ grows superexponentially.
See also Proposition \ref{Prop 0 rc BM} and the remark preceding it for the situation in the Brownian case.
\end{rem}

\medskip

\begin{rem}
Note that the expansions \eqref{Eq A E N}, \eqref{Eq A E tilde N} and \eqref{Eq A E tilde N bis}
do not involve the even powers $\psi_{2k, A}$.
We believe that this is related to the fact that one cannot restrict $\psi_{2k, A}$
to $A$ in the first place; see Remark \ref{Rem non sep} and
Corollary \ref{Cor restr A}.
Also, the integer powers of $\vert\log \varepsilon\vert$
do not appear.
\end{rem}

\medskip

\begin{rem}
\label{Rem an Le Gall}
We see Theorem \ref{Thm A E FPS} as an analogue of Le Gall's expansion for the
Wiener sausage ($\varepsilon$-neighborhood of a Brownian path) in dimension $2$ 
\cite{LeGallLocTime,LeGallStFlour}; see Section \ref{Subsec Le Gall sausage}.
The leading exponents of $1/\vert\log \varepsilon\vert$ differ:
$1$ in the Brownian case, and $1/2$ in the GFF case.
However, the increment of the exponents is the same in both cases, $1$.
In both cases, the leading coefficients are positive measures,
while the higher order terms involve more complicated renormalized quantities.
The Brownian and the Gaussian expansions could perhaps also be related through the fact that the first passage sets correspond to clusters of Brownian excursions and loops,
with however infinitely many Brownian loops of all small scales \cite{ALS2}.
We do not develop this point of view in this paper.
However, the renormalized intersection local times of these Brownian trajectories correspond to the even Wick powers $:\Phi^{2k}:$ \cite{LeJan2011Loops},
those that actually do not appear in the asymptotic expansions above.
So the correspondence between the Brownian expansion and the Gaussian expansion through the Brownian representations of the GFF is actually not straightforward,
and is yet to be understood.
See Section \ref{Subsec Wiener to FPS} for an in-depths discussion.
\end{rem}

\bigskip

The expansion \eqref{Eq A E N} provides a way to explicitly express
the restricted off Wick powers $\psi_{2n+1, A}$ in terms of $A$, by induction on the degree.
The expression for $\psi_{1, A}=\nu_{A}$ coincides with that of Theorem \ref{Thm Mink ALS1} (Minkowski content).
Further,
\begin{displaymath}
\psi_{3, A} = \lim_{\varepsilon \to 0}
- 
\dfrac{3}{2\pi}
\vert \log\varepsilon\vert^{3/2}
\Big(
\ind_{\Ns_{\varepsilon}(A)} - \dfrac{1}{2} \vert\log\varepsilon\vert^{1/2} \psi_{1,A}
\Big),
\end{displaymath}
with convergence in $L^{2}(d\PP,\sigma(\Phi),H^{-\eta}(\C))$.
For higher degrees we get the following, which is just a rephrasing of \eqref{Eq A E N}:
\begin{multline*}
\psi_{2n+1, A} = 
\\
\lim_{\varepsilon \to 0} (-1)^{n}\dfrac{2^{n} n! (n+1/2)}{(2\pi)^{n}}\vert\log\varepsilon\vert^{n+1/2}
\Big(
\ind_{\Ns_{\varepsilon}(A)} - 
\dfrac{1}{\sqrt{2\pi}}
\sum_{k=0}^{n-1} (-1)^{k}
\dfrac{1}{2^{k} k! (k+1/2)}
\dfrac{\psi_{2k+1,A}}{\big(\frac{1}{2\pi}\vert\log \varepsilon\vert\big)^{k + 1/2}}
\Big),
\end{multline*}
with convergence in $L^{2}(d\PP,\sigma(\Phi),H^{-\eta}(\C))$.

Another expression for $\psi_{2n+1, A}$
can be obtained by combining different scales $\varepsilon^{\alpha_{i}}$.

\begin{cor}
\label{Cor psi Vandermonde}
Let $n\geq 1$
and let $\alpha_{n}>\alpha_{n-1}>\dots >\alpha_{1}>\alpha_{0}=1$.
Let $(c_{0}, c_{1}, \dots, c_{n})$
be the unique solution (Vandermonde) to the linear system
\begin{equation}
\label{Eq Vandermonde sys}
\forall k\in \{0,\dots, n-1\},
\sum_{i=0}^{n} c_{i} \alpha_{i}^{-(k+1/2)} = 0,
\qquad
\sum_{i=0}^{n} c_{i} \alpha_{i}^{-(n+1/2)} = 1.
\end{equation}
Then
\begin{equation}
\label{Eq psi scales}
\psi_{2n+1, A} = \lim_{\varepsilon \to 0} (-1)^{n}\dfrac{2^{n} n! (n+1/2)}{(2\pi)^{n}}\vert\log\varepsilon\vert^{n+1/2}
\sum_{i=0}^{n} c_{i} \ind_{\Ns_{\varepsilon^{\alpha_{i}}}(A)},
\end{equation}
with convergence in $L^{2}(d\PP,\sigma(\Phi),H^{-\eta}(\C))$ for $\eta>0$.
\end{cor}

\begin{proof}
For $\ind_{\Ns_{\varepsilon^{\alpha_{i}}}(A)}$ we have the following asymptotic expansion:
\begin{displaymath}
\ind_{\Ns_{\varepsilon^{\alpha_{i}}}(A)}
=
\dfrac{1}{\sqrt{2\pi}}
\sum_{k=0}^{n} (-1)^{k}
\dfrac{\alpha_{i}^{-(k+1/2)}}{2^{k} k! (k+1/2)}
\dfrac{\psi_{2k+1,A}}{\big(\frac{1}{2\pi}\vert\log \varepsilon\vert\big)^{k + 1/2}}
+
o(\vert\log \varepsilon\vert^{-(n + 1/2)}).
\end{displaymath}
Then \eqref{Eq psi scales} is a consequence of the linear system \eqref{Eq Vandermonde sys}.
Note that the matrix $(\alpha_{i}^{-(k+1/2)})_{0\leq i,k\leq n}$,
associated to the linear system \eqref{Eq Vandermonde sys},
is invertible,
because up to a multiplication of rows by $\alpha_{i}^{-1/2}$
it is a Vandermonde matrix.
\end{proof}

In Section \ref{Sec psi GMC},
we will present a different, but related,
description of the fields $\psi_{2n+1, A}$.

\bigskip

The expansions \eqref{Eq A E N} and \eqref{Eq A E tilde N} imply expansions of second moments into integer powers of
$1/\vert\log \varepsilon\vert$, as below.

\begin{cor}
\label{Cor 2 pt}
Let $\eta>0$. Then for every $N\geq 0$ and every deterministic functions or generalized functions
$f_{1},f_{2}\in H^{\eta}(\C)$ (in particular for every $f_{1},f_{2}$ smooth compactly supported in $\C$ test functions),
\begin{multline*}
\E\Big[\int_{\Ns_{\varepsilon}(A)} f_{1}(z)\,d^{2} z~\int_{\Ns_{\varepsilon}(A)} f_{2}(w)\,d^{2} w\Big]
=
\\
\sum_{n=1}^{N+1}
\Big(
\sum_{k=0}^{n-1}\dfrac{\E[(\psi_{2k+1, A},f_{1})\,(\psi_{2(n-k)-1, A},f_{2})]}{k!(n-1-k)! (k+1/2)(n-k-1/2)}
\Big)
\dfrac{(-1)^{n-1} \pi^{n-1}}{\vert\log \varepsilon\vert^{n}}
~+~ o(\vert\log \varepsilon\vert^{-(N+1)}).
\end{multline*}
If furthermore $f_{1}$ and $f_{2}$ are compactly supported in $D$, then
\begin{multline*}
\E\Big[\int_{\widetilde{\Ns}_{\varepsilon}(A)} f_{1}(z)\,d^{2} z~
\int_{\widetilde{\Ns}_{\varepsilon}(A)} f_{2}(w)\,d^{2} w\Big]
=
\\
\sum_{n=1}^{N+1}
\Big(
\sum_{k=0}^{n-1}\dfrac{\E[(\tilde{\psi}_{2k+1, A},f_{1})\,(\tilde{\psi}_{2(n-k)-1, A},f_{2})]}
{k!(n-1-k)! (k+1/2)(n-k-1/2)}
\Big)
\dfrac{(-1)^{n-1} \pi^{n-1}}{\vert\log \varepsilon\vert^{n}}
~+~ o(\vert\log \varepsilon\vert^{-(N+1)}).
\end{multline*}
\end{cor}

\begin{proof}
Let us first deal with the case of $\Ns_{\varepsilon}(A)$. 
Note that
\begin{multline*}
\E[\,\vert(\psi_{2k+1, A},f_{1})\,(\psi_{2(n-k)-1, A},f_{2})\vert\,]
\\\leq \E\big[\Vert\psi_{2k+1, A}\Vert_{H^{-\eta}(\C)}^{2}\big]^{1/2}\,
\E\big[\Vert\psi_{2(n-k)-1, A}\Vert_{H^{-\eta}(\C)}^{2}\big]^{1/2}
\,\Vert f_{1}\Vert_{H^{\eta}(\C)}
\,\Vert f_{2}\Vert_{H^{\eta}(\C)}
\\\leq \E\big[\Vert :\Phi^{2k+1}:\Vert_{H^{-\eta}(\C)}^{2}\big]^{1/2}\,
\E\big[\Vert :\Phi^{2(n-k) -1}:\Vert_{H^{-\eta}(\C)}^{2}\big]^{1/2}
\,\Vert f_{1}\Vert_{H^{\eta}(\C)}
\,\Vert f_{2}\Vert_{H^{\eta}(\C)}
<+\infty.
\end{multline*}
For both $\int_{\widetilde{\Ns}_{\varepsilon}(A)} f_{1}$ and $\int_{\widetilde{\Ns}_{\varepsilon}(A)} f_{2}$
we use the expansion \eqref{Eq A E N} up to order
$\vert\log \varepsilon\vert^{-(N+1/2)}$, test it against $f_{1}$, respectively $f_{2}$, and then take the expectation
of the product.

In the case of $\widetilde{\Ns}_{\varepsilon}(A)$,
our additional assumptions on $f_{1}$ and $f_{2}$
imply that there is a smooth function $f_{0}:D\rightarrow [0,1]$, compactly supported in $D$,
such that $f_{1} = f_{0}f_{1}$ and $f_{2} = f_{0}f_{2}$.
Then we use the expansion \eqref{Eq A E tilde N}.
\end{proof}

\begin{rem}
It would be interesting to show that all the above second moments have densities, 
that is to say two-point correlation functions defined for 
$z,w\in D$ with $z\neq w$, 
and that the second moment expansions also hold at the level of two-point correlation functions.
\end{rem}

\bigskip

Next we give a first intuition why the combinatorial coefficients
\begin{equation}
\label{Eq comb coeff}
(-1)^{k}
\dfrac{1}{2^{k} k! (k+1/2)}
\end{equation}
should appear in the expansions \eqref{Eq A E N}, \eqref{Eq A E tilde N} and \eqref{Eq A E tilde N bis}.
Further heuristics will be given in Section \ref{Subsec algeb} and in Section \ref{Subsec psi GMC multiscale}.

By taking the expectation, we have
\begin{displaymath}
\E[ \ind_{\Ns_{\varepsilon}(A)} ]
=
\PP(z\in \Ns_{\varepsilon}(A)) \ind_{z\in D}.
\end{displaymath}
By conformal invariance, 
the probability $\PP(z\in \Ns_{\varepsilon}(A))$
is the same whatever $z\in D$.
Let $T_{0}$ be the first hitting time of level $0$ for a standard one-dimension Brownian motion starting from $v$.
By Theorem \ref{Thm law CR FPS}, 
the r.v. $\CR(z,D\setminus A)/ \CR(z,D)$ has the same distribution as $e^{-2\pi T_{0}}$. 
Thus,
\begin{displaymath}
\PP(z\in \Ns_{\varepsilon}(A)) = \PP_{v}\Big(T_{0}>\frac{1}{2\pi}\vert\log\varepsilon\vert\Big).
\end{displaymath}
The law of $T_{0}$ has the density
\begin{displaymath}
\ind_{t>0} \dfrac{v}{\sqrt{2\pi} t^{3/2}}e^{-v^{2}/(2t)}.
\end{displaymath}
Check for instance \cite[Formula 2.0.2]{BorodinSalminen2015}.
Thus,
\begin{displaymath}
\PP(z\in \Ns_{\varepsilon}(A))
=
\int_{\frac{1}{2\pi}\vert\log\varepsilon\vert}^{+\infty}
\dfrac{v}{\sqrt{2\pi} t^{3/2}}e^{-v^{2}/(2t)}\, dt.
\end{displaymath}
By expanding $e^{-v^{2}/(2t)}$ into power series, we get
\begin{equation}
\label{Eq series 1}
\PP(z\in \Ns_{\varepsilon}(A))
=
\dfrac{1}{\sqrt{2\pi}}
\sum_{k=0}^{+\infty} (-1)^{k}
\dfrac{1}{2^{k} k! (k+1/2)}
\dfrac{v^{2k+1}}{\big(\frac{1}{2\pi}\vert\log \varepsilon\vert\big)^{k + 1/2}}.
\end{equation}
But now,
\begin{displaymath}
\E[\psi_{2k+1, A}] = \E[\, :\Phi^{2k+1}:\, ] = v^{2k+1}\ind_{D}.
\end{displaymath}
So, finally,
\begin{equation}
\label{Eq eq expect}
\E[ \ind_{\Ns_{\varepsilon}(A)} ]
=
\dfrac{1}{\sqrt{2\pi}}
\sum_{k=0}^{+\infty} (-1)^{k}
\dfrac{1}{2^{k} k! (k+1/2)}
\dfrac{\E[\psi_{2k+1, A}]}{\big(\frac{1}{2\pi}\vert\log \varepsilon\vert\big)^{k + 1/2}}.
\end{equation}
Of course, an equality in expectation does not imply the expansion \eqref{Eq A E N}.
On top of that, it is unclear whether
on the right-hand side of \eqref{Eq eq expect}
one can interchange the expectation and the infinite sum.
In our proof of Theorem \ref{Thm A E FPS}, we will use an enhanced version of the above argument,
which involves conditional expectations with respect to a filtration.

Now let us consider the case of $\widetilde{\Ns}_{\varepsilon}(A)$.
Let $\varepsilon\in (0,1)$ and $z\in D$ such that
$\CR(z,D)>\varepsilon$.
Then
\begin{eqnarray*}
\PP(z\in \widetilde{\Ns}_{\varepsilon}(A)) 
&=&
\PP_{v}\Big(T_{0}>\frac{1}{2\pi}(\vert\log\varepsilon\vert + \log\CR(z,D))\Big)
\\
&=&
\dfrac{1}{\sqrt{2\pi}}
\sum_{k=0}^{+\infty} (-1)^{k}
\dfrac{1}{2^{k} k! (k+1/2)}
\dfrac{v^{2k+1}}{\big(\frac{1}{2\pi}(\vert\log \varepsilon\vert + \log\CR(z,D))\big)^{k + 1/2}}.
\end{eqnarray*}
By reexpanding the powers of $1/(\vert\log \varepsilon\vert + \log\CR(z,D))$
into powers of $1/\vert\log \varepsilon\vert$
(provided $\CR(z,D)\in (\varepsilon, \varepsilon^{-1})$), we get
that $\PP(z\in \widetilde{\Ns}_{\varepsilon}(A))$ equals
\begin{equation}
\label{Eq reexp 1}
\dfrac{1}{\sqrt{2\pi}}
\sum_{k=0}^{+\infty} (-1)^{k}
\dfrac{1}{2^{k} k! (k+1/2)}
\Big(\sum_{j=0}^{k} \dfrac{(2k+1)!}{2^{j} j! (2(k-j)+1)!}
\,\Big(\dfrac{1}{2\pi} \log \CR(z,D)\Big)^{j}\,v^{2(k-j)+1}
\Big)
\dfrac{1}{\big(\frac{1}{2\pi}\vert\log \varepsilon\vert\big)^{k + 1/2}}.
\end{equation}
See Proposition \ref{Prop reexp} in the forthcoming Section \ref{Subsec cond TVS} for the details.
So, given a smooth cutoff function $f_{0}:D\rightarrow [0,1]$, compactly supported in $D$,
and given $\varepsilon\in (0,1)$ small enough such that for every
$z\in \operatorname{Supp}(f_{0})$,
$\CR(z,D)\in (\varepsilon, \varepsilon^{-1})$,
we have
\begin{displaymath}
\E[f_{0}\ind_{\widetilde{\Ns}_{\varepsilon}(A)} ]
=
\dfrac{1}{\sqrt{2\pi}}
\sum_{k=0}^{+\infty} (-1)^{k}
\dfrac{1}{2^{k} k! (k+1/2)}
\dfrac{\E[f_{0}\,\tilde{\psi}_{2k+1, A}]}{\big(\frac{1}{2\pi}\vert\log \varepsilon\vert\big)^{k + 1/2}}.
\end{displaymath}

\subsection{Conditioning on two-valued sets}
\label{Subsec cond TVS}

Let $b>v\vee 2\lambda$, where $2\lambda$ is the height gap of the GFF.
We consider $A_{0,b}$ the two-valued set of our GFF $\Phi$ (boundary condition $v$)
of levels $0$ and $b$.
Let $\ell_{0,b}$ be the label function of $(\Phi, A_{0,b})$.
Then $A_{0,b}\subset A$ a.s.
Moreover, the sigma-algebra of $(A_{0,b},\ell_{0,b})$
is contained in the sigma-algebra of $A$.
We will further consider the conditional expectations with respect to $(A_{0,b},\ell_{0,b})$.

First, we have
\begin{displaymath}
\E[\psi_{2k+1, A} \vert A_{0,b},\ell_{0,b}] = 
\E[\, :\Phi^{2k+1}:\, \vert A_{0,b},\ell_{0,b}].
\end{displaymath}
By Proposition \ref{Prop decomp Wick TVS},
\begin{displaymath}
\E[\psi_{2k+1, A} \vert A_{0,b},\ell_{0,b}] = Q_{2k+1}(\ell_{0,b}, V_{A_{0,b}}),
\end{displaymath}
where for $z\in D\setminus A_{0,b}$,
\begin{displaymath}
V_{A_{0,b}}(z) = \dfrac{1}{2\pi}\log \dfrac{\CR(z,D)}{\CR(z,D\setminus A_{0,b})} .
\end{displaymath}
Note that for $z\in D\setminus A_{0,b}$ such that $\ell_{0,b}(z) = b$, we have
\begin{displaymath}
Q_{2k+1}(\ell_{0,b}(z), V_{A_{0,b}}(z)) = 
\sum_{j=0}^{k} (-1)^{j}\dfrac{(2k+1)!}{2^{j} j! (2(k-j)+1)!}
\,V_{A_{0,b}}(z)^{j}\,b^{2(k-j)+1}.
\end{displaymath}
As for $z\in D\setminus A_{0,b}$ such that $\ell_{0,b}(z) = 0$,
\begin{displaymath}
Q_{2k+1}(\ell_{0,b}(z), V_{A_{0,b}}(z)) = 0.
\end{displaymath}

Now let us consider
$\E[\tilde{\psi}_{2k+1, A} \vert A_{0,b},\ell_{0,b}]$,
where $\tilde{\psi}_{2k+1, A}$ is given by \eqref{Eq def tilde psi}.
We have
\begin{displaymath}
\E[\tilde{\psi}_{2k+1, A} \vert A_{0,b},\ell_{0,b}](z) = 
\sum_{j=0}^{k} \dfrac{(2k+1)!}{2^{j} j! (2(k-j)+1)!}
\,\Big(\dfrac{1}{2\pi} \log \CR(z,D)\Big)^{j}\,
Q_{2(k-j)+1}(\ell_{0,b}(z), V_{A_{0,b}}(z)).
\end{displaymath}
Further, according to the change of variance formula \eqref{Eq change var},
\begin{eqnarray*}
\E[\tilde{\psi}_{2k+1, A} \vert A_{0,b},\ell_{0,b}](z) &=&
Q_{2k+1}\Big(\ell_{0,b}(z), -\frac{1}{2\pi}\log \CR(z,D\setminus A_{0,b})\Big)
\\
&=&
\ind_{\ell_{0,b}(z) = b}\,
Q_{2k+1}\Big(b, -\frac{1}{2\pi}\log \CR(z,D\setminus A_{0,b})\Big).
\end{eqnarray*}

Let us study now the conditional expectations
$\E[ \ind_{\Ns_{\varepsilon}(A)} \vert A_{0,b},\ell_{0,b}]$
and 
$\E[ \ind_{\widetilde{\Ns}_{\varepsilon}(A)} \vert A_{0,b},\ell_{0,b}]$.
If $z\in D\setminus A_{0,b}$ such that $\ell_{0,b}(z) = 0$,
then the connected component of $z$ in $D\setminus A$
equals to that in $D\setminus A_{0,b}$.
So, in this case,
\begin{displaymath}
\CR(z,D\setminus A) = \CR(z,D\setminus A_{0,b}).
\end{displaymath}
Therefore,
\begin{equation}
\label{Eq triv 0 loop}
\E[ \ind_{\Ns_{\varepsilon}(A)} \vert A_{0,b},\ell_{0,b}] \ind_{D\setminus A_{0,b}\cap \{\ell_{0,b} = 0\}}
= \ind_{\{ \CR(z,D\setminus A_{0,b})<\varepsilon \CR(z,D)\}} \ind_{D\setminus A_{0,b}\cap \{\ell_{0,b} = 0\}},
\end{equation}
and
\begin{displaymath}
\E[ \ind_{\widetilde{\Ns}_{\varepsilon}(A)} \vert A_{0,b},\ell_{0,b}] \ind_{D\setminus A_{0,b}\cap \{\ell_{0,b} = 0\}} =
\ind_{\{ \CR(z,D\setminus A_{0,b})<\varepsilon \}} \ind_{D\setminus A_{0,b}\cap \{\ell_{0,b} = 0\}}.
\end{displaymath}

By contrast, the case of $z\in D\setminus A_{0,b}$ such that $\ell_{0,b}(z) = b$
is non-trivial.
Fix $\varepsilon\in (0,1)$ and $z\in D$. 
Conditionally on $(A_{0,b},\ell_{0,b})$,
on the event $\{\ell_{0,b}(z) = b\}$,
the r.v.
$\CR(z,D\setminus A)/\CR(z,D\setminus A_{0,b})$ is distributed as
$e^{-2\pi T_{0}}$,
where $T_{0}$ is the first hitting time of $0$ of a standard one-dimensional Brownian motion starting from $b$; 
see Theorem \ref{Thm law CR FPS}.
Therefore,
\begin{displaymath}
\PP(z\in \Ns_{\varepsilon}(A)\vert A_{0,b},\ell_{0,b},\ell_{0,b}(z) = b) =
\PP_{b}\Big(T_{0}> 0\vee\Big(\frac{1}{2\pi}\vert\log\varepsilon\vert - V_{A_{0,b}}(z)\Big)\Big).
\end{displaymath}
As in \eqref{Eq series 1}, this in turn equals
\begin{displaymath}
\dfrac{1}{\sqrt{2\pi}}
\sum_{k=0}^{+\infty} (-1)^{k}
\dfrac{1}{2^{k} k! (k+1/2)}
\dfrac{b^{2k+1}}{\big(\frac{1}{2\pi}\vert\log \varepsilon\vert - V_{A_{0,b}}(z) \big)^{k + 1/2}},
\end{displaymath}
provided $2\pi V_{A_{0,b}}(z)<\vert \log\varepsilon\vert$.
In the same way, on the event $\{\ell_{0,b}(z) = b, \CR(z,D\setminus A_{0,b})>\varepsilon\}$,
we have the equality
\begin{displaymath}
\PP(z\in \widetilde{\Ns}_{\varepsilon}(A)\vert A_{0,b},\ell_{0,b}) =
\dfrac{1}{\sqrt{2\pi}}
\sum_{k=0}^{+\infty} (-1)^{k}
\dfrac{1}{2^{k} k! (k+1/2)}
\dfrac{b^{2k+1}}{\big(\frac{1}{2\pi}\vert\log \varepsilon\vert - V_{A_{0,b}}(z) +  \frac{1}{2\pi}\log\CR(z,D)\big)^{k + 1/2}}.
\end{displaymath}
Now we want to reexpand the expressions above into powers of $1/\vert\log \varepsilon\vert$.
This is similar to \eqref{Eq reexp 1}, and we give the details in Proposition \ref{Prop reexp} and 
Proposition \ref{Prop cond series} below.

\begin{prop}
\label{Prop reexp}
For every $w\in\C$, $x>0$ and $y\in\R$ such that $\vert y\vert <x$,
the following holds:
\begin{equation}
\label{Eq two sum fin}
\sum_{k=0}^{+\infty}
\dfrac{1}{2^{k} k! (k+1/2)}
\Big(\sum_{j=0}^{k}\dfrac{(2k+1)!}{2^{j} j! (2(k-j)+1)!}
\,\vert y\vert^{j}\,\vert w\vert^{2(k-j)+1}
\Big)
\dfrac{1}{x^{k + 1/2}}
<+\infty,
\end{equation}
and
\begin{multline}
\label{Eq reexp w x y}
\sum_{k=0}^{+\infty} (-1)^{k}
\dfrac{1}{2^{k} k! (k+1/2)}
\dfrac{w^{2k+1}}{(x-y)^{k + 1/2}}
\\
=
\sum_{k=0}^{+\infty} (-1)^{k}
\dfrac{1}{2^{k} k! (k+1/2)}
\Big(\sum_{j=0}^{k}(-1)^{j} \dfrac{(2k+1)!}{2^{j} j! (2(k-j)+1)!}
\,y^{j}\,w^{2(k-j)+1}
\Big)
\dfrac{1}{x^{k + 1/2}}
\\=
\sum_{k=0}^{+\infty} (-1)^{k}
\dfrac{1}{2^{k} k! (k+1/2)}
Q_{2k+1}(w,y)
\dfrac{1}{x^{k+1/2}}.
\end{multline}
\end{prop}
\begin{proof}
Let us first check \eqref{Eq two sum fin}.
In the left-hand side of \eqref{Eq two sum fin},
consider the change of variables 
$m = k-j$.
Then \eqref{Eq two sum fin} equals
\begin{displaymath}
\sum_{m=0}^{+\infty}
\dfrac{\vert w\vert^{2m+1}}{2^{m} m! (m+1/2)}
\dfrac{1}{x^{m+1/2}}
\sum_{j=0}^{+\infty}\dfrac{(2(m+j)+1)!m!(m+1/2)}{2^{2j} j!(m+j)!(m+j+1/2)(2 m + 1)!}\dfrac{\vert y\vert^{j}}{x^{j}}
\end{displaymath}
Further, we can write
\begin{multline*}
\dfrac{(2(m+j)+1)!m!(m+1/2)}{2^{2j} j!(m+j)!(m+j+1/2)(2 m + 1)!}
=
\dfrac{1}{2^{j}j!}
\dfrac{2^{m} m!}{(2m)!}
\dfrac{(2(m+j))!}{2^{m+j} (m+j)!}
=
\dfrac{1}{2^{j}j!} 
\prod_{\substack{2m<q<2(m+j)\\ q \text{~odd}}} q
\\
=
\dfrac{1}{2^{j}j!}
\prod_{l=0}^{j-1}(2(m+l)+1)
=
\dfrac{1}{j!}
\prod_{l=0}^{j-1}((m+l)+1/2).
\end{multline*}
Moreover, since $\vert y\vert/x <1$,
\begin{displaymath}
\sum_{j=0}^{+\infty}\dfrac{1}{j!}\Big(\prod_{l=0}^{j-1}((m+l)+1/2)\Big)\dfrac{\vert y\vert^{j}}{x^{j}}
= \dfrac{1}{(1 - \vert y\vert/x)^{m+1/2}}.
\end{displaymath}
Therefore, \eqref{Eq two sum fin} equals
\begin{displaymath}
\sum_{m=0}^{+\infty}
\dfrac{\vert w\vert^{2m+1}}{2^{m} m! (m+1/2)}
\dfrac{1}{(x-\vert y\vert)^{m+1/2}}.
\end{displaymath}
Above we recognize a power series in $\vert w\vert^{2}/(x-\vert y\vert)$ which has an infinite radius of convergence.
Therefore, the sum is finite.
As for the identity \eqref{Eq reexp w x y}, it follows from a similar calculation.
\end{proof}

\begin{prop}
\label{Prop cond series}
Fix $\varepsilon\in (0,1)$ and $z\in D$.
On the event $\{\ell_{0,b}(z) = b,\, 2\pi V_{A_{0,b}}(z)<\vert \log\varepsilon\vert\}$,
\begin{displaymath}
\PP(z\in \Ns_{\varepsilon}(A)\vert A_{0,b},\ell_{0,b}) =
\sum_{k=0}^{+\infty} (-1)^{k}
\dfrac{1}{2^{k} k! (k+1/2)}
\dfrac{\E[\psi_{2k+1, A} \vert A_{0,b},\ell_{0,b}](z)}{\big(\frac{1}{2\pi}\vert\log \varepsilon\vert\big)^{k + 1/2}}.
\end{displaymath}
On the event $\{\ell_{0,b}(z) = b, \CR(z,D\setminus A_{0,b})\in (\varepsilon,\varepsilon^{-1})\}$,
\begin{displaymath}
\PP(z\in \widetilde{\Ns}_{\varepsilon}(A)\vert A_{0,b},\ell_{0,b}) =
\sum_{k=0}^{+\infty} (-1)^{k}
\dfrac{1}{2^{k} k! (k+1/2)}
\dfrac{\E[\tilde{\psi}_{2k+1, A} \vert A_{0,b},\ell_{0,b}](z)}{\big(\frac{1}{2\pi}\vert\log \varepsilon\vert\big)^{k + 1/2}}.
\end{displaymath}
\end{prop}

\subsection{Bounding the error terms near the TVS}
\label{Subsec error terms TVS}

Denote
\begin{equation}
\label{Eq U 0 b eps}
U_{0,b,\varepsilon} = \{ z\in D\setminus A_{0,b}\,\vert \CR(z,D\setminus A_{0,b})<\varepsilon\CR(z,D)\},
\end{equation}
\begin{displaymath}
\widetilde{U}_{0,b,\varepsilon} = \{ z\in D\setminus A_{0,b}\,\vert \CR(z,D\setminus A_{0,b})<\varepsilon\}.
\end{displaymath}

\begin{lemma}
\label{Lem area 0}
Let be $\eta>0$, $\varepsilon\in (0,1)$ and $b\geq (2v)\vee(2\lambda)$. Then
\begin{equation}
\label{Eq bound 0 loop 1}
\E\Big[\big\Vert\ind_{U_{0,b,\varepsilon}}\big\Vert_{H^{-\eta}(\C)}^{2}\Big]^{1/2}
\leq \dfrac{2}{\sqrt{\pi}}\big\Vert\ind_{D}\big\Vert_{H^{-\eta}(\C)}\,
\varepsilon^{\pi/(32 b^{2})},
\end{equation}
and
\begin{equation}
\label{Eq bound 0 loop 2}
\E\Big[\big\Vert\ind_{\widetilde{U}_{0,b,\varepsilon}}\big\Vert_{H^{-\eta}(\C)}^{2}\Big]^{1/2}
\leq \dfrac{2}{\sqrt{\pi}}
\Big(\int_{D^{2}}
\Big(\Big(\dfrac{\varepsilon}{\CR(z,D)}\Big)\wedge 1\Big)^{\pi/(16 b^{2})}
\LK_{\eta}(\vert w-z\vert)\,d^{2}z\,d^{2}w\Big)^{1/2}.
\end{equation}
Fix a deterministic smooth cutoff function $f_{0}:D\rightarrow [0,1]$, compactly supported in $D$.
Assume that
\begin{displaymath}
\varepsilon\,\leq \inf_{z\in\operatorname{Supp}(f_{0})}\CR(z,D).
\end{displaymath}
Then
\begin{equation}
\label{Eq bound 0 loop 3}
\E\Big[\big\Vert f_{0}\ind_{\widetilde{U}_{0,b,\varepsilon}}\big\Vert_{H^{-\eta}(\C)}^{2}\Big]^{1/2}
\leq \dfrac{2}{\sqrt{\pi}}
\big\Vert f_{0}\big\Vert_{H^{-\eta}(\C)}
\Big(
\dfrac{\varepsilon}{\inf_{\operatorname{Supp}(f_{0})}\CR(z,D)}
\Big)^{\pi/(32 b^{2})}.
\end{equation}
\end{lemma}
\begin{proof}
We have that
\begin{displaymath}
\big\Vert\ind_{U_{0,b,\varepsilon}}\big\Vert_{H^{-\eta}(\C)}^{2}
=\int_{U_{0,b,\varepsilon}\times U_{0,b,\varepsilon}}
\LK_{\eta}(\vert w-z\vert)\,d^{2}z\,d^{2}w
\leq
\int_{U_{0,b,\varepsilon}\times D}
\LK_{\eta}(\vert w-z\vert)\,d^{2}z\,d^{2}w.
\end{displaymath}
Thus,
\begin{displaymath}
\E\Big[\big\Vert\ind_{U_{0,b,\varepsilon}}\big\Vert_{H^{-\eta}(\C)}^{2}\Big]\leq 
\int_{D^{2}}
\PP(z\in U_{0,b,\varepsilon})
\LK_{\eta}(\vert w-z\vert)\,d^{2}z\,d^{2}w.
\end{displaymath}
By conformal invariance, $\PP(z\in U_{0,b,\varepsilon})$ is the same whatever $z\in D$.
Theorem \ref{Thm ASW CR ell a b} gives an expression of this probability in terms of a one-dimensional standard Brownian motion starting from $v$.
Denote by $T_{0,b}$ its first exit time from the interval $(0,b)$.
Then, for every $z\in D$,
\begin{displaymath}
\PP(z\in U_{0,b,\varepsilon})
=
\PP_{v}\Big(T_{0,b}>\dfrac{1}{2\pi}\vert\log\varepsilon\vert\Big).
\end{displaymath}
Since the exit time from $(2v-b,b)$ is larger or equal to that from $(0,b)$,
\begin{displaymath}
\PP_{v}\Big(T_{0,b}>\dfrac{1}{2\pi}\vert\log\varepsilon\vert\Big)
\leq 
\PP_{v}\Big(T_{2v-b,b}>\dfrac{1}{2\pi}\vert\log\varepsilon\vert\Big)
\leq
\PP_{0}\Big(T_{-b,b}>\dfrac{1}{2\pi}\vert\log\varepsilon\vert\Big).
\end{displaymath}
By Fourier decomposition, for every $t\geq 0$,
\begin{displaymath}
\PP_{0}(T_{-b,b}>t)
=\dfrac{4}{\pi}\sum_{n=0}^{+\infty} (-1)^{n} \dfrac{1}{2n+1}e^{-\frac{(2n+1)^{2}\pi^{2}}{8 b^{2}}t}.
\end{displaymath}
Note that the terms of the above alternate sum are decreasing in absolute value.
Thus, the value of the sum is upper-bounded by the first term:
\begin{equation}
\label{Eq tail -b b}
\PP_{0}(T_{-b,b}>t)\leq \dfrac{4}{\pi}e^{-\frac{\pi^{2}}{8 b^{2}}t},
\qquad
\PP_{0}\Big(T_{-b,b}>\dfrac{1}{2\pi}\vert\log\varepsilon\vert\Big)\leq 
\dfrac{4}{\pi}\varepsilon^{\pi/(16 b^{2})}.
\end{equation}
Then \eqref{Eq bound 0 loop 1} follows.

Similarly,
\begin{displaymath}
\E\Big[\big\Vert \ind_{\widetilde{U}_{0,b,\varepsilon}}\big\Vert_{H^{-\eta}(\C)}^{2}\Big]\leq 
\int_{D^{2}}
\PP(z\in \widetilde{U}_{0,b,\varepsilon})
\LK_{\eta}(\vert w-z\vert)\,d^{2}z\,d^{2}w,
\end{displaymath}
and
\begin{displaymath}
\PP(z\in \widetilde{U}_{0,b,\varepsilon})
=\PP_{v}\Big(T_{0,b}>\dfrac{1}{2\pi}(\vert\log\varepsilon\vert + \log \CR(z,D))\Big).
\end{displaymath}
If $\CR(z,D)\geq \varepsilon$, then by \eqref{Eq tail -b b} left,
\begin{displaymath}
\PP(z\in \widetilde{U}_{0,b,\varepsilon})
\leq \dfrac{4}{\pi}\Big(\dfrac{\varepsilon}{\CR(z,D)}\Big)^{\pi/(16 b^{2})}.
\end{displaymath}
If $\CR(z,D)<\varepsilon$, then simply
\begin{displaymath}
\PP(z\in \widetilde{U}_{0,b,\varepsilon})\leq 1 < \dfrac{4}{\pi}.
\end{displaymath}
So \eqref{Eq bound 0 loop 2} follows.
The bound \eqref{Eq bound 0 loop 3} is similar.
\end{proof}

\begin{lemma}
\label{Lem V area 0}
Let be $\eta>0$, $\varepsilon\in (0,1)$ and $b\geq (2v)\vee(2\lambda)$.
Let $j\geq 1$.
Then
\begin{equation}
\label{Eq norm V A 0 b eps}
\E\Big[\Big\Vert\ind_{U_{0,b,\varepsilon}} 
V_{A_{0,b}}^{j}
\Big\Vert_{H^{-\eta}(\C)}^{2}\Big]^{1/2}
\leq \dfrac{j C_{j}}{2^{j-1}\pi^{j+1/2}}\big\Vert\ind_{D}\big\Vert_{H^{-\eta}(\C)}\,
\Big(
\dfrac{32 b^{2}}{\pi}
\Big)^{j}
\Big(
1\vee
\Big(
\dfrac{\pi}{32 b^{2}} \vert\log\varepsilon \vert
\Big)
\Big)^{j-1}
\varepsilon^{\pi/(32 b^{2})},
\end{equation}
where
\begin{equation}
\label{Eq C j int}
C_{j} = 
\int_{0}^{+\infty}
\Big(
1 + y\Big)^{j-1}
e^{-y}\,dy.
\end{equation}
Similarly,
given a deterministic smooth cutoff function $f_{0}:D\rightarrow [0,1]$, compactly supported in $D$,
and $\varepsilon\in (0,1)$, with
\begin{displaymath}
\varepsilon\,\leq \inf_{z\in\operatorname{Supp}(f_{0})}\CR(z,D),
\end{displaymath}
we have
\begin{multline}
\label{Eq norm log CR A 0 b eps}
\E\Big[\Big\Vert f_{0}\ind_{\widetilde{U}_{0,b,\varepsilon}} 
\vert \log \CR(z,D\setminus A_{0,b})\vert^{j}
\Big\Vert_{H^{-\eta}(\C)}^{2}\Big]^{1/2}
\\
\leq
\dfrac{2j C_{j}}{\sqrt{\pi}}
\big\Vert f_{0}\big\Vert_{H^{-\eta}(\C)}
\Big(
\dfrac{32 b^{2}}{\pi}
\Big)^{j}
\Big(
1\vee
\Big(
\dfrac{\pi}{32 b^{2}} \vert\log\varepsilon \vert
\Big)
\Big)^{j-1}
\Big(
\dfrac{\varepsilon}{\inf_{\operatorname{Supp}(f_{0})}\CR(z,D)}
\Big)^{\pi/(32 b^{2})},
\end{multline}
where $C_{j}$ is still given by \eqref{Eq C j int}.
\end{lemma}

\begin{proof}
For every $z\in D\setminus A_{0,b}$,
\begin{displaymath}
V_{A_{0,b}}(z)^{j}
=
\dfrac{j}{(2\pi)^{j}}
\int_{1}^{\frac{\CR(z,D)}{\CR(z,D\setminus A_{0,b})}}
\log(u)^{j-1}
\dfrac{du}{u},
\end{displaymath}
and
\begin{displaymath}
\ind_{z\in U_{0,b,\varepsilon}}
V_{A_{0,b}}(z)^{j}
=
\dfrac{j}{(2\pi)^{j}}
\int_{\varepsilon^{-1}}^{+\infty}
\ind_{z\in U_{0,b,u^{-1}}}
\log(u)^{j-1}
\dfrac{du}{u}.
\end{displaymath}
Thus,
\begin{displaymath}
\E\Big[\Big\Vert\ind_{U_{0,b,\varepsilon}} 
V_{A_{0,b}}^{j}
\Big\Vert_{H^{-\eta}(\C)}^{2}\Big]^{1/2}
\leq
\dfrac{j}{(2\pi)^{j}}
\int_{\varepsilon^{-1}}^{+\infty}
\E\Big[\big\Vert\ind_{U_{0,b,u^{-1}}}\big\Vert_{H^{-\eta}(\C)}^{2}\Big]^{1/2}
\log(u)^{j-1}
\dfrac{du}{u}.
\end{displaymath}
By Lemma \ref{Lem area 0}, bound \eqref{Eq bound 0 loop 1},
\begin{displaymath}
\E\Big[\Big\Vert\ind_{U_{0,b,\varepsilon}} 
V_{A_{0,b}}^{j}
\Big\Vert_{H^{-\eta}(\C)}^{2}\Big]^{1/2}
\leq
\dfrac{j}{2^{j-1}\pi^{j+1/2}} \big\Vert\ind_{D}\big\Vert_{H^{-\eta}(\C)}
\int_{\varepsilon^{-1}}^{+\infty}
u^{-\pi/(32 b^{2})}
\log(u)^{j-1}
\dfrac{du}{u}.
\end{displaymath}
By performing the change of variables
\begin{displaymath}
y = \dfrac{\pi}{32 b^{2}} \log(\varepsilon u),
\end{displaymath}
we get
\begin{multline*}
\E\Big[\Big\Vert\ind_{U_{0,b,\varepsilon}} 
V_{A_{0,b}}^{j}
\Big\Vert_{H^{-\eta}(\C)}^{2}\Big]^{1/2}
\\
\leq
\dfrac{j}{2^{j-1}\pi^{j+1/2}} \big\Vert\ind_{D}\big\Vert_{H^{-\eta}(\C)}
\Big(
\dfrac{32 b^{2}}{\pi}
\Big)^{j}
\varepsilon^{\pi/(32 b^{2})}
\int_{0}^{+\infty}
\Big(
\dfrac{\pi}{32 b^{2}} \vert\log\varepsilon \vert
+ y\Big)^{j-1}
e^{-y}\,dy.
\end{multline*}
If $\frac{\pi}{32 b^{2}} \vert\log\varepsilon\vert\leq 1$,
then
\begin{displaymath}
\int_{0}^{+\infty}
\Big(
\dfrac{\pi}{32 b^{2}} \vert\log\varepsilon \vert
+ y\Big)^{j-1}
e^{-y}\,dy
\leq 
\int_{0}^{+\infty}
\Big(
1 + y\Big)^{j-1}
e^{-y}\,dy.
\end{displaymath}
If $\frac{\pi}{32 b^{2}} \vert\log\varepsilon\vert > 1$,
then
\begin{displaymath}
\int_{0}^{+\infty}
\Big(
\dfrac{\pi}{32 b^{2}} \vert\log\varepsilon \vert
+ y\Big)^{j-1}
e^{-y}\,dy
\leq 
\Big(
\dfrac{\pi}{32 b^{2}} \vert\log\varepsilon \vert\Big)^{j-1}
\int_{0}^{+\infty}
\Big(
1 + y\Big)^{j-1}
e^{-y}\,dy.
\end{displaymath}
This concludes the proof of \eqref{Eq norm V A 0 b eps}.
The proof of \eqref{Eq norm log CR A 0 b eps} is similar,
and relies on the bound \eqref{Eq bound 0 loop 3}.
\end{proof}

\begin{lemma}
\label{Lem small boundary means fast enough decay}
Assume that the condition \eqref{Eq cond small boundary} holds:
the area of the $\varepsilon$-neighborhood of $\partial D$
decay faster than any power of $1/\vert \log\varepsilon \vert$.
Fix $\eta>0$. Let $b$ depend on $\varepsilon$, $b=b(\varepsilon)$,
with values in $[(2v)\vee(2\lambda),+\infty)$.
We further assume that as $\varepsilon\to 0$,
\begin{equation}
\label{Eq slow growth b eps}
\dfrac{\vert \log \varepsilon\vert}{b(\varepsilon)^{2}}\gg \log \vert\log\varepsilon\vert,
\end{equation}
that is to say
\begin{displaymath}
b(\varepsilon) = o\Big(\dfrac{\vert \log \varepsilon\vert^{1/2}}{(\log \vert\log\varepsilon\vert)^{1/2}}\Big).
\end{displaymath}
Then for every $\beta>0$, as $\varepsilon\to 0$,
\begin{equation}
\label{Eq small log eps beta 1}
\E\Big[\big\Vert\ind_{\widetilde{U}_{0,b(\varepsilon),\varepsilon}}\big\Vert_{H^{-\eta}(\C)}^{2}\Big]^{1/2}
= o(\vert \log\varepsilon \vert^{-\beta}).
\end{equation}
Similarly, for every $j\geq 1$ and $\beta>0$,
\begin{equation}
\label{Eq small log eps beta 2}
\E\Big[\Big\Vert \ind_{\widetilde{U}_{0,b(\varepsilon),\varepsilon}} 
\vert \log \CR(z,D\setminus A_{0,b})\vert^{j}
\Big\Vert_{H^{-\eta}(\C)}^{2}\Big]^{1/2}
= o(\vert \log\varepsilon \vert^{-\beta}).
\end{equation}
\end{lemma}

\begin{proof}
Let us first deal with \eqref{Eq small log eps beta 1}.
We use the bound \eqref{Eq bound 0 loop 2} (Lemma \ref{Lem area 0}).
First observe that
\begin{multline*}
\Big(\int_{D^{2}}
\Big(\Big(\dfrac{\varepsilon}{\CR(z,D)}\Big)\wedge 1\Big)^{\pi/(16 b(\varepsilon)^{2})}
\LK_{\eta}(\vert w-z\vert)\,d^{2}z\,d^{2}w\Big)^{1/2}
\\
\leq I_{\eta,D}^{1/2}
\,
\Big(\int_{D}
\Big(\Big(\dfrac{\varepsilon}{\CR(z,D)}\Big)\wedge 1\Big)^{\pi/(16 b(\varepsilon)^{2})}
\,d^{2}z\Big)^{1/2},
\end{multline*}
where
\begin{displaymath}
I_{\eta,D} = 2\pi \int_{0}^{\diam(D)}\LK_{\eta}(r)\, r\, dr <+\infty.
\end{displaymath}
Then we cut the integral
\begin{displaymath}
\int_{D}
\Big(\Big(\dfrac{\varepsilon}{\CR(z,D)}\Big)\wedge 1\Big)^{\pi/(16 b(\varepsilon)^{2})}
\,d^{2}z
\end{displaymath}
into two parts,
on the subsets
\begin{displaymath}
\{z\in D \vert \CR(z,D)<\varepsilon^{1/2}\}
\qquad
\text{and}
\qquad
\{z\in D \vert \CR(z,D)\geq\varepsilon^{1/2}\}.
\end{displaymath}
The first part is simply bounded by
\begin{displaymath}
\operatorname{Leb}(\{z\in D \vert \CR(z,D)<\varepsilon^{1/2}\}).
\end{displaymath}
By the assumption \eqref{Eq cond small boundary} and the distortion inequality
\eqref{Eq Koebe},
we get that this area is $o(\vert \log\varepsilon \vert^{-2\beta})$,
which corresponds to what we want.
The second part is bounded by
\begin{displaymath}
\vert D\vert \, \varepsilon^{\pi/(32 b(\varepsilon)^{2})},
\end{displaymath}
where $\vert D\vert$ is the area of $D$.
Then by \eqref{Eq slow growth b eps},
this is again $o(\vert \log\varepsilon \vert^{-2\beta})$.
So we get \eqref{Eq small log eps beta 1}.

Now let us deal with \eqref{Eq small log eps beta 2}.
First note that, as in the proof of Lemma \ref{Lem V area 0},
for $\varepsilon\in (0,1)$,
\begin{displaymath}
\E\Big[\Big\Vert \ind_{\widetilde{U}_{0,b(\varepsilon),\varepsilon}} 
\vert \log \CR(z,D\setminus A_{0,b})\vert^{j}
\Big\Vert_{H^{-\eta}(\C)}^{2}\Big]^{1/2}
\leq
j\int_{\varepsilon^{-1}}^{+\infty}
\E\Big[\big\Vert\ind_{U_{0,b(\varepsilon),u^{-1}}}\big\Vert_{H^{-\eta}(\C)}^{2}\Big]^{1/2}
\log(u)^{j-1}
\dfrac{du}{u}.
\end{displaymath}
Then, as previously,
\begin{displaymath}
\E\Big[\big\Vert\ind_{U_{0,b(\varepsilon),u^{-1}}}\big\Vert_{H^{-\eta}(\C)}^{2}\Big]^{1/2}
\leq\dfrac{2}{\sqrt{\pi}}
I_{\eta,D}^{1/2}
\Big(
\operatorname{Leb}(\{z\in D \vert \CR(z,D)<u^{-1/2}\})
+
\vert D\vert \, u^{-\pi/(32 b(\varepsilon)^{2})}
\Big)^{1/2},
\end{displaymath}
and we verify that we get $o(\vert \log\varepsilon \vert^{-\beta})$.
\end{proof}

\subsection{Bounding the difference between the $\varepsilon$-neighborhood and its conditional expectation}
\label{Subsec diff ind cond}

Our goal here is to bound the difference
\begin{displaymath}
\ind_{\Ns_{\varepsilon}(A)}
-\E\big[\ind_{\Ns_{\varepsilon}(A)}\big\vert  A_{0,b},\ell_{0,b}\big],
\end{displaymath}
and similarly for $\widetilde{\Ns}_{\varepsilon}(A)$.
For $z,w\in D$, we will denote by
\begin{equation}
\label{Eq evt same cc}
\Big\{z\stackrel{D\setminus A_{0,b}}{\longleftrightarrow} w\Big\}
\end{equation}
the event that $z$ and $w$ belong to the same connected component of
$D\setminus A_{0,b}$.
For $z\in D\setminus A_{0,b}$,
we will denote by $\Gamma_{0,b}(z)$
the loop of $A_{0,b}$ that surrounds $z$,
that is to say the boundary of the connected component of $z$ in $D\setminus A_{0,b}$.
Then the event $\eqref{Eq evt same cc}$ is the same as
$\{\Gamma_{0,b}(z) = \Gamma_{0,b}(w)\}$.

\begin{lemma}
\label{Lem bound ind same loop}
Let $\eta>0$, $\varepsilon>0$ and $b> v\vee (2\lambda)$.
Then
\begin{displaymath}
\E\Big[\Big\Vert
\ind_{\Ns_{\varepsilon}(A)}
-\E\big[\ind_{\Ns_{\varepsilon}(A)}\big\vert  A_{0,b},\ell_{0,b}\big]
\Big\Vert_{H^{-\eta}(\C)}^{2}\Big]
\leq
\int_{D^{2}}
\PP\Big(z\stackrel{D\setminus A_{0,b}}{\longleftrightarrow} w, \ell_{0,b}(z) = b\Big)
\LK_{\eta}(\vert w-z\vert)\,d^{2}z\,d^{2}w,
\end{displaymath}
and
\begin{displaymath}
\E\Big[\Big\Vert
\ind_{\widetilde{\Ns}_{\varepsilon}(A)}
-\E\big[\ind_{\widetilde{\Ns}_{\varepsilon}(A)}\big\vert  A_{0,b},\ell_{0,b}\big]
\Big\Vert_{H^{-\eta}(\C)}^{2}\Big]
\leq
\int_{D^{2}}
\PP\Big(z\stackrel{D\setminus A_{0,b}}{\longleftrightarrow} w, \ell_{0,b}(z) = b\Big)
\LK_{\eta}(\vert w-z\vert)\,d^{2}z\,d^{2}w.
\end{displaymath}
Given a deterministic smooth cutoff function $f_{0}:D\rightarrow [0,1]$, compactly supported in $D$,
\begin{multline*}
\E\Big[\Big\Vert
f_{0}\ind_{\widetilde{\Ns}_{\varepsilon}(A)}
-\E\big[f_{0}\ind_{\widetilde{\Ns}_{\varepsilon}(A)}\big\vert  A_{0,b},\ell_{0,b}\big]
\Big\Vert_{H^{-\eta}(\C)}^{2}\Big]
\\
\leq
\int_{D^{2}}
\PP\Big(z\stackrel{D\setminus A_{0,b}}{\longleftrightarrow} w, \ell_{0,b}(z) = b\Big)
\LK_{\eta}(\vert w-z\vert)
f_{0}(z)f_{0}(w)\,d^{2}z\,d^{2}w.
\end{multline*}
\end{lemma}

\begin{proof}
We have that
\begin{multline*}
\E\Big[\Big\Vert
\ind_{\Ns_{\varepsilon}(A)}
-\E\big[\ind_{\Ns_{\varepsilon}(A)}\big\vert  A_{0,b},\ell_{0,b}\big]
\Big\Vert_{H^{-\eta}(\C)}^{2}\Big]
=
\\
\int_{D^{2}}
\E\Big[
\PP(z,w \in \Ns_{\varepsilon}(A)\vert A_{0,b},\ell_{0,b}) - 
\PP(z\in \Ns_{\varepsilon}(A)\vert A_{0,b},\ell_{0,b})
\PP(w\in \Ns_{\varepsilon}(A)\vert A_{0,b},\ell_{0,b})
\Big]
\LK_{\eta}(\vert w-z\vert)\,d^{2}z\,d^{2}w.
\end{multline*}
Note that $\PP(z,w\in D\setminus A_{0,b})=1$.
Conditionally on $(A_{0,b},\ell_{0,b})$
and on the event $\Big\{z\stackrel{D\setminus A_{0,b}}{\nleftrightarrow} w\Big\}$,
the events $\{z\in \Ns_{\varepsilon}(A)\}$ and $\{w\in \Ns_{\varepsilon}(A)\}$
are independent.
Thus, on the event $\Big\{z\stackrel{D\setminus A_{0,b}}{\nleftrightarrow} w\Big\}$,
\begin{displaymath}
\PP(z,w \in \Ns_{\varepsilon}(A)\vert A_{0,b},\ell_{0,b}) - 
\PP(z\in \Ns_{\varepsilon}(A)\vert A_{0,b},\ell_{0,b})
\PP(w\in \Ns_{\varepsilon}(A)\vert A_{0,b},\ell_{0,b})
=0.
\end{displaymath}
On the event
$\{z\stackrel{D\setminus A_{0,b}}{\longleftrightarrow} w, \ell_{0,b}(z) = 0\}$
(and then necessarily $\ell_{0,b}(w)=\ell_{0,b}(z) = 0$)
we have
\begin{multline*}
\PP(z,w \in \Ns_{\varepsilon}(A)\vert A_{0,b},\ell_{0,b}) - 
\PP(z\in \Ns_{\varepsilon}(A)\vert A_{0,b},\ell_{0,b})
\PP(w\in \Ns_{\varepsilon}(A)\vert A_{0,b},\ell_{0,b})
\\
=
\ind_{z,w \in \Ns_{\varepsilon}(A)}
-
\ind_{z \in \Ns_{\varepsilon}(A)}
\ind_{w \in \Ns_{\varepsilon}(A)} = 0.
\end{multline*}
See \eqref{Eq triv 0 loop}.
On the event $\{z\stackrel{D\setminus A_{0,b}}{\longleftrightarrow} w, \ell_{0,b}(z) = b\}$,
the absolute value
\begin{displaymath}
\vert \PP(z,w \in \Ns_{\varepsilon}(A)\vert A_{0,b},\ell_{0,b}) - 
\PP(z\in \Ns_{\varepsilon}(A)\vert A_{0,b},\ell_{0,b})
\PP(w\in \Ns_{\varepsilon}(A)\vert A_{0,b},\ell_{0,b})\vert 
\end{displaymath}
is simply bounded by $1$.
The case of $\widetilde{\Ns}_{\varepsilon}(A)$ is similar.
\end{proof}

Further we will consider the extremal distance (same as effective electrical resistance)
between $\partial D$ and the loops $\Gamma_{0,b}(z)$ of $A_{0,b}$.
It will be denoted $\ED(\partial D,\Gamma_{0,b}(z))$.
See Section \ref{Subsubsec TVS}.

\begin{lemma}
\label{Lem bound ED}
Let $z,w\in D$ and $b> v\vee (2\lambda)$. Then
\begin{equation}
\label{Eq bound ED}
\PP\Big(z\stackrel{D\setminus A_{0,b}}{\longleftrightarrow} w, \ell_{0,b}(z) = b\Big)
\leq
\PP(\ED(\partial D,\Gamma_{0,b}(z))\leq G_{D}(z,w) + \log(2)/\pi, \ell_{0,b}(z) = b).
\end{equation}
\end{lemma}
\begin{proof}
Since both probabilities in \eqref{Eq bound ED} are conformally invariant,
we may replace the domain $D$ by the unit disk $\D$,
the point $z$ by $0$, and $w$ by a point $r\in [0,1)$
such that
\begin{displaymath}
G_{\D}(0,r) = G_{D}(z,w).
\end{displaymath}
Since
\begin{displaymath}
G_{\D}(0,r) = \dfrac{1}{2\pi} \log r^{-1},
\end{displaymath}
we get that
\begin{displaymath}
r = e^{- 2\pi G_{D}(z,w)}.
\end{displaymath}
Consider the loop $\Gamma^{\D}_{0,b}(0)$ of the TVS $A^{\D}_{0,b}$
in the unit disk $\D$.
On the event $\{\Gamma^{\D}_{0,b}(0) = \Gamma^{\D}_{0,b}(r)\}$ it surrounds both $0$ and the point $r$.
Thus, according to Proposition 2.5 in \cite{ALS3},
on the event $\{\Gamma^{\D}_{0,b}(0) = \Gamma^{\D}_{0,b}(r)\}$,
\begin{equation}
\label{Eq Grotzsch}
r \leq 4 e^{-2\pi\ED(\partial \D, \Gamma^{\D}_{0,b}(0))},
\end{equation}
that is to say
\begin{displaymath}
\ED(\partial \D, \Gamma^{\D}_{0,b}(0))\leq \dfrac{1}{2\pi} \log(4/r)
=G_{\D}(0,r) + \dfrac{\log(2)}{\pi}.
\end{displaymath}
So \eqref{Eq bound ED} follows.
The inequality \eqref{Eq Grotzsch} is a distortion bound that can be deduced from
Grötzsch's theorem \cite[Theorem 4-6]{Ahlfors2010ConfInv}.
See \cite{ALS3} for details.
\end{proof}

\begin{lemma}
\label{Lem ED Gauss}
Let $z,w\in D$ and $b\in 2\lambda\N$ such that
$b\geq v+2\lambda$. 
Then for every $L>0$,
\begin{displaymath}
\PP(\ED(\partial D,\Gamma_{0,b}(z))\leq L, \ell_{0,b}(z) = b)
\leq
e^{-(b-v-2\lambda)^{2}/(2L)}.
\end{displaymath}
\end{lemma}

\begin{proof}
Let $(W_{t})_{t\geq 0}$ be a one-dimensional standard Brownian motion starting from $0$.
Let $T_{-v,b-v}$ the first exit time from the interval $(-v,b-v)$:
\begin{displaymath}
T_{-v,b-v} = \inf\{ t>0\vert W_{t}\not\in (-v,b-v)\}.
\end{displaymath}
Let $\tau_{-v,b-v}$ be the last passage time
\begin{displaymath}
\tau_{-v,b-v} = 0\vee\sup\{t\in[0,T_{-v,b-v}]\vert W_{t} = b-v -2\lambda \}.
\end{displaymath}
According to Theorem \ref{Thm ED ALS3},
\begin{displaymath}
\PP(\ED(\partial D,\Gamma_{0,b}(z))\leq L, \ell_{0,b}(z) = b)
=
\PP_{0}\big(\tau_{-v,b-v}\leq L, W_{T_{-v,b-v}} = b-v\big).
\end{displaymath}
Further,
\begin{displaymath}
\PP_{0}\big(\tau_{-v,b-v}\leq L, W_{T_{-v,b-v}} = b-v\big)
\leq
\PP_{0}\Big(\sup_{t\in[0,L]}W_{t}\geq b-v-2\lambda\Big).
\end{displaymath}
But $\sup_{t\in[0,L]}W_{t}$ has the same distribution as $\vert W_{L}\vert$.
Thus,
\begin{multline*}
\PP_{0}\Big(\sup_{t\in[0,L]}W_{t}\geq b-v-2\lambda\Big)
=
\PP_{0}(\vert W_{L}\vert\geq b-v-2\lambda)
=
\dfrac{2}{\sqrt{2\pi L}}
\int_{b-v-2\lambda}^{+\infty}
e^{-x^{2}/(2L)}\,dx
\\
\leq
e^{-(b-v-2\lambda)^{2}/(2L)}
\,
\dfrac{2}{\sqrt{2\pi L}}
\int_{0}^{+\infty}
e^{-y^{2}/(2L)}\,dy
=
e^{-(b-v-2\lambda)^{2}/(2L)}.
\end{multline*}
This concludes.
\end{proof}

By combining Lemmas \eqref{Lem bound ind same loop},
\eqref{Lem bound ED} and \eqref{Lem ED Gauss},
we immediately get the following.

\begin{cor}
\label{Cor bound ind diff 1}
Let $\eta>0$, $\varepsilon>0$ and $b\in2\lambda\N$ with $b\geq v + 2\lambda$.
Then
\begin{displaymath}
\E\Big[\Big\Vert
\ind_{\Ns_{\varepsilon}(A)}
-\E\big[\ind_{\Ns_{\varepsilon}(A)}\big\vert  A_{0,b},\ell_{0,b}\big]
\Big\Vert_{H^{-\eta}(\C)}^{2}\Big]
\qquad
\text{and}
\qquad
\E\Big[\Big\Vert
\ind_{\widetilde{\Ns}_{\varepsilon}(A)}
-\E\big[\ind_{\widetilde{\Ns}_{\varepsilon}(A)}\big\vert  A_{0,b},\ell_{0,b}\big]
\Big\Vert_{H^{-\eta}(\C)}^{2}\Big]
\end{displaymath}
are bounded by
\begin{equation}
\label{Eq ind norm Gauss}
\int_{D^{2}}
e^{-(b-v-2\lambda)^{2}/(2\log(2)/\pi + 2 G_{D}(z,w))}
\LK_{\eta}(\vert w-z\vert)
\,d^{2}z\,d^{2}w.
\end{equation}
Given a deterministic smooth cutoff function $f_{0}:D\rightarrow [0,1]$, compactly supported in $D$,
then
\begin{displaymath}
\E\Big[\Big\Vert
f_{0}\ind_{\widetilde{\Ns}_{\varepsilon}(A)}
-\E\big[f_{0}\ind_{\widetilde{\Ns}_{\varepsilon}(A)}\big\vert  A_{0,b},\ell_{0,b}\big]
\Big\Vert_{H^{-\eta}(\C)}^{2}\Big]
\end{displaymath}
is again bounded by \eqref{Eq ind norm Gauss}.
\end{cor}

We further want a more explicit bound than \eqref{Eq ind norm Gauss}.
We will deal only with the case $\eta\in (0,1)$,
as by definition \eqref{Eq def Sobo norm}, 
the $H^{-\eta}(\C)$ norms are non-increasing in $\eta$.

\begin{lemma}
\label{Lem better bound}
Let $\eta\in (0,1)$ and $b\geq v + 2\lambda + 2\log(2)\pi$.
Then
\begin{multline*}
\int_{D^{2}}
e^{-(b-v-2\lambda)^{2}/(2\log(2)/\pi + 2 G_{D}(z,w))}
\LK_{\eta}(\vert w-z\vert)
\,d^{2}z\,d^{2}w
\\
\leq 
e^{-((2\pi\eta)\wedge 1)(b-v-2\lambda - 2\log(2)/\pi)}
(c_{\eta}\vert D\vert \diam(D)^{2\eta} +
\Vert 1_{D} \Vert_{H^{-\eta}(\C)}^{2}
),
\end{multline*}
where $\vert D\vert$ is the Lebesgue measure (area) of $D$,
and $c_{\eta}>0$ is a constant depending only on $\eta$.
\end{lemma}
\begin{proof}
Take $z,w\in D$. 
Since $D$ is contained in a disk of center $z$ and radius $\diam(D)$, 
We have that
\begin{equation}
\label{Eq simple bound G D}
G_{D}(z,w)\leq \dfrac{1}{2\pi}\log\Big(\dfrac{\diam(D)}{\vert w-z\vert}\Big).
\end{equation}
Let $r(b)>0$ be such that
\begin{displaymath}
\dfrac{2\log(2)}{\pi} + \dfrac{1}{\pi}\log\Big(\dfrac{\diam(D)}{r(b)}\Big)
=
b-v-2\lambda,
\end{displaymath}
that is to say
\begin{displaymath}
r(b) = e^{-\pi(b-v-2\lambda - 2\log(2)/\pi)}\, \diam(D).
\end{displaymath}
Then
\begin{displaymath}
\int_{\substack{(z,w)\in D^{2}\\ \vert w-z\vert\geq r(b)}}
e^{-(b-v-2\lambda)^{2}/(2\log(2)/\pi + 2 G_{D}(z,w))}
\LK_{\eta}(\vert w-z\vert)
\,d^{2}z\,d^{2}w
\leq
e^{-(b-v-2\lambda)} \Vert 1_{D} \Vert_{H^{-\eta}(\C)}^{2}.
\end{displaymath}
Further,
\begin{displaymath}
\int_{\substack{(z,w)\in D^{2}\\ \vert w-z\vert< r(b)}}
e^{-(b-v-2\lambda)^{2}/(2\log(2)/\pi + 2 G_{D}(z,w))}
\LK_{\eta}(\vert w-z\vert)
\,d^{2}z\,d^{2}w
\leq
\int_{\substack{(z,w)\in D^{2}\\ \vert w-z\vert< r(b)}}
\LK_{\eta}(\vert w-z\vert)
\,d^{2}z\,d^{2}w.
\end{displaymath}
By \eqref{Eq asymp kernel},
\begin{displaymath}
\int_{\substack{(z,w)\in D^{2}\\ \vert w-z\vert< r(b)}}
\LK_{\eta}(\vert w-z\vert)
\,d^{2}z\,d^{2}w
\leq c'_{\eta}\vert D\vert\,r(b)^{2\eta},
\end{displaymath}
for some constant $c'_{\eta}>0$ depending only on $\eta$.
By combining, we get the desired bound. 
\end{proof}

\begin{cor}
\label{Cor Eq ind norm exp}
Let $\eta\in (0,1)$.
Then there are constants $c,C>0$,
depending only on $\eta$ and $v$, such that
for every $\varepsilon>0$ and $b\in2\lambda\N$ with $b\geq v + 2\lambda$,
\begin{displaymath}
\E\Big[\Big\Vert
\ind_{\Ns_{\varepsilon}(A)}
-\E\big[\ind_{\Ns_{\varepsilon}(A)}\big\vert  A_{0,b},\ell_{0,b}\big]
\Big\Vert_{H^{-\eta}(\C)}^{2}\Big]
\qquad
\text{and}
\qquad
\E\Big[\Big\Vert
\ind_{\widetilde{\Ns}_{\varepsilon}(A)}
-\E\big[\ind_{\widetilde{\Ns}_{\varepsilon}(A)}\big\vert  A_{0,b},\ell_{0,b}\big]
\Big\Vert_{H^{-\eta}(\C)}^{2}\Big]
\end{displaymath}
are bounded by
\begin{equation}
\label{Eq ind norm exp}
Ce^{-cb}
(\vert D\vert \diam(D)^{2\eta} +
\Vert 1_{D} \Vert_{H^{-\eta}(\C)}^{2}
).
\end{equation}
Given a deterministic smooth cutoff function $f_{0}:D\rightarrow [0,1]$, compactly supported in $D$,
then
\begin{displaymath}
\E\Big[\Big\Vert
f_{0}\ind_{\widetilde{\Ns}_{\varepsilon}(A)}
-\E\big[f_{0}\ind_{\widetilde{\Ns}_{\varepsilon}(A)}\big\vert  A_{0,b},\ell_{0,b}\big]
\Big\Vert_{H^{-\eta}(\C)}^{2}\Big]
\end{displaymath}
is again bounded by \eqref{Eq ind norm exp}.
\end{cor}

\subsection{Bounding the difference between fields and their conditional expectations}
\label{Subsec error fields}

Here we will bound the differences
\begin{displaymath}
\psi_{2k+1,A} - \E[\psi_{2k+1, A} \vert A_{0,b},\ell_{0,b}]
\qquad
\text{and}
\qquad
\tilde{\psi}_{2k+1,A} - \E[\tilde{\psi}_{2k+1, A} \vert A_{0,b},\ell_{0,b}]
.
\end{displaymath}
We begin with $\psi_{2k+1,A} - \E[\psi_{2k+1, A} \vert A_{0,b},\ell_{0,b}]$.

\begin{lemma}
\label{Lem Delta psi 1}
Let $\eta>0$, $k\geq 0$ and $b\geq v\vee(2\lambda)$.
Then
\begin{multline*}
\E\Big[\big\Vert\psi_{2k+1,A} - \E[\psi_{2k+1, A} \vert A_{0,b},\ell_{0,b}]
\big\Vert_{H^{-\eta}(\C)}^{2}
\Big]^{1/2}
\leq
\\
\sum_{\substack{1\leq j\leq 2k+1 \\ 0\leq l\leq \lfloor(2k+1-j)/2\rfloor}}
\dfrac{(2k+1)!}{2^{l}j!l!(2k+1-j-2l)!}
\E\Big[\big\Vert
\ell_{0,b}^{2k+1-j-2l}V_{A_{0,b}}^{l}
\,:\Phi_{D\setminus A_{0,b}}^{j} :
\ind_{\{z\in A_{0,b}\text{~or~}\ell_{0,b}(z)=b\}}
\big\Vert_{H^{-\eta}(\C)}^{2}
\Big]^{1/2},
\end{multline*}
where the notations are those of \eqref{Eq decomp Wick TVS}.
\end{lemma}

\begin{proof}
First, we claim that
\begin{displaymath}
\psi_{2k+1,A} = \psi_{2k+1,A}\ind_{\{z\in A_{0,b}\text{~or~}\ell_{0,b}(z)=b\}}.
\end{displaymath}
Indeed $\psi_{2k+1,A}$ is supported on $A$ (Corollary \ref{Cor support psi}),
and a.s., $A\subset \{z\in \overline{D}\vert z\in A_{0,b}\text{~or~}\ell_{0,b}(z)=b\}$.
Since the random compact set $\{z\in A_{0,b}\text{~or~}\ell_{0,b}(z)=b\}$
is measurable w.r.t. $(A_{0,b},\ell_{0,b})$,
we have that
\begin{eqnarray*}
\E[\psi_{2k+1, A} \vert A_{0,b},\ell_{0,b}]
&=&
\E[\psi_{2k+1, A} \ind_{\{z\in A_{0,b}\text{~or~}\ell_{0,b}(z)=b\}}\vert A_{0,b},\ell_{0,b}]
\\&=&
\E[\psi_{2k+1, A} \vert A_{0,b},\ell_{0,b}]
\ind_{\{z\in A_{0,b}\text{~or~}\ell_{0,b}(z)=b\}}.
\end{eqnarray*}
Similarly,
\begin{displaymath}
\E[\,:\Phi^{2k+1}: \ind_{\{z\in A_{0,b}\text{~or~}\ell_{0,b}(z)=b\}}\vert A_{0,b},\ell_{0,b}]
=
\E[\,:\Phi^{2k+1}: \vert A_{0,b},\ell_{0,b}]\ind_{\{z\in A_{0,b}\text{~or~}\ell_{0,b}(z)=b\}}.
\end{displaymath}
Further, since the sigma-algebra of $(A_{0,b},\ell_{0,b})$
is contained in the sigma-algebra of $A$,
we have
\begin{displaymath}
\E[\,:\Phi^{2k+1}: \ind_{\{z\in A_{0,b}\text{~or~}\ell_{0,b}(z)=b\}}\vert A]
=
\E[\,:\Phi^{2k+1}: \vert A]\ind_{\{z\in A_{0,b}\text{~or~}\ell_{0,b}(z)=b\}}
=\psi_{2k+1, A}.
\end{displaymath}
Thus,
\begin{multline*}
\E\Big[\big\Vert\psi_{2k+1,A} - \E[\psi_{2k+1, A} \vert A_{0,b},\ell_{0,b}]
\big\Vert_{H^{-\eta}(\C)}^{2}
\Big]
\\=
\E\Big[\big\Vert
\E[\,:\Phi^{2k+1}: \ind_{\{z\in A_{0,b}\text{~or~}\ell_{0,b}(z)=b\}}\vert A]
- \E[\,:\Phi^{2k+1}: \ind_{\{z\in A_{0,b}\text{~or~}\ell_{0,b}(z)=b\}}\vert A_{0,b},\ell_{0,b}]
\big\Vert_{H^{-\eta}(\C)}^{2}
\Big]
\\
=
\E\Big[\big\Vert
(\,:\Phi^{2k+1}:
- \E[\,:\Phi^{2k+1}: \vert A_{0,b},\ell_{0,b}])
\ind_{\{z\in A_{0,b}\text{~or~}\ell_{0,b}(z)=b\}}
\big\Vert_{H^{-\eta}(\C)}^{2}
\Big]
\\-
\E\Big[\big\Vert
(\,:\Phi^{2k+1}:
- \E[\,:\Phi^{2k+1}: \vert A])
\ind_{\{z\in A_{0,b}\text{~or~}\ell_{0,b}(z)=b\}}
\big\Vert_{H^{-\eta}(\C)}^{2}
\Big].
\end{multline*}
Consequently,
\begin{multline*}
\E\Big[\big\Vert\psi_{2k+1,A} - \E[\psi_{2k+1, A} \vert A_{0,b},\ell_{0,b}]
\big\Vert_{H^{-\eta}(\C)}^{2}
\Big]^{1/2}
\\\leq 
\E\Big[\big\Vert
(\,:\Phi^{2k+1}:
- \E[\,:\Phi^{2k+1}: \vert A_{0,b},\ell_{0,b}])
\ind_{\{z\in A_{0,b}\text{~or~}\ell_{0,b}(z)=b\}}
\big\Vert_{H^{-\eta}(\C)}^{2}
\Big]^{1/2}.
\end{multline*}
By \eqref{Eq decomp Wick TVS}, 
\begin{multline*}
\E[\,:\Phi^{2k+1}:]
- \E[\,:\Phi^{2k+1}: \vert A_{0,b},\ell_{0,b}]
=\sum_{j=1}^{2k+1}
\dfrac{(2k+1)!}{j!(2k+1 -j)!}
Q_{2k+1-j}(\ell_{0,b},V_{A_{0,b}})\,:\Phi_{D\setminus A_{0,b}}^{j} :
\\
=
\sum_{\substack{1\leq j\leq 2k+1 \\ 0\leq l\leq \lfloor(2k+1-j)/2\rfloor}}
\dfrac{(2k+1)!}{2^{l}j!l!(2k+1-j-2l)!}(-1)^{l}
\ell_{0,b}^{2k+1-j-2l}V_{A_{0,b}}^{l}
\,:\Phi_{D\setminus A_{0,b}}^{j} :\,,
\end{multline*}
and so we get the desired bound.
\end{proof}

\begin{lemma}
\label{Lem 2 point cor D A 0 b}
Let $\eta>0$, $k\geq 0$, $j\geq 1$, $l\geq 0$, with $j+2l\leq 2k+1$, and $b\geq v\vee(2\lambda)$.
Then
\begin{multline*}
\E\Big[\big\Vert
\ell_{0,b}^{2k+1-j-2l}V_{A_{0,b}}^{l}
\,:\Phi_{D\setminus A_{0,b}}^{j} :
\ind_{\{z\in A_{0,b}\text{~or~}\ell_{0,b}(z)=b\}}
\big\Vert_{H^{-\eta}(\C)}^{2}
\Big]
=
\\
j!
b^{4k+2-2j- 4l}
\int_{D^{2}}
\E
\big[
V_{A_{0,b}}(z)^{l}V_{A_{0,b}}(w)^{l}
G_{D\setminus A_{0,b}}(z,w)^{j}
\ind_{\{\ell_{0,b}(z)=\ell_{0,b}(w)=b\}}
\big]
\LK_{\eta}(\vert w-z\vert)
\,d^{2}z\,d^{2}w.
\end{multline*}
\end{lemma}

\begin{proof}
We can replace $\ind_{\{z\in A_{0,b}\text{~or~}\ell_{0,b}(z)=b,~w\in A_{0,b}\text{~or~}\ell_{0,b}(w)=b\}}$ 
by $\ind_{\{\ell_{0,b}(z)=\ell_{0,b}(w)=b\}}$
because a.s., $A_{0,b}$ has $0$ Lebesgue measure.
If $\ell_{0,b}(z)=\ell_{0,b}(w)=b$, then of course
\begin{displaymath}
\ell_{0,b}(z)^{2k+1-j-2l}\ell_{0,b}(w)^{2k+1-j-2l} = b^{4k+2-2j- 4l}.
\end{displaymath}
Finally, we use that conditionally on $(A_{0,b},\ell_{0,b})$,
the two-point correlation function of the field $:\Phi_{D\setminus A_{0,b}}^{j} :$
is $j! G_{D\setminus A_{0,b}}(z,w)^{j}$.
\end{proof}

\begin{lemma}
\label{Lem a s bound}
Let $j\geq 1$ and $l\geq 0$.
There is a constant $C_{j,l}>0$ depending only on $j$ and $l$
(in particular depending neither on $D$ nor $b$),
such that a.s., for every $z,w\in D$ with $z\neq w$,
\begin{multline}
\label{Eq a s bound 2}
V_{A_{0,b}}(z)^{l}V_{A_{0,b}}(w)^{l}
G_{D\setminus A_{0,b}}(z,w)^{j}
\ind_{\{\ell_{0,b}(z)=\ell_{0,b}(w)=b\}}
\\
\leq C_{j,l} (1\vee\diam(D)^{j/4})
\Big(
1\vee\log\Big(\dfrac{2(1\vee\diam(D))}{\vert w-z\vert}\Big)^{j+2l}
\Big)
\,
\ind_{z\stackrel{D\setminus A_{0,b}}{\longleftrightarrow} w,\,\ell_{0,b}(z) = b}.
\end{multline}
\end{lemma}

\begin{proof}
First of all, $G_{D\setminus A_{0,b}}(z,w)$ is $0$ unless $z$ and $w$ belong to the same connected component
of $D\setminus A_{0,b}$.
This gives rise to the indicatrix 
$\ind_{z\stackrel{D\setminus A_{0,b}}{\longleftrightarrow} w,\,\ell_{0,b}(z) = b}$
in \eqref{Eq a s bound 2}.
Further, clearly $G_{D\setminus A_{0,b}}(z,w)\leq G_{D}(z,w)$.
So, if $l=0$ then the desired bound follows 
from \eqref{Eq simple bound G D}.

If $l\geq 1$, the bound \eqref{Eq a s bound 2} is more involved.
By the distortion inequalities \eqref{Eq Koebe},
\begin{displaymath}
\dfrac{\CR(z,D)}{\CR(z,D\setminus A_{0,b})}
\leq \dfrac{4 d(z,\partial D)}{d(z,A_{0,b})}
\leq \dfrac{4 \diam(D)}{d(z,A_{0,b})\wedge d(w,A_{0,b})}.
\end{displaymath}
A similar bound holds for $\CR(w,D)/\CR(w,D\setminus A_{0,b})$.
Thus,
\begin{displaymath}
V_{A_{0,b}}(z)^{l}V_{A_{0,b}}(w)^{l}
\leq \dfrac{1}{(2\pi)^{2 l}}
\log\Big(\dfrac{4 \diam(D)}{d(z,A_{0,b})\wedge d(w,A_{0,b})}\Big)^{2l}.
\end{displaymath}
We also apply the bound \eqref{Eq est Green} (Proposition \ref{Prop Green}) 
to $G_{D\setminus A_{0,b}}(z,w)$:
\begin{eqnarray*}
G_{D\setminus A_{0,b}}(z,w)
&\leq&
C\Big( 1\wedge \Big(
\dfrac{d(z,A_{0,b})\wedge d(w,A_{0,b})}{\vert w-z\vert}\Big)^{1/2}
\Big)
\log\Big(\dfrac{2 \diam (D\setminus A_{0,b})}{\vert w-z\vert}\Big)
\\
&\leq&
C\Big( 1\wedge \Big(
\dfrac{d(z,A_{0,b})\wedge d(w,A_{0,b})}{\vert w-z\vert}\Big)^{1/2}
\Big)
\log\Big(\dfrac{2 \diam (D)}{\vert w-z\vert}\Big),
\end{eqnarray*}
where $C>0$ is a universal constant.
Therefore,
\begin{multline*}
V_{A_{0,b}}(z)^{l}V_{A_{0,b}}(w)^{l}
G_{D\setminus A_{0,b}}(z,w)^{j}
\\\leq
\dfrac{C^{j}}{(2\pi)^{l}}
\Big( 1\wedge \Big(
\dfrac{d(z,A_{0,b})\wedge d(w,A_{0,b})}{\vert w-z\vert}\Big)^{j/2}
\Big)
\log\Big(\dfrac{4 \diam(D)}{d(z,A_{0,b})\wedge d(w,A_{0,b})}\Big)^{2l}
\log\Big(\dfrac{2 \diam (D)}{\vert w-z\vert}\Big)^{j}.
\end{multline*}
Further, we will distinguish two cases:
$d(z,A_{0,b})\wedge d(w,A_{0,b})\leq \vert w-z\vert^{2}$
and
$d(z,A_{0,b})\wedge d(w,A_{0,b})> \vert w-z\vert^{2}$.

Consider the first case $d(z,A_{0,b})\wedge d(w,A_{0,b})\leq \vert w-z\vert^{2}$.
Then 
\begin{multline*}
V_{A_{0,b}}(z)^{l}V_{A_{0,b}}(w)^{l}
G_{D\setminus A_{0,b}}(z,w)^{j}
\\\leq
\dfrac{C^{j}}{(2\pi)^{l}}
(d(z,A_{0,b})\wedge d(w,A_{0,b}))^{j/4}
\log\Big(\dfrac{4 \diam(D)}{d(z,A_{0,b})\wedge d(w,A_{0,b})}\Big)^{2l}
\log\Big(\dfrac{2 \diam (D)}{\vert w-z\vert}\Big)^{j}
\\
\leq
\dfrac{C^{j}}{(2\pi)^{l}}
\Big(\dfrac{d(z,A_{0,b})\wedge d(w,A_{0,b})}{\diam (D)}\Big)^{j/4}
\log\Big(\dfrac{4 \diam(D)}{d(z,A_{0,b})\wedge d(w,A_{0,b})}\Big)^{2l}
\\
\times \diam (D)^{j/4}
\log\Big(\dfrac{2 \diam (D)}{\vert w-z\vert}\Big)^{j}
.
\end{multline*}
Let  be
\begin{displaymath}
C^{\ast}_{j,l} = \sup_{x\geq 1} ( x^{-j/4} \log (4x)^{2l}) <+\infty.
\end{displaymath}
Then
\begin{displaymath}
V_{A_{0,b}}(z)^{l}V_{A_{0,b}}(w)^{l}
G_{D\setminus A_{0,b}}(z,w)^{j}
\leq 
\dfrac{C^{j}C^{\ast}_{j,l}}{(2\pi)^{l}}
\diam (D)^{j/4}
\log\Big(\dfrac{2 \diam (D)}{\vert w-z\vert}\Big)^{j}.
\end{displaymath}

Now consider the second case $d(z,A_{0,b})\wedge d(w,A_{0,b})> \vert w-z\vert^{2}$.
Then
\begin{multline*}
\log\Big(\dfrac{4 \diam(D)}{d(z,A_{0,b})\wedge d(w,A_{0,b})}\Big)
\leq
\log\Big(\dfrac{4 \diam(D)}{\vert w-z\vert^{2}}\Big)
\\=
2
\log\Big(\dfrac{2 \diam(D)^{1/2}}{\vert w-z\vert}\Big)
\leq 
2
\log\Big(\dfrac{2 (1\vee\diam(D))}{\vert w-z\vert}\Big).
\end{multline*}
So in both case we get a bound of desired form.
\end{proof}

By combining Lemmas \ref{Lem 2 point cor D A 0 b} and \ref{Lem a s bound}
we immediately get the following.

\begin{cor}
\label{Cor cool bound}
Let $\eta>0$, $k\geq 0$, $j\geq 1$, $l\geq 0$, with $j+2l\leq 2k+1$, and $b\geq v\vee(2\lambda)$.
Then,
\begin{multline*}
\E\Big[\big\Vert
\ell_{0,b}^{2k+1-j-2l}V_{A_{0,b}}^{l}
\,:\Phi_{D\setminus A_{0,b}}^{j} :
\ind_{\{z\in A_{0,b}\text{~or~}\ell_{0,b}(z)=b\}}
\big\Vert_{H^{-\eta}(\C)}^{2}
\Big]
\leq
\\
C_{j,l}(1\vee\diam(D)^{j/4})
b^{4k+2-2j - 4l}
\\
\times
\int_{D^{2}}
\PP
\Big(
z\stackrel{D\setminus A_{0,b}}{\longleftrightarrow} w,\,\ell_{0,b}(z) = b
\Big)
\Big(
1\vee\log\Big(\dfrac{2(1\vee\diam(D))}{\vert w-z\vert}\Big)^{j+2l}
\Big)
\LK_{\eta}(\vert w-z\vert)
\,d^{2}z\,d^{2}w,
\end{multline*}
where $C_{j,l}>0$ is a constant depending only on $j$ and $l$.
\end{cor}

By further applying Lemmas \ref{Lem bound ED} and \ref{Lem ED Gauss},
we get the following.

\begin{cor}
\label{Cor cool bound 2}
Let $\eta>0$, $k\geq 0$, $j\geq 1$, $l\geq 0$, with $j+2l\leq 2k+1$.
Let $b\in2\lambda\N$ with $b\geq v + 2\lambda$.
Then
\begin{multline*}
\E\Big[\big\Vert
\ell_{0,b}^{2k+1-j-2l}V_{A_{0,b}}^{l}
\,:\Phi_{D\setminus A_{0,b}}^{j} :
\ind_{\{z\in A_{0,b}\text{~or~}\ell_{0,b}(z)=b\}}
\big\Vert_{H^{-\eta}(\C)}^{2}
\Big]
\leq
\\
C_{j,l}(1\vee\diam(D)^{j/4})
b^{4k+2-2j - 4l}
\\
\times
\int_{D^{2}}
e^{-(b-v-2\lambda)^{2}/(2\log(2)/\pi + 2 G_{D}(z,w))}
\Big(
1\vee\log\Big(\dfrac{2(1\vee\diam(D))}{\vert w-z\vert}\Big)^{j+2l}
\Big)
\LK_{\eta}(\vert w-z\vert)
\,d^{2}z\,d^{2}w,
\end{multline*}
where $C_{j,l}>0$ is a constant depending only on $j$ and $l$.
\end{cor}

\begin{lemma}
\label{Lem better bound 2}
Let $\eta\in (0,1)$, $j\geq 1$, $l\geq 0$, and $b\geq v + 2\lambda + 2\log(2)\pi$.
Then
\begin{multline*}
\int_{D^{2}}
e^{-(b-v-2\lambda)^{2}/(2\log(2)/\pi + 2 G_{D}(z,w))}
\Big(
1\vee\log\Big(\dfrac{2(1\vee\diam(D))}{\vert w-z\vert}\Big)^{j+2l}
\Big)
\LK_{\eta}(\vert w-z\vert)
\,d^{2}z\,d^{2}w
\\
\leq 
e^{-((2\pi\eta)\wedge 1)(b-v-2\lambda - 2\log(2)/\pi)}
\\
\times
(c_{\eta,j,l}\vert D\vert \diam(D)^{2\eta}
(1\vee(\pi b + 2\log(2)-\log(1\wedge\diam(D)))^{j+2l}) 
+ I_{\eta,j,l}(D) 
),
\end{multline*}
where $\vert D\vert$ is the Lebesgue measure (area) of $D$,
and $c_{\eta,j,l}>0$ is a constant depending on $\eta, j, l$, and
\begin{displaymath}
I_{\eta,j,l}(D) = 
\int_{D^{2}}
\Big(
1\vee\log\Big(\dfrac{2(1\vee\diam(D))}{\vert w-z\vert}\Big)^{j+2l}
\Big)
\LK_{\eta}(\vert w-z\vert)
\,d^{2}z\,d^{2}w.
\end{displaymath}
\end{lemma}

\begin{proof}
The proof is similar to that of Lemma \ref{Lem better bound}.
As there, we take
\begin{displaymath}
r(b) = e^{-\pi(b-v-2\lambda - 2\log(2)/\pi)}\, \diam(D),
\end{displaymath}
and separate the integral over $D^{2}$ into two parts:
$\vert w-z\vert\geq r(b)$ and $\vert w-z\vert < r(b)$.
We omit further details.
\end{proof}

\begin{cor}
\label{Cor bref psi k j l}
Let $\eta\in (0,1)$, $k\geq 0$, $j\geq 1$ and $l\geq 0$, with $j+2l\leq 2k+1$.
Then there are constants $c,C>0$
($c=c(\eta)$,
$C=C(D,\eta,k,j,l,v)$),
such that for every $b\in2\lambda\N$ with $b\geq v + 2\lambda$,
\begin{displaymath}
\E\Big[\big\Vert
\ell_{0,b}^{2k+1-j-2l}V_{A_{0,b}}^{l}
\,:\Phi_{D\setminus A_{0,b}}^{j} :
\ind_{\{z\in A_{0,b}\text{~or~}\ell_{0,b}(z)=b\}}
\big\Vert_{H^{-\eta}(\C)}^{2}
\Big]
\leq C e^{-c b}.
\end{displaymath}
\end{cor}

\begin{cor}
\label{Cor bref psi k}
Let $\eta\in (0,1)$ and $k\geq 0$.
Then there are constants $c,C>0$
($c=c(\eta)$,
$C=C(D,\eta,k,v)$),
such that for every $b\in2\lambda\N$ with $b\geq v + 2\lambda$,
\begin{displaymath}
\E\Big[\big\Vert\psi_{2k+1,A} - \E[\psi_{2k+1, A} \vert A_{0,b},\ell_{0,b}]
\big\Vert_{H^{-\eta}(\C)}^{2}
\Big]^{1/2}
\leq C e^{-c b}.
\end{displaymath}
\end{cor}

\medskip

Next we deal with
\begin{displaymath}
\tilde{\psi}_{2k+1,A} - \E[\tilde{\psi}_{2k+1, A} \vert A_{0,b},\ell_{0,b}]
\qquad
\text{and}
\qquad
f_{0}\tilde{\psi}_{2k+1,A} - \E[f_{0}\tilde{\psi}_{2k+1, A} \vert A_{0,b},\ell_{0,b}]
,
\end{displaymath}
where we recall that $\tilde{\psi}_{2k+1,A}$ is given by \eqref{Eq def tilde psi},
and $f_{0}$ is a deterministic cutoff function as in \eqref{Eq A E tilde N}.
We will use that $\tilde{\psi}_{2k+1,A}$ is a linear combination of
$(\log \CR(z,D))^{q}\psi_{2(k-q)+1,A}$,
with constant coefficients.
We would like to emphasize that the function
$z\mapsto (\log \CR(z,D))^{q}$ is deterministic.

Similarly to Lemma \ref{Lem Delta psi 1},
the following holds.

\begin{lemma}
\label{Lem Delta psi CR 1}
Let $\eta>0$, $k\geq 0$, $q\in\{1,\dots,k-1\}$, and $b\geq v\vee(2\lambda)$.
Then
\begin{multline*}
\E\Big[\big\Vert(\log \CR(z,D))^{q}\psi_{2(k-q)+1,A} - \E[(\log \CR(z,D))^{q}\psi_{2(k-q)+1,A} \vert A_{0,b},\ell_{0,b}]
\big\Vert_{H^{-\eta}(\C)}^{2}
\Big]^{1/2}
\leq
\\
\sum_{\substack{1\leq j\leq 2(k-q)+1 \\ 0\leq l\leq \lfloor(2(k-q)+1-j)/2\rfloor}}
\dfrac{(2(k-q)+1)!}{2^{l}j!l!(2(k-q)+1-j-2l)!} \times
\\
\times
\E\Big[\big\Vert
\ell_{0,b}^{2(k-q)+1-j-2l}
(\log \CR(z,D))^{q}
V_{A_{0,b}}^{l}
\,:\Phi_{D\setminus A_{0,b}}^{j} :
\ind_{\{z\in A_{0,b}\text{~or~}\ell_{0,b}(z)=b\}}
\big\Vert_{H^{-\eta}(\C)}^{2}
\Big]^{1/2}.
\end{multline*}
If $f_{0}$ is a deterministic cutoff function as in \eqref{Eq A E tilde N}, then
\begin{multline*}
\E\Big[\big\Vert f_{0}(\log \CR(z,D))^{q}\psi_{2(k-q)+1,A} 
- \E[f_{0}(\log \CR(z,D))^{q}\psi_{2(k-q)+1,A} \vert A_{0,b},\ell_{0,b}]
\big\Vert_{H^{-\eta}(\C)}^{2}
\Big]^{1/2}
\leq
\\
\sum_{\substack{1\leq j\leq 2(k-q)+1 \\ 0\leq l\leq \lfloor(2(k-q)+1-j)/2\rfloor}}
\dfrac{(2(k-q)+1)!}{2^{l}j!l!(2(k-q)+1-j-2l)!} \times
\\
\times
\E\Big[\big\Vert
f_{0}
\ell_{0,b}^{2(k-q)+1-j-2l}
(\log \CR(z,D))^{q}
V_{A_{0,b}}^{l}
\,:\Phi_{D\setminus A_{0,b}}^{j} :
\ind_{\{z\in A_{0,b}\text{~or~}\ell_{0,b}(z)=b\}}
\big\Vert_{H^{-\eta}(\C)}^{2}
\Big]^{1/2}.
\end{multline*}
\end{lemma}

Similarly to Lemma \ref{Lem 2 point cor D A 0 b},
the following holds.

\begin{lemma}
\label{Lem 2 point cor CR D A 0 b}
Let $\eta>0$, $k\geq 0$, $q\in\{1,\dots,k-1\}$, $j\geq 1$, $l\geq 0$, with $j+2l\leq 2(k-q)+1$, 
and $b\geq v\vee(2\lambda)$.
Then
\begin{multline*}
\E\Big[\big\Vert
\ell_{0,b}^{2(k-q)+1-j-2l}
(\log \CR(z,D))^{q}
V_{A_{0,b}}^{l}
\,:\Phi_{D\setminus A_{0,b}}^{j} :
\ind_{\{z\in A_{0,b}\text{~or~}\ell_{0,b}(z)=b\}}
\big\Vert_{H^{-\eta}(\C)}^{2}
\Big]
=
\\
j!
b^{4(k-q)+2-2j- 4l}
\int_{D^{2}}
(\log \CR(z,D))^{q}
(\log \CR(w,D))^{q}
\LK_{\eta}(\vert w-z\vert)
\\
\times
\E
\big[
V_{A_{0,b}}(z)^{l}V_{A_{0,b}}(w)^{l}
G_{D\setminus A_{0,b}}(z,w)^{j}
\ind_{\{\ell_{0,b}(z)=\ell_{0,b}(w)=b\}}
\big]
\,d^{2}z\,d^{2}w
.
\end{multline*}
If $f_{0}$ is a deterministic cutoff function as in \eqref{Eq A E tilde N}, then
\begin{multline*}
\E\Big[\big\Vert
f_{0}
\ell_{0,b}^{2(k-q)+1-j-2l}
(\log \CR(z,D))^{q}
V_{A_{0,b}}^{l}
\,:\Phi_{D\setminus A_{0,b}}^{j} :
\ind_{\{z\in A_{0,b}\text{~or~}\ell_{0,b}(z)=b\}}
\big\Vert_{H^{-\eta}(\C)}^{2}
\Big]
=
\\
j!
b^{4(k-q)+2-2j- 4l}
\int_{D^{2}}
f_{0}(z)f_{0}(w)
(\log \CR(z,D))^{q}
(\log \CR(w,D))^{q}
\LK_{\eta}(\vert w-z\vert)
\\
\times
\E
\big[
V_{A_{0,b}}(z)^{l}V_{A_{0,b}}(w)^{l}
G_{D\setminus A_{0,b}}(z,w)^{j}
\ind_{\{\ell_{0,b}(z)=\ell_{0,b}(w)=b\}}
\big]
\,d^{2}z\,d^{2}w
.
\end{multline*}
\end{lemma}

The case with the cutoff function $f_{0}$ then follows immediately,
since $z\mapsto f_{0}(z)(\log \CR(z,D))^{q}$ is a bounded function.
We get then as in Corollary \ref{Cor bref psi k} the following bound.

\begin{cor}
\label{Cor bref psi k CR f 0}
Let $\eta\in (0,1)$ and $k\geq 0$.
Let $f_{0}$ be a deterministic cutoff function as in \eqref{Eq A E tilde N}.
Then there are constants $c,C>0$
($c=c(\eta)$,
$C=C(D,f_{0},\eta,k,v)$),
such that for every $b\in2\lambda\N$ with $b\geq v + 2\lambda$,
\begin{displaymath}
\E\Big[\big\Vert f_{0}\tilde{\psi}_{2k+1,A} - \E[f_{0}\tilde{\psi}_{2k+1, A} \vert A_{0,b},\ell_{0,b}]
\big\Vert_{H^{-\eta}(\C)}^{2}
\Big]^{1/2}
\leq C e^{-c b}.
\end{displaymath}
\end{cor}

Now let us consider the more involved case without cutoff.
Next bound is similar to Lemma \ref{Lem a s bound} (case $q=0$).

\begin{lemma}
\label{Lem a s bound q}
Let $q\geq 1$, $j\geq 1$ and $l\geq 0$.
There is a constant $C_{q,j,l}>0$ depending only on $q$, $j$ and $l$
(in particular depending neither on $D$ nor $b$),
such that a.s., for every $z,w\in D$ with $z\neq w$,
\begin{multline*}
\vert\log \CR(z,D)\vert^{q}
\vert\log \CR(w,D)\vert^{q}
V_{A_{0,b}}(z)^{l}V_{A_{0,b}}(w)^{l}
G_{D\setminus A_{0,b}}(z,w)^{j}
\ind_{\{\ell_{0,b}(z)=\ell_{0,b}(w)=b\}}
\leq
\\
C_{q,j,l} 
(1\vee \log \diam (D))^{2q}
(1\vee\diam(D)^{j/8})
\Big(
1\vee\log\Big(\dfrac{2(1\vee\diam(D))}{\vert w-z\vert}\Big)^{j+2l+2q}
\Big)
\,
\ind_{z\stackrel{D\setminus A_{0,b}}{\longleftrightarrow} w,\,\ell_{0,b}(z) = b}.
\end{multline*}
\end{lemma}

\begin{proof}
Since $G_{D\setminus A_{0,b}}(z,w)\leq G_{D}(z,w)$,
\begin{multline*}
\vert\log \CR(z,D)\vert^{q}
\vert\log \CR(w,D)\vert^{q}
V_{A_{0,b}}(z)^{l}V_{A_{0,b}}(w)^{l}
G_{D\setminus A_{0,b}}(z,w)^{j}
\ind_{\{\ell_{0,b}(z)=\ell_{0,b}(w)=b\}}
\leq
\\
\big(
\vert\log \CR(z,D)\vert^{q}
\vert\log \CR(w,D)\vert^{q}
G_{D}(z,w)^{j/2}
\big)
V_{A_{0,b}}(z)^{l}V_{A_{0,b}}(w)^{l}
G_{D\setminus A_{0,b}}(z,w)^{j/2}
\ind_{\{\ell_{0,b}(z)=\ell_{0,b}(w)=b\}}.
\end{multline*}
As in Lemma \ref{Lem a s bound} ($j/2$ does not need to be integer),
\begin{multline*}
V_{A_{0,b}}(z)^{l}V_{A_{0,b}}(w)^{l}
G_{D\setminus A_{0,b}}(z,w)^{j/2}
\ind_{\{\ell_{0,b}(z)=\ell_{0,b}(w)=b\}}
\\
\leq C (1\vee\diam(D)^{j/8})
\Big(
1\vee\log\Big(\dfrac{2(1\vee\diam(D))}{\vert w-z\vert}\Big)^{j/2+2l}
\Big)
\,
\ind_{z\stackrel{D\setminus A_{0,b}}{\longleftrightarrow} w,\,\ell_{0,b}(z) = b},
\end{multline*}
for some constant $C>0$ depending on $j$ and $l$.
Next we use the inequality \eqref{Eq est Green} (Proposition \ref{Prop Green}):
\begin{displaymath}
G_{D}(z,w) \leq C'\Big( 1\wedge \Big(
\dfrac{d(z,\partial D)\wedge d(w,\partial D)}{\vert w-z\vert}\Big)^{1/2}
\Big)
\log\Big(\dfrac{2 \diam (D)}{\vert w-z\vert}\Big),
\end{displaymath}
for some universal constant $C'>0$.
Moreover, by \eqref{Eq Koebe},
\begin{displaymath}
\vert\log \CR(z,D)\vert
\leq \vert\log d(z,\partial D)\vert + 2 \log(2).
\end{displaymath}
Thus,
\begin{multline*}
\vert\log \CR(z,D)\vert^{q}
\vert\log \CR(w,D)\vert^{q}
G_{D}(z,w)^{j/2}
\leq
\\
(C')^{j/2}
(\vert\log d(z,\partial D)\vert + 2 \log(2))^{q}
(\vert\log d(w,\partial D)\vert + 2 \log(2))^{q}
\\
\times
\Big( 1\wedge \Big(
\dfrac{d(z,\partial D)\wedge d(w,\partial D)}{\vert w-z\vert}\Big)^{j/4}
\Big)
\log\Big(\dfrac{2 \diam (D)}{\vert w-z\vert}\Big)^{j/2}
.
\end{multline*}
Further, we will distinguish two cases:
$d(z,\partial D)\wedge d(w,\partial D)\leq \vert w-z\vert^{2}$
and
$d(z,\partial D)\wedge d(w,\partial D)> \vert w-z\vert^{2}$.

In the case $d(z,\partial D)\wedge d(w,\partial D)\leq \vert w-z\vert^{2}$,
\begin{displaymath}
1\wedge \Big(
\dfrac{d(z,\partial D)\wedge d(w,\partial D)}{\vert w-z\vert}\Big)^{j/4}
\leq
1\wedge (d(z,\partial D)\wedge d(w,\partial D))^{j/8},
\end{displaymath}
and further,
\begin{multline*}
(\vert\log d(z,\partial D)\vert + 2 \log(2))^{q}
(\vert\log d(w,\partial D)\vert + 2 \log(2))^{q}
\Big( 1\wedge \Big(
\dfrac{d(z,\partial D)\wedge d(w,\partial D)}{\vert w-z\vert}\Big)^{j/4}
\Big)
\\
\leq
C^{\ast}_{q,j}(2\log(2)+0\vee \log \diam (D))^{2q}
\leq
(2\log(2) + 1)^{2q} C^{\ast}_{q,j}(1\vee \log \diam (D))^{2q}
,
\end{multline*}
where
\begin{displaymath}
C^{\ast}_{q,j} = \sup_{x\in (0,1]}\Big(1 + \dfrac{\vert \log x \vert}{2\log(2)}\Big)^{2q} x^{j/8}
< +\infty.
\end{displaymath}

In the case $d(z,\partial D)\wedge d(w,\partial D) > \vert w-z\vert^{2}$,
we have
\begin{multline*}
(\vert\log d(z,\partial D)\vert + 2 \log(2))^{q}
(\vert\log d(w,\partial D)\vert + 2 \log(2))^{q}
\\
\leq
((2\vert \log \vert w-z\vert\vert)\vee\vert \log (1\vee\diam(D))\vert + 2\log (2))^{2q}
\\
\leq
4^{q}
\Big(
1
+
\dfrac{1}{\log 2}
\Big)^{2q}
(1\vee \log \diam (D))^{2q}
\Big(
1\vee\log\Big(\dfrac{2(1\vee\diam(D))}{\vert w-z\vert}\Big)^{2q}
\Big).
\end{multline*}
This concludes the proof.
\end{proof}

Armed with the bound of Lemma \ref{Lem a s bound q},
we can proceed as in the derivation of Corollary \ref{Cor bref psi k}.
We get the following.

\begin{cor}
\label{Cor bref psi k CR}
Let $\eta\in (0,1)$ and $k\geq 0$.
Then there are constants $c,C>0$
($c=c(\eta)$,
$C=C(D,f_{0},\eta,k,v)$),
such that for every $b\in2\lambda\N$ with $b\geq v + 2\lambda$,
\begin{displaymath}
\E\Big[\big\Vert \tilde{\psi}_{2k+1,A} - \E[\tilde{\psi}_{2k+1, A} \vert A_{0,b},\ell_{0,b}]
\big\Vert_{H^{-\eta}(\C)}^{2}
\Big]^{1/2}
\leq C e^{-c b}.
\end{displaymath}
\end{cor}

\subsection{Bounding the tails of power series}
\label{Subsec bound series}

Here we are going to bound the tails of the power series in $\vert\log \varepsilon\vert^{-1}$,
\begin{displaymath}
\sum
(-1)^{k}
\dfrac{1}{2^{k} k! (k+1/2)}
\dfrac{\E[\psi_{2k+1, A} \vert A_{0,b},\ell_{0,b}]}{\big(\frac{1}{2\pi}\vert\log \varepsilon\vert\big)^{k + 1/2}}
\qquad
\text{and}
\qquad
\sum
(-1)^{k}
\dfrac{1}{2^{k} k! (k+1/2)}
\dfrac{\E[\tilde{\psi}_{2k+1, A} \vert A_{0,b},\ell_{0,b}]}{\big(\frac{1}{2\pi}\vert\log \varepsilon\vert\big)^{k + 1/2}}
,
\end{displaymath}
appearing in Proposition \ref{Prop cond series}.

Let $N\geq 0$ and $z\in D\setminus A_{0,b}$ such that $2\pi V_{A_{0,b}}<\vert\log\varepsilon\vert$.
Then
\begin{align*}
\sum_{k=N+1}^{+\infty}
(-1)^{k}&
\dfrac{1}{2^{k} k! (k+1/2)}
\dfrac{\E[\psi_{2k+1, A} \vert A_{0,b},\ell_{0,b}](z)}{\big(\frac{1}{2\pi}\vert\log \varepsilon\vert\big)^{k + 1/2}}
\\
&=
\sum_{k=N+1}^{+\infty}
(-1)^{k}
\dfrac{1}{2^{k} k! (k+1/2)}
\dfrac{Q_{2k+1}(\ell_{0,b}(z),V_{A_{0,b}}(z))}{\big(\frac{1}{2\pi}\vert\log \varepsilon\vert\big)^{k + 1/2}}
\\
&=
\ind_{\ell_{0,b}(z)=b}
\sum_{k=N+1}^{+\infty}
(-1)^{k}
\dfrac{1}{2^{k} k! (k+1/2)}
\dfrac{Q_{2k+1}(b,V_{A_{0,b}}(z))}{\big(\frac{1}{2\pi}\vert\log \varepsilon\vert\big)^{k + 1/2}}
\\
&=
\ind_{\ell_{0,b}(z)=b}
\sum_{k=N+1}^{+\infty}
\sum_{j=0}^{k}
(-1)^{k+j}
\dfrac{(2k+1)!}{2^{k+j} k! j! (2(k-j)+1)! (k+1/2)}
\dfrac{b^{2(k-j)+1}V_{A_{0,b}}(z)^{j}}{\big(\frac{1}{2\pi}\vert\log \varepsilon\vert\big)^{k + 1/2}}.
\end{align*}

Similarly, for $z\in D\setminus A_{0,b}$ such that $\CR(z,D\setminus A_{0,b})\in (\varepsilon,\varepsilon^{-1})$,
\begin{multline*}
\sum_{k=N+1}^{+\infty}
(-1)^{k}
\dfrac{1}{2^{k} k! (k+1/2)}
\dfrac{\E[\tilde{\psi}_{2k+1, A} \vert A_{0,b},\ell_{0,b}](z)}{\big(\frac{1}{2\pi}\vert\log \varepsilon\vert\big)^{k + 1/2}}
\\
=
\ind_{\ell_{0,b}(z)=b}
\sum_{k=N+1}^{+\infty}
(-1)^{k}
\dfrac{1}{2^{k} k! (k+1/2)}
\dfrac{Q_{2k+1}(b,-\frac{1}{2\pi}\log \CR(z,D\setminus A_{0,b}))}
{\big(\frac{1}{2\pi}\vert\log \varepsilon\vert\big)^{k + 1/2}}
\\
=
\ind_{\ell_{0,b}(z)=b}
\sum_{k=N+1}^{+\infty}
\sum_{j=0}^{k}
(-1)^{k}
\dfrac{(2k+1)!}{2^{k+j} k! j! (2(k-j)+1)! (k+1/2)}
\dfrac{b^{2(k-j)+1}(\frac{1}{2\pi}\log \CR(z,D\setminus A_{0,b}))^{j}}
{\big(\frac{1}{2\pi}\vert\log \varepsilon\vert\big)^{k + 1/2}}.
\end{multline*}

Denote
\begin{multline*}
F_{N}(u) = 
\sum_{k=N+1}^{+\infty}
\sum_{j=0}^{k}
\dfrac{(2k+1)!}{2^{k+j} k! j! (2(k-j)+1)! (k+1/2)}
\,u^{k-(N+1)}
\\
=\sum_{m=0}^{+\infty}
\dfrac{1}{2^{m} m! (m+1/2)}
\dfrac{u^{m -(N+1)}}{(1-u)^{m+1/2}}
-
\sum_{k=0}^{N}
\sum_{j=0}^{k}
\dfrac{(2k+1)!}{2^{k+j} k! j! (2(k-j)+1)! (k+1/2)}
\,u^{k-(N+1)}.
\end{multline*}
The radius of convergence of the above power series in $u$ is $1$.
The following is now immediate.

\begin{lemma}
\label{Lem bound series}
Fix $N\geq 0$. Let $z\in D\setminus A_{0,b}$ such that $2\pi V_{A_{0,b}}<\vert\log\varepsilon\vert$.
Then
\begin{displaymath}
\Big\vert
\sum_{k=N+1}^{+\infty}
(-1)^{k}
\dfrac{1}{2^{k} k! (k+1/2)}
\dfrac{\E[\psi_{2k+1, A} \vert A_{0,b},\ell_{0,b}](z)}{\big(\frac{1}{2\pi}\vert\log \varepsilon\vert\big)^{k + 1/2}}
\Big\vert
\leq
\dfrac{1}{\sqrt{2\pi}}
\big(u^{N+3/2}F_{N}(u)\big)
\Big(
u=
\dfrac{b^{2}\vee V_{A_{0,b}}(z)}{\frac{1}{2\pi}\vert\log\varepsilon\vert}
\Big).
\end{displaymath}
Let $z\in D\setminus A_{0,b}$ such that $\CR(z,D\setminus A_{0,b})\in (\varepsilon,\varepsilon^{-1})$.
Then
\begin{multline*}
\Big\vert
\sum_{k=N+1}^{+\infty}
(-1)^{k}
\dfrac{1}{2^{k} k! (k+1/2)}
\dfrac{\E[\tilde{\psi}_{2k+1, A} \vert A_{0,b},\ell_{0,b}](z)}{\big(\frac{1}{2\pi}\vert\log \varepsilon\vert\big)^{k + 1/2}}
\Big\vert
\\
\leq
\dfrac{1}{\sqrt{2\pi}}
\big(u^{N+3/2}F_{N}(u)\big)
\Big(
u=
\dfrac{b^{2}\vee (\frac{1}{2\pi}\vert \log \CR(z,D\setminus A_{0,b})\vert)}{\frac{1}{2\pi}\vert\log\varepsilon\vert}
\Big).
\end{multline*}
\end{lemma}

\subsection{Proof of Theorem \ref{Thm A E FPS}}
\label{Subsec proof expansion}

We are now ready to prove Theorem \ref{Thm A E FPS}.

\begin{proof}[Proof of Theorem \ref{Thm A E FPS}]
We start with \eqref{Eq A E N}.
Fix $\eta\in (0,1)$ and an integer $N\geq 0$.
Let $\varepsilon\in (0,e^{-1})$,
$\delta\in(\varepsilon,e^{-1})$ and
$b\in 2\lambda \N$ with $b\geq (2v)\vee(v + 2\lambda)$.
Recall the notation $U_{0,b,\varepsilon}$ \eqref{Eq U 0 b eps}.
We will further use the notations
\begin{displaymath}
U_{0,b,\varepsilon}^{(0)}
=
\{ z\in D\setminus A_{0,b}\,\vert \CR(z,D\setminus A_{0,b})<\varepsilon\CR(z,D)
\text{ and } \ell_{0,b}(z)=0\},
\end{displaymath}
\begin{displaymath}
U_{0,b,\delta}^{(b)}
=
\{ z\in D\setminus A_{0,b}\,\vert \CR(z,D\setminus A_{0,b})<\delta\CR(z,D)
\text{ and } \ell_{0,b}(z)=b\}.
\end{displaymath}
Proposition \ref{Prop cond series} ensures that one can bound the difference
\begin{displaymath}
\Big\Vert
\ind_{\Ns_{\varepsilon}(A)}
-
\dfrac{1}{\sqrt{2\pi}}
\sum_{k=0}^{N} (-1)^{k}
\dfrac{1}{2^{k} k! (k+1/2)}
\dfrac{\psi_{2k+1,A}}{\big(\frac{1}{2\pi}\vert\log \varepsilon\vert\big)^{k + 1/2}}
\Big\Vert_{H^{-\eta}(\C)}
\end{displaymath}
by the sum of the sum of the Sobolev norms of the following terms:
\begin{enumerate}
\item 
$\ind_{U_{0,b,\varepsilon}^{(0)}}$;
\item $\ind_{U_{0,b,\delta}^{(b)}}$;
\item $\ind_{\Ns_{\varepsilon}(A)}
-\E\big[\ind_{\Ns_{\varepsilon}(A)}\big\vert  A_{0,b},\ell_{0,b}\big]$;
\item $\dfrac{1}{\sqrt{2\pi}}
\sum_{k=0}^{N} (-1)^{k}
\dfrac{1}{2^{k} k! (k+1/2)}
\dfrac{\psi_{2k+1,A}-\E[\psi_{2k+1,A}\vert A_{0,b},\ell_{0,b}]}
{\big(\frac{1}{2\pi}\vert\log \varepsilon\vert\big)^{k + 1/2}}$;
\item $\ind_{U_{0,b,\delta}^{(b)}} \dfrac{1}{\sqrt{2\pi}} 
\sum_{k=0}^{N} (-1)^{k}
\dfrac{1}{2^{k} k! (k+1/2)}
\dfrac{\E[\psi_{2k+1,A}\vert A_{0,b},\ell_{0,b}]}
{\big(\frac{1}{2\pi}\vert\log \varepsilon\vert\big)^{k + 1/2}}$;
\item $\ind_{D\setminus (A_{0,b}\cup U_{0,b,\varepsilon}^{(0)}\cup U_{0,b,\delta}^{(b)})}
\dfrac{1}{\sqrt{2\pi}}
\sum_{k=N+1}^{+\infty} (-1)^{k} \dfrac{1}{2^{k} k! (k+1/2)}
\dfrac{\E[\psi_{2k+1,A}\vert A_{0,b},\ell_{0,b}]}
{\big(\frac{1}{2\pi}\vert\log \varepsilon\vert\big)^{k + 1/2}}$.
\end{enumerate}

For the terms (1) and (2), we will use the result of Section \ref{Subsec error terms TVS}.
Since $U_{0,b,\varepsilon}^{(0)}\subset U_{0,b,\varepsilon}$
and $U_{0,b,\delta}^{(b)}\subset U_{0,b,\delta}$,
\begin{displaymath}
\big\Vert\ind_{U_{0,b,\varepsilon}}^{(0)}\big\Vert_{H^{-\eta}(\C)}
\leq
\big\Vert\ind_{U_{0,b,\varepsilon}}\big\Vert_{H^{-\eta}(\C)},
\qquad
\big\Vert\ind_{U_{0,b,\delta}}^{(b)}\big\Vert_{H^{-\eta}(\C)}
\leq
\big\Vert\ind_{U_{0,b,\delta}}\big\Vert_{H^{-\eta}(\C)}.
\end{displaymath}
Then, by Lemma \ref{Lem area 0},
\begin{displaymath}
\E\Big[\big\Vert\ind_{U_{0,b,\varepsilon}}^{(0)}\big\Vert_{H^{-\eta}(\C)}^{2}\Big]^{1/2}
\leq \dfrac{2}{\sqrt{\pi}}\big\Vert\ind_{D}\big\Vert_{H^{-\eta}(\C)}\,
\varepsilon^{\pi/(32 b^{2})},
\end{displaymath}
\begin{displaymath}
\E\Big[\big\Vert\ind_{U_{0,b,\delta}}^{(b)}\big\Vert_{H^{-\eta}(\C)}^{2}\Big]^{1/2}
\leq \dfrac{2}{\sqrt{\pi}}\big\Vert\ind_{D}\big\Vert_{H^{-\eta}(\C)}\,
\delta^{\pi/(32 b^{2})}.
\end{displaymath}

For the term (3), we will use the result of Section \ref{Subsec diff ind cond}.
According to Corollary \ref{Cor Eq ind norm exp},
\begin{displaymath}
\E\Big[\Big\Vert
\ind_{\Ns_{\varepsilon}(A)}
-\E\big[\ind_{\Ns_{\varepsilon}(A)}\big\vert  A_{0,b},\ell_{0,b}\big]
\Big\Vert_{H^{-\eta}(\C)}^{2}\Big]^{1/2}
\leq
Ce^{-cb}
(\vert D\vert^{1/2} \diam(D)^{\eta} +
\Vert 1_{D} \Vert_{H^{-\eta}(\C)}
),
\end{displaymath}
for some constants $c,C>0$ depending on $\eta$ and the boundary condition $v$.

For the term (4), we will use the result of Section \ref{Subsec error fields}.
According to Corollary \ref{Cor bref psi k},
\begin{displaymath}
\E\Big[\Big\Vert
\dfrac{1}{\sqrt{2\pi}}
\sum_{k=0}^{N} (-1)^{k}
\dfrac{1}{2^{k} k! (k+1/2)}
\dfrac{\psi_{2k+1,A}-\E[\psi_{2k+1,A}\vert A_{0,b},\ell_{0,b}]}
{\big(\frac{1}{2\pi}\vert\log \varepsilon\vert\big)^{k + 1/2}}
\Big\Vert_{H^{-\eta}(\C)}^{2}\Big]^{1/2}
\leq
Ce^{-cb},
\end{displaymath}
where $c,C>0$ are constants depending on $D$, $\eta$, $N$ and $v$.

For the term (5), we will again rely on Section \ref{Subsec error terms TVS}.
The term (5) can be bounded by a linear combination of terms of form
\begin{equation}
\label{Eq b V eps bound}
\ind_{U_{0,b,\delta}}b^{2(k-j)+1}V_{A_{0,b}}^{j}\vert \log \varepsilon\vert^{-(k+1/2)},
\end{equation}
with the coefficients of the combination being some constants depending only on $k$ and $j$.
Since $\varepsilon\leq e^{-1}$,
$\delta\leq e^{-1}$ (so that $V_{A_{0,b}}\geq 1/(2\pi)$ on $U_{0,b,\delta}$),
and provided $b\geq 1$, 
the quantity \eqref{Eq b V eps bound} is bounded by
\begin{displaymath}
\ind_{U_{0,b,\delta}}(2\pi)^{N}b^{2N+1}V_{A_{0,b}}^{N}\vert \log \varepsilon\vert^{-1/2}.
\end{displaymath}
Then, by Lemma \ref{Lem V area 0}, we get that
\begin{displaymath}
\E\Big[\Big\Vert
\ind_{U_{0,b,\delta}^{(b)}} \dfrac{1}{\sqrt{2\pi}} 
\sum_{k=0}^{N} (-1)^{k}
\dfrac{1}{2^{k} k! (k+1/2)}
\dfrac{\E[\psi_{2k+1,A}\vert A_{0,b},\ell_{0,b}]}
{\big(\frac{1}{2\pi}\vert\log \varepsilon\vert\big)^{k + 1/2}}
\Big\Vert_{H^{-\eta}(\C)}^{2}\Big]^{1/2}
\end{displaymath}
is bounded by
\begin{displaymath}
C
\big\Vert\ind_{D}\big\Vert_{H^{-\eta}(\C)}\,
b^{3N+1}
\Big(
1\vee
\Big(
\dfrac{\pi}{32 b^{2}} \vert\log\delta \vert
\Big)
\Big)^{N-1}
\vert \log \varepsilon\vert^{-1/2}
\,
\delta^{\pi/(32 b^{2})},
\end{displaymath}
where $C>0$ is a constant depending only on $N$.
The above can be further bounded by
\begin{displaymath}
C'\big\Vert\ind_{D}\big\Vert_{H^{-\eta}(\C)}\,
b^{3N+1} \vert\log\delta \vert^{N}\,
\delta^{\pi/(32 b^{2})}.
\end{displaymath}

For the term (6), we will use the result of Section \ref{Subsec bound series}.
Accord to Lemma \ref{Lem bound series},
\begin{multline*}
\E\Big[\Big\Vert
\ind_{D\setminus (A_{0,b}\cup U_{0,b,\varepsilon}^{(0)}\cup U_{0,b,\delta}^{(b)})}
\dfrac{1}{\sqrt{2\pi}}
\sum_{k=N+1}^{+\infty} (-1)^{k}
\dfrac{1}{2^{k} k! (k+1/2)}
\dfrac{\E[\psi_{2k+1,A}\vert A_{0,b},\ell_{0,b}]}
{\big(\frac{1}{2\pi}\vert\log \varepsilon\vert\big)^{k + 1/2}}
\Big\Vert_{H^{-\eta}(\C)}^{2}\Big]^{1/2}
\\
\leq
\dfrac{1}{\sqrt{2\pi}}
\Big(\dfrac{(2\pi b^{2})\vee\vert\log\delta\vert}
{\vert\log\varepsilon\vert}\Big)^{N + 3/2}
F_{N}\Big(\dfrac{(2\pi b^{2})\vee\vert\log\delta\vert}
{\vert\log\varepsilon\vert}\Big)
\big\Vert\ind_{D}\big\Vert_{H^{-\eta}(\C)}.
\end{multline*}

By adding up all the bounds, we get that
\begin{displaymath}
\E\Big[
\Big\Vert
\ind_{\Ns_{\varepsilon}(A)}
-
\dfrac{1}{\sqrt{2\pi}}
\sum_{k=0}^{N} (-1)^{k}
\dfrac{1}{2^{k} k! (k+1/2)}
\dfrac{\psi_{2k+1,A}}{\big(\frac{1}{2\pi}\vert\log \varepsilon\vert\big)^{k + 1/2}}
\Big\Vert_{H^{-\eta}(\C)}^{2}
\Big]^{1/2}
\end{displaymath}
is smaller or equal to
\begin{equation}
\label{Eq big bound eps delta b}
C\Big(\varepsilon^{c/b^{2}} + b^{3N+1} \vert\log\delta \vert^{N}\delta^{c/b^{2}} + e^{-c b}
+
\Big(\dfrac{(2\pi b^{2})\vee\vert\log\delta\vert}
{\vert\log\varepsilon\vert}\Big)^{N + 3/2}
F_{N}\Big(\dfrac{(2\pi b^{2})\vee\vert\log\delta\vert}
{\vert\log\varepsilon\vert}\Big)
\Big),
\end{equation}
where the constants $c,C>0$ may in general depend on $D$, $\eta$, $N$, $v$,
but crucially not on $\varepsilon$, $\delta$ or $b$.

Further, we want to choose $\delta$ and $b$ depending on $\varepsilon$
($\delta = \delta(\varepsilon)$, $b=b(\varepsilon)$),
such that \eqref{Eq big bound eps delta b} is $o(\vert\log \varepsilon\vert^{-(N+1/2)})$.
For this, it is sufficient to have all the following conditions satisfied:
\begin{itemize}
\item $b \gg \log \vert \log \varepsilon \vert$;
\item $\log b \ll \log \vert \log \varepsilon \vert$;
\item $\log \vert \log \delta \vert \ll \log \vert \log \varepsilon \vert$;
\item $\vert \log \delta \vert \gg b^{2} \log \vert \log \varepsilon \vert$.
\end{itemize}
Note that some of these conditions work in the opposite directions.
Yet, there is plenty of room to satisfy them all.
For instance, one can take
\begin{equation}
\label{Eq delta b eps}
\delta(\varepsilon) = e^{-(\log \vert \log \varepsilon \vert)^{6}},
\qquad
b(\varepsilon) = 2\lambda\lfloor(\log \vert \log \varepsilon \vert)^{2}\rfloor
\end{equation}
(so that $b(\varepsilon)$ is an integer multiple of the height gap $2\lambda$).
This concludes the proof of \eqref{Eq A E N}.

Now let us deal with the asymptotic expansion of $\ind_{\widetilde{\Ns}_{\varepsilon}(A)}$,
\eqref{Eq A E tilde N} and \eqref{Eq A E tilde N bis}.
Recall the notation $\widetilde{U}_{0,b,\varepsilon}$ \eqref{Eq U 0 b eps}.
We further define
\begin{displaymath}
\widetilde{U}_{0,b,\varepsilon}^{(0)}
=
\{ z\in D\setminus A_{0,b}\,\vert \CR(z,D\setminus A_{0,b})<\varepsilon\text{ and } \ell_{0,b}(z)=0\},
\end{displaymath}
\begin{displaymath}
\widetilde{U}_{0,b,\delta}^{(b)}
=
\{ z\in D\setminus A_{0,b}\,\vert \CR(z,D\setminus A_{0,b})<\delta\text{ and } \ell_{0,b}(z)=b\}.
\end{displaymath}

Fix $\eta\in (0,1)$ and an integer $N\geq 0$.
As previously, we take
$\varepsilon\in (0,e^{-1})$,
$\delta\in(\varepsilon,e^{-1})$ and
$b\in 2\lambda \N$ with $b\geq (2v)\vee(v + 2\lambda)$.

Given the identity of Proposition \ref{Prop cond series},
we get that the difference
\begin{displaymath}
\ind_{\widetilde{\Ns}_{\varepsilon}(A)}
-
\dfrac{1}{\sqrt{2\pi}}
\sum_{k=0}^{N} (-1)^{k}
\dfrac{1}{2^{k} k! (k+1/2)}
\dfrac{\tilde{\psi}_{2k+1,A}}{\big(\frac{1}{2\pi}\vert\log \varepsilon\vert\big)^{k + 1/2}}
\end{displaymath}
can be decomposed into six terms as follows:
\begin{enumerate}
\item 
$\ind_{\widetilde{U}_{0,b,\varepsilon}^{(0)}}$;
\item $\ind_{\widetilde{U}_{0,b,\delta}^{(b)}}
\E\big[\ind_{\widetilde{\Ns}_{\varepsilon}(A)}\big\vert  A_{0,b},\ell_{0,b}\big]$;
\item $\ind_{\widetilde{\Ns}_{\varepsilon}(A)}
-\E\big[\ind_{\widetilde{\Ns}_{\varepsilon}(A)}\big\vert  A_{0,b},\ell_{0,b}\big]$;
\item $\dfrac{1}{\sqrt{2\pi}}
\sum_{k=0}^{N} (-1)^{k}
\dfrac{1}{2^{k} k! (k+1/2)}
\dfrac{\tilde{\psi}_{2k+1,A}-\E[\tilde{\psi}_{2k+1,A}\vert A_{0,b},\ell_{0,b}]}
{\big(\frac{1}{2\pi}\vert\log \varepsilon\vert\big)^{k + 1/2}}$;
\item $\ind_{\widetilde{U}_{0,b,\delta}^{(b)}} \dfrac{1}{\sqrt{2\pi}} 
\sum_{k=0}^{N} (-1)^{k}
\dfrac{1}{2^{k} k! (k+1/2)}
\dfrac{\E[\tilde{\psi}_{2k+1,A}\vert A_{0,b},\ell_{0,b}]}
{\big(\frac{1}{2\pi}\vert\log \varepsilon\vert\big)^{k + 1/2}}$;
\item $\ind_{D\setminus (A_{0,b}\cup \widetilde{U}_{0,b,\varepsilon}^{(0)}\cup \widetilde{U}_{0,b,\delta}^{(b)})}
\dfrac{1}{\sqrt{2\pi}}
\sum_{k=N+1}^{+\infty} (-1)^{k} \dfrac{1}{2^{k} k! (k+1/2)}
\dfrac{\E[\tilde{\psi}_{2k+1,A}\vert A_{0,b},\ell_{0,b}]}
{\big(\frac{1}{2\pi}\vert\log \varepsilon\vert\big)^{k + 1/2}}$.
\end{enumerate}
In the presence of a cutoff function $f_{0}$ as in \eqref{Eq A E tilde N},
every term should be multiplied by $f_{0}$.

Further, the proof of the expansion \eqref{Eq A E tilde N},
with the cutoff function $f_{0}$,
is very similar to that of \eqref{Eq A E N}.
We make the same choice \eqref{Eq delta b eps} for
$\delta(\varepsilon)$ and $b(\varepsilon)$.
We use Lemma \ref{Lem area 0} to bound the terms (1) and (2).
For the term (3) we use Corollary \eqref{Cor Eq ind norm exp}.
For (4) we use Corollary \eqref{Cor bref psi k CR f 0}.
To bound (5) we use Lemma \ref{Lem V area 0}.
For (6) we apply Lemma \ref{Lem bound series}.

For the expansion \eqref{Eq A E tilde N bis} (without cutoff)
we have first to assume that the area of a small neighborhood of 
$\partial D$ decays faster than any power of $\log$
(condition \eqref{Eq cond small boundary}).
Further, we cannot use the same $\delta(\varepsilon)$ as in \eqref{Eq delta b eps}.
Indeed, we want to use Lemma \ref{Lem small boundary means fast enough decay}
to control the terms (2) and (5).
But to get the desired decay, we need
$\vert\log \delta(\varepsilon)\vert$
to grow at least like a power of $\vert \log \varepsilon\vert$
as $\varepsilon\to 0$,
and this is not the case in \eqref{Eq delta b eps}.
But on the other hand, we are also constrained by
the bound of Lemma \ref{Lem bound series}:
we need $(\vert\log \delta(\varepsilon)\vert/\vert \log \varepsilon\vert)^{N+3/2}$
to decay faster than $\vert \log \varepsilon\vert^{-(N+1/2)}$.
Therefore, we make a choice of $\delta(\varepsilon)$ also depending on $N$, the order of the expansion,
\begin{displaymath}
\delta(\varepsilon) = \exp(-\vert \log \varepsilon\vert^{1/(N+2)}),
\end{displaymath}
and keep the same choice of $b(\varepsilon)$ \eqref{Eq delta b eps}.
Then we use Lemma \ref{Lem small boundary means fast enough decay}
to bound the terms (1), (2) and (5),
Corollary \ref{Cor Eq ind norm exp} to bound (3),
and Corollary \ref{Cor bref psi k CR} to bound (4).
\end{proof}

\section{An algebraic viewpoint on the asymptotic expansions: consistency conditions and associated equations}
\label{Sec algeb}

\subsection{Consistency of the asymptotic expansion for the FPS under change of domain}
\label{Subsec algeb}

This section will be heuristic for the most part. 
We will see that the asymptotic expansion \eqref{Eq A E N}
(Theorem \ref{Thm A E FPS}) satisfies a consistency property under
change of domain.
This consistency property is closely related to
the algebraic identity \eqref{Eq reexp w x y} (Proposition \ref{Prop reexp}),
and this also characterizes the combinatorial coefficients \eqref{Eq comb coeff}
appearing in \eqref{Eq A E N}, up to a global multiplicative constant.

Let $D$ and $\widehat{D}$ be two open bounded non-empty simply-connected domains,
such that the intersection $D\cap\widehat{D}$ is non-empty.
Let $f_{0}$ be a smooth function compactly supported in $D\cap\widehat{D}$.
Consider boundary conditions $v>0$ on $\partial D$
and $\hat{v}>0$ on $\partial \widehat{D}$.
Denote $\Phi$ the GFF on $D$ with boundary condition $v$,
and $\widehat{\Phi}$ the GFF on $\widehat{D}$ with boundary condition $\hat{v}$.
Let $A$ be the FPS of $\Phi$ of level $0$ and
$\widehat{A}$ the FPS of $\widehat{\Phi}$ of level $0$.

The laws of the restrictions $\Phi_{\vert \Supp(f_{0})}$
and $\widehat{\Phi}_{\vert \Supp(f_{0})}$ are mutually absolutely continuous,
and therefore, one can couple $\Phi$ and $\widehat{\Phi}$ on the same probability space such that
\begin{displaymath}
\PP\big(\ind_{\Supp(f_{0})}\Phi\equiv \ind_{\Supp(f_{0})}\widehat{\Phi}\big) > 0.
\end{displaymath}
Actually, one can show that a stronger coupling holds.
Let $U$ be an open subset of $D\cap\widehat{D}$ such that
$\overline{U}\subset D\cap\widehat{D}$ and $\Supp(f_{0})\subset U$.
Let $\mathbf{E}$ be the event defined by the following conditions:
\begin{itemize}
\item $\ind_{U}\Phi\equiv \ind_{U}\widehat{\Phi}$;
\item $A\cap U = \widehat{A} \cap U$;
\item $\forall z\in U\setminus A = U\setminus \widehat{A}$, 
the connected component of $z$ in $D\setminus A$ is contained in $D\cap\widehat{D}$
and it equals the connected component of $z$ in $D\setminus \widehat{A}$.
\end{itemize}

\begin{claim}
One can couple $\Phi$ and $\widehat{\Phi}$ on the same probability space such that
$\PP(\mathbf{E})>0$.
\end{claim}

To get the last two conditions, one can for instance use the representation of first passage sets as clusters of Brownian loops and excursions (see Section \ref{Subsubsec FPS}), 
with the Brownian loops contained in $D\cap\widehat{D}$ being the same in both cases.
Then one samples $0$-boundary GFF-s inside the holes of $A$ and $\widehat{A}$
which coincide on $U$ with positive probability.
We leave the details for the reader.

Now let be
\begin{displaymath}
\psi_{n} = \E\big[:\Phi^{n}:\big\vert A\big],
\qquad
\hat{\psi}_{n} = \E\big[:\widehat{\Phi}^{n}:\big\vert A\big].
\end{displaymath}
On the event $\mathbf{E}$,
$f_{0}\psi_{1}$ and $f_{0}\hat{\psi}_{1}$ coincide.
However, this is no longer true for higher powers $n\geq 2$.
This is due to the fact that the counterterms in $\psi_{n}$ and $\hat{\psi}_{n}$ (sub-leading degree)
are domain-dependent. Here is the exact relation between $\psi_{n}$ and $\hat{\psi}_{n}$.
Denote, for $z\in D\setminus A$, respectively $z\in \widehat{D}\setminus \widehat{A}$,
\begin{displaymath}
V_{A}(z) = \dfrac{1}{2\pi}\log\Big(\dfrac{\CR(z,D)}{\CR(z,D\setminus A)}\Big),
\qquad
V_{\widehat{A}}(z) = \dfrac{1}{2\pi}\log\Big(\dfrac{\CR(z,\widehat{D})}{\CR(z,\widehat{D}\setminus \widehat{A})}\Big).
\end{displaymath}
For $z\in D\cap\widehat{D}$, denote
\begin{displaymath}
V_{D,\widehat{D}}(z)
=
\dfrac{1}{2\pi}\log\Big(\dfrac{\CR(z,\widehat{D})}{\CR(z,D)}\Big).
\end{displaymath}
Note that unlike $V_{\widehat{A}}$ and $V_{A}$,
the function $V_{D,\widehat{D}}$ is deterministic.
Also, on the event $\mathbf{E}$, a.s.,
for every $z\in D\cap\widehat{D}$,
\begin{displaymath}
V_{D,\widehat{D}}(z) = V_{\widehat{A}}(z) - V_{A}(z).
\end{displaymath}
This is because on $\mathbf{E}$, a.s., 
for every $z\in D\cap\widehat{D}$,
\begin{displaymath}
\CR(z,D\setminus A) = \CR(z,\widehat{D}\setminus \widehat{A}).
\end{displaymath}

\begin{lemma}
\label{Lem psi hat psi}
On the event $\mathbf{E}$, 
for every $n\geq 1$,
\begin{displaymath}
f_{0}\hat{\psi}_{n}
=
\sum_{0\leq k\leq \lfloor n/2\rfloor}
(-1)^{k} \dfrac{n!}{2^{k} k! (n-2k)!} V_{D,\widehat{D}}^{k}f_{0}\psi_{n-2k}.
\end{displaymath}
\end{lemma}

\begin{proof}
Let $\nu_{\widehat{A}} = \hat{\psi}_{1}$ be the Minkowski content measure on $\widehat{A}$,
corresponding to the restriction of the GFF $\widehat{\Phi}$ to its FPS $\widehat{A}$.
Let $\nu_{\widehat{A},\varepsilon}$ be the convolution 
$\nu_{\widehat{A}}\ast \rho_{\varepsilon}$, 
where $(\rho_{\varepsilon})_{\varepsilon>0}$ is a family of mollificators as in
Section \ref{Subsec Wick}.
Let $V_{\widehat{A},\varepsilon}$ be
\begin{displaymath}
V_{\widehat{A},\varepsilon}(z) = g_{\widehat{D},\varepsilon,\varepsilon}(z,z) - 
g_{\widehat{D}\setminus \widehat{A},\varepsilon,\varepsilon}(z,z).
\end{displaymath}
According to Proposition \ref{Prop psi eps},
\begin{displaymath}
\hat{\psi}_{n} = \lim_{\varepsilon\to 0}
Q_{n}(\nu_{\widehat{A},\varepsilon},V_{\widehat{A},\varepsilon}),
\end{displaymath}
where the limit is in $L^{2}(d\PP,\sigma(\Phi,\widehat{\Phi}),H^{-\eta}(\C))$ for $\eta>0$.
Therefore,
\begin{displaymath}
\ind_{\mathbf{E}} f_{0} \hat{\psi}_{n} = \lim_{\varepsilon\to 0}
\ind_{\mathbf{E}} f_{0}
Q_{n}(\nu_{\widehat{A},\varepsilon},V_{\widehat{A},\varepsilon}),
\end{displaymath}
with the limit in $L^{2}(d\PP,\sigma(\Phi,\widehat{\Phi}),H^{-\eta}(\C))$.
On the event $\mathbf{E}$, and for $\varepsilon<d(\Supp(f_{0}),\partial U)$,
$\nu_{\widehat{A},\varepsilon}$ and $\nu_{A,\varepsilon}$ coincide on $\Supp(f_{0})$.
Therefore, 
\begin{displaymath}
\ind_{\mathbf{E}} f_{0} \hat{\psi}_{n} = \lim_{\varepsilon\to 0}
\ind_{\mathbf{E}} f_{0}
Q_{n}(\nu_{A,\varepsilon},V_{\widehat{A},\varepsilon}),
\end{displaymath}
By the change of variance formula \eqref{Eq change var},
\begin{displaymath}
f_{0}
Q_{n}(\nu_{A,\varepsilon},V_{\widehat{A},\varepsilon}) = 
\sum_{0\leq k\leq \lfloor n/2\rfloor}
(-1)^{k} \dfrac{n!}{2^{k} k! (n-2k)!} (V_{\widehat{A},\varepsilon}-V_{A,\varepsilon})^{k}
f_{0} Q_{n-2k}(\nu_{A,\varepsilon},V_{A,\varepsilon}).
\end{displaymath}
On the event $\mathbf{E}$, for every $z\in \Supp(f_{0})$,
\begin{displaymath}
V_{\widehat{A},\varepsilon}(z)-V_{A,\varepsilon}(z)
= g_{\widehat{D},\varepsilon,\varepsilon}(z,z) - g_{D,\varepsilon,\varepsilon}(z,z).
\end{displaymath}
Therefore,
\begin{displaymath}
\ind_{\mathbf{E}} f_{0} \hat{\psi}_{n} = \lim_{\varepsilon\to 0}
\ind_{\mathbf{E}}
\sum_{0\leq k\leq \lfloor n/2\rfloor}
(-1)^{k} \dfrac{n!}{2^{k} k! (n-2k)!} 
(g_{\widehat{D},\varepsilon,\varepsilon}(z,z) - g_{D,\varepsilon,\varepsilon}(z,z))^{k}
f_{0} Q_{n-2k}(\nu_{A,\varepsilon},V_{A,\varepsilon}),
\end{displaymath}
with the limit in $L^{2}(d\PP,\sigma(\Phi,\widehat{\Phi}),H^{-\eta}(\C))$.
Given that $g_{\widehat{D},\varepsilon,\varepsilon}(z,z) - g_{D,\varepsilon,\varepsilon}(z,z)$
converges to $V_{D,\widehat{D}}(z)$ uniformly on compact subsets of $D\cap\widehat{D}$,
we get that the right-hand side also converges in $L^{2}(d\PP,\sigma(\Phi,\widehat{\Phi}),H^{-\eta}(\C))$
to
\begin{displaymath}
\ind_{\mathbf{E}}
\sum_{0\leq k\leq \lfloor n/2\rfloor}
(-1)^{k} \dfrac{n!}{2^{k} k! (n-2k)!} V_{D,\widehat{D}}^{k}f_{0}\psi_{n-2k}.
\end{displaymath}
This implies the desired almost sure equality.
\end{proof}

Let $\Ns_{\varepsilon}(\widehat{A})$ be the neighborhood of $\widehat{A}$ in $\widehat{D}\setminus \widehat{A}$
given by
\begin{displaymath}
\Ns_{\varepsilon}(\widehat{A}) = 
\{ z\in \widehat{D}\setminus \widehat{A} \vert \CR(z,\widehat{D}\setminus \widehat{A})<\varepsilon \CR(z,\widehat{D})\}.
\end{displaymath}
According to Theorem \ref{Thm A E FPS},
for every $N\geq 0$,
\begin{displaymath}
f_{0}\ind_{\Ns_{\varepsilon}(\widehat{A})}
=
\dfrac{1}{\sqrt{2\pi}}
\sum_{k=0}^{N} (-1)^{k}
\dfrac{1}{2^{k} k! (k+1/2)}
\dfrac{f_{0}\hat{\psi}_{2k+1}}{\big(\frac{1}{2\pi}\vert\log \varepsilon\vert\big)^{k + 1/2}}
~+~o(\vert\log\varepsilon\vert^{-(N+1/2)}).
\end{displaymath}
By reexpressing the $\hat{\psi}_{2k+1}$ in terms of the $\psi_{2j+1}$,
we get that on the event $\mathbf{E}$,
\begin{multline}
\label{Eq exp N hat A 1}
\ind_{\mathbf{E}}f_{0}\ind_{\Ns_{\varepsilon}(\widehat{A})}
=
\\
\dfrac{\ind_{\mathbf{E}}f_{0}}{\sqrt{2\pi}}
\sum_{k=0}^{N} 
\dfrac{(-1)^{k}}{2^{k} k! (k+1/2)}
\dfrac{1}{\big(\frac{1}{2\pi}\vert\log \varepsilon\vert\big)^{k + 1/2}}
\Big(\sum_{j=0}^{k}\dfrac{(-1)^{j}(2k+1)!}{2^{j}j!(2(k-j)+1)!}V_{D,\widehat{D}}^{j}\psi_{2(k-j)+1}\Big)
\\
+~o(\vert\log\varepsilon\vert^{-(N+1/2)}).
\end{multline}

Next, let $\check{\Ns}_{\varepsilon}(A)$ be the neighborhood of $A\cap \widehat{D}$
in $(D\cap\widehat{D})\setminus A$ defined via the conformal radius in $\widehat{D}$ and not $D$:
\begin{displaymath}
\check{\Ns}_{\varepsilon}(A) = 
\{ z\in (D\cap\widehat{D})\setminus A \vert \CR(z,D\setminus A)<\varepsilon \CR(z,\widehat{D})\}.
\end{displaymath}
Denote
\begin{displaymath}
\check{\varepsilon}(\varepsilon,z) = \varepsilon \dfrac{\CR(z,\widehat{D})}{\CR(z,D)} .
\end{displaymath}
The condition $\CR(z,D\setminus A)<\varepsilon \CR(z,\widehat{D})$
is equivalent to
\begin{displaymath}
\CR(z,D\setminus A)<\check{\varepsilon}(\varepsilon,z)\CR(z,D).
\end{displaymath}
The function $z\mapsto \check{\varepsilon}(\varepsilon,z)$ is deterministic an smooth on
$D\cap\widehat{D}$.
Moreover, as $\varepsilon\to 0$,
$\vert\log\check{\varepsilon}(\varepsilon,z)\vert \sim \vert\log\varepsilon\vert$,
with the equivalent being uniform on compact subsets of $D\cap\widehat{D}$.
We claim that the asymptotic expansion \eqref{Eq A E N} for $A$ holds
if we replace $\varepsilon$ by $\check{\varepsilon}(\varepsilon,z)$.
It can be proved along the same lines, but we will not detail it here.

\begin{claim}
For every $N\geq 0$,
\begin{displaymath}
f_{0}\ind_{\check{\Ns}_{\varepsilon}(A)} =
\dfrac{1}{\sqrt{2\pi}}
\sum_{k=0}^{N} (-1)^{k}
\dfrac{1}{2^{k} k! (k+1/2)}
\dfrac{f_{0}\psi_{2k+1}}{\big(\frac{1}{2\pi}\vert\log \check{\varepsilon}(\varepsilon,z)\vert\big)^{k + 1/2}}
~+~o(\vert\log\varepsilon\vert^{-(N+1/2)}). 
\end{displaymath}
\end{claim}

Now, on the event $\mathbf{E}$, 
$\check{\Ns}_{\varepsilon}(A)\cap\Supp(f_{0}) = \Ns_{\varepsilon}(\widehat{A})\cap\Supp(f_{0})$.
Moreover, for
\begin{displaymath}
\varepsilon < \inf_{w\in \Supp(f_{0})} \dfrac{\CR(w,D)}{\CR(w,\widehat{D})},
\end{displaymath}
we have that for all $z\in \Supp(f_{0})$,
\begin{displaymath}
\dfrac{1}{2\pi}\vert \log \check{\varepsilon}(\varepsilon,z)\vert
=
\dfrac{1}{2\pi}\vert \log \varepsilon\vert - V_{D,\widehat{D}}(z).
\end{displaymath}
Therefore,
\begin{equation}
\label{Eq exp N hat A 2}
\ind_{\mathbf{E}}f_{0}\ind_{\Ns_{\varepsilon}(\widehat{A})}=
\ind_{\mathbf{E}}
\dfrac{1}{\sqrt{2\pi}}
\sum_{k=0}^{N} (-1)^{k}
\dfrac{1}{2^{k} k! (k+1/2)}
\dfrac{f_{0}\psi_{2k+1}}{\big(\frac{1}{2\pi}\vert\log \varepsilon\vert - V_{D,\widehat{D}}\big)^{k + 1/2}}
~+~o(\vert\log\varepsilon\vert^{-(N+1/2)}).
\end{equation}

So now, one the event $\mathbf{E}$, we have two expansions for
$f_{0}\ind_{\Ns_{\varepsilon}(\widehat{A})}$,
\eqref{Eq exp N hat A 1} and \eqref{Eq exp N hat A 2}.
Actually, the two expansion are the same.
Indeed, in \eqref{Eq exp N hat A 2} one can reexpand
\begin{displaymath}
\dfrac{1}{\big(\frac{1}{2\pi}\vert\log \varepsilon\vert - V_{D,\widehat{D}}\big)^{k + 1/2}}
\end{displaymath}
into powers of $\vert\log \varepsilon\vert$,
and then one gets exactly \eqref{Eq exp N hat A 1}.
This is ensured by the algebraic identity \eqref{Eq reexp w x y} (Proposition \ref{Prop reexp}).
The coincidence between \eqref{Eq exp N hat A 1} and \eqref{Eq exp N hat A 2} is a consistency
of the asymptotic expansion \eqref{Eq A E N} (Theorem \ref{Thm A E FPS})
under change of domain.
Indeed, \eqref{Eq exp N hat A 1} has been obtained by applying \eqref{Eq A E N}
to the FPS $\widehat{A}$ in $\widehat{D}$,
and \eqref{Eq exp N hat A 2} by applying \eqref{Eq A E N} to the FPS $A$ in $D$.
On the event that $A$ and $\widehat{A}$ coincide on some region of $D\cap\widehat{D}$,
the two yield the same result.

\medskip

Further, if we had \textit{a priori} assumed an asymptotic expansion of form
\begin{displaymath}
\ind_{\Ns_{\varepsilon}(A)}
=
\sum_{k=0}^{N} a_{k}
\dfrac{\psi_{2k+1,A}}{\big(\frac{1}{2\pi}\vert\log \varepsilon\vert\big)^{k + 1/2}}
~+~o(\vert\log\varepsilon\vert^{-(N+1/2)}),
\end{displaymath}
(which by the way is still a strong assumption),
then the condition of consistency under change of domain would have automatically given us that
for all $k\geq 1$,
\begin{equation}
\label{Eq a k a 0}
a_k = \dfrac{(-1)^k}{2^{k+1} k! (k+1/2)} a_{0}.
\end{equation}
Indeed, as explained above, the consistency under change of domain is equivalent to the following identity at the level of formal (i.e. not necessarily convergent) power series:
\begin{equation}
\label{Eq change domain algeb}
\sum_{k=0}^{+\infty} a_{k}
\dfrac{w^{2k+1}}{(x-y)^{k + 1/2}}
=
\sum_{k=0}^{+\infty} a_{k}
Q_{2k+1}(w,y)
\dfrac{1}{x^{k+1/2}}.
\end{equation}
Then, the following holds.

\begin{prop}
\label{Prop change domain unique}
The sequences $(a_k)_{k\geq 0}\in\R^{\N}$ satisfying the formal identity \eqref{Eq change domain algeb},
with the convention that
\begin{equation}
\label{Eq exp y x k half}
\dfrac{1}{(1-y/x)^{k+1/2}} =
\sum_{j=0}^{+\infty}\dfrac{1}{j!}\Big(\prod_{l=0}^{j-1}((k+l)+1/2)\Big)\dfrac{ y^{j}}{x^{j}},
\end{equation}
are exactly the sequences of form \eqref{Eq a k a 0},
with $a_0\in\R$ being arbitrary.
\end{prop}

\begin{proof}
Note that the equation \eqref{Eq change domain algeb}
is linear in the sequence $(a_k)_{k\geq 0}\in\R^{\N}$,
therefore the solutions form a linear subspace of $\R^{\N}$.
The statement tells that the set of solutions is actually a one-dimensional line.
The identity \eqref{Eq reexp w x y} (Proposition \ref{Prop reexp}) ensures already that the sequences of form 
\eqref{Eq a k a 0} are solutions to \eqref{Eq change domain algeb}.
Now let us see the converse.

With the expansion \eqref{Eq exp y x k half},
we see that \eqref{Eq change domain algeb} is equivalent to the following system of linear equations:
\begin{displaymath}
\dfrac{1}{j!}\Big(\prod_{l=0}^{j-1}((k-j+l)+1/2)\Big)
a_{k-j}
=(-1)^{j}\dfrac{(2k+1)!}{2^{j}j!(2(k-j)+1)!} a_{k},
\end{displaymath}
indexed by couples $(k,j)$ with $k\geq j\geq 1$. By taking $j=k$, we in particular get that
\begin{displaymath}
(-1)^{k}\dfrac{(2k+1)!}{2^{k}} a_{k}
= \Big(\prod_{l=0}^{k-1}(l+1/2)\Big) a_0,
\end{displaymath}
which is exactly \eqref{Eq a k a 0}.
\end{proof}

\subsection{A remark on Le Gall's expansion for the 2D Wiener sausage: consistency under change of mass parameter}
\label{Subsec BM algeb}

Just as the asymptotic expansion in the Gaussian case satisfies a consistency constraint
(see Section \ref{Subsec algeb} above) which is related to an algebraic identity for Hermite
polynomials (Proposition \ref{Prop reexp}),
so too there is a consistency constraint in the case of the asymptotic expansion
of the Wiener sausage (see Section \ref{Subsec Le Gall sausage})
which is related to an algebraic identity for generalized Laguerre polynomials of order $-1$,
which we will present in this section.
In \cite[page 301]{LeGallLocTime}, Le Gall briefly mentions that his expansion (Theorem \ref{Thm Le Gall})
is consistent under the change of the mass parameter $M$,
but neither details this nor attaches importance to it, 
seeing it as a mere computational fact. 
Yet, we believe that this point is actually important,
and we will develop it here.
This will provide both a stronger analogy between the Gaussian and the Brownian setting,
and another connection to the umbral composition (see Section \ref{Subsec umbral}).

Consider two mass parameters $M, \widehat{M} >0$, with $M>\widehat{M}$.
The two exponential times $\zeta_{M}$ and $\zeta_{\widehat{M}}$ are naturally coupled
together, with
\begin{displaymath}
\mathbb{P}(\zeta_{M} = \zeta_{\widehat{M}}) = \dfrac{\widehat{M}}{M},
\end{displaymath}
and conditionally on $\zeta_{\widehat{M}}>\zeta_{M}$,
$\zeta_{\widehat{M}}-\zeta_{M}$ being distributed as the unconditioned $\zeta_{\widehat{M}}$.
Take now a standard Brownian motion on $\C$,
$(B_{t})_{t\geq 0}$, independent from the couple $(\zeta_{M},\zeta_{\widehat{M}})$,
with $B_{0}=0$.

On the event $\{\zeta_{M} = \zeta_{\widehat{M}}\}$,
the occupation fields $\Theta_{\zeta_{M}}$ and $\Theta_{\zeta_{\widehat{M}}}$
are the same. 
However, the higher renormalized powers $:\Theta_{\zeta_{M}}^{n}:$
and $:\Theta_{\zeta_{\widehat{M}}}^{n}:$, $n\geq 2$, are not equal.
This is because the parameter $M$ enters the definition of $:\Theta_{\zeta_{M}}^{n}:$
\eqref{Eq renorm mult loc time} through $h_{M}(r)$.
The exact relation between the powers $:\Theta_{\zeta_{\widehat{M}}}^{n}:$
and $:\Theta_{\zeta_{M}}^{k}:$
follows from the change of normalization identity \eqref{Eq change norm Laguerre} for the
generalized Laguerre polynomials and the asymptotic expansion \eqref{Eq h M exp}
for $h_{M}(r)$.
In this way, we get the following.

\begin{prop}
\label{Prop powers loc time change norm}
With the notations above,
on the event $\{\zeta_{M} = \zeta_{\widehat{M}}\}$,
a.s., for every $n\geq 2$,
the following decomposition holds:
\begin{displaymath}
:\Theta_{\zeta_{\widehat{M}}}^{n}:~=
\sum_{k=1}^{n}(-1)^{n-k} \dfrac{n! (n-1)!}{(n-k)! k! (k-1)! }
(\cst(\widehat{M}) - \cst(M))^{n-k}
:\Theta_{\zeta_{M}}^{k}: 
\, .
\end{displaymath}
\end{prop}

On the event $\{\zeta_{M} = \zeta_{\widehat{M}}\}$, the Wiener sausages 
$S_\varepsilon(\zeta_M)$ and $S_{\varepsilon}(\zeta_{\widehat{M}})$
are equal.
For $S_\varepsilon(\zeta_M)$ one has the expansion
\begin{displaymath}
\sum_{n=1}^{N} (-1)^{n-1} \dfrac{1}{n!}\dfrac{:\Theta_{\zeta_M}^{n}:}
{\big(\frac{1}{\pi}\vert \log \varepsilon\vert + \cst (M)\big)^{n}}
\, + o(\vert \log\varepsilon\vert^{-N}).
\end{displaymath}
For $S_{\varepsilon}(\zeta_{\widehat{M}})$ one has the expansion
\begin{multline*}
\sum_{n=1}^{N} (-1)^{n-1} \dfrac{1}{n!}\dfrac{1}
{\big(\frac{1}{\pi}\vert \log \varepsilon\vert + \cst (\widehat{M})\big)^{n}}
\Big( \sum_{k=1}^{n}(-1)^{n-k} \dfrac{n! (n-1)!}{(n-k)! k! (k-1)! }
(\cst(\widehat{M}) - \cst(M))^{n-k}
:\Theta_{\zeta_{M}}^{k}: \Big)
\\ + o(\vert \log\varepsilon\vert^{-N}).
\end{multline*}
The two expansions should be the same. 
They are indeed, and this follows from the following identity for
generalized Laguerre polynomials, which is an analogue of
Proposition \ref{Prop reexp} for Hermite polynomials,
but with different exponents and coefficients.

\begin{prop}
\label{Prop reexp Laguerre}
For every $w\in\C$, $x>0$ and $y\in\R$ such that $\vert y\vert <x$,
the following holds:
\begin{equation}
\label{Eq finite sum Laguerre}
\sum_{n=1}^{+\infty}
\dfrac{1}{n!}
\Big(\sum_{k=1}^{n}\dfrac{n! (n-1)!}{(n-k)! k! (k-1)! }
\,\vert y\vert^{n-k}\,\vert w\vert^{k}
\Big)
\dfrac{1}{x^{n}}
<+\infty,
\end{equation}
and
\begin{eqnarray}
\label{Eq Laguerre reexp}
\sum_{n=1}^{+\infty} (-1)^{n-1}
\dfrac{1}{n!}
\dfrac{w^{n}}{(x-y)^{n}}
&=&
\sum_{n=1}^{+\infty} (-1)^{n-1}
\dfrac{1}{n!}
\Big(\sum_{k=1}^{n}(-1)^{n-k}\dfrac{n! (n-1)!}{(n-k)! k! (k-1)! }
\,y^{n-k}\,w^{n}
\Big)
\dfrac{1}{x^{n}}
\\ 
\nonumber
&=&
\sum_{n=1}^{+\infty} (-1)^{n-1}
\dfrac{1}{n!}
\Lambda_{n}(w,y)
\dfrac{1}{x^{n}}.
\end{eqnarray}
\end{prop}

\begin{proof}
Let us deal with \eqref{Eq finite sum Laguerre}.
One can write it as
\begin{displaymath}
\sum_{n=1}^{+\infty}
\sum_{k=1}^{n}\dfrac{(n-1)!}{(n-k)! k! (k-1)! }
\,\vert y\vert^{n-k}\,\vert w\vert^{k}
\dfrac{1}{x^{n}}
=
\sum_{k=1}^{+\infty}
\sum_{m=0}^{+\infty}
\dfrac{(m+k-1)!}{m! k! (k-1)!}
\vert y\vert^{m}\,\vert w\vert^{k}
\dfrac{1}{x^{m+k}}
,
\end{displaymath}
where we performed the change of variables $m=n-k$.
Further, we use that
\begin{displaymath}
\sum_{m=0}^{+\infty}
\dfrac{(m+k-1)!}{m!(k-1)!}
\,\dfrac{\vert y\vert^{m}}{x^{m}}
=
\dfrac{1}{(1-\vert y\vert/x)^{k}}.
\end{displaymath}
Therefore, \eqref{Eq finite sum Laguerre} equals
\begin{displaymath}
\sum_{k=1}^{+\infty}\dfrac{1}{k!}\dfrac{\vert w\vert^{k}}{(x-\vert y\vert)^{k}}
=
\exp\Big(\dfrac{\vert w\vert}{x-\vert y\vert}\Big) -1 <+\infty.
\end{displaymath}
The identity \eqref{Eq Laguerre reexp} follows from a similar computation.
\end{proof}

A uniqueness property similar to Proposition \ref{Prop change domain unique}
in the Gaussian case also holds.

\begin{prop}
\label{Prop reexp Laguerre unique}
The sequences $(a_n)_{n\geq 1}\in\R^{\N\setminus\{0\}}$ satisfying the formal identity
\begin{equation}
\label{Eq reexp Leguerre a n}
\sum_{n=1}^{+\infty} a_{n}
\dfrac{w^{n}}{(x-y)^{n}}
=
\sum_{n=1}^{+\infty} a_{n}
\Lambda_{n}(w,y)
\dfrac{1}{x^{n}},
\end{equation}
with the convention that
\begin{displaymath}
\dfrac{1}{(1-y/x)^{n}} =
\sum_{m=0}^{+\infty}
\dfrac{(m+n-1)!}{m!(n-1)!}
\,\dfrac{y^{m}}{x^{m}}
\end{displaymath}
are exactly the sequences of form
\begin{equation}
\label{Eq a n a 1 Laguerre}
a_{n} = (-1)^{n-1}
\dfrac{1}{n!} a_{1},
\end{equation}
with $a_1\in\R$ being arbitrary.
\end{prop}

\begin{proof}
Indeed, the equation \eqref{Eq reexp Leguerre a n} is equivalent to the system
\begin{displaymath}
 \dfrac{(n-1)!}{(n-k)!(k-1)!}
 a_{k} = (-1)^{n-k}\dfrac{n! (n-1)!}{(n-k)! k! (k-1)!}a_{n},
\end{displaymath}
indexed by the couples $(n,k)$ with $n\geq k\geq 1$.
This is the same as
\begin{displaymath}
a_{n} = (-1)^{n-k} \dfrac{k!}{n!}a_{k},
\end{displaymath}
and by taking $k=1$, we get \eqref{Eq a n a 1 Laguerre}.
\end{proof}

\section{A description of the fields $\psi_{n,A}$ from the germs of the multiplicative chaos}
\label{Sec psi GMC}

\subsection{The description and its proof}
\label{Subsec psi GMC 1}

As in Sections \ref{Sec Wick FPS} and \ref{Sec A E},
let $D\subset \C$ be an open, non-empty, bounded, connected and simply connected domain.
Let be a constant $v>0$ and $\Phi$ a GFF on $D$ with constant
boundary condition $v$ on $\partial D$.
Let $A$ be the first passage set (FPS) of $\Phi$
from level $v$ to level $0$.
Recall the fields $\psi_{n,A}$ introduced in Section \ref{Subsec cond Wick FPS}.
There is a description of the $\psi_{n,A}$-s that come from series expansion of the
Gaussian multiplicative chaos (GMC) at $\gamma=0$ \eqref{Eq GMC Wick gen}.
It holds for both odd and even powers.

\begin{thm}
\label{Thm psi GMC}
Let $n\geq 1$.
Then the field $\psi_{n,A}$ is the limit, as $\gamma\to 0^{+}$,
of the following continuous function supported of $D\setminus A$:
\begin{equation}
\label{Eq psi n as lim func}
\ind_{D\setminus A}
\He_{n}(\gamma V_{A}^{1/2})V_{A}^{n/2} \exp\Big(-\dfrac{\gamma^{2}}{2} V_{A}\Big),
\end{equation}
where $\He_{n}$ is the $n$-th probabilistic Hermite polynomial \eqref{Eq Herm 3},
and $V_{A}$ is the function given by \eqref{Eq V A CR}.
The convergence holds both in
$L^{2}(d\PP,\sigma(\Phi),H^{-\eta}(\C))$ 
and almost surely in $H^{-\eta}(\C)$ (for the Sobolev norm) for every $\eta>0$.
\end{thm}

\begin{proof}
We will follow a similar argument used to derive Theorem \ref{Thm Mink ALS1} in \cite{ALS1}.
Fix $\eta>0$.
Let $\gamma\in (-2\sqrt{(1\wedge\eta)\pi},2\sqrt{(1\wedge\eta)\pi})$,
so that
\begin{displaymath}
\ind_{D}: e^{-\gamma \Phi} :
~=\ind_{D}\sum_{n\geq 0}(-1)^{n}\dfrac{\gamma^{n}}{n!}:\Phi^{n}:\,
~=\ind_{D}+\sum_{n\geq 1}(-1)^{n}\dfrac{\gamma^{n}}{n!}:\Phi^{n}:,
\end{displaymath}
with convergence in $L^{2}(d\PP,\sigma(\Phi),H^{-\eta}(\C))$.
This is just the identity \eqref{Eq GMC Wick gen}
with $-\gamma$ instead of $\gamma$.
By taking the conditional expectation w.r.t. $A$, we get
\begin{displaymath}
\E\big[\ind_{D}: e^{-\gamma \Phi} :\,\vert A\big]
=\ind_{D}+\sum_{n\geq 1}(-1)^{n}\dfrac{\gamma^{n}}{n!}\E\big[\,:\Phi^{n}:\,\vert A\big]
=\ind_{D}+\sum_{n\geq 1}(-1)^{n}\dfrac{\gamma^{n}}{n!}\psi_{n,A},
\end{displaymath}
with convergence in $L^{2}(d\PP,\sigma(\Phi),H^{-\eta}(\C))$.
Moreover, for our range of $\gamma$,
\begin{displaymath}
\sum_{n\geq 1}\dfrac{\vert\gamma\vert^{n}}{n!}\E\big[\Vert\psi_{n,A}\Vert_{H^{-\eta}(\C)}^{2}\big]^{1/2}
\leq
\sum_{n\geq 1}\dfrac{\vert\gamma\vert^{n}}{n!}\E\big[\Vert :\Phi^{n}:\Vert_{H^{-\eta}(\C)}^{2}\big]^{1/2}
< +\infty.
\end{displaymath}
So the convergence of the above power series in $\gamma$ also holds almost surely for the $H^{-\eta}(\C)$ norm.
So we get that for every $n\geq 1$,
\begin{equation}
\label{Eq psi n deriv}
\psi_{n,A}
=
\lim_{\gamma \to 0} (-1)^{n}\dfrac{d^{n}}{d\gamma^{n}} \E\big[\ind_{D}: e^{-\gamma \Phi} :\,\vert A\big],
\end{equation}
where the convergence holds both in $L^{2}(d\PP,\sigma(\Phi),H^{-\eta}(\C))$
and a.s. for the $H^{-\eta}(\C)$ norm.

Now recall the decomposition \eqref{Eq decomp Phi}:
\begin{displaymath}
\Phi = \nu_{A} + \Phi_{D\setminus A},
\end{displaymath}
where $\nu_{A}$ is a positive measure supported on A,
measurable w.r.t. $A$
(Minkowski content, Theorem \ref{Thm Mink ALS1})
and conditionally on $A$,
$\Phi_{D\setminus A}$ is distributed as a GFF on $D\setminus A$
with $0$ boundary condition.
Aru, Powell and Sep\'ulveda showed in \cite[Proposition 4.1]{APS} that for $\gamma\in (0,2\sqrt{2\pi})$,
\begin{equation}
\label{Eq GMC A - gamma}
\ind_{D}: e^{-\gamma \Phi} :\, =\,
\ind_{D\setminus A}
\exp\Big(-\dfrac{\gamma^{2}}{2} V_{A}\Big): e^{-\gamma \Phi_{D\setminus A}} :\, .
\end{equation}
The factor
\begin{displaymath}
\exp\Big(-\dfrac{\gamma^{2}}{2} V_{A}\Big) = \Big(\dfrac{\CR(z,D)}{\CR(z,D\setminus A)}\Big)^{-\gamma^{2}/(4\pi)}
\end{displaymath}
accounts for the change of domain: $: e^{-\gamma \Phi} :$ is renormalized with $G_{D}$,
whereas $: e^{-\gamma \Phi_{D\setminus A}} :$ is renormalized with $G_{D\setminus A}$.
In general, one can show that
\begin{displaymath}
\ind_{D}: e^{-\gamma \Phi} :\, =\,
\ind_{D\setminus A}
\exp\Big(-\dfrac{\gamma^{2}}{2} V_{A}\Big): e^{-\gamma \Phi_{D\setminus A}} :\,
+
\lim_{\varepsilon\to 0} e^{-\gamma\nu_{A,\varepsilon}}
\exp\Big(-\dfrac{\gamma^{2}}{2} V_{A,\varepsilon}\Big).
\end{displaymath}
So the identity \eqref{Eq GMC A - gamma} comes from the fact that for $\gamma>0$,
\begin{displaymath}
e^{-\gamma\nu_{A,\varepsilon}(z)}
\exp\Big(-\dfrac{\gamma^{2}}{2} V_{A,\varepsilon}(z)\Big)
\leq 
\ind_{d(z,A)<\varepsilon}
\exp\Big(-\dfrac{\gamma^{2}}{2} V_{A,\varepsilon}(z)\Big),
\end{displaymath}
and it converges to $0$ as $\varepsilon\to 0$.
(Note that \eqref{Eq GMC A - gamma} does not hold for $\gamma < 0$.)
Therefore, for $\gamma\in (0,2\sqrt{2\pi})$,
\begin{displaymath}
\E\big[\ind_{D}: e^{-\gamma \Phi} :\,\vert A\big]
=
\ind_{D\setminus A}
\exp\Big(-\dfrac{\gamma^{2}}{2} V_{A}\Big)
\E\big[\,: e^{-\gamma \Phi_{D\setminus A}} :\,\vert A\big]
= 
\ind_{D\setminus A}
\exp\Big(-\dfrac{\gamma^{2}}{2} V_{A}\Big).
\end{displaymath}
By combining with \eqref{Eq psi n deriv},
we get that
\begin{displaymath}
\psi_{n,A}
=
\lim_{\gamma \to 0^{+}} (-1)^{n}\ind_{D\setminus A}
\dfrac{d^{n}}{d\gamma^{n}} \exp\Big(-\dfrac{\gamma^{2}}{2} V_{A}\Big).
\end{displaymath}
By combining with the identity \eqref{Eq Herm 1} defining the Hermite polynomials
$\He_{n}$, we get \eqref{Eq psi n as lim func}.
\end{proof}

\begin{rem}
\label{Rem concentration}
Above Theorem \ref{Thm psi GMC} gives an expression of the fields $\psi_{n,A}$ as 
limits of functions in $\LC^{\infty}(\C)$
which are zero outside $D\setminus A$.
If $n$ is odd, 
then \eqref{Eq psi n as lim func}
converges to $0$ uniformly on compact subsets of 
$D\setminus A$, because then $\He_{n}(0)=0$.
However, the convergence to $0$ does not hold in the Sobolev spaces $H^{-\eta}(\C)$:
the function \eqref{Eq psi n as lim func} concentrates near $A$,
and has a non-trivial functional limit supported on $A$,
that is $\psi_{n,A}$.
If $n=2k$ is even, then 
\begin{displaymath}
\He_{n}(0) = (-1)^{k}\dfrac{(2k)!}{2^{k}k!}.
\end{displaymath}
In this way, the function \eqref{Eq psi n as lim func}
converges, uniformly on compact subsets of $D\setminus A$, 
towards the function
\begin{displaymath}
(-1)^{k}\dfrac{(2k)!}{2^{k}k!}V_{A}^{k}.
\end{displaymath}
The limit in $H^{-\eta}(\C)$ is however different,
$\psi_{2k}$.
All this is consistent with Corollary \ref{Cor support psi}.
\end{rem}

\begin{rem}
\label{Rem non holomorph}
We would like to point out an important subtlety.
The identity \eqref{Eq GMC A - gamma} holds for $\gamma\in (0,2\sqrt{2\pi})$
but not for $\gamma<0$.
But shouldn't two holomorphic functions in $\gamma$ coinciding on $(0,2\sqrt{2\pi})$
also coincide on a neighborhood of $0$
(principle of isolated zeroes)?
The only option is that the right-hand side of \eqref{Eq GMC A - gamma} is does not depend holomorphically on $\gamma$ is a neighborhood of $0$.
Formally, one could expand
\begin{displaymath}
\ind_{D\setminus A}
\exp\Big(-\dfrac{\gamma^{2}}{2} V_{A}\Big): e^{-\gamma \Phi_{D\setminus A}} :
\end{displaymath}
into
\begin{displaymath}
\ind_{D\setminus A} + \sum_{n\geq 1}(-1)^{n}\frac{\gamma^{n}}{n!}
\Big(\sum_{0\leq k\leq \lfloor n/2\rfloor}(-1)^{k}
\dfrac{n!}{2^{k} k! (n-2k)!} V_{A}^{k} :\Phi_{D\setminus A}^{n-2k}:
\Big).
\end{displaymath}
However, already the term in $\gamma^{2}$,
$(:\Phi_{D\setminus A}^{2}: - V_{A})/2$
is problematic.
The field $:\Phi_{D\setminus A}^{2}:$
is a well defined element of $L^{2}(d\PP,\sigma(\Phi),H^{-\eta}(\C))$,
but
\begin{displaymath}
\int_{D\setminus A} V_{A} = +\infty~\text{a.s. },
\end{displaymath}
as explained in Remark \ref{Rem non sep}.
\end{rem}

\subsection{Relation for odd powers to the asymptotic expansion}
\label{Subsec psi GMC multiscale}

For odd integers $n$, the expression \eqref{Eq psi n as lim func} for $\psi_{n,A}$
and the asymptotic expansion \eqref{Eq A E N} are closely related.
The function \eqref{Eq psi n as lim func} is a linear combination of terms of form
\begin{displaymath}
\gamma^{n-2k}\ind_{D\setminus A}
V_{A}^{n-k} \exp\Big(-\dfrac{\gamma^{2}}{2} V_{A}\Big),
\end{displaymath}
with $k\in \{0,\dots,\lfloor n/2\rfloor\}$.

Recall the notation $\Ns_{\varepsilon}(A)$ \eqref{Eq not N eps A}.

\begin{lemma}
\label{Lem beta gamma int}
Let $\beta, \gamma >0$.
Then
\begin{displaymath}
\ind_{D\setminus A}
V_{A}^{\beta} \exp\Big(-\dfrac{\gamma^{2}}{2} V_{A}\Big)
=\dfrac{2^{\beta}}{\gamma^{2\beta}}
\int_{0}^{+\infty}
\ind_{\Ns_{e^{-4\pi y/\gamma^{2}}}(A)}
\Big(\dfrac{\beta}{y}-1\Big)
y^{\beta} e^{-y}\, dy.
\end{displaymath}
\end{lemma}

\begin{proof}
For every $x\geq 0$, we can write
\begin{eqnarray*}
x^{\beta}\exp\Big(-\dfrac{\gamma^{2}}{2}x\Big)
&=&
\int_{0}^{x}
\Big(\dfrac{\beta}{u}-\dfrac{\gamma^{2}}{2}\Big)
u^{\beta}\exp\Big(-\dfrac{\gamma^{2}}{2}u\Big)\, du
\\
&=&
\int_{0}^{+\infty}
\ind_{x>u}
\Big(\dfrac{\beta}{u}-\dfrac{\gamma^{2}}{2}\Big)
u^{\beta}\exp\Big(-\dfrac{\gamma^{2}}{2}u\Big)\, du .
\end{eqnarray*}
Therefore,
\begin{displaymath}
\ind_{D\setminus A}
V_{A}^{\beta} \exp\Big(-\dfrac{\gamma^{2}}{2} V_{A}\Big)
=
\int_{0}^{+\infty}
\ind_{V_{A}>u}
\Big(\dfrac{\beta}{u}-\dfrac{\gamma^{2}}{2}\Big)
u^{\beta}\exp\Big(-\dfrac{\gamma^{2}}{2}u\Big)\, du .
\end{displaymath}
By definition, the subset
$\{z\in D\setminus A \vert V_{A}(z)>u\}$
is $\Ns_{e^{-2\pi u}}(A)$.
Further, we perform the change of variable
$y = 2u/\gamma^{2}$
to get the desired expression.
\end{proof}

\begin{prop}
\label{Prop exp beta gamma}
Let $N\geq 0$ and $\beta>N + 1/2$.
Then, for every $\eta>0$,
as $\gamma\to 0^{+}$,
the following asymptotic expansion holds in $L^{2}(d\PP,\sigma(\Phi),H^{-\eta}(\C))$:
\begin{displaymath}
\ind_{D\setminus A}
V_{A}^{\beta} \exp\Big(-\dfrac{\gamma^{2}}{2} V_{A}\Big)
= 
\dfrac{1}{\sqrt{2\pi}}
\sum_{j=0}^{N} (-1)^{j}
\dfrac{\Gamma(\beta -j-1/2) \gamma^{2j+1 - 2\beta}}{2^{2j+1/2 - \beta} j!}
\,\psi_{2j+1,A}
+
o(\gamma^{2N+1 - 2\beta}).
\end{displaymath}
\end{prop}

\begin{proof}
Let $y>0$. 
By \eqref{Eq A E N},
\begin{equation}
\label{Eq exp fixed y}
\ind_{\Ns_{e^{-4\pi y/\gamma^{2}}}(A)}
= 
\dfrac{1}{\sqrt{2\pi}}
\sum_{j=0}^{N} (-1)^{j}
\dfrac{\gamma^{2j+1}}{2^{2j+1/2} j! (j+1/2)}
\dfrac{\psi_{2j+1,A}}{y^{j + 1/2}}
+
o(\gamma^{2N+1}),
\end{equation}
where $o(\gamma^{2N+1})$ is in the $L^{2}(d\PP,\sigma(\Phi),H^{-\eta}(\C))$ sense.
The idea now is to integrate the above expansion for $y\in (0,+\infty)$
against
\begin{displaymath}
\Big(\dfrac{\beta}{y}-1\Big)
y^{\beta} e^{-y}\, dy,
\end{displaymath}
as in Lemma \ref{Lem beta gamma int}.
The condition $\beta>N + 1/2$ ensures the integrability of all the terms of the expansion.
To be more precise, we will split the integral into two parts,
from $0$ to $\gamma^{\theta}$ and from $\gamma^{\theta}$ to $+\infty$,
where
\begin{displaymath}
\theta = 1 + \dfrac{2N+1}{2\beta}.
\end{displaymath}
Since $\theta < 2$, $e^{-4\pi \gamma^{\theta}/\gamma^{2}}$ converges to $0$,
and therefore one can interchange the expansion \eqref{Eq exp fixed y}
and the integral on $(\gamma^{\theta},+\infty)$.
For the first part, given that the Sobolev norms 
$\Vert\ind_{\Ns_{e^{-4\pi y/\gamma^{2}}}(A)}\Vert_{H^{-\eta}(\C)}$
are uniformly and deterministically bounded by
$\Vert\ind_{D}\Vert_{H^{-\eta}(\C)}$,
\begin{displaymath}
\int_{0}^{\gamma^{\theta}}\ind_{\Ns_{e^{-4\pi y/\gamma^{2}}}(A)}
\Big(\dfrac{\beta}{y}-1\Big)
y^{\beta} e^{-y}\, dy
= O(\gamma^{\theta \beta}) = o(\gamma^{2N+1}).
\end{displaymath}
So we get that
\begin{multline*}
\ind_{D\setminus A}
V_{A}^{\beta} \exp\Big(-\dfrac{\gamma^{2}}{2} V_{A}\Big)
=
\\
\dfrac{1}{\sqrt{2\pi}}
\sum_{j=0}^{N} (-1)^{j}
\dfrac{\gamma^{2j+1 - 2\beta}}{2^{2j+1/2 - \beta} j! (j+1/2)}
\Big(\int_{0}^{+\infty}
\Big(\dfrac{\beta}{y}-1\Big)
y^{\beta - (j+1/2)} e^{-y}\, dy
\Big)
\psi_{2j+1,A}
+
o(\gamma^{2N+1 - 2\beta}) .
\end{multline*}
Further,
\begin{eqnarray*}
\int_{0}^{+\infty}
\Big(\dfrac{\beta}{y}-1\Big)
y^{\beta - (j+1/2)} e^{-y}\, dy
&=& \beta \Gamma(\beta -j-1/2) - \Gamma(\beta -j+1/2) 
\\
&=& \Big(j+\dfrac{1}{2}\Big)\Gamma(\beta -j-1/2).
\end{eqnarray*}
So we get the desired expansion.
\end{proof}

\begin{cor}
\label{Cor exp V A n k}
Let $n\geq 1$ be an odd integer and let $k\in \{0,\dots,\lfloor n/2\rfloor\}$.
Then, for every $\eta>0$,
as $\gamma\to 0^{+}$,
the following asymptotic expansion holds in $L^{2}(d\PP,\sigma(\Phi),H^{-\eta}(\C))$:
\begin{displaymath}
\gamma^{n-2k}\ind_{D\setminus A}
V_{A}^{n-k} \exp\Big(-\dfrac{\gamma^{2}}{2} V_{A}\Big)
= 
\dfrac{1}{2^{n-k-1}}
\sum_{j=0}^{\lfloor n/2\rfloor} (-1)^{j}
\dfrac{(2(n-1-k-j))!}{(n-1-k-j)! j!}\,\gamma^{2j+1 - n}
\,\psi_{2j+1,A}
+
o(1).
\end{displaymath}
\end{cor}

\begin{proof}
We apply Proposition \ref{Prop exp beta gamma}
with $N= \lfloor n/2\rfloor$
and $\beta = n-k\geq \lfloor n/2\rfloor+1>\lfloor n/2\rfloor +1/2$.
We use that
\begin{displaymath}
\Gamma(n-k -j-1/2) = 
\dfrac{(2(n-1-k -j))!\sqrt{\pi}}{2^{2(n-1-k -j)}(n-1-k -j)!}.
\qedhere
\end{displaymath}
\end{proof}

Corollary \ref{Cor exp V A n k} implies that for $n$ odd, the function
\begin{displaymath}
\ind_{D\setminus A}
\He_{n}(\gamma V_{A}^{1/2})V_{A}^{n/2} \exp\Big(-\dfrac{\gamma^{2}}{2} V_{A}\Big)
\end{displaymath}
appearing in \eqref{Eq psi n as lim func}
has the following asymptotic expansion, in $L^{2}(d\PP,\sigma(\Phi),H^{-\eta}(\C))$,
as $\gamma\to 0^{+}$:
\begin{displaymath}
\dfrac{1}{2^{n-1}}
\sum_{\substack{0\leq k\leq \lfloor n/2\rfloor\\0\leq j\leq \lfloor n/2\rfloor}}
(-1)^{j+k}
\dfrac{n! (2(n-1-k-j))!}{(n-2k)! k! (n-1-k-j)! j!}\,\gamma^{2j+1 - n}
\,\psi_{2j+1,A}
+
o(1).
\end{displaymath}
This expansion is consistent with Theorem \ref{Thm psi GMC}
because the following combinatorial identity holds.

\begin{prop}
\label{Prop combi 0 1}
Let $n\in \N$, $n$ odd.
Then the following holds.
\begin{enumerate}
\item For every $j\in \{0,\dots, \lfloor n/2\rfloor-1\}$,
\begin{displaymath}
\dfrac{1}{2^{n-1}}
\sum_{k=0}^{\lfloor n/2\rfloor}
(-1)^{j+k}
\dfrac{n! (2(n-1-k-j))!}{(n-2k)! k! (n-1-k-j)! j!}
= 0.
\end{displaymath}
\item In the case $j = \lfloor n/2\rfloor$,
\begin{displaymath}
\dfrac{1}{2^{n-1}}
\sum_{k=0}^{\lfloor n/2\rfloor}
(-1)^{j+k}
\dfrac{n! (2(n-1-k-j))!}{(n-2k)! k! (n-1-k-j)! j!}
= 1.
\end{displaymath}
\end{enumerate}
\end{prop}
\begin{proof}
We write $n=2m+1$, so that $\lfloor n/2\rfloor = m$.
Let $j\in \{0,\dots, m-1\}$. Then
\begin{multline*}
\sum_{k=0}^{m}
(-1)^{k}
\dfrac{(2(2m-k-j))!}{(2(m-k)+1)! k! (2m-k-j)!}
=
\sum_{k=0}^{m}
(-1)^{k}
\dfrac{(2(m-k)+2(m-j))!}{(2(m-k)+1)! k! ((m-k)+(m-j))!}
\\
=
\dfrac{1}{m!}
\sum_{k=0}^{m}
(-1)^{k}
\Big(
\prod_{l=2}^{2(m-j)} (2(m-k)+l)
\Big)
\dfrac{(m-k)!}{((m-k)+(m-j))!}
\dfrac{m!}{k!(m-k)!}
\\
=
\dfrac{1}{m!}
\sum_{k=0}^{m}
(-1)^{k}
\dfrac{\prod_{l=2}^{2(m-j)} (2(m-k)+l)}
{\prod_{p=1}^{m-j}((m-k)+p)}
\dfrac{m!}{k!(m-k)!}\, .
\end{multline*}
For the second equality we used that $m-j\geq 1$.
Further,
\begin{displaymath}
\dfrac{\prod_{l=2}^{2(m-j)} (2(m-k)+l)}
{\prod_{p=1}^{m-j}((m-k)+p)}
=
2^{m-j}
\prod_{q=1}^{m-j-1}(2(m-k+q)+1)
.
\end{displaymath}
Thus,
\begin{multline}
\label{Eq red poly k}
\sum_{k=0}^{m}
(-1)^{k}
\dfrac{(2(2m-k-j))!}{(2(m-k)+1)! k! (2m-k-j)!}
=
\\
\dfrac{2^{m-j}}{m!}
\sum_{k=0}^{m}
(-1)^{k}
\Big(
\prod_{q=1}^{m-j-1}(2(m-k+q)+1)
\Big)
\dfrac{m!}{k!(m-k)!}\, .
\end{multline}
But the factor
\begin{displaymath}
\prod_{q=1}^{m-j-1}(2(m-k+q)+1)
\end{displaymath}
is a polynomial in $k$ of degree $m-j-1$.
We claim that for every polynomial $P$ of degree less or equal to $m-1$,
\begin{displaymath}
\sum_{k=0}^{m}
(-1)^{k}
P(k)
\dfrac{m!}{k!(m-k)!} = 0.
\end{displaymath}
If $P$ is a constant, this is clear.
Otherwise, $P(k)$ can be written as a constant plus a linear combination of
\begin{displaymath}
k(k-1)\dots (k-l)
\end{displaymath}
with $l\in \{0,\dots, m-2\}$, and
\begin{displaymath}
\sum_{k=0}^{m}
(-1)^{k}
k(k-1)\dots (k-l)
\dfrac{m!}{k!(m-k)!}
=
(-1)^{m}
\Big(\dfrac{d^{l+1}}{(dx)^{l+1}} (x-1)^{m}\Big)_{x=1} = 0.
\end{displaymath}
So \eqref{Eq red poly k} equals $0$, and we obtain (1).

Let us consider the case $j=m$.
Then
\begin{multline}
\label{Eq case j m}
\dfrac{1}{4^{m}}
\sum_{k=0}^{m}
(-1)^{m-k}
\dfrac{(2m+1)!(2(2m-k-m))!}{(2(m-k)+1)! k! (2m-k-m)! m!}
=
\\
\dfrac{(2m+1)!}{4^{m}(m!)^{2}}
\sum_{k=0}^{m}
(-1)^{m-k}
\dfrac{1}{2(m-k)+1}
\dfrac{m!}{k!(m-k)!}\, .
\end{multline}
But
\begin{displaymath}
\dfrac{1}{2(m-k)+1} = \int_{0}^{1} x^{2(m-k)}\,dx.
\end{displaymath}
Thus, \eqref{Eq case j m} equals
\begin{displaymath}
\dfrac{(2m+1)!}{4^{m} (m!)^{2}}
\sum_{k=0}^{m}
(-1)^{m-k}
\dfrac{m!}{k!(m-k)!}
\int_{0}^{1} x^{2(m-k)}\,dx
=
\dfrac{(2m+1)!}{4^{m} (m!)^{2}}
\int_{0}^{1} (1-x^{2})^{m}\,dx
\end{displaymath}
By performing the change of variables $x=\sin \theta$, we get
\begin{displaymath}
\int_{0}^{1} (1-x^{2})^{m}\,dx
=
\int_{0}^{\pi/2} (\cos \theta)^{2m+1}\, d\theta\,.
\end{displaymath}
Here we recognize a Wallis' integral:
\begin{displaymath}
\int_{0}^{\pi/2} (\cos \theta)^{2m+1}\, d\theta = 
\dfrac{4^{m} (m!)^{2}}{(2m+1)!} .
\end{displaymath}
So we get (2).
\end{proof}

Thus, the asymptotic expansion \eqref{Eq A E N} (Theorem \ref{Thm A E FPS}) provides an alternative proof of the convergence in Theorem \ref{Thm psi GMC} for $n$ odd in the $L^{2}(d\PP,\sigma(\Phi),H^{-\eta}(\C))$ sense.
However, unlike Theorem \ref{Thm psi GMC}, this covers neither the case $n$ ever nor the almost sure convergence.
On top of that, the direct proof of Theorem \ref{Thm psi GMC} is much simpler than that of Theorem \ref{Thm A E FPS}.

Conversely, if one assumes an asymptotic expansion in $L^{2}(d\PP,\sigma(\Phi),H^{-\eta}(\C))$
of form
\begin{equation}
\label{Eq form expansion}
\ind_{\Ns_{\varepsilon}(A)}
=
\sum_{k=0}^{N}
\dfrac{F_{k}}{\vert\log \varepsilon\vert^{k + 1/2}}
\,
+ o(\vert\log \varepsilon\vert^{-(N + 1/2)}),
\end{equation}
then Theorem \ref{Thm psi GMC} implies that
\begin{displaymath}
F_{k} = 
(-1)^{k}
\dfrac{\pi^{k}}{k! (k+1/2)} \psi_{2k+1,A}.
\end{displaymath}
However, Theorem \ref{Thm psi GMC} does not directly provide the existence of
an expansion of form \eqref{Eq form expansion},
in particular with these powers of $\vert\log\varepsilon\vert$.

\section{Decomposition of Wick powers via the excursion clusters}
\label{Sec Wick exc decomp}

\subsection{The generational decomposition}
\label{Subsec Wick exc gen}

Here we will present the full decomposition of Wick powers of the GFF through the excursion clusters.
We are in the setting of Sections \ref{Subsubsec exc} and \ref{Subsubsec gen exc},
and use the notations therein.
So $D\subset \C$ is an open, non-empty, bounded, connected and simply connected domain.
Let $\Phi$ a GFF on $D$ with $0$ boundary condition on $\partial D$.

We start with the generational decomposition, 
corresponding to grouping the excursion clusters into generations as in \ref{Subsubsec gen exc}.
Consider the first generation.
The GFF $\Phi$ decomposes
\begin{displaymath}
\Phi = \nu^{\rm gen}_{0} + \Phi_{D\setminus A^{\rm gen, +}_{0}},
\end{displaymath}
where $\nu^{\rm gen}_{0}$ is a signed measure with finite total variation $\vert \nu^{\rm gen}_{0}\vert$,
supported on $A^{\rm gen, +}_{0}$,
and conditionally on $\mathcal{F}_{0}^{\rm gen, +}$,
the field $\Phi_{D\setminus A^{\rm gen, +}_{0}}$ is distributed as a GFF on $D\setminus A^{\rm gen, +}_{0}$
with $0$ boundary conditions.
We are in a setting pretty similar to that of Section \ref{Sec Wick FPS},
except that $\nu^{\rm gen}_{0}$ is a signed measure and not a positive measure,
and $A^{\rm gen, +}_{0}$ is not an FPS in $D$.
Actually, conditionally on $\mathcal{F}_{0}^{\rm gen}$,
$A^{\rm gen, +}_{0}$ is an FPS of the field
\begin{displaymath}
\vert \nu^{\rm gen}_{0}\vert  + \Phi_{D\setminus A^{\rm gen, +}_{0}}
\end{displaymath}
which conditionally is a GFF on $D\setminus A^{\rm gen}_{0}$
with boundary value $2\lambda$;
see Section \ref{Subsubsec gen exc}.
Because of this similarity of settings,
we will only provide the main steps of proofs
and emphasize the points where there are differences.

Denote
\begin{displaymath}
\psi^{\rm gen}_{n,0} = \E[\, :\Phi^{n}:\,\vert \mathcal{F}_{0}^{\rm gen, +}].
\end{displaymath}
Observe that $\psi^{\rm gen}_{1,0} = \nu^{\rm gen}_{0}$.
Denote by $V_{A^{\rm gen, +}_{0}}$ the function
\begin{displaymath}
V_{A^{\rm gen, +}_{0}}(z)
=
\dfrac{1}{2\pi}\log\Big(\dfrac{\CR(z,D)}{\CR(z,D\setminus A^{\rm gen, +}_{0})}\Big),
\end{displaymath}
defined on $D\setminus A^{\rm gen, +}_{0}$.
Let $(\rho_{\varepsilon})_{\varepsilon>0}$ be a family of mollificators as in
Sections \ref{Subsec Wick} and \ref{Sec Wick FPS}.
By mollifying, and with obvious notations, we have
\begin{displaymath}
\Phi_{\varepsilon} = \nu^{\rm gen}_{0,\varepsilon} + \Phi_{D\setminus A^{\rm gen, +}_{0}, \varepsilon},
\end{displaymath}
Denote
\begin{displaymath}
V_{A^{\rm gen, +}_{0},\varepsilon}(z)
= G_{D,\varepsilon,\varepsilon}(z,z) - 
G_{D\setminus A^{\rm gen, +}_{0},\varepsilon,\varepsilon}(z,z)
=
g_{D,\varepsilon,\varepsilon}(z,z) - 
g_{D\setminus A^{\rm gen, +}_{0},\varepsilon,\varepsilon}(z,z),
\end{displaymath}
where $G_{D,\varepsilon,\varepsilon}$ and $G_{D\setminus A^{\rm gen, +}_{0},\varepsilon,\varepsilon}$
are $\rho_{\varepsilon}$-convolutions
of the Green's functions in respective domain w.r.t. both variables,
and 
$g_{D,\varepsilon,\varepsilon}$ and $g_{D\setminus A^{\rm gen, +}_{0},\varepsilon,\varepsilon}$
are similar convolutions of constant order parts of Green's functions
(logarithmic singularity substracted).
Then $V_{A^{\rm gen, +}_{0},\varepsilon}$ is a well-defined and smooth function on $\C$,
with compact support.
An analogue of the decomposition of Lemma \ref{Lem decomp Phi eps} holds,
with similar computations.

\begin{lemma}
\label{Lem decomp Phi gen 1}
For $n\geq 1$ and $z\in\C$,
\begin{eqnarray*}
:\Phi^{n}_{\varepsilon}:(z)&=&
Q_{n}(\nu^{\rm gen}_{0,\varepsilon}(z),V_{A^{\rm gen, +}_{0},\varepsilon}(z)) +
Q_{n}(\Phi_{D\setminus A^{\rm gen, +}_{0},\varepsilon}(z),G_{D,\varepsilon,\varepsilon}(z,z))
\\
&&-\ind_{n \text{ even}}(-1)^{n/2}
\dfrac{n!}{2^{n/2} (n/2)!}
V_{A^{\rm gen, +}_{0},\varepsilon}(z)^{n/2}
\\&& +
\sum_{\substack{0\leq j\leq n-1\\0\leq k <\lfloor j/2\rfloor}}
(-1)^{k}\dfrac{n!}{2^{k} (n-j)! k! (j-2k)!}
V_{A^{\rm gen, +}_{0},\varepsilon}(z)^{k}
\nu^{\rm gen}_{0,\varepsilon}(z)^{n-j}
\,
:\Phi_{D\setminus A^{\rm gen, +}_{0},\varepsilon}^{j-2k}:(z).
\end{eqnarray*}
\end{lemma}

From this we deduce the analogues of 
Corollary \ref{Cor cond exp eps} and Proposition \ref{Prop psi eps}.

\begin{prop}
\label{Prop psi eps gen 0}
For $\varepsilon>0$,
\begin{displaymath}
\E\big[:\Phi^{n}_{\varepsilon}:\vert \mathcal{F}_{0}^{\rm gen, +}\big] =
Q_{n}(\nu^{\rm gen}_{0,\varepsilon},V_{A^{\rm gen, +}_{0},\varepsilon}).
\end{displaymath}
In particular, for every $\eta>0$ and $n\geq 1$,
\begin{displaymath}
\lim_{\varepsilon\to 0}
\E\Big[
\big\Vert
\psi^{\rm gen}_{n,0}
-
Q_{n}(\nu^{\rm gen}_{0,\varepsilon},V_{A^{\rm gen, +}_{0},\varepsilon})
\big\Vert_{H^{-\eta}(\C)}^{2}
\Big] = 0.
\end{displaymath}
\end{prop}

Next we would like to get rid of the mixed terms
$V_{A^{\rm gen, +}_{0},\varepsilon}^{k}
(\nu^{\rm gen}_{0,\varepsilon})^{n-j}
\,
:\Phi_{D\setminus A^{\rm gen, +}_{0},\varepsilon}^{j-2k}:$
appearing in Lemma \ref{Lem decomp Phi gen 1}.
To apply the procedure of Section \ref{Subsec prod 0},
we need an analogue of Lemma \ref{Lem moment nu eps},
that is to say a logarithmic bound
for $\E[\nu^{\rm gen}_{0,\varepsilon}(z)^{2l}]$.
We have better, an upper bound on
$\E[\vert\nu^{\rm gen}_{0,\varepsilon}\vert(z)^{2l}]$,
and of course 
$\vert\nu^{\rm gen}_{0,\varepsilon}(z)\vert\leq \vert\nu^{\rm gen}_{0,\varepsilon}\vert(z)$.

\begin{lemma}
\label{Lem moment nu vert eps}
Let $l\geq 1$.
There are constants $c_{l},c'_{l}>0$,
depending only on $l$,
same as in Lemma \ref{Lem moment nu eps},
such that for every $\varepsilon>0$ and $z\in\C$,
\begin{displaymath}
\E[\vert\nu^{\rm gen}_{0,\varepsilon}\vert(z)^{2 l}\vert \mathcal{F}_{m}^{\rm gen}]
= c_{l} (2\lambda)^{2l} + c'_{l} G_{D\setminus A^{\rm gen}_{0},\varepsilon,\varepsilon}(z,z)^{l}
~~\text{a.s.}
\end{displaymath}
In particular,
\begin{displaymath}
\E[\vert\nu^{\rm gen}_{0,\varepsilon}\vert(z)^{2 l}]
= c_{l} (2\lambda)^{2l} + c'_{l} G_{D,\varepsilon,\varepsilon}(z,z)^{l}.
\end{displaymath}
\end{lemma}

\begin{proof}
This comes from the fact that conditionally on $\mathcal{F}_{0}^{\rm gen}$, the field
\begin{displaymath}
\vert \nu^{\rm gen}_{0}\vert  + \Phi_{D\setminus A^{\rm gen, +}_{0}}
\end{displaymath}
is a GFF on $D\setminus A^{\rm gen}_{0}$
with boundary value $2\lambda$,
and $\vert \nu^{\rm gen}_{0}\vert$ is the corresponding first passage set measure.
So we can actually apply Lemma \ref{Lem moment nu eps}.
\end{proof}

Given the above, we can use the method of Sections \ref{Subsec prod 0} and \ref{Subsec Wick outside}
and obtain the following decomposition.

\begin{prop}
\label{Prop decomp gen 0 Wick}
Let $n\geq 1$. A.s., the Wick power $:\Phi^{n}:$ can be decomposed
\begin{equation}
\label{Eq decomp Wick gen 0}
:\Phi^{n}:~=\psi^{\rm gen}_{n,0}
+
\sum_{0\leq k <\lfloor n/2\rfloor}
(-1)^{k}\dfrac{n!}{2^{k} k! (n-2k)!}
V_{A^{\rm gen, +}_{0}}^{k}
\,
:\Phi_{D\setminus A^{\rm gen, +}_{0}}^{n-2k}:\,.
\end{equation}
Fix $\eta>0$. 
Then a.s.,
\begin{multline*}
\E\big[\Vert :\Phi^{n}:\Vert^{2}_{H^{-\eta}(\C)}\big\vert \mathcal{F}_{0}^{\rm gen, +}\big]
=
\Vert \psi^{\rm gen}_{n,0}\Vert^{2}_{H^{-\eta}(\C)}
\\+
\E\Big[
\Big\Vert
\sum_{0\leq k <\lfloor n/2\rfloor}
(-1)^{k}\dfrac{n!}{2^{k} k! (n-2k)!}
V_{ A^{\rm gen, +}_{0}}^{k}:\Phi_{D\setminus A^{\rm gen, +}_{0}}^{n-2k}:
\Big\Vert^{2}_{H^{-\eta}(\C)}
\Big\vert \mathcal{F}_{0}^{\rm gen, +}
\Big] .
\end{multline*}
In particular,
\begin{multline*}
\E\big[\Vert :\Phi^{n}:\Vert^{2}_{H^{-\eta}(\C)}\big]
=
\E\big[\Vert \psi^{\rm gen}_{n,0}\Vert^{2}_{H^{-\eta}(\C)}\big]
\\
+
\E\Big[
\Big\Vert
\sum_{0\leq k <\lfloor n/2\rfloor}
(-1)^{k}\dfrac{n!}{2^{k} k! (n-2k)!}
V_{ A^{\rm gen, +}_{0}}^{k}:\Phi_{D\setminus A^{\rm gen, +}_{0}}^{n-2k}:
\Big\Vert^{2}_{H^{-\eta}(\C)}
\Big].
\end{multline*}
\end{prop}

Next, we want to iterate the decomposition with further generations of excursion clusters.
At a surface level, the remainder term in \eqref{Eq decomp Wick gen 0},
that is to say
\begin{displaymath}
\sum_{0\leq k <\lfloor n/2\rfloor}
(-1)^{k}\dfrac{n!}{2^{k} k! (n-2k)!}
V_{A^{\rm gen, +}_{0}}^{k}
\,
:\Phi_{D\setminus A^{\rm gen, +}_{0}}^{n-2k}:\,.,
\end{displaymath}
looks like a complication.
It is not just $:\Phi_{D\setminus A^{\rm gen, +}_{0}}^{n}:$.
However, the change of variance identity \eqref{Eq change var} is there to simplify everything.
In what follows, we will just give the end result without the tedious but standard intermediate
computations.

So denote
\begin{displaymath}
\psi^{\rm gen}_{n,m} = \E[\, :\Phi_{D\setminus A^{\rm gen, +}_{m-1}}^{n}:\,\vert \mathcal{F}_{m}^{\rm gen, +}].
\end{displaymath}
Let be
\begin{displaymath}
V_{A^{\rm gen, +}_{m}}(z)
=
\dfrac{1}{2\pi}\log\Big(\dfrac{\CR(z,D)}{\CR(z,D\setminus A^{\rm gen, +}_{m})}\Big),
\end{displaymath}
for $z\in D\setminus A^{\rm gen, +}_{m}$.

\begin{prop}
\label{Prop decomp gen Wick iter}
Let $n\geq 1$ and $m\geq 1$.
A.s., the Wick power $:\Phi^{n}:$ can be decomposed
\begin{multline*}
:\Phi^{n}:~= \psi^{\rm gen}_{n,0}
+
\sum_{l=1}^{m}
\sum_{0\leq k <\lfloor n/2\rfloor}
(-1)^{k}\dfrac{n!}{2^{k} k! (n-2k)!}
V_{A^{\rm gen, +}_{l-1}}^{k}
\,
\psi^{\rm gen}_{n-2k,l}
\\
+
\sum_{0\leq k <\lfloor n/2\rfloor}
(-1)^{k}\dfrac{n!}{2^{k} k! (n-2k)!}
V_{A^{\rm gen, +}_{m}}^{k}
\,
:\Phi_{D\setminus A^{\rm gen, +}_{m}}^{n-2k}:
\, .
\end{multline*}
Moreover, for $\eta>0$,
one has the following $L^2$ decomposition of Sobolev norms:
\begin{eqnarray*}
\E\Big[
\Vert
:\Phi^{n}:
\Vert^{2}_{H^{-\eta}(\C)}
\Big]
&=&
\E\Big[
\Vert
\psi^{\rm gen}_{n,0}
\Vert^{2}_{H^{-\eta}(\C)}
\Big]
\\&&+
\sum_{l=1}^{m}
\E\Big[
\Big\Vert
\sum_{0\leq k <\lfloor n/2\rfloor}
(-1)^{k}\dfrac{n!}{2^{k} k! (n-2k)!}
V_{A^{\rm gen, +}_{l-1}}^{k}
\,
\psi^{\rm gen}_{n-2k,l}
\Big\Vert^{2}_{H^{-\eta}(\C)}
\Big]
\\
&&+
\E\Big[
\Big\Vert
\sum_{0\leq k <\lfloor n/2\rfloor}
(-1)^{k}\dfrac{n!}{2^{k} k! (n-2k)!}
V_{A^{\rm gen, +}_{m}}^{k}
\,
:\Phi_{D\setminus A^{\rm gen, +}_{m}}^{n-2k}:
\Big\Vert^{2}_{H^{-\eta}(\C)}
\Big]
\, .
\end{eqnarray*}
\end{prop}

Now let use deal with the remainder terms.

\begin{lemma}
\label{Lem conv 0 remainder Sob}
Let $k\geq 0$ and $j\geq 1$. Let $\eta >0$.
Then
\begin{displaymath}
\lim_{m\to +\infty}
\E\Big[
\Big\Vert
V_{A^{\rm gen, +}_{m}}^{k}
\,
:\Phi_{D\setminus A^{\rm gen, +}_{m}}^{j}:
\Big\Vert^{2}_{H^{-\eta}(\C)}
\Big] = 0.
\end{displaymath}
\end{lemma}

\begin{proof}
Indeed, one can express
\begin{displaymath}
\E\Big[
\Big\Vert
V_{A^{\rm gen, +}_{m}}^{k}
\,
:\Phi_{D\setminus A^{\rm gen, +}_{m}}^{j}:
\Big\Vert^{2}_{H^{-\eta}(\C)}
\Big\vert
\mathcal{F}_{m}^{\rm gen, +}
\Big]
\end{displaymath}
as
\begin{displaymath}
j!\,
\int
V_{A^{\rm gen, +}_{m}}(z)^{k}
V_{A^{\rm gen, +}_{m}}(w)^{k}
G_{D\setminus A^{\rm gen, +}_{m}}(z,w)^{j}
\LK_{\eta}(\vert w-z\vert)
d^{2}z
\,
d^{2}w
\, .
\end{displaymath}
Note that $V_{A^{\rm gen, +}_{m}}(z)$
blows up when $z$ approaches $A^{\rm gen, +}_{m}$.
Also note that $G_{D\setminus A^{\rm gen, +}_{m}}(z,w)$ is $0$
if $z$ and $w$ belong to different connected components of $D\setminus A^{\rm gen, +}_{m}$.
Further, we apply the bound \eqref{Eq est Green}
(Proposition \ref{Prop Green}) to
$G_{D\setminus A^{\rm gen, +}_{m}}(z,w)$.
To conclude, we use the fact that the maximum of the diameters of connected components of
$D\setminus A^{\rm gen, +}_{m}$ converges in probability to $0$ as
$m\to +\infty$.
For instance, one can see the latter fact through the nested CLE$_{4}$ and apply
\cite[Theorem 1.5]{AruPaponPowell23dust}.
\end{proof}

So we get the full generational decomposition of Wick powers.

\begin{thm}
\label{Thm decomp gen Wick}
Let $n\geq 1$.
The Wick power $:\Phi^{n}:$ can be decomposed
\begin{displaymath}
:\Phi^{n}:~= \psi^{\rm gen}_{n,0}
+
\sum_{m=1}^{+\infty}
\sum_{0\leq k <\lfloor n/2\rfloor}
(-1)^{k}\dfrac{n!}{2^{k} k! (n-2k)!}
V_{A^{\rm gen, +}_{m-1}}^{k}
\,
\psi^{\rm gen}_{n-2k,m}
\, ,
\end{displaymath}
where the convergence holds in $L^{2}(d\PP,\sigma(\Phi),H^{-\eta}(\C))$ for $\eta>0$.
Moreover, one has the following $L^2$ decomposition of Sobolev norms:
\begin{eqnarray*}
\E\Big[
\Vert
:\Phi^{n}:
\Vert^{2}_{H^{-\eta}(\C)}
\Big]
&=&
\E\Big[
\Vert
\psi^{\rm gen}_{n,0}
\Vert^{2}_{H^{-\eta}(\C)}
\Big]
\\&&+
\sum_{m=1}^{+\infty}
\E\Big[
\Big\Vert
\sum_{0\leq k <\lfloor n/2\rfloor}
(-1)^{k}\dfrac{n!}{2^{k} k! (n-2k)!}
V_{A^{\rm gen, +}_{m-1}}^{k}
\,
\psi^{\rm gen}_{n-2k,m}
\Big\Vert^{2}_{H^{-\eta}(\C)}
\Big]
\, .
\end{eqnarray*}
\end{thm}

Let us also recall the nature of the fields $\psi^{\rm gen}_{n,m}$.

\begin{prop}
\label{Prop struc psi n m}
If $n$ is odd, then a.s. the generalized function $\psi^{\rm gen}_{n,m}$
is supported on $A^{\rm gen, +}_{m}$.
If $n=2k\geq 2$ is even, then the generalized function $\psi^{\rm gen}_{n,m}$
also extends to $D\setminus A^{\rm gen, +}_{m}$,
where it coincides with the function
\begin{displaymath}
(-1)^{k}\dfrac{(2k)!}{2^{k}k!}
\Big(
\dfrac{1}{2\pi}
\log
\Big(
\dfrac{\CR(z, D\setminus A^{\rm gen, +}_{m-1})}{\CR(z, D\setminus A^{\rm gen, +}_{m})}
\Big)
\Big)^{k} .
\end{displaymath}
\end{prop}

\subsection{Fields associated to individual excursion clusters}
\label{Subsec fields indiv}

Here we will consider the individual excursion clusters $\clus_j$,
as in Section \ref{Subsubsec exc}.
We will denote by $\Phi_{j}$ the field
\begin{displaymath}
\Phi_{j} = \ind_{\inter(\Gamma_j)}\sigma_{j}\Phi .
\end{displaymath}
Conditionally on $\Gamma_j$,
the field $\Phi_{j}$ is distributed as a GFF on $\inter(\Gamma_j)$
with boundary condition $2\lambda$.
Moreover, the cluster $\clus_j$ is the first passage set to level $0$
of the field $\Phi_{j}$, and the positive measure $\nu_{j}$ is the associated FPS measure.
We will denote by $:\Phi_{j}^{n}:$ the Wick powers of $\Phi_{j}$,
where the renormalization is done w.r.t. the partial Green's function
$G_{\inter(\Gamma_j)}$ and not the total Green's function $G_{D}$.
We will denote by $\psi_{n,j}$ the fields
\begin{displaymath}
\psi_{n,j} = \E[:\Phi_{j}^{n}:\vert \clus_j].
\end{displaymath}
Here we will collect the properties of the fields $\psi_{n,j}$.
In the following Section \ref{Subsec Wick exc indiv}
we will relate the fields $\psi_{k,j}$
and the fields $\psi^{\rm gen}_{n,m}$ from the previous Section \ref{Subsec Wick exc gen}.

First, we restate the results of Sections \ref{Sec Wick FPS} and \ref{Sec psi GMC} in this context.
Strictly speaking, this is not a direct logical consequence of the results there,
because these are stated for fixed domains $D$, and here the domain $\inter(\Gamma_j)$
is random. 
However, it is easy to check that all the domain-dependent bounds are monotonic in the domains,
and of course $\inter(\Gamma_j)\subset D$.
So we get the following.
Denote
\begin{displaymath}
V_{\clus_j,\Gamma_j}(z)
=
\dfrac{1}{2\pi}
\log\Big(
\dfrac{\CR(z,\inter(\Gamma_j))}{\CR(z,\inter(\Gamma_j)\setminus \clus_j)}
\Big),
\end{displaymath}
and its approximation
\begin{displaymath}
V_{\clus_j,\Gamma_j,\varepsilon}(z)
=
G_{\inter(\Gamma_j),\varepsilon,\varepsilon}(z,z)
-
G_{\inter(\Gamma_j)\setminus \clus_j,\varepsilon,\varepsilon}(z,z).
\end{displaymath}

\begin{thm}
\label{Thm fields on C j}
Let $n\geq 1$.
Then
\begin{displaymath}
\psi_{n,j}
=
\lim_{\varepsilon\to 0}
Q_{n}(\nu_{j,\varepsilon},V_{\clus_j,\Gamma_j,\varepsilon}),
\end{displaymath}
with a convergence in $L^{2}(d\PP,\sigma(\Phi),H^{-\eta}(\C))$ for $\eta>0$.
Further,
\begin{displaymath}
\psi_{n,j}
=
\lim_{\gamma\to 0^{+}}
\ind_{\inter(\Gamma_j)\setminus \clus_j}
\He_{n}(\gamma V_{\clus_j,\Gamma_j}^{1/2})
V_{\clus_j,\Gamma_j}^{n/2} 
\exp\Big(-\dfrac{\gamma^{2}}{2} V_{\clus_j,\Gamma_j}\Big),
\end{displaymath}
with convergence both in $L^{2}(d\PP,\sigma(\Phi),H^{-\eta}(\C))$ 
and almost surely in $H^{-\eta}(\C)$, for $\eta>0$.
For $n$ odd, the field $\psi_{n,j}$ is a generalized function supported on $\clus_j$.
For $n=2k\geq 2$ even, the field $\psi_{n,j}$ also has an extension to
$\inter(\Gamma_j)\setminus \clus_j$, 
where it coincides with the function
\begin{displaymath}
(-1)^{k}\dfrac{(2k)!}{2^{k}k!}
V_{\clus_j,\Gamma_j}^{k} .
\end{displaymath}
\end{thm}

Next we give the asymptotic expansion for the cluster $\clus_j$.
Compared to Theorem \ref{Thm A E FPS},
we deal with random domains $\inter(\Gamma_j)$,
rather than deterministic domains.
In the proof of Theorem \ref{Thm A E FPS}, there are two types of domain-dependent bounds.
The first type depends either on the diameter or is an integral over the domain.
These bounds are monotononic w.r.t. the domain, and we keep a fairly explicit track of these in the proof of 
Theorem \ref{Thm A E FPS}.
The second type of bounds depends on the size of the boundary of the domain, and is related to
the condition \eqref{Eq cond small boundary}.
It appears only in Lemma \ref{Lem small boundary means fast enough decay}.
For the random domain $\inter(\Gamma_j)$ it translates into the condition
\begin{equation}
\label{Eq boundary cond CLE 4}
\forall \beta >0, \,
\E\big[
\operatorname{Leb}\{ z\in\inter(\Gamma_j)\vert d(z,\Gamma_j)<\varepsilon\}\big]
= o(\vert\log\varepsilon\vert^{-\beta}).
\end{equation}
We claim that this condition is satisfied with a very wide margin
since the dimension of an SLE$_{4}$ curve is $3/2$ \cite{BeffaraDimSLE},
and in particular, one has a polynomial decay $\varepsilon^{1/2 + o(1)}$.
However, we will not give a precise proof of \eqref{Eq boundary cond CLE 4},
because bridging the gap between what is written in the literature
and our precise setting appears to be both very standard and at the same time a very lengthy detour through SLE theory.
So here is the asymptotic expansion.

\begin{thm}
\label{Thm A E C j}
Let $j\geq 0$.
For $\varepsilon>0$, set
\begin{displaymath}
\Ns_{\varepsilon}(\clus_j) = \{ z\in \inter(\Gamma_j)\setminus \clus_j \vert 
\CR(z,\inter(\Gamma_j)\setminus \clus_j)<\varepsilon \CR(z,\inter(\Gamma_j))\},
\end{displaymath}
\begin{displaymath}
\widetilde{\Ns}_{\varepsilon}(\clus_j)
= \{ z\in \inter(\Gamma_j)\setminus \clus_j \vert \CR(z,\inter(\Gamma_j)\setminus \clus_j)<\varepsilon\}.
\end{displaymath}
Then, for every $N\geq 0$,
\begin{displaymath}
\ind_{\Ns_{\varepsilon}(\clus_j)}
=
\dfrac{1}{\sqrt{2\pi}}
\sum_{k=0}^{N} (-1)^{k}
\dfrac{1}{2^{k} k! (k+1/2)}
\dfrac{\psi_{2k+1,j}}{\big(\frac{1}{2\pi}\vert\log \varepsilon\vert\big)^{k + 1/2}}
~+ o(\vert\log\varepsilon\vert^{-(N+1/2)}),
\end{displaymath}
where $o(\vert\log\varepsilon\vert^{-(N+1/2)})$ is in the sense
$L^{2}(d\PP,\sigma(\Phi),H^{-\eta}(\C))$ for $\eta>0$.
Further,
\begin{displaymath}
\ind_{\widetilde{\Ns}_{\varepsilon}(\clus_j)}
=
\dfrac{1}{\sqrt{2\pi}}
\sum_{k=0}^{N} (-1)^{k}
\dfrac{1}{2^{k} k! (k+1/2)}
\dfrac{\tilde{\psi}_{2k+1,j}}{\big(\frac{1}{2\pi}
\vert\log \varepsilon\vert\big)^{k + 1/2}}
~+ o(\vert\log\varepsilon\vert^{-(N+1/2)}),
\end{displaymath}
where
\begin{displaymath}
\tilde{\psi}_{2k+1,j} = 
\sum_{j=0}^{k} \dfrac{(2k+1)!}{2^{j} j! (2(k-j)+1)!}
\,\Big(\dfrac{1}{2\pi} \log \CR(z,\inter(\Gamma_j))\Big)^{j}\,\psi_{2(k-j)+1,j}.
\end{displaymath}
\end{thm}

\begin{cor}
\label{Cor psi Vandermonde indiv exc}
Let $j\geq 0$.
Let $n\geq 1$
and let $\alpha_{n}>\alpha_{n-1}>\dots >\alpha_{1}>\alpha_{0}=1$.
Let $(c_{0}, c_{1}, \dots, c_{n})$
be the unique solution to the linear system
\begin{displaymath}
\forall k\in \{0,\dots, n-1\},
\sum_{i=0}^{n} c_{i} \alpha_{i}^{-(k+1/2)} = 0,
\qquad
\sum_{i=0}^{n} c_{i} \alpha_{i}^{-(n+1/2)} = 1.
\end{displaymath}
Then
\begin{displaymath}
\psi_{2n+1, j} = \lim_{\varepsilon \to 0} (-1)^{n}\dfrac{2^{n} n! (n+1/2)}{(2\pi)^{n}}\vert\log\varepsilon\vert^{n+1/2}
\sum_{i=0}^{n} c_{i} \ind_{\Ns_{\varepsilon^{\alpha_{i}}}(\clus_j)},
\end{displaymath}
with convergence in $L^{2}(d\PP,\sigma(\Phi),H^{-\eta}(\C))$ for $\eta>0$.
\end{cor}

\subsection{Decomposition through individual excursion clusters}
\label{Subsec Wick exc indiv}

Here we relate the fields $\psi_{k,j}$ from Section \ref{Subsec fields indiv}, 
associated to individual excursion clusters,
and the fields $\psi^{\rm gen}_{n,m}$ from Section \ref{Subsec Wick exc gen},
associated to generations of clusters.
We will omit some details of derivations, which at this stage of the paper are very standard.
Basically, given $j\in \Gen(m)$,
the restriction of $\psi^{\rm gen}_{n,m}$ to $\inter(\Gamma_j)$
is given by a linear combination of the $\psi_{n-2k,j}$,
with two points to take into account.
First, the sign $\sigma_{j}$ in the case $n$ is odd,
and then the fact that the renormalizations in the case of $\psi^{\rm gen}_{n,m}$
and that of $\psi_{n-2k,j}$ are done w.r.t. different Green's function.
The difference in Green's functions is handled as usually via the change of variance formula
\eqref{Eq change var}.
The precise relation is as follows.
Denote for $m\geq 0$,
\begin{displaymath}
V_{A^{\rm gen}_{m}}(z)
=
\dfrac{1}{2\pi}
\log\Big(
\dfrac{\CR(z,D)}{\CR(z,D\setminus A^{\rm gen}_{m})}
\Big),
\end{displaymath}
and
\begin{displaymath}
V_{A^{\rm gen}_{m},A^{\rm gen, +}_{m-1}}(z)
=
\dfrac{1}{2\pi}
\log\Big(
\dfrac{\CR(z,D\setminus A^{\rm gen, +}_{m-1})}{\CR(z,D\setminus A^{\rm gen}_{m})}
\Big),
\end{displaymath}
which are functions defined on $D\setminus A^{\rm gen}_{m}$.
By convention, for $m=0$, we set
$A^{\rm gen, +}_{-1}=\partial D$.
Note that for $j\in \Gen(m)$ and $z\in \inter(\Gamma_{j})$,
\begin{displaymath}
\CR(z,D\setminus A^{\rm gen}_{m}) = \CR(z,\inter(\Gamma_{j})).
\end{displaymath}

\begin{prop}
\label{Prop gen to indiv}
Let $n\geq 1$ odd.
Then, a.s., for every $m\geq 0$, for every $j\in \Gen(m)$
and every smooth function $f$ compactly supported in $\inter(\Gamma_j)$,
\begin{displaymath}
(\psi^{\rm gen}_{n,m},f)
=
\sigma_{j}
\sum_{0\leq k< \lfloor n/2\rfloor}
(-1)^{k} \dfrac{n!}{2^{k} k! (n-2k)!}
(V_{A^{\rm gen}_{m},A^{\rm gen, +}_{m-1}}^{k}\psi_{n-2k,j},f),
\end{displaymath}  
and
\begin{multline*}
\sum_{0\leq k <\lfloor n/2\rfloor}
(-1)^{k}\dfrac{n!}{2^{k} k! (n-2k)!}
(V_{A^{\rm gen, +}_{m-1}}^{k}
\psi^{\rm gen}_{n-2k,m},f)
=
\\
\sigma_{j}
\sum_{0\leq k< \lfloor n/2\rfloor}
(-1)^{k} \dfrac{n!}{2^{k} k! (n-2k)!}
(V_{A^{\rm gen}_{m}}^{k}\psi_{n-2k,j},f).
\end{multline*}
If $n\geq 2$ is even, then
\begin{displaymath}
(\psi^{\rm gen}_{n,m},f)
=
\sum_{0\leq k < \lfloor n/2\rfloor}
(-1)^{k} \dfrac{n!}{2^{k} k! (n-2k)!}
(V_{A^{\rm gen}_{m},A^{\rm gen, +}_{m-1}}^{k}\psi_{n-2k,j},f)
+
(-1)^{n/2} \dfrac{n!}{2^{n/2} (n/2)! }
(V_{A^{\rm gen}_{m},A^{\rm gen, +}_{m-1}}^{n/2},f),
\end{displaymath}
and
\begin{multline*}
\sum_{0\leq k <\lfloor n/2\rfloor}
(-1)^{k}\dfrac{n!}{2^{k} k! (n-2k)!}
(V_{A^{\rm gen, +}_{m-1}}^{k}
\psi^{\rm gen}_{n-2k,m},f)
=
\\
\sum_{0\leq k< \lfloor n/2\rfloor}
(-1)^{k} \dfrac{n!}{2^{k} k! (n-2k)!}
(V_{A^{\rm gen}_{m}}^{k}\psi_{n-2k,j},f)
+
(-1)^{n/2} \dfrac{n!}{2^{n/2} (n/2)! }
(V_{A^{\rm gen}_{m}}^{n/2} - V_{A^{\rm gen, +}_{m-1}}^{n/2},f)
.
\end{multline*}
\end{prop}

Given a generation $m\geq 0$,
we will enumerate the clusters
$(\clus_{m,j})_{j\geq 0}$ belonging to this generation, in the decreasing order of the diameter.
We will denote by $\sigma_{m,j}$ the corresponding signs,
by $\Gamma_{m,j}$ the corresponding outer boundaries,
and by $\psi_{n,m,j}$ the corresponding fields.

\begin{thm}
\label{Thm per cluster enum}
For the generation $m=0$ and $n\geq 1$ odd,
\begin{displaymath}
\psi^{\rm gen}_{n,m=0} = 
\sum_{j=0}^{+\infty}
\Big(
\sigma_{m=0,j}
\sum_{0\leq k< \lfloor n/2\rfloor}
(-1)^{k} \dfrac{n!}{2^{k} k! (n-2k)!}
V_{A^{\rm gen}_{0}}^{k}\psi_{n-2k,m=0,j}
\Big),
\end{displaymath}
where the infinite sum in $j$ converges in $L^{2}(d\PP,\sigma(\Phi),H^{-\eta}(\C))$ for $\eta>0$.
Moreover,
\begin{displaymath}
\E\Big[\Vert \psi^{\rm gen}_{n,m=0}\Vert_{H^{-\eta}(\C)}^{2}\Big] = 
\sum_{j=0}^{+\infty}
\E\Big[
\Big\Vert
\sum_{0\leq k< \lfloor n/2\rfloor}
(-1)^{k} \dfrac{n!}{2^{k} k! (n-2k)!}
V_{A^{\rm gen}_{0}}^{k}\psi_{n-2k,m=0,j}
\Big\Vert_{H^{-\eta}(\C)}^{2}
\Big].
\end{displaymath}
For $n\geq 2$ even,
\begin{multline*}
\psi^{\rm gen}_{n,m=0} =
\\
\sum_{j=0}^{+\infty}
\Big(
\Big(
\sum_{0\leq k < \lfloor n/2\rfloor}
(-1)^{k} \dfrac{n!}{2^{k} k! (n-2k)!}
V_{A^{\rm gen}_{0}}^{k}\psi_{n-2k,m=0,j}
\Big)
+
(-1)^{n/2} \dfrac{n!}{2^{n/2} (n/2)! }
\ind_{\inter(\Gamma_{m=0,j})}
V_{A^{\rm gen}_{0}}^{n/2}
\Big).
\end{multline*}
For the generations $m\geq 1$ and $n\geq 1$ odd,
\begin{multline*}
\sum_{0\leq k <\lfloor n/2\rfloor}
(-1)^{k}\dfrac{n!}{2^{k} k! (n-2k)!}
V_{A^{\rm gen, +}_{m-1}}^{k}
\psi^{\rm gen}_{n-2k,m}
=
\\
\sum_{j=0}^{+\infty}
\Big(
\sigma_{m,j}
\sum_{0\leq k< \lfloor n/2\rfloor}
(-1)^{k} \dfrac{n!}{2^{k} k! (n-2k)!}
V_{A^{\rm gen}_{m}}^{k}\psi_{n-2k,m,j}
\Big),
\end{multline*}
and
\begin{multline*}
\E\Big[
\Big\Vert
\sum_{0\leq k <\lfloor n/2\rfloor}
(-1)^{k}\dfrac{n!}{2^{k} k! (n-2k)!}
V_{A^{\rm gen, +}_{m-1}}^{k}
\psi^{\rm gen}_{n-2k,m}
\Big\Vert_{H^{-\eta}(\C)}^{2}
\Big]
=
\\
\sum_{j=0}^{+\infty}
\E\Big[
\Big\Vert
\sum_{0\leq k< \lfloor n/2\rfloor}
(-1)^{k} \dfrac{n!}{2^{k} k! (n-2k)!}
V_{A^{\rm gen}_{m}}^{k}\psi_{n-2k,m,j}
\Big)
\Big\Vert_{H^{-\eta}(\C)}^{2}
\Big].
\end{multline*}
For $n\geq 2$ even,
\begin{multline*}
\sum_{0\leq k <\lfloor n/2\rfloor}
(-1)^{k}\dfrac{n!}{2^{k} k! (n-2k)!}
V_{A^{\rm gen, +}_{m-1}}^{k}
\psi^{\rm gen}_{n-2k,m}
=
\\
\sum_{j=0}^{+\infty}
\Big(
\Big(
\sum_{0\leq k< \lfloor n/2\rfloor}
(-1)^{k} \dfrac{n!}{2^{k} k! (n-2k)!}
V_{A^{\rm gen}_{m}}^{k}\psi_{n-2k,j}
\Big)
+
\ind_{\inter(\Gamma_{m,j})}
(-1)^{n/2} \dfrac{n!}{2^{n/2} (n/2)! }
(V_{A^{\rm gen}_{m}}^{n/2} - V_{A^{\rm gen, +}_{m-1}}^{n/2})
\Big)
.
\end{multline*}
\end{thm}

\begin{proof}
Note that the exact form of the decompositions originates from Proposition \ref{Prop gen to indiv}.
We will focus on the case $m=0$, the case $m\geq 1$ being similar.
We will condition everything on $\mathcal{F}^{\rm gen}_{0}$,
that is to say the knowledge of all the outer boundaries $\Gamma_{m=0,j}$ and the signs $\sigma_{m=0,j}$.
Then the conditional expectation
\begin{displaymath}
\E\Big[\Vert \psi^{\rm gen}_{n,m=0}\Vert_{H^{-\eta}(\C)}^{2}\Big\vert \mathcal{F}^{\rm gen}_{0}\Big]
\end{displaymath}
is a linear combination of double integrals on $D\setminus A^{\rm gen}_{0}$
w.r.t. the kernels $\LK_{\eta}(\vert w-z\vert)G_{D\setminus A^{\rm gen}_{0}}(z,w)^{q}$
for $q\in\{0,1,\dots,n\}$.
We will not give the heavy exact expression, and just focus on the structure.
If $q\geq 1$, then $G_{D\setminus A^{\rm gen}_{0}}(z,w)^{q}=0$
whenever $z$ and $w$ belong to different connected components of $D\setminus A^{\rm gen}_{0}$.
For $q=0$, by convention, $G_{D\setminus A^{\rm gen}_{0}}(z,w)^{0}\equiv 1$.
From this we get that for $n$ odd,
\begin{multline*}
\E\Big[\Vert \psi^{\rm gen}_{n,m=0}\Vert_{H^{-\eta}(\C)}^{2}\Big\vert \mathcal{F}^{\rm gen}_{0}\Big] =
\\ 
\sum_{j=0}^{+\infty}
\E\Big[
\Big\Vert
\sum_{0\leq k< \lfloor n/2\rfloor}
(-1)^{k} \dfrac{n!}{2^{k} k! (n-2k)!}
V_{A^{\rm gen}_{0}}^{k}\psi_{n-2k,m=0,j}
\Big\Vert_{H^{-\eta}(\C)}^{2}
\Big\vert \mathcal{F}^{\rm gen}_{0}
\Big]
\\
+
2\sum_{0\leq j<j'}
\sigma_{m=0,j}\sigma_{m=0,j'}
\Big\langle
\E\Big[
\sum_{0\leq k< \lfloor n/2\rfloor}
(-1)^{k} \dfrac{n!}{2^{k} k! (n-2k)!}
V_{A^{\rm gen}_{0}}^{k}\psi_{n-2k,m=0,j}
\Big\vert \mathcal{F}^{\rm gen}_{0}
\Big],
\\
\E\Big[
\sum_{0\leq k< \lfloor n/2\rfloor}
(-1)^{k} \dfrac{n!}{2^{k} k! (n-2k)!}
V_{A^{\rm gen}_{0}}^{k}\psi_{n-2k,m=0,j'}
\Big\vert \mathcal{F}^{\rm gen}_{0}
\Big]
\Big\rangle_{H^{-\eta}(\C)}.
\end{multline*}
Note that
\begin{displaymath}
\E\Big[
\sum_{0\leq k< \lfloor n/2\rfloor}
(-1)^{k} \dfrac{n!}{2^{k} k! (n-2k)!}
V_{A^{\rm gen}_{0}}^{k}\psi_{n-2k,m=0,j}
\Big\vert \mathcal{F}^{\rm gen}_{0}
\Big]
=\ind_{\inter(\Gamma_{m=0,j})}Q_{n}(2\lambda, V_{A^{\rm gen}_{0}}),
\end{displaymath}
and in general this is not $0$.
Further, the double sum
\begin{displaymath}
\sum_{0\leq j<j'}
\sigma_{m=0,j}\sigma_{m=0,j'}
\Big\langle
\ind_{\inter(\Gamma_{m=0,j})}Q_{n}(2\lambda, V_{A^{\rm gen}_{0}}),
\ind_{\inter(\Gamma_{m=0,j'})}Q_{n}(2\lambda, V_{A^{\rm gen}_{0}})
\Big\rangle_{H^{-\eta}(\C)}
\end{displaymath}
is absolutely convergent.
Indeed, given exponents $q, q'\geq 0$,
\begin{multline*}
2\sum_{0\leq j<j'}
\Big\langle
\ind_{\inter(\Gamma_{m=0,j})}V_{A^{\rm gen}_{0}}^{q},
\ind_{\inter(\Gamma_{m=0,j})}V_{A^{\rm gen}_{0}}^{q'}
\Big\rangle_{H^{-\eta}(\C)}
\\
\leq
\big\langle
\ind_{D\setminus A^{\rm gen}_{0}}V_{A^{\rm gen}_{0}}^{q},
\ind_{D\setminus A^{\rm gen}_{0}}V_{A^{\rm gen}_{0}}^{q'}
\big\rangle_{H^{-\eta}(\C)}
\leq 
\big\langle
\ind_{D\setminus A^{\rm gen}_{0}}V_{A^{\rm gen}_{0}}^{2q},
1
\big\rangle_{H^{-\eta}(\C)}^{1/2}
\big\langle
\ind_{D\setminus A^{\rm gen}_{0}}V_{A^{\rm gen}_{0}}^{2q'},
1
\big\rangle_{H^{-\eta}(\C)}^{1/2}
\end{multline*}
Further, by conformal invariance, the law of $V_{A^{\rm gen}_{0}}(z)^{2q}$
does not depend on $z\in D$,
and it is given by Theorem \ref{Thm ASW CR ell a b}.
So $\E[V_{A^{\rm gen}_{0}}(z)^{2q}]$ is a finite constant.

Consider now the smaller sigma-algebra induced by $A^{\rm gen}_{0}$.
It contains the outer boundaries $\Gamma_{m=0,j}$, but not the signs $\sigma_{m=0,j}$.
Then, for $j\neq j'$,
\begin{displaymath}
\E\big[\sigma_{m=0,j}\sigma_{m=0,j'}\big\vert A^{\rm gen}_{0}\big]=0.
\end{displaymath}
Therefore,
\begin{multline*}
\E\Big[\Vert \psi^{\rm gen}_{n,m=0}\Vert_{H^{-\eta}(\C)}^{2}\Big\vert A^{\rm gen}_{0}\Big] =
\\ 
\sum_{j=0}^{+\infty}
\E\Big[
\Big\Vert
\sum_{0\leq k< \lfloor n/2\rfloor}
(-1)^{k} \dfrac{n!}{2^{k} k! (n-2k)!}
V_{A^{\rm gen}_{0}}^{k}\psi_{n-2k,m=0,j}
\Big\Vert_{H^{-\eta}(\C)}^{2}
\Big\vert A^{\rm gen}_{0}
\Big]
.
\end{multline*}
So this gives the orthogonal decomposition in square-norms in
$L^{2}(d\PP,\sigma(\Phi),H^{-\eta}(\C))$.
Further, we also have
\begin{multline*}
\E\Big[\Big\Vert \psi^{\rm gen}_{n,m=0}
-\sum_{j=0}^{j_0}
\Big(
\sum_{0\leq k< \lfloor n/2\rfloor}
(-1)^{k} \dfrac{n!}{2^{k} k! (n-2k)!}
V_{A^{\rm gen}_{0}}^{k}\psi_{n-2k,m=0,j}
\Big)
\Big\Vert_{H^{-\eta}(\C)}^{2}\Big\vert A^{\rm gen}_{0}\Big] =
\\ 
\sum_{j=j_0 +1}^{+\infty}
\E\Big[
\Big\Vert
\sum_{0\leq k< \lfloor n/2\rfloor}
(-1)^{k} \dfrac{n!}{2^{k} k! (n-2k)!}
V_{A^{\rm gen}_{0}}^{k}\psi_{n-2k,m=0,j}
\Big\Vert_{H^{-\eta}(\C)}^{2}
\Big\vert A^{\rm gen}_{0}
\Big]
.
\end{multline*}
The right-hand side converges a.s. to $0$ when $j_0\to +\infty$.

The case of $n$ even is somewhat different.
We have that
\begin{multline*}
\E\Big[
\Big(
\sum_{0\leq k < \lfloor n/2\rfloor}
(-1)^{k} \dfrac{n!}{2^{k} k! (n-2k)!}
V_{A^{\rm gen}_{0}}^{k}\psi_{n-2k,m=0,j}
\Big)
+
(-1)^{n/2} \dfrac{n!}{2^{n/2} (n/2)! }
\ind_{\inter(\Gamma_{m=0,j})}
V_{A^{\rm gen}_{0}}^{n/2}
\Big\vert 
\mathcal{F}^{\rm gen}_{0}
\Big]
\\
=\ind_{\inter(\Gamma_{m=0,j})}Q_{n}(2\lambda, V_{A^{\rm gen}_{0}}),
\end{multline*}
and
\begin{multline*}
\E\Big[\Big\Vert \psi^{\rm gen}_{n,m=0}
-
\sum_{j=0}^{j_0}
\Big(
\Big(
\sum_{0\leq k < \lfloor n/2\rfloor}
(-1)^{k} \dfrac{n!}{2^{k} k! (n-2k)!}
V_{A^{\rm gen}_{0}}^{k}\psi_{n-2k,m=0,j}
\Big)
\\
+
(-1)^{n/2} \dfrac{n!}{2^{n/2} (n/2)! }
\ind_{\inter(\Gamma_{m=0,j})}
V_{A^{\rm gen}_{0}}^{n/2}
\Big)
\Big\Vert_{H^{-\eta}(\C)}^{2}\Big\vert \mathcal{F}^{\rm gen}_{0}\Big]
\\
=
\sum_{j=j_0 + 1}^{+\infty}
\E\Big[\Big\Vert 
\Big(
\sum_{0\leq k < \lfloor n/2\rfloor}
(-1)^{k} \dfrac{n!}{2^{k} k! (n-2k)!}
V_{A^{\rm gen}_{0}}^{k}\psi_{n-2k,m=0,j}
\Big)
\\
+
(-1)^{n/2} \dfrac{n!}{2^{n/2} (n/2)! }
\ind_{\inter(\Gamma_{m=0,j})}
V_{A^{\rm gen}_{0}}^{n/2}
\Big\Vert_{H^{-\eta}(\C)}^{2}\Big\vert \mathcal{F}^{\rm gen}_{0}\Big]
\\
+
2 \sum_{j_0 + 1 \leq j<j'}
\Big\langle
\ind_{\inter(\Gamma_{m=0,j})}Q_{n}(2\lambda, V_{A^{\rm gen}_{0}}),
\ind_{\inter(\Gamma_{m=0,j'})}Q_{n}(2\lambda, V_{A^{\rm gen}_{0}})
\Big\rangle_{H^{-\eta}(\C)}.
\end{multline*}
Again this converges a.s. to $0$ as $j_{0}\to +\infty$.
\end{proof}

\begin{rem}
\label{Rem expect 0}
At the level of expectations,
for odd powers $n\geq 1$,
and for $z\in \inter(\Gamma_{m=0,j})$,
\begin{displaymath}
\E[\psi^{\rm gen}_{n,m=0}\vert\mathcal{F}^{\rm gen}_{0}](z)
=
\sigma_{m=0,j}
\ind_{\inter(\Gamma_{m=0,j})}Q_{n}(2\lambda, V_{A^{\rm gen}_{0}}).
\end{displaymath}
For even powers $n\geq 2$,
\begin{displaymath}
\E[\psi^{\rm gen}_{n,m=0}\vert\mathcal{F}^{\rm gen}_{0}]
=
\ind_{D\setminus A^{\rm gen}_{0}}Q_{n}(2\lambda, V_{A^{\rm gen}_{0}}).
\end{displaymath}
But 
\begin{displaymath}
\E[\psi^{\rm gen}_{n,m=0}]=
\E\big[\,:\Phi^{n}:\,\big]
=0,
\end{displaymath}
for every $n\geq 1$, odd or even.
For $n$ odd this is related to $\E[\sigma_{m=0,j}]=0$,
for every $j\geq 0$.
For $n\geq 2$ even, it implies that for every $z\in D$,
\begin{equation}
\label{Eq Q n 2 lambda 0 expect}
\E[Q_{n}(2\lambda, V_{A^{\rm gen}_{0}}(z))] = 0,
\end{equation}
where actually it was \textit{a priori} clear that the expectation is a constant not depending on $z$, just by conformal invariance.
Now, is there a way to see the relation \eqref{Eq Q n 2 lambda 0 expect},
without going through Wick powers of the GFF?
The answer is yes.
According to Theorem \ref{Thm ASW CR ell a b},
the r.v. $V_{A^{\rm gen}_{0}}(z)$ is distributed as
$T_{-2\lambda,2\lambda}$,
the first exit time from the interval
$(-2\lambda,2\lambda)$ by a standard Brownian motion
$(W_{t})_{t\geq 0}$ starting from $W_{0}=0$.
Thus,
\begin{displaymath}
\E[Q_{n}(2\lambda, V_{A^{\rm gen}_{0}}(z))]
=
\E_{0}[Q_{n}(2\lambda, T_{-2\lambda,2\lambda})].
\end{displaymath}
Since the Hermite polynomials of even degree are even, 
this is also the same as
\begin{displaymath}
\E_{0}[Q_{n}(W_{T_{-2\lambda,2\lambda}}, T_{-2\lambda,2\lambda})].
\end{displaymath}
But the stochastic process
$(Q_{n}(W_{t},t))_{t\geq 0}$ is actually a martingale.
For instance, for $n=2$,
$Q_{2}(W_{t},t) = W_{t}^{2} -t$,
which is well known to be a martingale,
but this also holds for all higher degrees.
See for instance \cite[Lemma 9.7]{JegoLupuQianFields}.
By the optional stopping theorem,
\begin{displaymath}
\E_{0}[Q_{n}(W_{T_{-2\lambda,2\lambda}}, T_{-2\lambda,2\lambda})]
=
\E_{0}[Q_{n}(W_{0}, 0)] = 0.
\end{displaymath}
\end{rem}

\begin{rem}
\label{Rem orth even}
A natural questions to ask is whether for $n$ even,
there is equality between
\begin{displaymath}
\E\Big[\Vert \psi^{\rm gen}_{n,m=0}\Vert_{H^{-\eta}(\C)}^{2}\Big]
\end{displaymath}
and
\begin{multline*}
\sum_{j=0}^{+\infty}
\E\Big[
\Big\Vert
\Big(
\sum_{0\leq k < \lfloor n/2\rfloor}
(-1)^{k} \dfrac{n!}{2^{k} k! (n-2k)!}
V_{A^{\rm gen}_{0}}^{k}\psi_{n-2k,m=0,j}
\Big)
\\
+
(-1)^{n/2} \dfrac{n!}{2^{n/2} (n/2)! }
\ind_{\inter(\Gamma_{m=0,j})}
V_{A^{\rm gen}_{0}}^{n/2}
\Big\Vert_{H^{-\eta}(\C)}^{2}
\Big],
\end{multline*}
and similarly for the generations $m\geq 1$.
The author does not have an answer to this question,
but sees no reason why such a thing should be true.
The issue is the term
\begin{displaymath}
2 \sum_{0 \leq j<j'}
\Big\langle
\ind_{\inter(\Gamma_{m=0,j})}Q_{n}(2\lambda, V_{A^{\rm gen}_{0}}),
\ind_{\inter(\Gamma_{m=0,j'})}Q_{n}(2\lambda, V_{A^{\rm gen}_{0}})
\Big\rangle_{H^{-\eta}(\C)}.
\end{displaymath}
Its expectation can be rewritten as
\begin{displaymath}
\int_{D^{2}}
\E\big[Q_{n}(2\lambda, V_{A^{\rm gen}_{0}}(z))Q_{n}(2\lambda, V_{A^{\rm gen}_{0}}(w))
\ind_{z,w \text{ in different con. comp. of } D\setminus A^{\rm gen}_{0}}
\big]
\LK_{\eta}(\vert w-z\vert)\,d^{2}z\,d^{2}w.
\end{displaymath}
So a related question is whether for all $z\neq w\in D$,
the two-point function
\begin{equation}
\label{Eq 2 point trunc}
\E\big[Q_{n}(2\lambda, V_{A^{\rm gen}_{0}}(z))Q_{n}(2\lambda, V_{A^{\rm gen}_{0}}(w))
\ind_{z,w \text{ in different con. comp. of } D\setminus A^{\rm gen}_{0}}
\big]
\end{equation}
is $0$. 
It is true that for every $z\in D$,
\begin{displaymath}
\E[Q_{n}(2\lambda, V_{A^{\rm gen}_{0}}(z))] = 0.
\end{displaymath}
See Remark \ref{Rem expect 0}.
However, the author does not see any good reason for
\eqref{Eq 2 point trunc} to be $0$ in general.
\end{rem}

By combining Theorem \ref{Thm decomp gen Wick} and Theorem \ref{Thm per cluster enum},
we get the following.

\begin{cor}
\label{Cor decomp gen plus enum}
For $n\geq 1$ odd,
the Wick power $:\Phi^{n}:$ can be decomposed
\begin{displaymath}
:\Phi^{n}:~= 
\sum_{m=0}^{+\infty}
\sum_{j=0}^{+\infty}
\Big(
\sigma_{m,j}
\sum_{0\leq k< \lfloor n/2\rfloor}
(-1)^{k} \dfrac{n!}{2^{k} k! (n-2k)!}
V_{A^{\rm gen}_{m}}^{k}\psi_{n-2k,m,j}
\Big),
\end{displaymath}
and 
\begin{displaymath}
\E\big[\big\Vert :\Phi^{n}:\big\Vert_{H^{-\eta}(\C)}^{2}\big]~= 
\sum_{m=0}^{+\infty}
\sum_{j=0}^{+\infty}
\E\Big[\Big\Vert
\sum_{0\leq k< \lfloor n/2\rfloor}
(-1)^{k} \dfrac{n!}{2^{k} k! (n-2k)!}
V_{A^{\rm gen}_{m}}^{k}\psi_{n-2k,m,j}
\Big\Vert_{H^{-\eta}(\C)}^{2}
\Big],
\end{displaymath}
for every $\eta>0$.
For $n\geq 2$ even,
the Wick power $:\Phi^{n}:$ can be decomposed
\begin{multline*}
:\Phi^{n}:~= 
\sum_{m=0}^{+\infty}
\sum_{j=0}^{+\infty}
\Big(
\Big(
\sum_{0\leq k< \lfloor n/2\rfloor}
(-1)^{k} \dfrac{n!}{2^{k} k! (n-2k)!}
V_{A^{\rm gen}_{m}}^{k}\psi_{n-2k,j}
\Big)
\\
+
\ind_{\inter(\Gamma_{m,j})}
(-1)^{n/2} \dfrac{n!}{2^{n/2} (n/2)! }
(V_{A^{\rm gen}_{m}}^{n/2} - V_{A^{\rm gen, +}_{m-1}}^{n/2})
\Big),
\end{multline*}
with the convention that for $m=0$, 
$V_{A^{\rm gen, +}_{-1}} = 0$.
\end{cor}

\subsection{A side remark: decomposition of Wick powers through nested CLE$_4$ and Miller-Sheffield coupling}
\label{Subsec remark nestes CLE 4}

Although this is not the focus of this work,
we would like to mention how the Wick powers fit into
the Miller-Sheffield description of the GFF through the nested CLE$_{4}$
and the labels performing a branching random walk with step $2\lambda$ on top of this nested structure.
So here again $\Phi$ will be a GFF with $0$ boundary condition.
Essentially, we are looking at the conditional expectation of $:\Phi^{n}:$
given the $k$-th generation CLE$_{4}$
(not to be confused with the $m$-th generation of excursion clusters)
and the corresponding labels.
Now, the gasket of $k$-th generation CLE$_{4}$ is a \textbf{thin}
local set of the GFF,
which induces a much simpler and less rich decomposition of $:\Phi^{n}:$
than in the \textbf{non-thin} case of first passage sets and excursion clusters.

A framework that can fit both the fixed generation CLE$_{4}$
and that of two-valued sets as in Section \ref{Subsec Wick TVS},
is that of \textit{bounded-type \textbf{thin} local sets} (BTLS)
introduced in \cite{ASW}.
So let $\mathcal{A}$ be such a BTLS,
which in particular is a random compact subset of
$\overline{D}$,
ling on the same probability space as $\Phi$,
but not necessarily measurable w.r.t. $\sigma(\Phi)$.
To simplify, we will additionally assume that $\mathcal{A}$ is a.s. connected and contains $\partial D$.
One can also deal with the case of $\mathcal{A}$ not being connected,
but this would require dealing with then non-simply connected domain
$D\setminus\mathcal{A}$, which we would like to avoid.
So $\mathcal{A}$ is a.s. connected and contains $\partial D$.
The GFF $\Phi$ admits then a decomposition
\begin{displaymath}
\Phi = h_{\mathcal{A}} + \Phi_{D\setminus\mathcal{A}},
\end{displaymath}
where
\begin{itemize}
\item $h_{\mathcal{A}}$ is a random harmonic function on $D\setminus\mathcal{A}$,
with moreover $\vert h_{\mathcal{A}}\vert$ being deterministically bounded;
\item conditionally on $(\mathcal{A}, h_{\mathcal{A}})$,
the field $\Phi_{D\setminus\mathcal{A}}$ is distributed as a GFF on
$D\setminus\mathcal{A}$ with $0$ boundary conditions.
\end{itemize}
We will denote by $:\Phi_{D\setminus\mathcal{A}}^{n}:$
the Wick powers of $\Phi_{D\setminus\mathcal{A}}$ renormalized
with the Green's function $G_{D\setminus\mathcal{A}}$,
and not the Green's function $G_{D}$.
Let $V_{\mathcal{A}}$ be the following function on $D\setminus\mathcal{A}$:
\begin{displaymath}
V_{\mathcal{A}}(z) = \dfrac{1}{2\pi}\log\Big(
\dfrac{\CR(z,D)}{\CR(z,D\setminus \mathcal{A})}
\Big).
\end{displaymath}
Next we state the decomposition of $:\Phi^{n}:$ via the BTLS
$\mathcal{A}$.
Its proof is straightforward given what we have done previously, so we omit it.
In particular, we use Lemma \ref{Lem Q n x y u v}.

\begin{prop}
\label{Prop decomp Wick BTLS}
With the notations above,
for every $n\geq 1$, the following decomposition holds:
\begin{displaymath}
:\Phi^{n}: ~=~
\sum_{j=0}^{n} \dfrac{n!}{j! (n-j)!} Q_{n-j}(h_{\mathcal{A}},V_{\mathcal{A}})\,:\Phi_{D\setminus \mathcal{A}}^{j} : \, ,
\end{displaymath}
In particular,
\begin{displaymath}
\E\big[ :\Phi^{n}: \big\vert \mathcal{A}, h_{\mathcal{A}}\big]
= Q_{n}(h_{\mathcal{A}},V_{\mathcal{A}}).
\end{displaymath}
\end{prop}

Note that the case of two-valued sets 
(Proposition \ref{Prop decomp Wick TVS} in Section \ref{Subsec Wick TVS})
is just a special case of the above Proposition \ref{Prop decomp Wick BTLS}.
Another special case is given by the $k$-th generation CLE$_{4}$,
which we will denote by $\mathcal{A}_{k}$.
Let $\ell_{k}$ be the corresponding label function
on $D\setminus \mathcal{A}_{k}$,
which takes values in
$\{-2\lambda k, -2\lambda (k-2),\dots, 2\lambda (k-2), 2\lambda k\}$.
With the notations above, $\ell_{k}$ corresponds to $h_{\mathcal{A}_{k}}$.

\begin{cor}
\label{Cor decomp Wick nested CLE 4}
For every $k\geq 0$ and $n\geq 1$, the following decomposition holds:
\begin{equation}
\label{Eq decomp Wick nested CLE 4}
:\Phi^{n}: ~=~
\sum_{j=0}^{n} \dfrac{n!}{j! (n-j)!} Q_{n-j}(\ell_{k},V_{\mathcal{A}_{k}})\,:\Phi_{D\setminus \mathcal{A}_{k}}^{j} : \, ,
\end{equation}
In particular,
\begin{displaymath}
\E\big[ :\Phi^{n}: \big\vert \mathcal{A}_{k}, \ell_{k}\big]
= Q_{n}(\ell_{k},V_{\mathcal{A}_{k}}).
\end{displaymath}
Further, for every $\eta >0$,
\begin{equation}
\label{Eq Wick limit nested CLE 4}
\lim_{k\to +\infty}
\E\big[
\big\Vert
\,:\Phi^{n}:\, - Q_{n}(\ell_{k},V_{\mathcal{A}_{k}})
\big\Vert_{H^{-\eta}(\C)}^{2}
\big] = 0.
\end{equation}
\end{cor}

\begin{proof}
The only not entirely straightforward point is 
\eqref{Eq Wick limit nested CLE 4}.
One could proceed with explicit estimates applied to the decomposition
\eqref{Eq decomp Wick nested CLE 4}
and argue that for $k$ large and $j\geq 1$,
$:\Phi_{D\setminus \mathcal{A}_{k}}^{j} :$ is sufficiently small in a sense
to compensate for $Q_{n-j}(\ell_{k},V_{\mathcal{A}_{k}})$
taking increasingly large values.
However, there is a more abstract and direct argument that circumvents all that.
The sequence of sigma-algebras
$(\sigma(\mathcal{A}_{k}, \ell_{k}))_{k\geq 0}$
is increasing and forms a filtration.
Moreover,
\begin{displaymath}
\bigvee_{k\geq 0} \sigma(\mathcal{A}_{k}, \ell_{k}) = \sigma(\Phi),
\end{displaymath}
up to negligible events (i.e. with probability $0$).
Therefore, \eqref{Eq Wick limit nested CLE 4} is just a convergence of $L^2$ closed martingales with values in 
the Sobolev space $H^{-\eta}(\C)$.
\end{proof}

\section{Conjectures and open questions}
\label{Sec open}

\subsection{The expansion of first passage sets for the Euclidean distance}
\label{Subsec FPS Euclid}

A natural question is what form takes the asymptotic expansion for first passage set (FPS)
(see Theorem \ref{Thm A E FPS} in Section \ref{Subsec Pres A E})
when one considers $\varepsilon$-neighborhoods
defined through Euclidean distance rather then conformal radius.
Here we will formulate a precise conjecture for this,
but at this stage we lack ingredients to actually prove it.
Our conjecture involves the Malliavin–Kontsevich–Suhov (MKS) measure on
SLE$_{4}$ loops,
as it is supposed to describe, up to scaling, the microscopic holes of an FPS.

So we consider Jordan loops $\Gamma$ in $\C$ that surround $0$,
that is to say separate $0$ from $\infty$.
A natural family of measure on such loops is given by the 
SLE$_{\kappa}$ MKS measures
for $\kappa\in (0,4]$,
which we will denote by
$\mu^{\rm loop}_{\text{SLE}_{\kappa}}$.
For references on MKS measures, see \cite{AiraultMalliavin01Unitarizing,KontsevichSuhov07MKS,KemppainenWerner16CLE,Zhan21SLEloop,BaverezJego24MKS}.
The measure $\mu^{\rm loop}_{\text{SLE}_{\kappa}}$ is scale invariant and has infinite total mass.
Given the scale invariance, it is natural to consider the Jordan loops $\Gamma$
up to change of scale, that is to say to identify
$\Gamma$ and $r\Gamma$ for any $r>0$.
We will denote by $[\Gamma]$ the corresponding equivalence class.
The infinite measure $\mu^{\rm loop}_{\text{SLE}_{\kappa}}$
can be factorized into a product measure:
a measure on a scaling factor (for instance $\CR(0,\inter(\Gamma))$ or $d(0,\Gamma)$)
times a \textbf{probability} measure on $[\Gamma]$.
So we have
\begin{displaymath}
\mu^{\rm loop}_{\text{SLE}_{\kappa}}
(\CR(0,\inter(\Gamma))\in (r, r+dr), d[\Gamma])
= C_{\kappa} \dfrac{dr}{r}\,\times\,\hat{\mu}^{\rm loop}_{\text{SLE}_{\kappa}}(d[\Gamma]),
\end{displaymath}
where $C_{\kappa}>0$ is a constant depending on $\kappa$,
and $\hat{\mu}^{\rm loop}_{\text{SLE}_{\kappa}}$ is a probability measure on $[\Gamma]$.
For our conjecture we will need the measure $\hat{\mu}^{\rm loop}_{\text{SLE}_{4}}$.

Now consider $\Phi$ a GFF with boundary condition $V>0$ and $A$ the FPS of $\Phi$ up to level $0$,
as in Sections \ref{Sec Wick FPS} and \ref{Sec A E}.
Consider the following $\varepsilon$-neighborhoods defined in terms of the Euclidean distance:
\begin{displaymath}
\Ns_{\varepsilon}^{\rm Eucl}(A) = \{ z\in D\setminus A \vert d(z,A)<\varepsilon d(z,\partial D)\},
\qquad
\widetilde{\Ns}^{\rm Eucl}_{\varepsilon}(A)
= \{ z\in D\setminus A \vert d(z,A)<\varepsilon\}.
\end{displaymath}
Let us first explain how the measure $\hat{\mu}^{\rm loop}_{\text{SLE}_{4}}$
appears in this context.
For $z\in D\setminus A$, we will denote by $\Gamma_{z}$
the boundary of the connected component of $z\in D\setminus A$,
which we translated by $-z$. 
In other words, we centered everything by sending $z$ to $0$.
In this way, $\Gamma_{z}$ is a Jordan loop in $\C$ separating $0$ from $\infty$.
We will denote by $[\Gamma_{z}]$ the equivalence class of $\Gamma_{z}$ under scaling.
Now, take $n\geq 1$ and consider $Z_{1,\varepsilon}, Z_{2,\varepsilon},\dots, Z_{n,\varepsilon}$
random points in $\widetilde{\Ns}^{\rm Eucl}_{\varepsilon}(A)$,
sampled, conditionally on $A$, in an i.i.d. way according to the probability measure
\begin{displaymath}
\dfrac{\ind_{z\in \widetilde{\Ns}^{\rm Eucl}_{\varepsilon}(A)}\,d^{2}z}
{\operatorname{Leb}(\widetilde{\Ns}^{\rm Eucl}_{\varepsilon}(A))}.
\end{displaymath}

\begin{conj}
\label{Conj MKS SLE 4}
Let $n\geq 1$.
For $\varepsilon>0$, consider the family of random variables
\begin{displaymath}
(A, (Z_{j,\varepsilon},[\Gamma_{Z_{j,\varepsilon}}])_{1\leq j\leq n}).
\end{displaymath}
As $\varepsilon\to 0$, this family converges in law towards
\begin{displaymath}
(A,(Z_{j},[\Gamma_{j}])_{1\leq j\leq n}),
\end{displaymath}
with the following properties.
\begin{itemize}
\item $A$ is the same FPS as previously.
\item $(Z_{j})_{1\leq j\leq n}$ are random points on $A$,
sampled, conditionally on $A$, in an i.i.d. way according to the probability measure
\begin{displaymath}
\dfrac{\nu_{A}}{\nu_{A}(A)},
\end{displaymath}
where $\nu_{A}$ is the FPS measure \eqref{Eq nu a Mink non norm}.
\item The family of equivalence classes of loops
$([\Gamma_{j}])_{1\leq j\leq n}$
is \textbf{independent} from
$(A,(Z_{j})_{1\leq j\leq n})$,
with an i.i.d. distribution according to the measure
$\hat{\mu}^{\rm loop}_{\text{SLE}_{4}}$.
\end{itemize}
\end{conj}

So, whereas for $\varepsilon>0$,
the loops $([\Gamma_{Z_{j,\varepsilon}}])_{1\leq j\leq n}$
are measurable w.r.t. $(A,(Z_{j,\varepsilon})_{1\leq j\leq n})$,
in the limit $\varepsilon\to 0$ we believe that they become independent.
The fact that the law of $[\Gamma_{Z_{j,\varepsilon}}]$
converges to $\hat{\mu}^{\rm loop}_{\text{SLE}_{4}}$
should be doable by using ideas similar to
Kemppainen-Werner \cite[Corollary 8]{KemppainenWerner16CLE}.
However, the asymptotic independence of loops from
the FPS and the points is the trickier part of
Conjecture \ref{Conj MKS SLE 4}.

Given $[\Gamma]$ an equivalence class, the ratio
\begin{equation}
\label{Eq ratio CR dist}
\dfrac{\CR(0,\inter{\Gamma})}{d(0,\Gamma)}
\end{equation}
does not dependent on the representative $\Gamma$ of the class $[\Gamma]$,
so it is a function of $[\Gamma]$.
Indeed, the conformal radius and the Euclidean distance scale in the same way.
So we can consider the ratio \eqref{Eq ratio CR dist}
as a random variable under the probability measure $\hat{\mu}^{\rm loop}_{\text{SLE}_{4}}$.
Given Conjecture \ref{Conj MKS SLE 4},
we believe it to be the limit law of
\begin{displaymath}
\dfrac{\CR(Z_{j,\varepsilon},D\setminus A)}{d(Z_{j,\varepsilon},A)}
\end{displaymath}
as $\varepsilon\to 0$,
with moreover asymptotic independence from
$A$ and $Z_{j,\varepsilon}$.

For $z\in D$ and $n\geq 1$, consider the following expectations w.r.t. $\hat{\mu}^{\rm loop}_{\text{SLE}_{4}}$:
\begin{eqnarray*}
F^{\text{SLE}_{4}}_{n}(z) &=& 
\hat{\mu}^{\rm loop}_{\text{SLE}_{4}}
\Big(\dfrac{1}{2\pi}
\Big(
\log
\Big(
\dfrac{\CR(0,\inter{\Gamma})}{d(0,\Gamma)}
\dfrac{d(z,\partial D)}{\CR(z,D)}
\Big)
\Big)^{n}\Big)
\\
&=&
\sum_{j=0}^{n}\dfrac{n!}{j!(n-j)!}
\hat{\mu}^{\rm loop}_{\text{SLE}_{4}}
\Big(\dfrac{1}{2\pi}
\Big(
\log
\Big(
\dfrac{\CR(0,\inter{\Gamma})}{d(0,\Gamma)}
\Big)
\Big)^{j}\Big)
\Big(\dfrac{1}{2\pi}
\Big(
\log
\Big(
\dfrac{d(z,\partial D)}{\CR(z,D)}
\Big)
\Big)^{n-j},
\end{eqnarray*}
\begin{eqnarray*}
\widetilde{F}^{\text{SLE}_{4}}_{n}(z) &=& 
\hat{\mu}^{\rm loop}_{\text{SLE}_{4}}
\Big(\dfrac{1}{2\pi}
\Big(
\log
\Big(
\dfrac{\CR(0,\inter{\Gamma})}{d(0,\Gamma)}
\dfrac{1}{\CR(z,D)}
\Big)
\Big)^{n}\Big)
\\
&=&
\sum_{j=0}^{n}\dfrac{(-1)^{n-j}n!}{j!(n-j)!}
\hat{\mu}^{\rm loop}_{\text{SLE}_{4}}
\Big(\dfrac{1}{2\pi}
\Big(
\log
\Big(
\dfrac{\CR(0,\inter{\Gamma})}{d(0,\Gamma)}
\Big)
\Big)^{j}\Big)
\Big(\dfrac{1}{2\pi}
\Big(
\log
\Big(
\CR(z,D)
\Big)
\Big)^{n-j}.
\end{eqnarray*}
As objects, $F^{\text{SLE}_{4}}_{n}$ and $\widetilde{F}^{\text{SLE}_{4}}_{n}$
are continuous deterministic functions on $D$.
By convention, we set
$F^{\text{SLE}_{4}}_{0}(z) = \widetilde{F}^{\text{SLE}_{4}}_{0} (z) =1$.
Recall the fields $\psi_{2k+1,A}$ as in Sections \ref{Sec Wick FPS} and \ref{Sec A E},
which are the restrictions of the odd Wick powers
$:\Phi^{2k+1}:$ to $A$.

\begin{conj}
\label{Conj expansion FPS Eucl}
Fix $\eta>0$.
The following expansions hold in $L^{2}(d\PP,\sigma(\Phi),H^{-\eta}(\C))$.
For every $N\geq 0$,
\begin{align*}
&\ind_{\Ns_{\varepsilon}^{\rm Eucl}(A)}
\\
&=
\dfrac{1}{\sqrt{2\pi}}
\sum_{k=0}^{N} 
\dfrac{(-1)^{k}}{2^{k} k! (k+1/2)}
\hat{\mu}^{\rm loop}_{\text{SLE}_{4}}
\Big(
\Big(\frac{1}{2\pi}\Big\vert\log \Big(\dfrac{\varepsilon\CR(0,\inter{\Gamma})d(z,\partial D)}
{d(0,\Gamma)\CR(z,D)}\Big)\Big\vert\Big)^{-(k + 1/2)}
\Big)
\psi_{2k+1,A}
\\&+o(\vert \log\varepsilon\vert^{-(N+1/2)})
\\
&=
\dfrac{1}{\sqrt{2\pi}}
\sum_{k=0}^{N} 
\dfrac{(-1)^{k}}{2^{k} k! (k+1/2)}
\dfrac{1}{\big(\frac{1}{2\pi}\vert\log\varepsilon\vert\big)^{k+1/2}}
\sum_{j=0}^{k} \dfrac{(-1)^{j}(2k+1)!}{2^{j} j! (2(k-j)+1)!}
\,F^{\text{SLE}_{4}}_{j}\,\psi_{2(k-j)+1,A}
\\&+o(\vert \log\varepsilon\vert^{-(N+1/2)}).
\end{align*}
Fix a deterministic smooth cutoff function $f_{0}:D\rightarrow [0,1]$, compactly supported in $D$.
Then  
\begin{align*}
&\ind_{\widetilde{\Ns}_{\varepsilon}^{\rm Eucl}(A)}f_{0}
\\
&=
\dfrac{1}{\sqrt{2\pi}}
\sum_{k=0}^{N} 
\dfrac{(-1)^{k}}{2^{k} k! (k+1/2)}
\hat{\mu}^{\rm loop}_{\text{SLE}_{4}}
\Big(
\Big(\frac{1}{2\pi}\Big\vert\log \Big(\dfrac{\varepsilon\CR(0,\inter{\Gamma})}
{d(0,\Gamma)\CR(z,D)}\Big)\Big\vert\Big)^{-(k + 1/2)}
\Big)
f_{0}\psi_{2k+1,A}
\\&+o(\vert \log\varepsilon\vert^{-(N+1/2)})
\\
&=
\dfrac{1}{\sqrt{2\pi}}
\sum_{k=0}^{N} 
\dfrac{(-1)^{k}}{2^{k} k! (k+1/2)}
\dfrac{1}{\big(\frac{1}{2\pi}\vert\log\varepsilon\vert\big)^{k+1/2}}
\sum_{j=0}^{k} \dfrac{(-1)^{j}(2k+1)!}{2^{j} j! (2(k-j)+1)!}
\,\widetilde{F}^{\text{SLE}_{4}}_{j}\,f_{0}\psi_{2(k-j)+1,A}
\\&+o(\vert \log\varepsilon\vert^{-(N+1/2)}).
\end{align*}
Further, assume that the condition \eqref{Eq cond small boundary} holds.
Then,
\begin{align*}
&\ind_{\widetilde{\Ns}_{\varepsilon}^{\rm Eucl}(A)} =
\\
&
\dfrac{1}{\sqrt{2\pi}}
\sum_{k=0}^{N} 
\dfrac{(-1)^{k}}{2^{k} k! (k+1/2)}
\dfrac{1}{\big(\frac{1}{2\pi}\vert\log\varepsilon\vert\big)^{k+1/2}}
\sum_{j=0}^{k} \dfrac{(-1)^{j}(2k+1)!}{2^{j} j! (2(k-j)+1)!}
\,\widetilde{F}^{\text{SLE}_{4}}_{j}\,\psi_{2(k-j)+1,A}
\\&+o(\vert \log\varepsilon\vert^{-(N+1/2)}).
\end{align*}
\end{conj}

\subsection{The expansion of Wiener sausages for the conformal radius}
\label{Subsec Wiener CR}

In Le Gall's expansion for the 2D Wiener sausage
(Theorem \ref{Thm Le Gall} in Section \ref{Subsec Le Gall sausage}), the $\varepsilon$-neighborhood of
the Brownian trajectory is defined in terms of the Euclidean distance.
To establish a full parallelism with the asymptotic expansion for first passage sets,
one may wonder what is the expansion for one Brownian trajectory in
the $\varepsilon$-neighborhood is defined in terms of the conformal radius.
We have a precise conjecture for this case, too.
Again it involves an MKS measure $\hat{\mu}^{\rm loop}_{\text{SLE}_{8/3}}$,
this time with $\kappa = 8/3$,
and again we lack ingredients to prove our conjecture.
Note that historically, the MKS measure $\mu^{\rm loop}_{\text{SLE}_{8/3}}$
was the first to be constructed,
by Werner in \cite{Werner08SLE_8_3_loops}, the $\text{SLE}_{8/3}$ loops appearing as outer boundaries of Brownian loops.

Here we will reuse the notations of Section \ref{Subsec Le Gall sausage}.
Let $n\geq 1$, $\varepsilon>0$, and let
$Z_{1,\varepsilon},Z_{2,\varepsilon},\dots, Z_{n,\varepsilon}$
be random points on the Wiener sausage $S_\varepsilon(\zeta_M)$,
which, conditionally on $B([0,\zeta_M])$, are i.i.d.
and sampled according to the probability measure
\begin{displaymath}
\dfrac{\ind_{z\in S_\varepsilon(\zeta_M)} d^{2}z}{\operatorname{Leb}(S_\varepsilon(\zeta_M))}.
\end{displaymath}
Since the dimension of the outer boundary of a $2D$ Brownian trajectory is
$4/3$ and the area of $S_\varepsilon(\zeta_M)$ is of order $1/\vert\log\varepsilon\vert$, 
a point $Z_{j,\varepsilon}$ will be disconnected by $B([0,\zeta_M])$
from $\infty$ with probability
$1-\varepsilon^{2/3 + o(1)}$.
We will consider $\Gamma_{Z_{j,\varepsilon}}$ the boundary of the connected component of
$Z_{j,\varepsilon}$ in $\C\setminus B([0,\zeta_M])$
(which is bounded with high probability)
translated by $-Z_{j,\varepsilon}$
(we center the picture by sending $Z_{j,\varepsilon}$ to $0$).
We will denote by $[\Gamma_{Z_{j,\varepsilon}}]$ the equivalence class of 
$Z_{j,\varepsilon}$ under scaling.
We first state an analogue of Conjecture \ref{Conj MKS SLE 4} in this context.

\begin{conj}
\label{Conj MKS SLE 8 3}
Let $n\geq 1$.
For $\varepsilon>0$, consider the family of random variables
\begin{displaymath}
(B([0,\zeta_M]), (Z_{j,\varepsilon},[\Gamma_{Z_{j,\varepsilon}}])_{1\leq j\leq n}).
\end{displaymath}
As $\varepsilon\to 0$, this family converges in law towards
\begin{displaymath}
(B([0,\zeta_M]),(Z_{j},[\Gamma_{j}])_{1\leq j\leq n}),
\end{displaymath}
with the following properties.
\begin{itemize}
\item $B([0,\zeta_M])$ is the same Brownian trajectory as previously.
\item $(Z_{j})_{1\leq j\leq n}$ are random points on $B([0,\zeta_M])$,
sampled, conditionally on $B([0,\zeta_M])$, in an i.i.d. way according to the probability measure
$\Theta_{\zeta_{M}}/\zeta_{M}$, 
where $\Theta_{\zeta_{M}}$ is the occupation measure \eqref{Eq occup meas}.
\item The family of equivalence classes of loops
$([\Gamma_{j}])_{1\leq j\leq n}$
is \textbf{independent} from
$(B([0,\zeta_M]),(Z_{j})_{1\leq j\leq n})$,
with an i.i.d. distribution according to the measure
$\hat{\mu}^{\rm loop}_{\text{SLE}_{8/3}}$.
\end{itemize}
\end{conj}

Consider now the $\varepsilon$-neighborhood defined in terms of the conformal radius:
\begin{displaymath}
S_{\varepsilon}^{\rm CR}(\zeta_M) = 
\{ z\in\C\vert \, 
z \text{ disconnected by } B([0,\zeta_M]) \text{ from } \infty,
\CR(z, \C\setminus B([0, \zeta_M]))<\varepsilon\}.
\end{displaymath}
We do not consider the points that are not disconnected by $B([0,\zeta_M])$ from
$\infty$ because the area of the 
$\varepsilon$-neighborhood in the unbounded connected component of
$\C\setminus B([0, \zeta_M])$
is of order $\varepsilon^{2/3 + o(1)}$,
which is negligible compared to any power of
$1/\vert\log\varepsilon\vert$.
Already in Le Gall's expansion one could remove from $S_\varepsilon(\zeta_M)$
the points that are not disconnected from $\infty$ and still get exactly the same
asymptotic expansion.

Let be the following constants,
obtained as expectations w.r.t. $[\Gamma]$ sampled according to $\hat{\mu}^{\rm loop}_{\text{SLE}_{8/3}}$:
\begin{displaymath}
C^{\text{SLE}_{8/3}}_{n,M}
=
\hat{\mu}^{\rm loop}_{\text{SLE}_{8/3}}
\Big(
\Big(
\dfrac{1}{\pi}
\log\Big(\dfrac{\CR(0,\inter{\Gamma})}{d(0,\Gamma)}\Big)
+ \cst (M)
\Big)^{n}
\Big).
\end{displaymath}
By convention, $C^{\text{SLE}_{8/3}}_{0,M}=1$.

\begin{conj}
\label{Conj Wiener sausage CR}
Fix $M>0$ and $\eta>0$.
The following asymptotic expansion holds in
$L^{2}(d\PP,\sigma((B_{t})_{0\leq t< \zeta_M}),H^{-\eta}(\C))$.
For every $N\geq 1$,
\begin{align*}
&\ind_{S_{\varepsilon}^{\rm CR}(\zeta_M)}
\\&
=
\sum_{n=1}^{N} \dfrac{(-1)^{n-1}}{n!}
\hat{\mu}^{\rm loop}_{\text{SLE}_{8/3}}
\Big(
\Big(\frac{1}{\pi}\vert \log (\varepsilon d(0,\Gamma)/\CR(0,\inter{\Gamma}))\vert + \cst (M)\Big)^{-n}
\Big)
:\Theta_{\zeta_M}^{n}:
\, + o(\vert \log\varepsilon\vert^{-N})
\\&
=
\sum_{n=1}^{N}  \dfrac{(-1)^{n-1}}{n!}
\dfrac{1}{\big(\frac{1}{\pi} \vert \log \varepsilon\vert\big)^{n}}
\sum_{k=1}^{n} 
\dfrac{n!(n-1)!}{(n-k)! k! (k-1)!} C^{\text{SLE}_{8/3}}_{n-k,M}\,
:\Theta_{\zeta_M}^{k}:
\, + o(\vert \log\varepsilon\vert^{-N}).
\end{align*}
\end{conj}

\subsection{The FPS expansion from the Wiener sausage expansion}
\label{Subsec Wiener to FPS}

Given that a first passage set can be represented as the closure of a cluster of
Brownian loops and Brownian boundary-to-boundary excursions in $D$,
a natural question is how to deduce the asymptotic expansion for the FPS from
Le Gall's asymptotic expansion for the Wiener sausage in dimension $2$.
We shall see that there are difficulties in this regard.
For this, one has to chose, either work all along with the Euclidean distance
or work all along with the conformal radius, but not mix the two
(check Conjectures \ref{Conj expansion FPS Eucl} and \ref{Conj Wiener sausage CR}).
The Euclidean distance is maybe simpler, because the Euclidean $\varepsilon$-neighborhood of a cluster
is the union of Euclidean $\varepsilon$-neighborhoods
of individual Brownian trajectories composing the cluster.
For the conformal radius the analogue is not true,
as the conformal radius is a more global variable.

So let $\LP$ be the random countable collection of Brownian-like trajectories composing the FPS $A$:
\begin{displaymath}
A = \overline{\bigcup_{\wp\in\LP}\operatorname{Range}(\wp)}.
\end{displaymath}
For $\delta>0$, denote
\begin{displaymath}
\LP_{\delta} = \{\wp\in \LP\vert \diam(\wp)<\delta\}.
\end{displaymath}
The subset $\LP_{\delta}$ is a.s. finite.
Denote by $S_{\varepsilon}(\LP_{\delta})$
the Euclidean $\varepsilon$-neighborhood of $\LP_{\delta}$.
Let $\Theta_{\LP_{\delta}}$ be the occupation measure of
$\LP_{\delta}$:
\begin{displaymath}
\Theta_{\LP_{\delta}} = 
\sum_{\wp\in \LP_{\delta}}
\Theta_{\wp}.
\end{displaymath}
This is an a.s. finite measure.
Let $:\Theta_{\LP_{\delta}}^{n}:$ be the renormalized powers of $\Theta_{\LP_{\delta}}$.
Several comments here.
\begin{itemize}
\item We regularize $\Theta_{\LP_{\delta}}$ by circle averages.
\item We apply to regularized $\Theta_{\LP_{\delta}}$
the generalized Laguerre polynomials $\Lambda_{n}(x,u)$
\eqref{Eq Lambda}.
\item A question arises: what kind of Green's function to use in the renormalization? 
There is choice, but the outcome $:\Theta_{\LP_{\delta}}^{n}:$ will depend on the choice.
The relation between different ways to renormalize $\Theta_{\LP_{\delta}}^{n}$
is given by the change of normalization indentity for the polynomials $\Lambda_{n}(x,u)$
\eqref{Eq change norm Laguerre}.
A natural choice in the context here is the massless Green's function on
$D$ with $0$ boundary condition rather than the massive Green's function on $\C$
used by Le Gall.
\item In the context of Brownian motion representations of the continuum GFF,
the normalization conventions for the GFF and that for the Brownian motion have to match.
With our convention for the GFF, one would rather use not the standard Brownian motion
(generator $\Delta/2$) but the Brownian motion with generator $\Delta$.
This is just a detail, but this impacts the coefficients in the asymptotic expansion of
Wiener sausages.
\end{itemize}
So, formally speaking, our renormalized powers $:\Theta_{\LP_{\delta}}^{n}:$
correspond to
\begin{displaymath}
\Lambda_{n}\big(\Theta_{\LP_{\delta}},G_{D}(z,z)\big),
\end{displaymath}
where $G_{D}$ is the same Green's function as in the GFF case;
see Section \ref{Subsec Wick}.
Now, the constant order term in $G_{D}(z,z)$ after the logarithmic singularity
is no longer $\cst(M)$ as in \eqref{Eq G M sing}, but
\begin{displaymath}
\dfrac{1}{2\pi}\log \CR(z,D).
\end{displaymath}
This affects the form of the asymptotic expansion for Wiener sausages,
but the two forms are related through the change of normalization formula \eqref{Eq change norm Laguerre} and the reexpansion identity \eqref{Eq Laguerre reexp}.
So, for \textbf{fixed} $\delta>0$, the expansion in $\varepsilon$ for
$\ind_{S_{\varepsilon}(\LP_{\delta})}$ is as follows.

\begin{claim}
\label{Claim Le Gall soup delta}
For simplicity, assume that the condition \eqref{Eq cond small boundary} holds,
that is to say the area of an $\varepsilon$-neighborhood of the boundary $\partial D$ is asymptotically smaller than
any power of $1/\vert\log\varepsilon\vert$.
Fix $\delta\in (0,\diam(D))$ and $\eta>0$.
The following asymptotic expansion holds in
$L^{2}(d\PP,\sigma(\LP),H^{-\eta}(\C))$.
For every $N\geq 1$,
\begin{equation}
\label{Eq exp clus delta}
\ind_{S_{\varepsilon}(\LP_{\delta})}
=
\sum_{n=1}^{N}
\dfrac{(-1)^{n-1}}{\big(\frac{1}{2\pi} \vert \log \varepsilon\vert\big)^{n}}
\sum_{k=1}^{n} 
\dfrac{(n-1)!}{(n-k)! k! (k-1)!} \Big(\dfrac{1}{2\pi}\log\CR(z,D)\Big)^{n-k}\,
:\Theta_{\LP_{\delta}}^{k}:
\, + o(\vert \log\varepsilon\vert^{-N}).
\end{equation}
\end{claim}

Note that Claim \ref{Eq exp clus delta} is not an immediate consequence of
Le Gall's result (Theorem \ref{Thm Le Gall}).
Indeed, one deals with Brownian excursions and Brownian bridges instead of an unconditioned Brownian trajectory.
Although a.s. finite, the number of trajectories in $\LP_{\delta}$ is still random.
Moreover, the trajectories are correlated because of the condition to belong to the same cluster.
But given this caveat, the author is pretty much confident that
\eqref{Eq exp clus delta} is the correct expansion.

Now let us compare the expansion \eqref{Eq exp clus delta} to
Conjecture \ref{Conj expansion FPS Eucl},
more precisely to the expansion for $\widetilde{\Ns}_{\varepsilon}^{\rm Eucl}(A)$ there.
Since our $\varepsilon$-neighborhoods are defined through Euclidean distance rather than conformal radius,
this is the right point of comparison.
Then one remarks several important discrepancies.
\begin{itemize}
\item The most important one,
the exponents of $1/\vert\log\varepsilon\vert$ that appear are not the same.
For $S_{\varepsilon}(\LP_{\delta})$ one has integer exponents
$\N\setminus\{0\}$.
For $\widetilde{\Ns}_{\varepsilon}^{\rm Eucl}(A)$ one has half-integer exponents
$\N\setminus\{0\} -1/2$.
So there is a translation of exponents by $-1/2$.
\item Further, the representation theorems (a.k.a. isomorphism theorems)
relate the renormalized intersection local times to even Wick powers of the GFF.
By contrast, in Conjecture \ref{Conj expansion FPS Eucl}
(as well as in Theorem \ref{Thm A E FPS}) appear the (restrictions of) odd Wick powers.
\item The combinatorial coefficients do no match,
although, given that the expectations w.r.t. the MKS measure $\hat{\mu}^{\rm loop}_{\text{SLE}_{4}}$
are black boxes, one does not really know.
\end{itemize}
A good sign though is that powers of $\log\CR(z,D)$ appear in both cases.

Further, in the right-hand side of \eqref{Eq exp clus delta}
one can fix $\varepsilon$ and send $\delta\to 0$,
and see what happens.
What should happen is that for every $n\geq 1$, the field-coefficient corresponding to
$\big(\frac{1}{2\pi} \vert \log \varepsilon\vert\big)^{-n}$,
that is
\begin{displaymath}
\sum_{k=1}^{n} 
\dfrac{(n-1)!}{(n-k)! k! (k-1)!} \Big(\dfrac{1}{2\pi}\log\CR(z,D)\Big)^{n-k}\,
:\Theta_{\LP_{\delta}}^{k}:\, ,
\end{displaymath}
diverges as $\delta\to 0$.
This is easy to verify for the leading coefficient
$\Theta_{\LP_{\delta}}$ ($n=1$).
For $\delta>0$ this is a finite measure,
but the occupation measure of the whole cluster, that is to say the total time of paths composing it,
is a.s. infinite because of the accumulation of small loops.
In general, it has been observed that the renormalization for Brownian loop soups requires an additional layer of renormalization compared to that for a finite number of Brownian trajectories,
so as to remove the ultraviolet divergence induced by the accumulation of small Brownian loops.
See this in \cite{LeJanMarcusRosen17PermWick} in the context of renormalized intersection local times
and in \cite{ABJLMultChaos} in the context of the multiplicative chaos.

So is there a way to connect the expansion for first passage sets to Le Gall's expansion for the Wiener sausage?
One possible approach is as follows.
First one has to take $\delta$ depending on  $\varepsilon$
(i.e. $\delta = \delta(\varepsilon)$)
in an appropriate way to be determined.
Then one writes
\begin{displaymath}
\dfrac{1}{\vert \log\varepsilon\vert^{n}} = 
\dfrac{\vert \log\varepsilon\vert^{-1/2}}{\vert \log\varepsilon\vert^{n-1/2}}.
\end{displaymath}
The $\vert \log\varepsilon\vert^{-1/2}$ in the numerator will serve to tame the divergence of the fields.
In the case $n=1$, one can expect the renormalized positive measure
$\vert \log\varepsilon\vert^{-1/2} \Theta_{\LP_{\delta(\varepsilon)}}$
to converge to a constant times the FPS measure $\nu_{A}$.
For higher order terms $(n\geq 2)$, one most likely needs to rearrange the terms by adding linear combinations of lower order terms so as to get something converging.
The $\vert \log\varepsilon\vert^{-1/2}$ in the numerator may also explain why at the end of the day we get the odd powers of the GFF instead of the even powers.
Indeed, morally speaking, dividing by  $\vert \log\varepsilon\vert^{1/2}$ is like dividing by
$\vert\Phi\vert$, or a regularized version of it at scale $\varepsilon$.

\subsection{Asymptotic expansion of Brownian loop soup clusters for central charge $c\in (0,1)$}
\label{Subsec other c}

The GFF is related via the representation theorems to a Brownian loop soup (BLS) of one specific intensity parameter.
In the literature there are two-different conventions for the parametrization of the intensity of the BLS.
According to the convention due to Lawler-Werner \cite{LawlerWerner2004ConformalLoopSoup}
and Sheffield-Werner \cite{SheffieldWerner2012CLE},
the intensity corresponding to the GFF is $c=1$.
According to the convention due to Le Jan \cite{LeJan2011Loops},
the intensity corresponding to the GFF is $\alpha =1/2$.
The general relation between $c$ and $\alpha$ is 
$c=2\alpha$.
Here we will use the Lawler-Werner and Sheffield-Werner convention
with $c$, since the parameter $c$ also corresponds to a central charge in
conformal field theory, and $c=1$ is the central charge of the GFF.

Here we will discuss the case of BLS of arbitrary intensity parameter $c$.
So let $\LL^{c}_{D}$ denote such a BLS in $D$. It is a Poisson point process of Brownian loops
in $D$ (Brownian bridges from a root to itself conditioned to stay in $D$).
$\LL^{c}_{D}$ contains countably infinitely many Brownian loops.
The family $\LL^{c}_{D}$ also satisfies the conformal invariance (in law) property,
up to time change and rerooting of loops.

Le Jan \cite{LeJan2011Loops} introduced for general $c>0$
the renormalized intersection local times of $\LL^{c}_{D}$,
which we will denote here by
$:\Theta_{\LL^{c}_{D}}^{n}:$ \,.
See also \cite{LeJanMarcusRosen17PermWick}.
Note that already for $n=1$,
that is say at the level of occupation measure,
the field $:\Theta_{\LL^{c}_{D}}:$
is obtained by renormalization and is no longer a positive measure.
Indeed, the positive measure $\Theta_{\LL^{c}_{D}}$ diverges in every open subset of $D$,
and one has to substract its diverging expectation to obtain a converging field
$:\Theta_{\LL^{c}_{D}}:$\, .
For general $n\geq 1$, the renormalization $:\Theta_{\LL^{c}_{D}}^{n}:$
is achieved in \cite{LeJan2011Loops} by using
scale-dependent perturbations of generalized Laguerre polynomials of order $c/2-1$.
For $c=1$, the joint distribution of
$(:\Theta_{\LL^{1}_{D}}^{n}:)_{n\geq 1}$
is the same as that of even Wick powers
$(2^{-n}\,:\Phi^{2n}:)_{n\geq 1}$,
where $\Phi$ is a GFF with $0$ boundary conditions.
This is part of the Brownian representations of the GFF.
For integer values of $c$, $c\in\N\setminus\{0\}$,
the BLS $\LL^{c}_{D}$ is related to the vector-valued GFF with $c$ scalar components,
and $(:\Theta_{\LL^{c}_{D}}^{n}:)_{n\geq 1}$ are distributed as renormalized powers of
the norm-squared of the vector-valued GFF.
Given these relations, we see, for general $c>0$,
the field $:\Theta_{\LL^{c}_{D}}^{n}:$ as an analogue of the even Wick power
$:\Phi^{2n}:$\,.

In \cite{SheffieldWerner2012CLE},
Sheffield and Werner studied the clusters formed by Brownian loops in $\LL^{c}_{D}$.
They showed that there is a phase transition at $c=1$.
For $c>1$, there is only one cluster that is everywhere dense in $D$.
But for $c\in (0,1]$, there are infinitely many different clusters.
Moreover, in this phase, the outer boundaries of outermost clusters are distributed as a
conformal loop ensemble CLE$_{\kappa}$ with
\begin{equation}
\label{Eq rel kappa c}
c = \dfrac{(3\kappa-8)(6-\kappa)}{2\kappa} .
\end{equation}
$c=1$ corresponds to $\kappa=4$ and the limit $c\to 0$ to $\kappa=8/3$.

In \cite{JegoLupuQianFields},
Jego, Lupu and Qian constructed analogues of the GFF $\Phi$ in the subcritical phase
$c\in (0,1)$.
Here we will denote these fields by $\fld_{c}$.
The construction considers considers the clusters formed by Brownian loops in $\LL^{c}_{D}$,
and adds an additional randomness given i.i.d. uniform signs
$\sigma(\clus)\in\{-1,1\}$ per cluster $\clus$ of $\LL^{c}_{D}$.
The restriction of the field $\fld_{c}$ to
$\overline{\clus}$ (topological closure of a cluster $\clus$)
is a finite positive (if $\sigma(\clus)=1$) or negative (if $\sigma(\clus)= -1$)
measure corresponding to a Minkowski content of $\clus$
in the gauge $\vert\log r\vert^{1-c/2 + o(1)} r^{2}$.
It is believed that the $o(1)$ in the exponents actually does not exist, but the proof does not provide such a precision.
Moreover, a signed decomposition of $\fld_{c}$ analogous to \eqref{Eq exc GFF} holds.
For $c=1$, $\fld_{c=1}$ is just the GFF $\Phi$,
but for $c\in (0,1)$, the $\fld_{c}$ are completely new fields
that have never appeared previously in the literature.

Now recall that since Le Jan, we already had analogues of the even Wick powers $:\Phi^{2n}:$,
given by $:\Theta_{\LL^{c}_{D}}^{n}:$\, .
The field $\fld_{c}$ should be though of as
$\sigma\Theta_{\LL^{c}_{D}}^{1-c/2}$,
that is to say a signed fractional power $1-c/2$ of the diverging occupation measure $\Theta_{\LL^{c}_{D}}$.
A further natural question is what would the the analogues of higher odd Wick powers
$:\Phi^{2n+1}:$\, , $n\geq 1$, of the GFF $\Phi$.
To this end, the authors in \cite[Conjecture 9.11]{JegoLupuQianFields}
conjecture the existence of renormalized signed fractional powers
$:\fld_{c} \Theta_{\LL^{c}_{D}}^{n}: \, = \, :\sigma\Theta_{\LL^{c}_{D}}^{n+1-c/2}:$\, $n\geq 1$.
The conjecture is based on a duality relation between the generalized Laguerre polynomials of order
$c/2-1$, $L_{n}^{(c/2-1)}$,
and the generalized Laguerre polynomials of order $1-c/2$, $L_{n}^{(1-c/2)}$.
Note that for $c=1$, the $L_{n}^{(-1/2)}$ are related to even Hermite polynomials
($\He_{2n}(x) = L_{n}^{(-1/2)}(x^{2})$),
and $L_{n}^{(1/2)}$ are related to odd Hermite polynomials
($\He_{2n+1}(x) = x L_{n}^{(1/2)}(x^{2})$).
So, to summarize,
Le Jan constructed in \cite{LeJan2011Loops}
the renormalized integer powers
$n\in\N\setminus/\{0\}$ of the diverging occupation measure $\Theta_{\LL^{c}_{D}}$.
Jego, Lupu and Qian constructed in \cite{JegoLupuQianFields}
the signed fractional power $1-c/2$ of $\Theta_{\LL^{c}_{D}}$,
and further conjectured the existence of 
renormalized signed fractional powers $(n-c/2)_{n\geq 2}$.

These fractional powers are completely new objects that have never been considered previously in the literature.
Unlike the integer powers $:\Theta_{\LL^{c}_{D}}^{n}:$\,
they are supposed to capture the organization of the BLS $\LL^{c}_{D}$
into clusters. 
Here, we will conjecture that the clusters of $\LL^{c}_{D}$
admit asymptotic expansions similar to Le Gall's expansion for the Wiener sausage (Theorem \ref{Thm Le Gall})
and to our expansions for first passage sets and excursion clusters of the GFF
(Theorems \ref{Thm A E FPS} and \ref{Thm A E C j}).
Moreover, the fields appearing in the expansion should be precisely the restrictions
of these fractional powers $:\fld_{c} \Theta_{\LL^{c}_{D}}^{n}:$\, .
Note that the dominant behavior of the $\varepsilon$-neighborhood of a cluster $\clus$ of $\LL^{c}_{D}$
should be, up to a constant,
\begin{displaymath}
\dfrac{\sigma(\clus) \fld_{c \vert \clus}}{\vert\log\varepsilon\vert^{1-c/2}}.
\end{displaymath}
This is proved in \cite{JegoLupuQianFields},
but only with a $\vert\log\varepsilon\vert^{-(1-c/2) + o(1)}$.

\begin{conj}
\label{Conj expansion c}
Fix $c\in (0,1)$.
The hypothetical fractional powers $:\fld_{c} \Theta_{\LL^{c}_{D}}^{n-1}:$ admit restriction to clusters of 
$\LL^{c}_{D}$.
That is to say, given a cluster $\clus$ of $\LL^{c}_{D}$, the field
\begin{displaymath}
\ind_{d(z,\clus)<\delta}\, :\fld_{c} \Theta_{\LL^{c}_{D}}^{n-1}:
\end{displaymath}
converges as $\delta\to 0$, in an appropriate sense (conditional $L^2$), towards a well defined generalized function
supported of $\overline{\clus}$,
which we will denote by $:\fld_{c} \Theta_{\LL^{c}_{D}}^{n-1}:_{\vert \overline{\clus}}$.
By contrast, the renormalized intersection local times $:\Theta_{\LL^{c}_{D}}^{n}:$
do not admit restrictions to clusters.
More precisely the integral
\begin{displaymath}
\int_{d(z,\clus)<\delta}:\Theta_{\LL^{c}_{D}}^{n}:
\end{displaymath}
converges in probability,
as $\delta\to 0$, to $\sigma(\clus)(-1)^{n-1}\infty$.

The $\varepsilon$-neighborhood of a cluster $\clus$ of $\LL^{c}_{D}$,
both for the Euclidean distance and conformal radius,
admits an asymptotic expansion in the $L^2$ sense, as $\varepsilon\to 0$,
into the powers $(\vert \log\varepsilon\vert^{-(n-c/2)})_{n\geq 1}$
of $\vert \log\varepsilon\vert$.
The coefficient of $\vert \log\varepsilon\vert^{-(n-c/2)}$
is a linear combination of fields
$\sigma(\clus):\fld_{c} \Theta_{\LL^{c}_{D}}^{k-1}:_{\vert \overline{\clus}}$
for $k\in\{1,\dots, n\}$.
By contrast, the renormalized intersection local times $:\Theta_{\LL^{c}_{D}}^{n}:$
do not appear in the expansion.
The expansion for the Euclidean distance and that for the conformal radius are related
through the expectations
\begin{displaymath}
\hat{\mu}^{\rm loop}_{\text{SLE}_{\kappa}}
\Big(
\log\Big(
\dfrac{\CR(0,\inter(\Gamma))}{d(0,\Gamma)}
\Big)^{k}
\Big)
\end{displaymath}
w.r.t. the MKS measure $\hat{\mu}^{\rm loop}_{\text{SLE}_{\kappa}}$,
where $\kappa$ and $c$ are related by \eqref{Eq rel kappa c}.
\end{conj}

\subsection{The relation to umbral calculus}
\label{Subsec deep umbral}

The last point we would like to discuss is more open and ``philosophical''.
It concerns the relation between the renormalization in dimension 2 and the umbral calculus \cite{RotaAll73Umbral,Roman84Umbral}.
The later arose out of a systematic study of remarkable identities satisfied by special polynomial sequences,
mostly orthogonal polynomials.
A key notion there is the umbral composition and the group structure on polynomial sequences.
See Section \ref{Subsec umbral}.
In this work we have seen that the change of normalization identities for Hermite polynomials \eqref{Eq change var}
and for generalized Laguerre polynomials of order $-1$ \eqref{Eq change norm Laguerre},
are naturally interpreted in terms of the umbral composition
and they correspond to one-parameter subgroups of the group of polynomial sequences.
Moreover, they also play an important role in the context of renormalization
as they describe how the Wick powers and the renormalized self-intersection local times transform if one changes conventions for the Green's function,
in particular if one changes the constant order term after the logarithmic singularity on the diagonal.
Moreover, these umbral identities \eqref{Eq change var} and \eqref{Eq change norm Laguerre}
are closely related to the asymptotic expansions both in the GFF case (Theorem \ref{Thm A E FPS})
and the Brownian case (Theorem \ref{Thm Le Gall} of Le Gall),
via the reexpansion identities \eqref{Eq reexp w x y} and \eqref{Eq Laguerre reexp}.
See the discussion of Section \ref{Sec algeb}.

Actually, change of normalization identities analogous to \eqref{Eq change norm Laguerre}
hold for generalized Laguerre polynomials of any order,
and all correspond to one-parameter subgroups of the group of polynomial sequences
(one subgroup per order).
In particular, this describes how the renormalized intersection local times of Brownian loop soups transform under a change of Green's function (of the constant order term after the logarithmic singularity).
In the context of Conjecture \ref{Conj expansion c},
we expect reexpansion identities similar to \eqref{Eq reexp w x y} and \eqref{Eq Laguerre reexp},
but with generalized Laguerre polynomials of other orders.

So the ``philosophical'' question is as follows.
What is the ``deep reason'' for the relation between renormalization and umbral calculus?
How far does it go?
Why the few examples we have of renormalized powers of fields are all related to one-parameter subgroups of
the group of polynomial sequences for the umbral composition?

\bibliographystyle{alpha}
\bibliography{titusbib}

\begin{thebibliography}{HvNVW16}

\bibitem[ABJL23]{ABJLMultChaos}
Elie Aïdékon, Nathanaël Berestycki, Antoine Jego, and Titus Lupu.
\newblock {Multiplicative chaos of the Brownian loop soup}.
\newblock {\em Proceedings of the London Mathematical Society}, 126:1254--1393,
  2023.

\bibitem[AF03]{AdamsFournierSobo}
Robert~A. Adams and John~J.F. Fournier.
\newblock {\em Sobolev spaces}, volume 140 of {\em Pure and Applied
  Mathematics}.
\newblock Academic press, 2nd edition, 2003.

\bibitem[Ahl10]{Ahlfors2010ConfInv}
Lars~Valerian Ahlfors.
\newblock {\em Conformal invariants: topics in geometric function theory},
  volume 371.
\newblock American Mathematical Society, 2010.

\bibitem[ALS20a]{ALS1}
Juhan Aru, Titus Lupu, and Avelio Sep\'{u}lveda.
\newblock First passage sets of the 2{D} {G}aussian free field.
\newblock {\em Probability Theory and Related Fields}, 176:1303–1355, 2020.

\bibitem[ALS20b]{ALS2}
Juhan Aru, Titus Lupu, and Avelio Sep\'{u}lveda.
\newblock The first passage sets of the 2{D} {G}aussian free field: convergence
  and isomorphism.
\newblock {\em Communications in Mathematical Physics}, 375:1885--1929, 2020.

\bibitem[ALS22]{ALS3}
Juhan Aru, Titus Lupu, and Avelio Sep{\'u}lveda.
\newblock {Extremal distance and conformal radius of a CLE$_4$ loop}.
\newblock {\em The Annals of Probability}, 50(2):509--558, 2022.

\bibitem[ALS23]{ALS4}
Juhan Aru, Titus Lupu, and Avelio Sep\'ulveda.
\newblock Excursion decomposition of the {2D} continuum {GFF}.
\newblock arXiv:2304.03150, 2023.

\bibitem[AM01]{AiraultMalliavin01Unitarizing}
Hélène Airault and Paul Malliavin.
\newblock {Unitarizing probability measures for representations of Virasoro
  algebra}.
\newblock {\em Journal de Mathématiques Pures et Appliquées}, 80(6):627--667,
  2001.

\bibitem[AMS63]{AronszajnMullaSzeptycki63Bessel}
Nachman Aronszajn, Fuad Mulla, and Pawel Szeptycki.
\newblock {On spaces of potentials connected with $L^p$ classes}.
\newblock {\em Annales de l’Institut Fourier}, 13(2):211--306, 1963.

\bibitem[APP23]{AruPaponPowell23dust}
Juhan Aru, Léonie Papon, and Ellen Powell.
\newblock {Thick points of the planar GFF are totally disconnected for all
  $\gamma\neq 0^{\ast}$}.
\newblock {\em Electronic Journal of Probability}, 28(85):1--24, 2023.

\bibitem[APS20]{APS}
Juhan Aru, Ellen Powell, and Avelio Sep{\'u}lveda.
\newblock Liouville measure as a multiplicative cascade via level sets of the
  {G}aussian free field.
\newblock {\em Annales de l'Institut Fourier}, 70(1):205--245, 2020.

\bibitem[Aru20]{Aru20GMCreview}
Juhan Aru.
\newblock {Gaussian multiplicative chaos through the lens of the 2D Gaussian
  free field}.
\newblock {\em Markov Processes and Related Fields}, 26(1):17--56, 2020.

\bibitem[AS61]{AronszajnSmith61Bessel1}
Nachman Aronszajn and Kennan~T. Smith.
\newblock {Theory of Bessel potentials. I}.
\newblock {\em Annales de l’Institut Fourier}, 11:385--475, 1961.

\bibitem[AS84]{AbramowitzStegun84}
Milton Abramowitz and Irene~A. Stegun, editors.
\newblock {\em Handbook of mathematical functions with formulas, graphs, and
  mathematical tables}.
\newblock A Wiley-Interscience Publication. John Wiley \& Sons, Inc., New York;
  John Wiley \& Sons, Inc., New York, 1984.
\newblock Reprint of the 1972 edition, Selected Government Publications.

\bibitem[AS18]{AruSepulveda18TVS}
Juhan Aru and Avelio Sep{\'u}lveda.
\newblock {Two-valued local sets of the 2D continuum Gaussian free field:
  connectivity, labels, and induced metrics}.
\newblock {\em Electronic Journal of Probability}, 23:1--35, 2018.

\bibitem[ASW19]{ASW}
Juhan Aru, Avelio Sep{\'u}lveda, and Wendelin Werner.
\newblock On bounded-type thin local sets of the two-dimensional {G}aussian
  free field.
\newblock {\em Journal of the Institute of Mathematics of Jussieu},
  18(3):591--618, 2019.

\bibitem[Bef08]{BeffaraDimSLE}
Vincent Beffara.
\newblock {The dimension of the SLE curves}.
\newblock {\em The Annals of Probability}, 36(4):1421--1452, 2008.

\bibitem[Ber17]{BerestyckiGMC}
Nathana\"{e}l Berestycki.
\newblock An elementary approach to {G}aussian multiplicative chaos.
\newblock {\em Electronic Communications in Probability}, 22(27):1--12, 2017.

\bibitem[BJ24]{BaverezJego24MKS}
Guillaume Baverez and Antoine Jego.
\newblock {The CFT of SLE loop measures and the Kontsevich--Suhov conjecture}.
\newblock arXiv:2407.09080, 2024.

\bibitem[BS15]{BorodinSalminen2015}
Andrei Borodin and Paavo Salminen.
\newblock {\em Handbook of {B}rownian {M}otion - {F}acts and {F}ormulae}.
\newblock Probability and {I}ts {A}pplications. Birkhäuser, 2nd, corrected
  edition, 2015.

\bibitem[DS11]{DuplantierSheffield}
Bertrand Duplantier and Scott Sheffield.
\newblock Liouville quantum gravity and {KPZ}.
\newblock {\em Inventiones Mathematicae}, 185(2):333--393, 2011.

\bibitem[Dub09]{Dubedat09PartitionFunc}
Julien Dub\'edat.
\newblock {SLE and the free field: Partition functions and couplings}.
\newblock {\em Journal of the American Mathematical Society}, 22(4):995--1054,
  2009.

\bibitem[Dyn84]{Dynkin1984Polynomials}
Evgeniy Dynkin.
\newblock {Polynomials of the occupation field and related random fields}.
\newblock {\em Journal of Functional Analysis}, 58:20--52, 1984.

\bibitem[HK71]{HoeghKrohn71}
Raphael Høegh-Krohn.
\newblock A general class of quantum fields without cut-offs in two space-time
  dimensions.
\newblock {\em Communications in Mathematical Physics}, 21:244--255, 1971.

\bibitem[HvNVW16]{HNVW16AnBanach1}
Tuomas Hytönen, Jan van Neerven, Mark Veraar, and Lutz Weis.
\newblock {\em Analysis in Banach Spaces. Volume I: Martingales and
  Littlewood-Paley Theory}, volume~63 of {\em A Series of Modern Surveys in
  Mathematics}.
\newblock Springer, 2016.

\bibitem[IM74]{ItoMcKean74Diffusions}
Kiyosi Itô and Henry~P. McKean.
\newblock {\em Diffusion processes and their sample paths}.
\newblock Classics in mathematics. Springer, 1974.

\bibitem[Jan97]{Janson1997GaussHilbSpaces}
Svante Janson.
\newblock {\em Gaussian Hilbert Spaces}, volume 129 of {\em Cambridge Tracts in
  Mathematics}.
\newblock Cambridge University Press, 1997.

\bibitem[Jeg20]{jegoBMC}
Antoine Jego.
\newblock Planar {B}rownian motion and {G}aussian multiplicative chaos.
\newblock {\em The Annals of Probability}, 48(4):1597--1643, 2020.

\bibitem[JLQ23]{JegoLupuQianFields}
Antoine Jego, Titus Lupu, and Wei Qian.
\newblock {Conformally invariant fields out of Brownian loop soups}.
\newblock arXiv:2307.10740, 2023.

\bibitem[Kah85]{Kahane85GMC}
Jean-Pierre Kahane.
\newblock Sur le chaos multiplicatif.
\newblock {\em Annales des Sciences Mathématiques du Québec}, 9(2):105--150,
  1985.

\bibitem[KM13]{KangMakarovGFFCFT}
Nam-Guy Kang and Nikolai~G. Makarov.
\newblock {\em Gaussian free ﬁeld and conformal ﬁeld theory}, volume 353 of
  {\em Astérisque}.
\newblock Soci{\'e}t{\'e} Math{\'e}matique de France, 2013.

\bibitem[KS07]{KontsevichSuhov07MKS}
Maxim Kontsevich and Yurii Suhov.
\newblock {On Malliavin measures, SLE, and CFT}.
\newblock {\em Proceedings of the Steklov Institute of Mathematics},
  258:107--153, 2007.

\bibitem[KW16]{KemppainenWerner16CLE}
Antti Kemppainen and Wendelin Werner.
\newblock {The nested simple conformal loop ensembles in the Riemann sphere}.
\newblock {\em Probability Theory and Related Fields}, 165(3-4):835–866,
  2016.

\bibitem[Law05]{LawlerConformallyInvariantProcesses}
Gregory~F. Lawler.
\newblock {\em Conformally invariant processes in the plane}, volume 114 of
  {\em Mathematical Surveys and Monographs}.
\newblock American Mathematical Society, 2005.

\bibitem[LG90]{LeGallLocTime}
Jean-François Le~Gall.
\newblock Wiener sausage and self-intersection local times.
\newblock {\em Journal of Functional Analysis}, 88(2):299--341, 1990.

\bibitem[LG92]{LeGallStFlour}
Jean-François Le~Gall.
\newblock {\em \'Ecole d'\'Eté de Probabilités de Saint-Flour XX - 1990},
  volume 1527 of {\em Lecture Notes in Mathematics}, chapter {Some properties
  of planar Brownian motion}.
\newblock Springer, 1992.

\bibitem[LJ11]{LeJan2011Loops}
Yves Le~Jan.
\newblock {\em Markov paths, loops and fields. École d'Été de Probabilités
  de Saint-Flour XXXVIII – 2008}, volume 2026 of {\em Lecture Notes in
  Mathematics}.
\newblock Springer, 2011.

\bibitem[LJMR17]{LeJanMarcusRosen17PermWick}
Yves Le~Jan, Michael~B. Marcus, and Jay Rosen.
\newblock {\em {Intersection local times, loop soups and permanental Wick
  powers}}, volume 247 of {\em Memoirs of the American Mathematical Society}.
\newblock American Mathematical Society, 2017.

\bibitem[LW04]{LawlerWerner2004ConformalLoopSoup}
Gregory~F. Lawler and Wendelin Werner.
\newblock The {B}rownian loop soup.
\newblock {\em Probability Theory and Related Fields}, 128:565--588, 2004.

\bibitem[MS16]{MS1}
Jason Miller and Scott Sheffield.
\newblock Imaginary geometry {I}: interacting {SLE}s.
\newblock {\em Probability Theory and Related Fields}, 164(3-4):553--705, 2016.

\bibitem[Oks83]{Oksendal83Beurling}
Bernt Oksendal.
\newblock Projection estimates for harmonic measure.
\newblock {\em Arkiv för Matematik}, 21(1-2):191--203, 1983.

\bibitem[QW19]{QianWerner19Clusters}
Wei Qian and Wendelin Werner.
\newblock Decomposition of {B}rownian loop-soup clusters.
\newblock {\em Journal of the European Mathematical Society},
  21(10):3225--3253, 2019.

\bibitem[RKO73]{RotaAll73Umbral}
Gian-Carlo Rota, David Kahaner, and Andrew Odlyzko.
\newblock {On the Foundations of Combinatorial Theory. VIII. Finite Operator
  Calculus}.
\newblock {\em Journal of Mathematical Analysis and Application}, 42:684--760,
  1973.

\bibitem[Rom84]{Roman84Umbral}
Steven Roman.
\newblock {\em The Umbral Calculus}, volume 111 of {\em Dover Books on
  Mathematics}.
\newblock Academic Press, Inc., 1st edition, 1984.

\bibitem[RV14]{RhodesVargas14GMCReview}
R\'{e}mi Rhodes and Vincent Vargas.
\newblock Gaussian multiplicative chaos and applications: a review.
\newblock {\em Probability Surveys}, 11:315--392, 2014.

\bibitem[Sep19]{Sepulveda19ThinLocalSet}
Avelio Sepúlveda.
\newblock On thin local sets of the {G}aussian free field.
\newblock {\em Annales de l'Institut Henri Poincaré, Probabilités et
  Statistiques}, 55(3):1797--1813, 2019.

\bibitem[She07]{Sheffield07GFF}
Scott Sheffield.
\newblock Gaussian free fields for mathematicians.
\newblock {\em Probability Theory and Related Fields}, 139(3):521--541, 2007.

\bibitem[Sim74]{Simon74EQFT}
Barry Simon.
\newblock {\em The {$P(\phi )\sb{2}$} {E}uclidean (quantum) field theory}.
\newblock Princeton Series in Physics. Princeton University Press, 1974.

\bibitem[SS09]{SchSh}
Oded Schramm and Scott Sheffield.
\newblock Contour lines of the two-dimensional discrete {G}aussian free field.
\newblock {\em Acta Mathematica}, 202(1):21--137, 2009.

\bibitem[SS13]{SchSh2}
Oded Schramm and Scott Sheffield.
\newblock A contour line of the continuum {G}aussian free field.
\newblock {\em Probability Theory and Related Fields}, 157(1-2):47--80, 2013.

\bibitem[SSV22]{SchougSepulvedaViklund22TVS}
Lukas Schoug, Avelio Sep{\'u}lveda, and Fredrik Viklund.
\newblock {Dimensions of two-valued sets via imaginary chaos}.
\newblock {\em International Mathematics Research Notices},
  2022(5):3219–3261, 2022.

\bibitem[SW12]{SheffieldWerner2012CLE}
Scott Sheffield and Wendelin Werner.
\newblock Conformal loop ensembles: the {M}arkovian characterization and the
  loop-soup construction.
\newblock {\em Annals of Mathematics}, 176(3):1827--1917, 2012.

\bibitem[Wer08]{Werner08SLE_8_3_loops}
Wendelin Werner.
\newblock The conformally invariant measure on self-avoiding loops.
\newblock {\em Journal of American Mathematical Society}, 21:137--169, 2008.

\bibitem[WP21]{PowWernerGFF}
Wendelin Werner and Ellen Powell.
\newblock {\em Lecture notes on the Gaussian free field}, volume~28 of {\em
  Cours Sp\'ecialis\'es}.
\newblock Soci{\'e}t{\'e} Math{\'e}matique de France, 2021.

\bibitem[WW17]{WaWu17LLGFF}
Menglu Wang and Hao Wu.
\newblock Level lines of {G}aussian free field {I}: zero-boundary {GFF}.
\newblock {\em Stochastic Processes and their Applications}, 127(4):1045--1124,
  2017.

\bibitem[Zha21]{Zhan21SLEloop}
Dapeng Zhan.
\newblock {TSLE loop measures}.
\newblock {\em Probability Theory and Related Fields}, 179(1-2):345–406,
  2021.

\end{thebibliography}

\end{document}